%% file: _PaperFmIid.tex
\documentclass{article}
\usepackage[
	backend=biber,
	datamodel=mydatamodel,
	style=alphabetic,
	natbib=true,
	url=false,
	doi=true,
	eprint=true,
	giveninits=true,
	isbn=false,
	maxbibnames=10,
]{biblatex}
\DeclareFieldFormat{mrnumber}{%
	MR\space
	\ifhyperref
	{\href{http://www.ams.org/mathscinet-getitem?mr=#1}{\nolinkurl{#1}}}
	{\nolinkurl{#1}}}
\DeclareFieldFormat{zbl}{%
	Zbl\space
	\ifhyperref
	{\href{https://zbmath.org/?q=an:#1}{\nolinkurl{#1}}}
	{\nolinkurl{#1}}}
\renewbibmacro*{doi+eprint+url}{%
	\iftoggle{bbx:doi}
	{\printfield{doi}}
	{}%
	\newunit\newblock
	\printfield{mrnumber}%
	\newunit\newblock
	\printfield{zbl}%
	\newunit\newblock
	\iftoggle{bbx:eprint}
	{\usebibmacro{eprint}}
	{}%
	\newunit\newblock
	\iftoggle{bbx:url}
	{\usebibmacro{url+urldate}}
	{}}
\usepackage[a4paper, margin=3cm]{geometry}
\usepackage[colorlinks,linkcolor=blue,citecolor=blue,urlcolor=blue]{hyperref}
\usepackage{authblk}
\usepackage{parskip} 
\input{packages}
\usepackage[capitalize,nameinlink,noabbrev]{cleveref}
\input{macros_general}
\input{macros_specific}
\newbox{\myorcidthanksbox}
\sbox{\myorcidthanksbox}{\large\includegraphics[height=1.8ex]{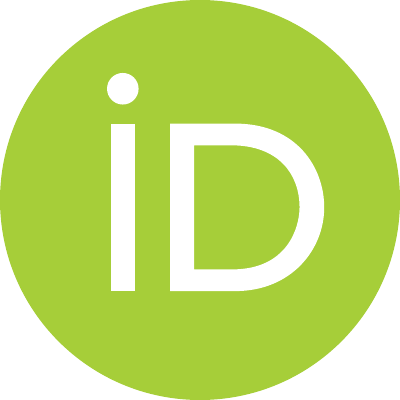}}
\newcommand{\orcidthanks}[1]{%
    \href{https://orcid.org/#1}{\raisebox{-0.5ex}{\usebox{\myorcidthanksbox}}\,#1}}
\usepackage{thmtools}

\input{macros_theorem}
\title{Transformed Fr\'echet Means for Robust Estimation in Hadamard Spaces}
\date{}
\author[1,2]{Christof Sch\"otz\thanks{christof.schoetz@tum.de, \orcidthanks{0000-0003-3528-4544}}}
\affil[1]{Technical University of Munich, Germany; Munich Climate Center; TUM School of Engineering and Design, Department of Aerospace and Geodesy, Earth System Modelling Group}
\affil[2]{Potsdam Institute for Climate Impact Research, Germany; Artificial Intelligence Group}
\begin{document}
\maketitle
\begin{abstract}
	\input{abstract.tex}
\end{abstract}
\tableofcontents
\input{sec_intro2.tex}

\input{sec_prelim.tex}
\input{sec_unbounded.tex}
\input{sec_bounded.tex}
\input{sec_discuss.tex}
\printbibliography
%

\clearpage

\pagenumbering{arabic} 
\setcounter{page}{1}

\vspace*{2cm}
\begin{center}
	{\LARGE \textbf{Supplement to "Transformed Fr\'echet Means for Robust Estimation in Hadamard Spaces"}}\\[1em]  
	{\large Christof Sch\"otz}  
\end{center}
\vspace{2em}

\renewcommand{\theequation}{S\arabic{equation}}
\renewcommand{\thefigure}{S\arabic{figure}}
\renewcommand{\thetable}{S\arabic{table}}
\renewcommand{\thesection}{S\arabic{section}}
\renewcommand{\thesubsection}{\thesection.\arabic{subsection}}
\renewcommand{\thetheorem}{\thesection.\arabic{theorem}}
\renewcommand{\thepage}{S\arabic{page}}

\renewcommand{\theHequation}{S\arabic{equation}}
\renewcommand{\theHfigure}{S\arabic{figure}}
\renewcommand{\theHtable}{S\arabic{table}}
\renewcommand{\theHsection}{S\arabic{section}}
\renewcommand{\theHsubsection}{\thesection.\arabic{subsection}}
\renewcommand{\theHtheorem}{\thesection.\arabic{theorem}}
\def\theHpage{S\arabic{page}}

\setcounter{equation}{0}
\setcounter{figure}{0}
\setcounter{table}{0}
\setcounter{section}{0}
\setcounter{subsection}{0}
\setcounter{theorem}{0}

\input{sec_powerfun.tex}
\input{sec_tran.tex}
\input{sec_loss.tex}
\input{sec_algo.tex}
\input{sec_chernoff.tex}
\input{sec_proof_prelim.tex}
\input{sec_proof_unbounded.tex}
\input{sec_proof_bounded.tex}
\input{sec_proof_opti.tex}
\end{document}

%% file: packages.tex
\usepackage{amsthm,amsmath,amsfonts,amssymb}
\usepackage{mathtools}
\usepackage{stmaryrd}
\usepackage{enumitem} 
\usepackage{graphicx} 
\usepackage{dsfont} 
\usepackage{tikz} 
\usetikzlibrary{calc,decorations.pathmorphing,shapes}
\usetikzlibrary{decorations.pathreplacing}
\usetikzlibrary {patterns,patterns.meta}
\usepackage{pgfplots}
\usepgfplotslibrary{colorbrewer}

%% file: macros_general.tex
\newcommand*{\mc}[1]{\mathcal{#1}}

\newcommand*{\ms}[1]{\mathsf{#1}}
\newcommand*{\mo}[1]{\mathbf{#1}}
\newcommand*{\mf}[1]{\mathfrak{#1}}
\newcommand{\N}{\mathbb{N}}
\newcommand{\Nn}{\mathbb{N}_0}

\newcommand{\Q}{\mathbb{Q}}
\newcommand{\R}{\mathbb{R}}
\newcommand{\Rp}{\R_{\geq0}}
\newcommand{\Rpp}{\R_{>0}}

\newcommand{\ball}{\mathrm{B}}
\newcommand{\lcb}{\left\lbrace} 
\newcommand{\rcb}{\right\rbrace} 
\newcommand{\cb}[1]{\lcb #1 \rcb} 
\newcommand{\lab}{\left[} 
\newcommand{\rab}{\right]} 
\newcommand{\ab}[1]{\lab #1 \rab} 
\newcommand{\abOf}[1]{\!\ab{#1}} 
\newcommand{\lb}{\left(} 
\newcommand{\rb}{\right)} 
\newcommand{\br}[1]{\lb #1 \rb} 
\newcommand{\brOf}[1]{\!\br{#1}} 
\newcommand{\absof}[1]{\vert #1 \vert} 
\newcommand{\absOf}[1]{\left\vert #1 \right\vert} 
\newcommand{\normof}[1]{\Vert#1\Vert}
\newcommand{\normOf}[1]{\left\Vert#1\right\Vert}

\newcommand{\ipof}[2]{\langle #1, #2\rangle}
\newcommand{\ipOf}[2]{\left\langle#1\,,\, #2\right\rangle}
\renewcommand{\Pr}{\mathbf{P}} 
\newcommand{\Prof}[1]{\Pr(#1)}
\newcommand{\PrOf}[1]{\Pr\brOf{#1}}
\newcommand{\E}{\mathbf{E}} 
\newcommand{\Eof}[1]{\E[#1]}

\newcommand{\EOf}[1]{\E\abOf{#1}}

\newcommand{\V}{\mathbf{V}} 

\newcommand{\VOf}[1]{\V\abOf{#1}}

\newcommand{\sizedMid}[2]{#1 \, \kern-\nulldelimiterspace\mathopen{}\left| \vphantom{#1}\,#2\right.\mathclose{}\kern-\nulldelimiterspace}
\newcommand{\setByEle}[2]{\cb{\sizedMid{#1}{#2}}}
\newcommand{\setByEleInText}[2]{\{#1 \mid #2\}}
\newcommand{\pr}{^\prime}
\newcommand{\prr}{^{\prime\prime}}

\newcommand{\compl}{^\mathsf{c}}%
\newcommand{\ind}{\mathds{1}}
\newcommand{\indOf}[1]{\ind_{\!#1}}%
\newcommand{\indOfOf}[2]{\ind_{\!#1}\!\brOf{#2}}%
\newcommand{\eqcm}{\,,} 
\newcommand{\eqfs}{\,.}
\newcommand{\dl}{\mathrm{d}} 
\newcommand{\kullback}{\mathrm{K\!L}} 
\DeclareMathOperator*{\argmin}{arg\,min}

%% file: macros_specific.tex
\newcommand{\ol}[2]{\overline{#1#2}}

\newcommand{\tran}{\tau}
\newcommand{\dtran}{\tau\pr}
\newcommand{\ddtran}{\tau\prr}

\newcommand{\ddrtran}{\tau^{\prime\oplus}}

\newcommand{\invdtran}{(\tau\pr)^{-1}}

\newcommand{\trans}{\tau_{\sqrt{}}}
\newcommand{\dtrans}{\tau\pr_{\sqrt{}}}

\newcommand{\setcc}{\mc S}
\newcommand{\setcciz}{\mc S_0^+}

\newcommand{\sign}{\mathrm{sign}}

\newcommand{\median}{\mathsf{med}}
\newcommand{\diam}{\mathsf{diam}}

\newcommand{\ld}{^\ominus}

\newcommand{\rd}{^\oplus}

\newcommand{\geodft}[2]{\gamma_{#1\to#2}} 

\newcommand{\Op}{\mathbf{O}_{\Pr}}

\newcommand{\cato}{$\mathsf{CAT}(0)$}

\newcommand{\ballopen}{\ball}
\newcommand{\ballclosed}{\overline{\ball}}

\newcommand{\bowtiechar}{%
	\tikz[baseline=(baseline), yscale=0.15, xscale=0.1]{
		\coordinate (baseline) at (0,-0.62);
		\path[draw=black, line join=round, rounded corners=0.15ex] (1,1) -- (1,-1) -- (0,-0.1) -- (-1, -1) -- (-1,1) -- (0, 0.1) -- cycle;
	}%
}

\newcommand{\bowtieset}{\bowtiechar}
\newcommand{\bowtiesetcompl}{\bowtiechar\compl\!}
\newcounter{sarrow}
\newcommand\xrsquigarrow[1]{%
    \stepcounter{sarrow}%
    \mathrel{\begin{tikzpicture}[baseline= {( $ (current bounding box.south) + (0,-0.5ex) $ )}]
            \node[inner sep=.5ex] (\thesarrow) {$\scriptstyle #1$};
            \path[draw,<-,decorate,
            decoration={zigzag,amplitude=0.7pt,segment length=1.2mm,pre=lineto,pre length=4pt}] 
            (\thesarrow.south east) -- (\thesarrow.south west);
    \end{tikzpicture}}%
}

\DeclareMathOperator{\artanh}{artanh}

\newcommand{\collOf}[1]{\ms{coll}\brOf{#1}}
\newcommand{\collof}[1]{\ms{coll}(#1)}

\newcommand{\suppof}[1]{\ms{supp}(#1)}

%% file: macros_theorem.tex
\def\postBoxSkip{1.0ex}
\def\postBoxSkipCmd{\vskip\postBoxSkip}
\def\preBoxSkip{1.0ex}
\def\preBoxSkipCmd{\vskip\preBoxSkip}
\declaretheoremstyle[
	bodyfont=\normalfont,
	postfoothook={\postBoxSkipCmd},
	preheadhook={\preBoxSkipCmd},
	mdframed={
		backgroundcolor = black!2,
		startcode={},
}]{ruledBoxStyle}
\declaretheoremstyle[
	bodyfont=\normalfont,
	postfoothook={\postBoxSkipCmd},
	preheadhook={\preBoxSkipCmd},
	mdframed={
		backgroundcolor=white,
}]{ruledBoxStyleWhite}
\declaretheoremstyle[
	bodyfont=\normalfont,
	postfoothook={\postBoxSkipCmd},
	preheadhook={\preBoxSkipCmd},
	mdframed={
		backgroundcolor=black!2,
		linecolor = black!2,
		tikzsetting = {
			draw = black,
			line width = 2pt,%
			dashed,%
			dash pattern = on 10pt off 3pt
		},
}]{dashedBoxStyle}
\declaretheoremstyle[
	bodyfont=\normalfont,
	postfoothook={\postBoxSkipCmd},
	preheadhook={\preBoxSkipCmd},
	mdframed={
		linecolor = white,
		startcode={},
		tikzsetting = {
			draw = black,
			line width = 1pt,%
			loosely dotted,
		},
	}
]{dashedStyle}
\declaretheoremstyle[
	bodyfont=\normalfont,
	postfoothook={\postBoxSkipCmd},
	preheadhook={\preBoxSkipCmd},
	mdframed={
		linecolor = black,
		innerlinewidth=1pt,outerlinewidth=1pt,
		middlelinewidth=1pt,
		linecolor=black,middlelinecolor=white,
		startcode={},
	}
]{doubleStyle}
\declaretheoremstyle[
	bodyfont=\normalfont,
	postfoothook={\postBoxSkipCmd},
	preheadhook={\preBoxSkipCmd},
	mdframed={
		backgroundcolor = black!4,
		linecolor = black!4,
		startcode={},
}]{boxStyle}
\declaretheoremstyle[
	headfont=\normalfont\itshape,
	notefont=\normalfont\itshape,
	notebraces={}{},
	bodyfont=\normalfont,
	qed=\qedsymbol,
	numbered=no,
	headindent=0pt,
	postheadspace=1ex,
	name={Proof},
	postheadhook={},
	mdframed={
		hidealllines = true,
		innerrightmargin = 0pt,
		innerleftmargin = 0pt,
		innertopmargin = 0pt,
		innerbottommargin = 0pt,
		leftmargin = 0pt,
		rightmargin = 0pt,
	}
]{proofStyle}
\declaretheoremstyle[
	bodyfont=\normalfont,
	postfoothook={\postBoxSkipCmd},
	preheadhook={\preBoxSkipCmd},
	mdframed={
		backgroundcolor = white,
		linecolor = black,
		startcode={},
		leftline = false,
		rightline = false,
}]{tobBottomStyle}
\declaretheoremstyle[
bodyfont=\normalfont,
]{standardStyle}
\declaretheorem[style=ruledBoxStyle,name=Definition, numberwithin=section]{definition}
\declaretheorem[style=ruledBoxStyle,name=Lemma,numberlike=definition]{lemma}
\declaretheorem[style=ruledBoxStyle,name=Proposition,numberlike=definition]{proposition}
\declaretheorem[style=ruledBoxStyle,name=Theorem,numberlike=definition]{theorem}
\declaretheorem[style=ruledBoxStyle,name=Corollary,numberlike=definition]{corollary}
\declaretheorem[style=boxStyle,name=Remark,numberlike=definition]{remark}
\declaretheorem[style=boxStyle,name=Notation,numberlike=definition]{notation}
\declaretheorem[style=dashedStyle,name=Example,numberlike=definition]{example}

%% file: abstract.tex
We establish finite-sample error bounds in expectation for transformed Fr\'echet means in Hadamard spaces under minimal assumptions. Transformed Fr\'echet means provide a unifying framework encompassing classical and robust notions of central tendency in metric spaces. Instead of minimizing squared distances as for the classical 2-Fr\'echet mean, we consider transformations of the distance that are nondecreasing, convex, and have a concave derivative. This class spans a continuum between median and classical mean. It includes the Fr\'echet median, power Fr\'echet means, and the (pseudo-)Huber mean, among others. We obtain the parametric rate of convergence under fewer than two moments, and a subclass of estimators exhibits a breakdown point of 1/2. Our results apply in general Hadamard spaces---including infinite dimensional Hilbert spaces and nonpositively curved geometries---and yield new insights even in Euclidean settings.

%% file: sec_intro2.tex
\section{Introduction}
\subsection{The Transformed Fr\'echet Mean}
%
The \emph{transformed Fr\'echet mean} (or $\tran$-Fr\'echet mean) \cite{varinequ} provides a unifying framework for classical and robust notions of centrality. Given a metric space $(\mc Q,d)$, a transformation $\tran\colon\Rp\to\R$, and a $\mc Q$-valued random variable $Y$, it is defined as any element
\begin{equation}
    m \in \argmin_{q\in\mc Q}\Eof{\tran(\ol Yq)}
    \eqcm
\end{equation}
where we write $\ol yq := d(y,q)$. Choosing $\tran(x)=x^{2}$ yields the classical $2$-Fr\'echet mean, while $\tran(x)=x$ gives the Fr\'echet median; in Euclidean spaces these reduce to the expectation and the geometric (or spatial) median, respectively.

We consider a large class of transformed Fr\'echet means where $\tran$ is a nondecreasing, convex function with concave derivative. Concavity of $\dtran$ caps growth at the quadratic rate and controls robustness: mass at $Y$ pulls at a candidate $q$ with strength $\dtran(\ol Yq)$, so concavity limits the relative pull of far-away mass ($\ol Yq$ large). Examples of such transformations include $\tran(x) = x^\alpha$ with $\alpha \in [1,2]$, the Huber loss $\tran(x) = x^2 \ind_{[0,1)}(x) + (2x-1) \ind_{[1,\infty)}(x)$ \cite{huber64}, the pseudo-Huber loss $\tran(x) = \sqrt{1+x^2}-1$ \cite{charbonnier94}, and $\tran(x) = \log(\cosh(x))$ \cite{green90}. The resulting means are in some sense between median and expectation and accordingly exhibit robustness to heavy tails and, in some cases, to contamination, as we show below. 

The transformed Fr\'echet mean $m$ is estimated by its empirical version $m_n$ based on $n\in\N$ independent and identically distributed (iid) copies $Y_1, Y_2, \dots, Y_n$ of $Y$, which is
\begin{equation}
    m_n \in \argmin_{q\in\mc Q} \sum_{i=1}^n \tran(\ol{Y_i}q)
    \eqfs
\end{equation}
As $\tran$ is fixed throughout a given context, we suppress the dependence of $m$ and $m_n$ on $\tran$ in our notation.

One strand of the robust statistics literature (see, e.g., \cite{Lugosi2019, Hubert2008}) takes as its target a functional whose plug-in estimator is itself not robust, such as the $2$-Fr\'echet mean, and achieves robustness by replacing the plug-in estimator with a more sophisticated procedure, often at the cost of introducing tuning parameters. In contrast, our target of estimation is the $\tran$-Fr\'echet mean $m$ itself, and we demonstrate robustness properties of its plug-in estimator $m_n$.

Treating the $\tran$-Fr\'echet mean as the estimand can be advantageous. On the one hand, for suitable transformations $\tran$, it remains estimable at the parametric rate even when $\Eof{\ol Ym^2}=\infty$, as we show below. In particular, transformations with bounded derivative ($\lim_{x\to\infty}\dtran(x)<\infty$) achieve this rate under extremely weak moment conditions. By contrast, without a finite second moment, the parametric rate is unattainable for the $2$-Fr\'echet mean, regardless of the estimator employed \cite[Theorem~3.1]{Devroye2016}. On the other hand, $\tran$-Fr\'echet means with unbounded influence ($\lim_{x\to\infty}\dtran(x)=\infty$), such as those generated by $\tran(x)=x^\alpha$ with $1<\alpha<2$, still retain quantitative information about the tails of the distribution of $Y$, unlike the (Fr\'echet) median. They are therefore particularly well suited for heavy-tailed settings in which the notion of central tendency is expected to reflect tail behavior.

Beyond these favorable robustness properties, the following facts support treating $m$ as an estimand in its own right. First, in a general metric space the $2$-Fr\'echet mean is less canonical than in Euclidean space: assuming $\tran(0)=0$, our conditions imply that $\sqrt{\tran}\circ d$ is a metric on $\mc Q$ inducing the same topology as $d$, and the $\tran$-Fr\'echet mean in $(\mc Q,d)$ is exactly the $2$-Fr\'echet mean in $(\mc Q,\sqrt{\tran}\circ d)$. Second, in $\R^{k}$ the empirical $\tran$-Fr\'echet mean is the maximum likelihood estimator of the location parameter $m$ in the family with Lebesgue density proportional to $\exp(-\tran(\normof{y-m})/b)$, for every value of the scale parameter $b>0$; under standard regularity conditions it is asymptotically efficient in this model \cite[Theorem~5.39 and Chapter~8]{Vaart1998}.
\subsection{Hadamard Spaces}
Our results are set in the framework of \emph{Hadamard spaces}, that is, complete geodesic metric spaces (each pair of points is connected by a geodesic) of nonpositive curvature in the sense of Alexandrov (geodesic triangles are at least as ``thin'' as their Euclidean counterparts). They are also called \emph{global NPC spaces} or \emph{complete \cato{} spaces}. A Hadamard space $(\mc Q, d)$ can be defined as a complete metric space with the following property: For all $y_0, y_1 \in \mc Q$ there is an $m\in\mc Q$ such that 
\begin{equation}
    \frac12 \ol{y_0}q^2 + \frac12 \ol{y_1}q^2 - \frac14 \ol{y_0}{y_1}^2 \geq  \ol qm^2
\end{equation}
for all $q\in\mc Q$. In this case, $m$ is the midpoint between $y_0$ and $y_1$. More details on the geometry of Hadamard spaces can be found in the textbooks \cite{burago01,bacak14b}.
Prominent examples of Hadamard spaces include:
\begin{itemize}
    \item Euclidean and, more generally, Hilbert spaces \cite[Prop.\ 3.5]{sturm03};
    \item Hadamard manifolds, i.e., complete, simply connected Riemannian manifolds with nonpositive sectional curvature \cite[Prop.\ 3.1]{sturm03};
    \item $\R$-trees (also called metric trees), i.e., geodesic spaces containing no subset homeomorphic to a circle \cite{evans2008probability};
    \item open books, i.e., disjoint copies of a half-space glued together along their spines, \cite{Hotz2013};
    \item the space of phylogenetic trees with the Billera--Holmes--Vogtmann metric \cite{billera01};
    \item the cone of symmetric positive definite matrices with the affine-invariant metric, for which the $2$-Fr\'echet mean coincides with the matrix geometric mean \cite{bhatia06};
    \item tangent cones of Hadamard spaces, suitably completed \cite[Thm.\ 1.2.17]{bacak14b}, \cite[Thm.\ 9.1.44]{burago01}.
\end{itemize}
The class of Hadamard spaces is closed under a variety of natural constructions, including passing to closed convex subsets, taking isometric images, forming products and $L^2$-spaces of Hadamard-space-valued maps, and certain gluings \cite[Sec.\ 3]{sturm03}. Importantly, they are not required to be finite dimensional (e.g., in the Hausdorff sense) or separable. These examples and closure properties illustrate the broad applicability of the Hadamard space framework.

\subsection{Results}
\subsubsection{Power Fr\'echet Means}
A particularly important subclass of transformed Fr\'echet means arises when 
\begin{equation}
    \tran(x) = x^\alpha, \qquad \alpha \in\Rpp
    \eqcm
\end{equation}
in which case $\tran$-Fr\'echet means are known as \emph{power Fr\'echet means} or \emph{$\alpha$-Fr\'echet means}. 
We restrict $\alpha \in (1,2]$ and denote the $\alpha$-Fr\'echet mean as $m = \argmin_{q\in\mc Q} \Eof{\ol Yq^\alpha}$ with its empirical counterpart $m_n = \argmin_{q\in\mc Q} \sum_{i=1}^n \ol{Y_i}q^\alpha$. These minimizers are unique if $\Eof{\ol Yq^\alpha} < \infty$ for one (and hence all) $q\in\mc Q$ \cite[Corollary 5.8]{varinequ}. 
We aim for an upper bound on the expected loss, but at the same time want to show robustness against heavy tails. Here, heavy tails refers to the case $\Eof{\ol Ym^2} = \infty$. Unfortunately, this implies $\Eof{\ol m{m_n}^2} = \infty$, so that the square loss does not allow for useful convergence-rate results.
Instead, setting $\sigma_\varphi := \Eof{\ol Ym^\varphi}$ for $\varphi\in\R$, $\phi := \frac{2-\alpha}{\alpha-1}$, and assuming $\sigma_\alpha < \infty$, we obtain 
\begin{equation}\label{eq:intro:pw:simple}
    \EOf{\ol{m}{m_n}^\alpha}^{\frac1\alpha}
    \leq
    c_\alpha \br{
        \sigma_{\alpha-1}^{\phi}\, \sigma_{2\alpha-2}^{\frac12}\, n^{-\frac12}
        +
        \sigma_\alpha^{\frac1\alpha}\, n^{-\frac{1}{2(\alpha-1)}}
    }
    \qquad\text{for all } n\in\N
\end{equation}
(implied by \cref{cor:power:stdloss:sharp}), 
where we use the convention here and in the following that $c_\alpha\in\Rpp$ is a constant depending only on $\alpha$ and each occurrence of $c_\alpha$ may refer to a different value.
Furthermore, in every Hadamard space containing more than one point, there are distributions such that $\sigma_{\alpha-2}, \sigma_{2\alpha-2} \in \Rpp$ and
\begin{equation}
    \EOf{\ol m{m_n}^{\alpha}}^{\frac1\alpha}
    \geq 
    c_\alpha \sigma_{\alpha-2}^{-1} \sigma_{2\alpha-2}^{\frac12}  n^{-\frac12}
    \qquad
    \text{for all sufficiently large } n\in\N
    \eqcm
\end{equation}
see \cref{cor:lower:alpha}.
Thus, \eqref{eq:intro:pw:simple} is asymptotically optimal up to a constant and the gap between $\sigma_{\alpha-2}^{-1}$ and $\sigma_{\alpha-1}^{\frac{2-\alpha}{\alpha-1}}$, where Jensen's inequality yields $\sigma_{\alpha-2}^{-1} \leq \sigma_{\alpha-1}^{\frac{2-\alpha}{\alpha-1}}$.

\Cref{thm:power:main} gives a slightly sharper finite-sample bound that implies
\eqref{eq:intro:pw:simple}. The loss it uses is the quantity by which the variance inequality (\cref{thm:infdtr:vi})
bounds the excess risk from below, and hence the quantity our proof technique
controls directly. It penalizes large values of $\ol m{m_n}$ only with power
$\alpha$ while retaining an $L^2$-type error for small values of $\ol m{m_n}$:
\begin{equation}
	\EOf{\ol m{m_n}^2 \br{\sigma_{\alpha-1} + \ol m{m_n}^{\alpha-1}}^{-\phi}}
	\leq
	c_\alpha n^{-1} \br{\sigma_{2\alpha-2}	\sigma_{\alpha-1}^{\phi} + n^{-\phi}\sigma_{\alpha}}
    \qquad\text{for all } n\in\N
    \eqfs
\end{equation}
Strikingly, the only assumptions required for this parametric rate are the moment condition $\sigma_\alpha < \infty$ (i.e., less than two moments if $\alpha<2$) and the Hadamard structure of $\mc Q$. Moreover, the highest-order moment $\sigma_\alpha$ appears in the rate multiplied by a factor that decreases with $n$, so that for sufficiently large samples the risk is controlled primarily by the lower-order moment $\sigma_{2\alpha-2}$. This highlights the robustness of power Fr\'echet means, making them particularly attractive in settings with heavy-tailed data.

While $\ol{m}{m_n} = \Op(n^{-1/2})$ is the optimal rate for general distributions (\cref{cor:lower:alpha}), we show that convergence is faster for distributions with large small ball probabilities near $m$ (\cref{cor:fast:entropy:power}): Assume there are $b\in\Rpp$ and $\beta \in (\alpha, 2)$ such that 
\begin{equation}\label{eq:intro:smallball}
    \PrOf{\ol Ym \leq x} \geq b x^{\beta-\alpha}
\end{equation}
for all sufficiently small $x\in\Rpp$. Then we have $\ol{m}{m_n} = \Op(n^{-\frac{1}{2(\beta-1)}})$ and this rate is optimal for distributions satisfying \eqref{eq:intro:smallball} and $\sigma_{2\alpha-2} < \infty$ (\cref{cor:lower:alpha}). For the fast rates, in contrast to our other results, we require the space $\mc Q$ to be finite dimensional in some sense, as we use a different proof technique. Apart from this requirement, fast rates can occur in any Hadamard space whenever $\alpha<2$.
\subsubsection{General Transformations}
For general nondecreasing, convex transformations $\tran$ with concave derivative $\dtran$, we distinguish two cases:

If $\lim_{x\to\infty}\dtran(x) = \infty$, as in the case of $\tran(x)=x^\alpha$, $\alpha \in (1, 2]$, we essentially require the moment $\Eof{\dtran(\ol Ym)^2/\ddtran(\ol Ym)}$ to be finite to obtain the parametric rates of convergence (\cref{cor:infdtr:improved})
\begin{equation}
    \tran^{-1}\brOf{\EOf{\tran(\ol m{m_n})}} = \mo O\brOf{n^{-\frac12}}
    \qquad\text{and}\qquad 
    \EOf{\ol m{m_n}} = \mo O\brOf{n^{-\frac12}}
    \eqfs
\end{equation}
The aforementioned moment enters the error bound multiplied by a factor decreasing with $n$, and the bound is dominated by $\Eof{\dtran(\ol Ym)^2}$ for large $n$. Thus, the $\tran$-Fr\'echet mean is robust to heavy-tailed distributions.

If $\lim_{x\to\infty}\dtran(x) < \infty$ and $\ddtran(x)>0$ for all $x\in\Rpp$, as in the case of the pseudo-Huber loss $\tran(x)=\sqrt{1 + x^2}-1$, we require a minimal moment condition $\Eof{\ol Ym^\xi} < \infty$ for some $\xi \in\Rpp$, however small, to establish a parametric rate of convergence. Furthermore, we establish exponential tail bounds for estimation error in \cref{thm:finiteD:convex,thm:tailbound}. This allows us to obtain a bound of the mean squared error even though the second moment $\Eof{\ol Ym^2}$ may be infinite,
\begin{equation}
    \EOf{\ol m{m_n}^2} = \mo O\brOf{n^{-1}}
\end{equation}
(\cref{thm:posddrtran:main}).
Moreover, $\lim_{x\to\infty}\dtran(x) < \infty$ implies that the $\tran$-Fr\'echet mean has a breakdown point of $1/2$ (\cref{thm:breakdown}); hence, it is robust to heavy-tailed distributions and contaminated data.
\subsubsection{Fr\'echet Median}
For the Fr\'echet median, i.e., $\tran(x)=x$, we obtain under some conditions
\begin{equation}
	\EOf{\ol m{m_n}^2} = \mo O\brOf{n^{-1}}
    \eqcm
\end{equation}
(\cref{thm:median:main}).
We make a minimal moment requirement $\Eof{\ol Ym^\xi}<\infty$ for some $\xi\in\Rpp$, however small. Furthermore, we assume that $Y$ is not concentrated on a so-called bow tie
(\cref{def:bowtie}). In Hilbert spaces and Hadamard manifolds, this reduces
to requiring that $Y$ not be concentrated on a geodesic
(\cref{cor:median:notonageodesic}), which in the linear case is the condition
commonly imposed for the spatial median \cite{Chakraborty2014}.
Furthermore, we show exponential tail bounds for the Fr\'echet median, see \cref{thm:median:tailbound}. In particular, for $r \in\Rpp$, \cref{cor:median:tailbound} implies
\begin{equation}
	\PrOf{\ol m{m_n} > 4 r} \leq  \br{2\Prof{\ol Ym > r}^{\frac13}}^n
	\eqfs
\end{equation}
\subsection{Proof Technique}
The proofs of the convergence rate results follow the ideas of algorithm stability, which have been applied in the context of Fr\'echet means in \cite{Escande2024,brunel2025}. In contrast to chaining-based proofs \cite{schoetz19,Ahidar2020}, this allows us to obtain results not cursed by dimension, i.e., we do not require any notion of dimension to be finite or covering numbers to be bounded in some way.

For the algorithm stability proof, we build on two key results from prior work: the quadruple inequality \cite{quadruple} and the variance inequality \cite{varinequ} for transformations $\tran$ in Hadamard spaces. Both results assume that $\tran$ is nondecreasing and convex with a concave derivative. Both are highly nontrivial, and both are essential to our arguments.

Beyond these two properties, our proofs go substantially beyond classical algorithm stability arguments and the corresponding results for the $2$-Fr\'echet mean \cite{Escande2024,brunel2025}. The reason is that the variance inequality for general $\tran$ carries a distribution-dependent factor, which is one of the main technical obstacles to deriving results in expectation.
\subsection{Related Literature}
The $2$-Fr\'echet mean (also called barycenter or center of mass) was introduced in \cite{frechet48}; a foundational treatment in Hadamard spaces can be found in \cite{sturm03}.
State-of-the-art strong laws of large numbers for power Fr\'echet means ($\tran(x)=x^\alpha$) in general metric spaces are derived in \cite{jaffe2024}, while laws of large numbers for transformed Fr\'echet means are established in \cite{schoetz22}.

For rates of convergence, \cite{Petersen2019,schoetz19,Ahidar2020} use approaches related to chaining \cite{talagrand21} resulting in bounds that are cursed by dimension, meaning that they slow down in infinite dimensions. This is suboptimal, since elementary calculations in Hilbert spaces give dimension-free rates for the arithmetic mean. While the chaining-based results apply in great generality (apart from the dimension requirement), stricter geometric assumptions allow the construction of a tangent cone with Hilbert space structure, which yields convergence rates for the $2$-Fr\'echet mean that are not cursed by dimension \cite{LeGouic2023}. Furthermore, proofs based on algorithm stability \cite{Escande2024,brunel2025} reduce assumptions further while retaining the non-cursed convergence rates for $2$-Fr\'echet means.

While Fr\'echet means typically exhibit a parametric rate of convergence, in some settings the geometry of the underlying metric space induces slower rates (smeariness) \cite{eltzner19} or a positive probability of perfect estimation with finite samples (stickiness) \cite{Lammers2023}.

The $1$-Fr\'echet mean or Fr\'echet median ($\tran(x)=x$) generalizes the notion of spatial median (also called geometric median) in normed spaces. In Euclidean spaces, the spatial median is well understood \cite{mottonen10,Minsker2024}, and many results extend to Banach spaces \cite{Kemperman1987,Chakraborty2014,minsker15,romon2023quantiles}. Furthermore, the median on Riemannian manifolds is studied in \cite{yang10}. Practical computation of medians and means in Hadamard spaces is addressed in \cite{Bacak2014}, and the underlying proximal point methodology \cite{Bacak2013} adapts to transformed Fr\'echet means with nondecreasing convex $\tran$.

In the context of robust statistics in metric spaces, median-of-means estimators were examined in \cite{Yun2023,kim2025}. Like the transformation-function approach considered here, these estimators balance between the median and the classical mean. Another example of such a trade-off is the trimmed Fr\'echet mean \cite{oliveira2025,bartl2025}, which has been shown to be minimax optimal under adversarial sample contamination.

Beyond the power Fr\'echet means, further forms of transformed Fr\'echet means have been studied: \cite{romon2023convex} consider convex transformations in metric trees---a specific class of Hadamard space; \cite{lee2025} consider the Huber and pseudo-Huber losses on Riemannian manifolds; and \cite{Brizzi2026} consider Fr\'echet means with convex transformations in the Wasserstein space. Fundamental properties such as existence and uniqueness of transformed Fr\'echet means in Hadamard spaces were studied in \cite{varinequ}.

Beyond the $2$-Fr\'echet mean, no rate-of-convergence results under comparably minimal assumptions are available in Hadamard spaces, and several results in this paper are new even in Euclidean spaces, e.g., the finite-sample bound for power Fr\'echet means.
\subsection{Outline}
The remainder of the article is organized as follows. In \cref{sec:preliminaries}, we introduce the necessary background on transformed Fr\'echet means, the class of transformations under consideration, and the geometry of Hadamard spaces. \Cref{sec:unbounded} establishes finite-sample error bounds with explicit constants for power Fr\'echet means and, more generally, for transformations satisfying $\lim_{x\to\infty}\dtran(x)=\infty$. The complementary class of transformations, characterized by $\lim_{x\to\infty}\dtran(x)<\infty$, is studied in \cref{sec:bounded}, where we derive exponential tail bounds and corresponding finite-sample error bounds. This section also covers the special case $\tran(x)=x$, which yields the Fr\'echet median. Finally, \cref{sec:optimal} discusses settings in which convergence rates exceeding the parametric rate are attainable and examines the optimality of the obtained results.

%% file: sec_prelim.tex
\section{Preliminaries}\label{sec:preliminaries}
Throughout the article, we use the convention $0^0:=1$.
\subsection{Nondecreasing, Convex Functions with Concave Derivative}\label{sec:ncfcd}
\begin{definition}\label{def:tran}
    Let $\setcc$ be the set of nondecreasing convex functions $\tran\colon\Rp\to\R$ that are differentiable on $\Rpp$ with concave derivative $\dtran$. We extend the domain of $\dtran$ to $\Rp$ by setting  $\dtran(0) := \lim_{x\searrow0} \dtran(x)$, which exists, as $\dtran$ is nonnegative and nondecreasing.
\end{definition}
Although convex functions need not be differentiable everywhere, they are differentiable Lebesgue almost everywhere. In this work we restrict attention to the subclass for which the derivative exists on all of $\Rpp$, which is sufficient for our purposes.
For technical reasons it is often more convenient to work with $\setcciz\subset\setcc$, the subset of strictly increasing functions $\tran\in\setcc$ with $\tran(0) = 0$,
\begin{align}
    \setcciz
    &:=
    \setByEle{\tran\in\setcc}{\tran(0) = 0 \text{ and } \forall x\in\Rpp\colon \dtran(x)>0}
    \\&=
    \setByEle{x \mapsto \tran(x)-\tran(0)}{\tran\in\setcc}\setminus\cb{x\mapsto 0}
    \eqfs
\end{align}
This is not restrictive, as we essentially only exclude constant functions.
To be able to talk about derivatives of $\tran\in\setcc$ at $0$ and second derivatives, let us recall the definition of the one-sided derivatives.
\begin{notation}\label{nota:leftrightderiv}
    Let $A\subset \R$ and $f\colon A \to \R$. Let $x_0\in A$ such that there is $\epsilon>0$ such that $(x_0-\epsilon, x_0] \subset A$. Then denote the left derivative of $f$ at $x_0$ as $f\ld(x_0) := \lim_{x \nearrow x_0} \frac{f(x) - f(x_0)}{x-x_0}$ if the limit exists. Similarly, for $x_0\in A$ with $\epsilon>0$ such that $[x_0, x_0+\epsilon) \subset A$, we denote the right derivative of $f$ at $x_0$ as $f\rd(x_0) := \lim_{x \searrow x_0} \frac{f(x) - f(x_0)}{x-x_0}$ if the limit exists.
\end{notation}
Let us note some basic continuity properties of functions in $\setcc$ and existence of one-sided derivatives. See \cite{varinequ} for proofs.
\begin{lemma}\label{lmm:tran:continuous}
    Let $\tran\in\setcc$. Then
    \begin{enumerate}[label=(\roman*)]
        \item $\tran\rd(0)\in \Rp$ exists and $\tran\rd(0) = \dtran(0)$;
        \item $\ddrtran(x) \in \Rp$ exists for all $x\in\Rpp$, $\ddrtran(0) \in [0,\infty]$ exists, and $\ddrtran(0) = \lim_{x\searrow0} \ddrtran(x)$;
		\item $\tran\colon\Rp\to\R$ and $\dtran\colon\Rp\to\Rp$ are continuous and nondecreasing, and $\ddrtran\colon\Rp\to[0,\infty]$ is nonincreasing.
    \end{enumerate}
\end{lemma}
Further important properties of the functions $\tran \in \setcc$ are listed in \cref{ssec:qi} and \cref{sec:tran}. Moreover, this function class is studied in \cite{Pinelis2015} in the context of von Bahr-Esseen bounds for martingales.
\subsection{Geometry}
We state some basic definitions regarding geodesics and convexity in metric spaces and define the bow tie set, which we will use in the context of Fr\'echet medians.
Let $(\mc Q, d)$ be a nonempty metric space and denote $\ol qp := d(q,p)$ for $q,p\in\mc Q$.
\begin{definition}\label{def:geodesic}
	Let $I\subset \R$ be convex.
	\begin{enumerate}[label=(\roman*)]
		\item A function $\gamma\colon I \to \mc Q$ is called \emph{geodesic} if and only if
		\begin{equation}
			\ol{\gamma(r)}{\gamma(t)} =  \ol{\gamma(r)}{\gamma(s)} + \ol{\gamma(s)}{\gamma(t)}
		\end{equation}
		for all $r,s,t\in I$ with $r<s<t$.
		\item Let  $\gamma\colon I \to \mc Q$ be a geodesic. If there is $L\in\Rp$ such that $\ol{\gamma(s)}{\gamma(t)} = L\absof{s-t}$ for all $s,t\in I$, then the geodesic is said to have \emph{constant speed}. If $L = 1$, we call $\gamma$ a \emph{unit-speed geodesic}.
		\item
		The metric space $(\mc Q, d)$ is called \emph{unique geodesic space}, if and only if each pair of points $(q,p) \in \mc Q^2$ is connected by a unique unit-speed geodesic $\geodft qp \colon [0, \ol qp] \to \mc Q$ so that $\geodft qp(0) = q$ and $\geodft qp(\ol qp) = p$.
	\end{enumerate}
\end{definition}
Hadamard spaces are unique geodesic spaces. Next, we define the bow tie, which was introduced and illustrated in \cite{varinequ} for the study of the Fr\'echet median in general Hadamard spaces. \Cref{fig:bowtie} illustrates the bow tie in $\R^2$.
\begin{figure}
    \definecolor{myOrange}{HTML}{F16E04}
    \definecolor{myPurple}{HTML}{A152D6}
    \definecolor{metricSpacepColor}{HTML}{000000}
    \begin{center}
        \begin{tikzpicture}[blend group=normal]
            \fill[black!10] (-6,-3) rectangle (6,3);
            \draw[step=1,white,line width=0.3pt] (-6,-3) grid (6,3);
            \draw[draw=none,pattern={Dots[radius=1pt,distance=5pt,xshift=2.5pt,yshift=2.5pt]},pattern color=myPurple] (-6,-2) -- (6,2) -- (6,-2) -- (-6,2) -- cycle;
            \draw[draw=none,pattern={Dots[radius=1pt,distance=5pt]},pattern color=myOrange] (-6,-8/3) -- (6,4/3) -- (6,-4/3) -- (-6,8/3) -- cycle;
            \draw[line width=0.5pt,myPurple] (-6,-2) -- (6,2);
            \draw[line width=0.5pt,myPurple] (6,-2) -- (-6,2);
            \draw[line width=0.5pt,myOrange] (-6,-8/3) -- (6,4/3);
            \draw[line width=0.5pt,myOrange] (6,-4/3) -- (-6,8/3);
            \draw[line width=1pt] (-6,0) -- (6,0);
            \fill (0,0) circle (2pt);
            \fill (2,0) circle (2pt);
            \node[above,fill=black!10,yshift=2pt, inner sep=1pt] at (-5,0) {$\gamma$};
            \node[below,fill=black!10,yshift=-2pt, inner sep=1pt] at (0,0) {$q$};
            \node[below,fill=black!10,yshift=-2pt, inner sep=1pt] at (2,0) {$p$};
            \node (bowtie) at (1,2) {$\bowtieset(q, p, w)$};
            \draw[->,>=stealth] (bowtie) -- (-2.5,1.6);
            \draw[->,>=stealth] (bowtie) -- (-0.8,1);
            \draw[->,>=stealth] (bowtie) -- (1,0.5);
            \draw[->,>=stealth] (bowtie) -- (2.8,1);
            \draw[->,>=stealth] (bowtie) -- (4.5,1.6);
        \end{tikzpicture}
        \caption{Visualization of the bow tie 
            in $\R^2$ between the knots $q = (0,0)$ and $p = (2,0)$ with \emph{widening} $w =  10^{-\frac12}$. The Euclidean plane is depicted as a gray area with unit grid. Shown in black is the geodesic (extended) geodesic $\gamma = \geodft qp\colon\R\to\R^2$ with $\gamma(0) = q$ and $\gamma(\ol qp) = p$. As we are in a Euclidean space, we have $\ol y\gamma\rd(0) = -\cos(\angle(y,q,p))$, where $\angle(y,q,p)$ is the angle at $q$ in the triangle with corners $y,q,p$. Thus, the purple area $\setByEleInText{y\in\R^2}{\ol y\gamma\rd(0)^2 \geq 1 - w^2}$ consists of the points $y$ that fulfill $|\sin(\angle(y,q,p))| \leq w = 10^{-\frac12}$; equivalently the geodesics $\geodft qy$ have absolute slopes $|\tan(\angle(y,q,p))| \leq \sqrt{w^2/(1-w^2)} = 1/3$. The same characterization at the other knot, $\ol y\gamma\ld(\ol qp) = \cos(\angle(y,p,q))$, yields the set $\setByEleInText{y\in\R^2}{\ol y\gamma\ld(\ol qp)^2 \geq 1 - w^2}$, shown as the orange dotted area. The bow tie 
            is the union of both dotted areas.}\label{fig:bowtie}
    \end{center}
\end{figure}
\begin{definition}\label{def:bowtie}
	Assume $\mc Q$ is a unique geodesic space.
	Let $q,p\in\mc Q$ with $q\neq p$.
	The \emph{bow tie} between the \emph{knots} $q$ and $p$ with \emph{widening} $w \in [0,1]$ is the set
	\begin{equation}
		\bowtieset(q, p, w) := \setByEle{y \in \mc Q}{\max\brOf{\ol y{\geodft qp}\rd(0)^2, \ol y{\geodft qp}\ld(\ol qp)^2} \geq 1-w^2}
		\eqfs
	\end{equation}
	Furthermore, set $\bowtieset(q, q, w) := \{q\}$ for all $q\in\mc Q$ and $w \in [0,1)$ and $\bowtieset(q, q, 1) := \mc Q$.
\end{definition}
The notion of convexity can be transferred to Hadamard spaces, see, e.g., \cite[chapter 2]{bacak14b}. We use the term \emph{convex} here, but some authors prefer \emph{geodesically convex} in this context.
\begin{definition}\label{def:convex}
	Assume $\mc Q$ is a unique geodesic space.
	\begin{enumerate}[label=(\roman*)]
		\item A set $A\subset \mc Q$ is called \emph{convex} if and only if, for any $q,p\in A, q\neq p$, we have $\geodft qp\subset A$.
		\item A function $f\colon\mc Q\to\R$ is called \emph{convex} if and only if, for any $q,p\in\mc Q, q\neq p$, we have $f \circ \geodft qp$ is convex.
	\end{enumerate}
\end{definition}
\begin{notation}
    For a metric space $\mc Q$, a point $p\in\mc Q$, and $r \in \Rp$, denote the open and closed ball with center $p$ and radius $r$ as 
    \begin{equation}
        \ballopen(p, r) := \setByEleInText{q\in\mc Q}{\ol qp < r}\quad\text{and}\quad
        \ballclosed(p, r) := \setByEleInText{q\in\mc Q}{\ol qp \leq r}
        \eqcm\quad\text{respectively}\eqfs
    \end{equation}
\end{notation}
\subsection{Basic Setup}\label{ssec:prelim:setup}
Throughout the remainder of the paper, unless explicitly stated otherwise, we assume the following setup: Let $(\mc Q, d)$ be a Hadamard space. For $q,p\in\mc Q$, we denote $\ol{q}{p} := d(q, p)$. This metric space is equipped with its Borel-$\sigma$-algebra. Let $\Pr$ be a probability measure on a measurable space $\Omega$. The expectation of measurable functions $X \colon \Omega \to \R$ is denoted as $\Eof{X}$ if it exists. Let $Y$ be a measurable function $Y \colon \Omega \to \mc Q$, i.e., a $\mc Q$-valued random variable. We always assume that there is a separable set $\mc Y \subset \mc Q$ such that $\Prof{Y\in \mc Y} = 1$. Let $n\in\N$ and let the samples $Y_1, Y_2, \dots, Y_n$ be independent and identically distributed copies of $Y$.
Let $\tran \in \setcciz$. Let $o\in\mc Q$ be an arbitrary reference point. Assume $\Eof{\dtran(\ol Yo)} < \infty$. Let the population and sample $\tran$-Fr\'echet means be
\begin{equation}
	m \in \argmin_{q\in\mc Q} \EOf{\tran(\ol {Y}q)-\tran(\ol {Y}o)}\eqcm\qquad
	m_n \in \argmin_{q\in\mc Q} \sum_{j=1}^n \tran(\ol {Y_j}q)
	\eqcm
\end{equation}
where $m$ does not depend on the choice of $o$, see \cref{ssec:tfm}.
\begin{remark}
    We follow \cite[beginning of Section 4]{sturm03} and avoid measurability issues by the requirement that the distribution of $Y$ is concentrated on a separable set $\mc Y$.
    Results based on this requirement are, in some sense, equivalent to results based on the condition that the whole space $\mc Q$ is separable:
    As $\mc Q$ is Hadamard, the closed convex hull $\overline{\ms{conv}}(\mc Y)$ exists and is separable, see \cref{lmm:convsepa}.
    Hence, we have $\Prof{Y\in \overline{\ms{conv}}(\mc Y)} = 1$ and $(\overline{\ms{conv}}(\mc Y), d)$ is a Hadamard space.
    By \cite[Proposition 5.2]{varinequ}, the $\tran$-Fr\'echet mean set in $\mc Q$ is the same as in $\overline{\ms{conv}}(\mc Y)$ because $\overline{\ms{conv}}(\mc Y)$ is closed and convex.
\end{remark}
\subsection{Transformed Fr\'echet Mean}\label{ssec:tfm}
Basic properties of the transformed Fr\'echet mean were derived in \cite{varinequ}. We briefly summarize the essential concepts here and refer to \cite{varinequ} for proofs and further details.

The assumption $\Eof{\dtran(\ol Yo)} < \infty$ implies $\Eof{\absof{\tran(\ol Yq)-\tran(\ol Yp)}} < \infty$ for all $q,p\in\mc Q$ \cite[Proposition 3.2]{varinequ}. In this case, we define the $\tran$-Fr\'echet mean set as
\begin{equation}
	M := \argmin_{q\in\mc Q} \EOf{\tran(\ol Yq)-\tran(\ol Yo)}
	\eqfs
\end{equation}
The set $M$ is nonempty, closed, convex, and bounded \cite[Proposition 5.2]{varinequ}. If $\mc Q$ is locally compact, then $M$ is compact \cite[Remark 5.3]{varinequ}. 
Local compactness of $\mc Q$ may not be required for compactness of $M$: By \cite[Example 2.5 and Corollary 3.10]{jaffe2024}, if $\tran(x) = x^\alpha$ with $\alpha \geq 1$, and $\mc Q$ is a separable Hadamard space, then $M$ is compact.
The set $M$ does not depend on the choice of $o$ \cite[Proposition 5.2]{varinequ}.
Moreover, if $\Eof{\tran(\ol Yo)}$ is finite, then $\Eof{\tran(\ol Yq)}$ is finite for all $q\in\mc Q$ (\cref{lmm:tran:add}) and $M = \argmin_{q\in\mc Q} \Eof{\tran(\ol Yq)}$.

Let $x_0 := \inf\setByEleInText{x\in\Rpp}{\ddrtran(x) = 0}$, with the convention $\inf\emptyset = \infty$. If $m \in M$ and $\Prof{\ol Ym < x_0} > 0$, then $M= \{m\}$ \cite[Corollary 5.7]{varinequ}. Thus, if $\ddrtran(x) > 0$ for all $x\in\Rpp$, then the $\tran$-Fr\'echet mean is unique. Alternatively, if $\mc Q$ is separable and the support of $Y$ is convex, then the $\tran$-Fr\'echet mean is unique for any $\tran\in\setcciz$ \cite[Corollary 6.14]{varinequ}.

Let the empirical transformed Fr\'echet mean set be
\begin{equation}
	M_n := \argmin_{q\in\mc Q} \sum_{i=1}^n \tran(\ol {Y_i}q) = \argmin_{q\in\mc Q} \frac1n \sum_{i=1}^n \br{\tran(\ol {Y_i}q)-\tran(\ol {Y_i}o)}
	\eqfs
\end{equation}
This estimator of $M$ satisfies a strong law of large numbers \cite{schoetz22, Evans2024, jaffe2024} under a transformed first-moment condition, which in this setting amounts to $\Eof{\dtran(\ol Yo)} < \infty$ for transformed Fr\'echet means with $\tran \in \setcciz$. If $M$ is not a singleton, the convergence guaranteed by the strong law is generally one-sided: convergent subsequences of $m_n \in M_n$ have limits in $M$, but not every $m \in M$ arises as the limit of an empirical sequence. See \cite{Blanchard2025} for a relaxation technique that yields convergence in the Hausdorff metric, i.e., two-sided convergence.
\subsection{Quadruple Inequality}\label{ssec:qi}
The first central ingredient for the main proofs is a quadruple inequality, detailed in \cite{quadruple}.
Quadruple inequalities generalize the Cauchy--Schwarz inequality of Hilbert spaces $\mc H$ with the square transformation, i.e.,
\begin{equation}
	\normOf{y-q}^2 - \normOf{y-p}^2 - \normOf{z-q}^2 + \normOf{z-p}^2 = 2 \ipof{p-q}{y-z} \leq 2 \normOf{q-p} \normOf{y-z}
\end{equation} 
for all $q,p,y,z\in\mc H$, to Hadamard spaces and transformations in $\setcciz$.
\begin{proposition}[{\cite[Theorem 1]{quadruple}}]\label{thm:infdtr:qi}
    For all $q,p,y,z\in\mc Q$,
    \begin{equation}\label{eq:qi:tran}
        \tran(\ol yq) -  \tran(\ol yp) - \tran(\ol zq) + \tran(\ol zp) \leq 2 \, \ol qp\, \dtran(\ol yz)
        \eqfs
    \end{equation}
\end{proposition}
\cref{thm:infdtr:qi} implies the symmetrized quadruple inequality
\begin{equation}
	\absOf{\tran(\ol yq) -  \tran(\ol yp) - \tran(\ol zq) + \tran(\ol zp)} \leq  2 \min\brOf{\ol qp\, \dtran(\ol yz),\, \ol yz\, \dtran(\ol qp)}
\end{equation}
for all $q,p,y,z\in\mc Q$.
The constant $2$ on the right-hand side of \eqref{eq:qi:tran} can be slightly improved for $\tran(x) = x^\alpha$ to the optimal constant $2^{2-\alpha}$:
\begin{proposition}[{\cite[Theorem 3]{schoetz19}}]\label{thm:power:qi}
	Let $\alpha\in[1,2]$.
	Then, for all $q,p,y,z\in\mc Q$,
	\begin{equation}
		\ol yq^\alpha- \ol yp^\alpha - \ol zq^\alpha + \ol zp^\alpha \leq 2^{2-\alpha} \alpha \, \ol qp\, \ol yz^{\alpha-1}
		\eqfs
	\end{equation}
\end{proposition}
\subsection{Variance Inequality}
The second central ingredient for the main proofs in this article is a variance inequality, which is discussed in detail in \cite{varinequ}.
Transformed Fr\'echet means are defined by minimizing the objective function $q \mapsto \Eof{\tran(\ol Yq)}$.
Variance inequalities relate differences in the value of the objective function to the distance between its arguments.
\begin{proposition}\label{thm:infdtr:vi}
    Define
    \begin{equation}
        v(t) = \inf_{q\in\mc Q\setminus\ballopen(m,t)}\EOf{\tran(\ol Yq)-\tran(\ol Ym)}
        \eqfs
    \end{equation}
    Then $v \colon \Rp\to\Rp$ and $\Rpp\to\Rp,t\mapsto v(t)/t$ are nondecreasing with $v(0)=0$. Furthermore:
    \begin{enumerate}[label=(\roman*)]
        \item \label{thm:infdtr:vi:default}
        We have
        \begin{equation}
            v(t) \geq \frac12 t^2 \EOf{\ddrtran\brOf{\ol Ym + t}}
            \eqfs
        \end{equation}
        \item \label{thm:infdtr:vi:convex}
        Assume $\ddrtran$ is convex. Then
        \begin{equation}
            v(t) \geq \frac12 t^2 \EOf{\ddrtran\brOf{\ol Ym + \frac13 t}}
            \eqfs
        \end{equation}
        \item \label{thm:infdtr:vi:quantile}
        Let $r\in\Rpp$ and set $\rho : =\Prof{\ol Ym\leq r}$. Then
        \begin{equation}
            v(t) \geq \frac\rho2 t^2 \ddrtran\brOf{r + t}
            \eqfs
        \end{equation}        
        \item \label{thm:infdtr:vi:infty}
        Assume $\diam(\mc Q) = \infty$. Then
        \begin{equation}
            \lim_{t\to\infty} \frac{v(t)}{\tran(t)} = 1
            \eqfs
        \end{equation}
        \item \label{thm:infdtr:vi:combined}
        Let $t_0 \in \Rpp$ such that $\PrOf{\ol Ym \leq t_0} > 0$ and $\ddrtran(2t_0)>0$. Then, there is a constant $a\in(0,1]$ depending only on $t_0$, $\tran$, and the distribution of $Y$, such that
        \begin{equation}
            v(t)
            \geq 
            \begin{cases}
                \frac12 \PrOf{\ol Ym \leq t_0} \ddrtran(2t_0)\,t^2 &\text{for } t \leq t_0\eqcm\\
                a \tran(t) &\text{for } t\geq t_0 \eqfs\\
            \end{cases}
        \end{equation}
    \end{enumerate}
\end{proposition}
The proposition is based on \cite[Theorem 3.5, Proposition 5.2, Theorem 5.4]{varinequ} with additional arguments for \ref{thm:infdtr:vi:convex}, \ref{thm:infdtr:vi:quantile}, and \ref{thm:infdtr:vi:combined} as detailed in \cref{sec:proof:prelim}. 
\begin{remark}
    \begin{enumerate}[label=(\roman*)]
        \item
        \cref{thm:infdtr:vi} shows that the objective function in the defining minimizing problem of the $\tran$-Fr\'echet mean typically grows at least quadratically near the minimizer and, farther away, behaves like the transformation $\tran$ itself. 
        \item 
        Note that although the map $q\mapsto \Eof{\tran(\ol Yq)-\tran(\ol Ym)}$ is geodesically convex \cite[Proposition 5.2]{varinequ}, the function $v$ need not be convex, see \cref{exa:v:notconvex}. Instead, the function fulfills a weaker property: $v$ is star-shaped, meaning that $v(at) \leq av(t)$ for all $t\in\Rp$ and $a\in[0,1]$, which is equivalent to $t\mapsto v(t)/t$ being nondecreasing on $\Rpp$ since $v(0)=0$.
    \end{enumerate}
\end{remark}

If $\tran(x) = x$, then $\ddrtran(x) = 0$ and the bounds of \cref{thm:infdtr:vi} in terms of $\ddrtran$ are not helpful. In this case, we can still obtain a nontrivial variance inequality if $Y$ is not concentrated on a bow tie (\cref{def:bowtie}):
\begin{proposition}[{\cite[Theorem 6.15]{varinequ}}]\label{thm:median:vi}
    Let $\tran(x)=x$ so that $m$ is a Fr\'echet median. 
    Let $w\in[0,1]$.
    Let $q\in\mc Q\setminus \cb{m}$.
    Then
    \begin{equation}\label{eq:near:median}
        \EOf{\ol Yq - \ol Ym} 
        \geq 
        \frac12 w^2  \,\ol qm^2\, \EOf{\br{\ol Ym + \ol qm}^{-1} \indOf{\bowtiesetcompl(m, q, w)}(Y)}
        \eqfs
    \end{equation}
\end{proposition}

%% file: sec_unbounded.tex
\section{Unbounded Influence}\label{sec:unbounded}
In this section, we establish rates of convergence in expectation for $\tran$-Fr\'echet means with $\lim_{x\to\infty}\dtran(x) = \infty$. The resulting (unique) $\tran$-Fr\'echet means are robust to some heavy-tailed distributions, but not to arbitrary adversarial corruption of data (breakdown point $0$). The complementary case, where $\lim_{x\to\infty}\dtran(x) < \infty$, will be addressed in \cref{sec:bounded}. We start with transformations $\tran(x)=x^\alpha$ in \cref{sec:power}, where we get the cleanest results, and discuss a more general class of transformations in \cref{sec:tailrobust}.
\subsection{Power Fr\'echet Means}\label{sec:power}
We consider power Fr\'echet means, i.e., $\tran$-Fr\'echet means with $\tran(x) = x^\alpha$. We restrict to $\alpha\in (1, 2]$, which makes $\tran$ nondecreasing and convex with concave derivative allowing us to use the quadruple and variance inequalities. We exclude the case of the Fr\'echet median, $\alpha=1$, which is special and is treated in \cref{sec:bounded}. We derive convergence rates in expectation with explicit constants, see \cref{thm:power:main,cor:power:stdloss:sharp,rem:power:consts}. We illustrate the main result by applying it in the case $\alpha= \frac32$ in \cref{cor:power:threehalfs}. Moreover \cref{thm:power:lower:lin} shows the influence of a single extreme observation on the $\alpha$-Fr\'echet mean.

\begin{notation}\label{def:power:moments}
    Let $\varphi\in\R$. Use the convention $0^0:= 1$ and $0^{\varphi} := \infty$ for $\varphi<0$. Define the $\varphi$-moment of $Y$ as
    \begin{equation}
        \sigma_{\varphi} := \EOf{\ol Ym^\varphi}
        \eqfs
    \end{equation}
\end{notation}
\begin{remark}
    When we are only interested in whether a $\varphi$-moment is finite or not, the reference point does not matter if $\varphi \geq0$: Using the triangle inequality and \cref{lmm:power:subadd}, 
    \begin{equation}
        \EOf{\ol Yq^\varphi} \leq 2^{\max(0, \varphi-1)} \br{\EOf{\ol Yp^\varphi} + \ol qp^\varphi}
        \qquad \text{for all $q,p\in\mc Q$.}
    \end{equation}
\end{remark}
Now we state our main theorem on the finite-sample error bound in expectation for power Fr\'echet means.
\begin{theorem}\label{thm:power:main}
    Let $\alpha \in (1,2]$ and $\tran(x)=x^\alpha$.
    Set $\phi := \frac{2-\alpha}{\alpha-1} \in \Rp$.
    Assume $\Eof{\ol{Y}o^\alpha}<\infty$.
    Then, for all $n\in\N$,
    \begin{equation}\label{eq:power:main}
        \EOf{\ol{m}{m_n}^2 \br{\sigma_{\alpha-1} + \ol{m}{m_n}^{\alpha-1}}^{-\phi}}
        \leq
        C_\alpha\, n^{-1} \br{\sigma_{\alpha-1}^{\phi}\, \sigma_{2\alpha-2} + n^{-\phi}\, \sigma_\alpha}
        \eqcm
    \end{equation}
    where $C_\alpha\in\Rpp$ is a constant depending only on $\alpha$.
\end{theorem}
\begin{remark}\mbox{}
    \begin{enumerate}[label=(\roman*)]
        \item We measure the deviation of the sample power Fr\'echet mean $m_n$ from its population version with the loss $x \mapsto x^2(\sigma_{\alpha-1} + x^{\alpha-1})^{-\phi}$, which is the lower bound in the variance inequality of \cref{lmm:power:vi}. For $x\to 0$, this loss is the squared loss $\sigma_{\alpha-1}^{-\phi} x^2 (1 + \mo o(1))$. For $x\to \infty$, it is $x^\alpha (1 + \mo o(1))$. Thus, tails are weighted less than with the standard $L^2$-loss.
        \item Even though we effectively bound the $\alpha$-moment of the error, the $\alpha$-moment of $Y$ only has a vanishing contribution to the rate. The rate is dominated by the lower $(2\alpha-2)$-moment.
        \item The correspondence between $\phi$ and $\alpha$ is as follows: $\phi$ decreases from $\infty$ to $0$ as $\alpha$ increases from $1$ to $2$.
        \item Explicit constants for the bound in \cref{thm:power:main} are given in \cref{rem:power:consts}. Let us note here that $\sup_{\alpha\in[1+\epsilon, 2]}C_\alpha < \infty$ for all $\epsilon\in\Rpp$, while $\liminf_{\alpha\searrow1} C_\alpha = \infty$.
    \end{enumerate}
\end{remark}
As the loss in \cref{thm:power:main} is nonstandard, we use \cref{lmm:loss:toalpha} to derive a bound on the more common $L^\alpha$-norm.
\begin{corollary}\label{cor:power:stdloss:sharp}
    Let $\alpha \in (1,2]$ and $\tran(x)=x^\alpha$. Set $\phi := \frac{2-\alpha}{\alpha-1}$.
    Assume $\Eof{\ol{Y}o^\alpha}<\infty$.
    Then, for all $n\in\N$,
    \begin{equation}\label{eq:power:stdloss:sharp}
        \EOf{\ol{m}{m_n}^\alpha}^{\frac1\alpha}
        \leq
        C_\alpha \br{
            \sigma_{\alpha-1}^{\phi}\, \sigma_{2\alpha-2}^{\frac12}\, n^{-\frac12}
            +
            \sigma_{\alpha-1}^{\frac\phi2}\, \sigma_{\alpha}^{\frac12}\, n^{-\frac{1}{2(\alpha-1)}}
            +
            \sigma_\alpha^{\frac1\alpha}\, n^{-\frac{1}{\alpha(\alpha-1)}}
        }
        \eqcm
    \end{equation}
    where $C_\alpha\in\Rpp$ is a constant depending only on $\alpha$.
\end{corollary}
Applying \cref{thm:power:main} with the explicit constants given in \cref{rem:power:consts} yields the following bound for $\alpha = \frac32$ (\cref{lmm:loss:toalpha} is used for the bound in $L^\alpha$-norm).
\begin{corollary}\label{cor:power:threehalfs}
    Set $\tran(x) = x^{\frac32}$.
    Assume $\Eof{\ol Yo^{\frac32}}<\infty$. Then, for all $n\in\N$,
    \begin{equation}
        \EOf{\frac{\ol{m}{m_n}^2}{\sigma_{1/2} + \ol{m}{m_n}^{1/2}}}
        \leq
        \frac{193\, \sigma_1\, \sigma_{\frac12}}{n} + \frac{38\, \sigma_{\frac32}}{n^2}
    \end{equation}
    and
    \begin{equation}
        \EOf{\ol{m}{m_n}^{\frac32}}^{\frac23}
        \leq
        \frac{41\, \sigma_1^{\frac12}\, \sigma_{\frac12}}{n^{\frac12}} 
        + 
        \frac{7\, \sigma_{\frac12}^{\frac12}\, \sigma_{\frac32}^{\frac12}}{n} 
        + 
        \frac{18\, \sigma_{\frac32}^{\frac23}}{n^{\frac43}}
        \eqfs
    \end{equation}
\end{corollary}

While the $\alpha$-Fr\'echet mean is robust with respect to heavy-tailed distributions in that a parametric rate for estimation is achieved with less than $2$ moments, we next show a negative result with regard to robustness with respect to outliers: The sample $\alpha$-Fr\'echet mean depends linearly on extreme samples. This implies that our assumption of a finite $\alpha$-moment is necessary to obtain a rate of convergence in $L^\alpha$-loss such as in \cref{cor:power:stdloss:sharp}.
\begin{theorem}\label{thm:power:lower:lin}
    Let $\tran(x) = x^\alpha$ with $\alpha \in (1,2]$.
    \begin{enumerate}[label=(\roman*)]
        \item 
        Let $y_1, \dots, y_n \in \mc Q$. 
        Assume $n\geq 2$.
        Denote 
        \begin{equation}
            m_{1:n} = \argmin_{q\in\mc Q} \sum_{i=1}^n \ol{y_i}{q}^\alpha
            \qquad\text{and}\qquad
            m_{2:n} = \argmin_{q\in\mc Q} \sum_{i=2}^n \ol{y_i}{q}^\alpha
            \eqfs
        \end{equation}
        Assume 
        \begin{equation}
            \ol {y_1}{m_{2:n}} \geq \br{\frac{16 n}{n-1}\sum_{i=2}^n \ol{y_i}{m_{2:n}}^{\alpha-1}}^{\frac{1}{\alpha-1}}
            \eqfs
        \end{equation}
        Then
        \begin{equation}
            \ol{m_{1:n}}{m_{2:n}} \geq 2^{-\frac{\alpha+4}{\alpha-1}} n^{-\frac{1}{\alpha-1}} \ol{y_1}{m_{2:n}}
            \eqfs
        \end{equation}
        \item 
        Drop the standing assumption $\Eof{\ol Yo^{\alpha-1}} < \infty$.
        Assume there is $\xi\in\Rpp$ such that $\Eof{\ol {Y}o^\xi} = \infty$. Then $\Eof{\ol {m_n}q^\xi} = \infty$ for all $q\in \mc Q$. 
    \end{enumerate}
\end{theorem}
\subsection{General Transformations}\label{sec:tailrobust}
Here we discuss transformations with $\lim_{x\to\infty}\dtran(x) = \infty$ in general. This implies that $\ddrtran(x)>0$ on $\Rp$ and that $\dtran$ is strictly increasing with a well-defined inverse $\invdtran\colon[\dtran(0), \infty) \to \Rp$, which we extend to the domain $\Rp$ by setting $\invdtran(x)=0$ for $x\in [0, \dtran(0))$.
As a moment assumption, we will require conditions similar to $\Eof{\dtran(\ol Ym)^2 / \ddrtran(\ol Ym)} < \infty$. This is at least as strong as $\Eof{\tran(\ol Ym)} < \infty$ (see \cref{lmm:fractranmoment}), which allows us to define the $\tran$-Fr\'echet mean as the minimizer of $q\mapsto \Eof{\tran(\ol Yq)}$.

We define the notation $\sigma_{f}$ for the moment induced by a function $f$, and $\hat\sigma_{f}$ for its empirical version.
\begin{notation}\label{def:infdtr:moments}
    Let $f\colon \Rp \to \Rp$ be a measurable function. 
    Define
    \begin{equation}
        \sigma_{f} := \EOf{f(\ol Ym)} 
        \eqcm\qquad
        \hat\sigma_{f} := \frac1n \sum_{j=1}^n f(\ol {Y_j}m)
        \eqfs
    \end{equation}
\end{notation}
Now we state our main theorem on the excess risk of transformed Fr\'echet means with unbounded influence. The result can be translated to a bound on the error $\ol{m}{m_n}$ using a variance inequality (\cref{thm:infdtr:vi}), see \cref{cor:infdtr:improved} below.
\begin{theorem}\label{thm:infdtr:improved}
    Assume $\lim_{x\to\infty}\dtran(x) = \infty$.
    Denote $g(x) := \ddrtran(7 x)^{-1}$ and $h(x) := g(\invdtran(12 x))$.
    For $p\in\Rpp$, set
    \begin{align}
        S_{n,p} &:= \max\brOf{
            \sigma_{g^p}, 
            2 h(\sigma_{\dtran})^p,
            \EOf{ h(2\hat\sigma_{\dtran})^p}
        }\eqcm
        \\
        V_{n,p} &:= \frac1n  \EOf{\dtran(\ol{Y}m)^{2p} g(\ol{Y}{m})^p}  + 	\EOf{\dtran(\ol{Y}m)^{2p} h(2n^{-1} \dtran(\ol{Y}m))^p}\eqfs
    \end{align}
    Let $\chi \in  \median(\ol Ym)$ be a median of $\ol Ym$. 
    Let $p,q > 1$ such that $\frac1p + \frac1q = 1$.
    Set 
    \begin{equation}
        r_0 := \max\brOf{\chi, 2 \invdtran\brOf{16 \sigma_{\dtran}}}\eqfs
    \end{equation}
    Then
    \begin{equation}
        \EOf{\tran(\ol Y{m_n}) - \tran(\ol Ym)}
        \leq
        \frac{16}{n} \min\brOf{4 \sigma_{(\dtran)^{2}} S_{n,1} + V_{n,1},\, \frac{4 \sigma_{(\dtran)^2}}{\ddrtran(4r_0)} + b_n} 
        \eqcm
    \end{equation}
    where
    \begin{equation}
        b_n := \br{V_{n,p} + 4 \sigma_{(\dtran)^{2p}} S_{n,p}}^{\frac1p} \br{\exp\brOf{-\frac{n}{16}} +  \frac2n \br{\frac{\sigma_{(\dtran)^2}}{\sigma_{\dtran}^2}-1}}^{\frac1q}
        \eqfs
    \end{equation}
\end{theorem}
\begin{remark}
    \begin{enumerate}[label=(\roman*)]
        \item
        Using a variance inequality (\cref{thm:infdtr:vi} \ref{thm:infdtr:vi:default}), we have 
        \begin{equation}
            \EOf{\ol m{m_n}^2 \ddrtran(\ol Ym + \ol m{m_n})} \leq 2 \EOf{\tran(\ol Y{m_n}) - \tran(\ol Ym)}
            \eqcm
        \end{equation}
        which can be further bounded to obtain more usable bounds, see \cref{cor:infdtr:improved} below.
        \item 
        The occurrence of the median $\chi$ is arbitrary insofar as it can be replaced by any other quantile for a probability in $(0,1)$ with a suitable change of constants. By minimizing over these quantile probabilities, some constants could potentially be improved, but the theorem will not have optimal constants either way. Thus, we refrain from further optimization in this regard.
        \item
        Let us be imprecise for the sake of illustrating this result. We approximate $\tran(x) \approx x^2 \ddrtran(x) \approx \dtran(x)^2/\ddrtran(x)$, which is a valid approximation at least for $\tran(x) = x^\alpha$, $\alpha\in(1, 2]$. Then we effectively bound a risk of the loss $\tran$ applied to $\ol m{m_n}$ (for large $\ol m{m_n}$). One might expect a moment term $\Eof{\tran(\ol Ym)}$ to come up in such a risk bound. And indeed it does (in the form of functions related to $\dtran(x)^2/\ddrtran(x)$). But this moment is multiplied by factors that vanish for $n\to\infty$ so that the dominating moment is $\sigma_{(\dtran)^2} = \Eof{\dtran(\ol Ym)^2}$, which is a lower order moment (except when $\tran(x) \approx x^2$ for large $x$). Thus, not only do $\tran$-Fr\'echet means require just a $\tran$-moment instead of a second moment for a parametric rate of convergence, the dominating moment in the rate is of the even lower order $(\dtran)^2$.
    \end{enumerate}
\end{remark}
\begin{corollary}\label{cor:infdtr:improved}
    Use the setting of \cref{thm:infdtr:improved}. The asymptotic notation $\mo o$ and $\mo O$ refers to the limit $n\to\infty$.
    \begin{enumerate}[label=(\roman*)]
        \item\label{cor:infdtr:improved:simple}
        Assume $\EOf{\dtran(\ol{Y}m)^{2} h(\dtran(\ol{Y}m))} < \infty$ and $\lim_{x\searrow0}\ddrtran(x) = \infty$.
        Then
        \begin{equation}
            \EOf{\ol m{m_n}^2  \ddrtran(\chi + \ol m{m_n})}
            \leq
            \frac{256}{n} \br{\sigma_{(\dtran)^{2}} S_{n,1} + \mathbf{o}(1)}
            \eqfs
        \end{equation}
        \item \label{cor:infdtr:improved:low}
        Let $\epsilon>0$.
        Assume
        $\EOf{\dtran(\ol{Y}m)^{2(1+\epsilon)} h(\dtran(\ol{Y}m))^{1+\epsilon}} < \infty$,
        $\sup_{n\in\N}\EOf{h(2\hat\sigma_{\dtran})^{1+\epsilon}} < \infty$,
        and $\sigma_{g^{1+\epsilon}} < \infty$.
        Then
        \begin{equation}
            \EOf{\ol m{m_n}^2  \ddrtran(\chi + \ol m{m_n})}
            \leq
            \frac{256}{n} \br{\sigma_{(\dtran)^2}\ddrtran(4r_0)^{-1} + \mathbf{o}(1)}
            \eqfs
        \end{equation}
        \item\label{cor:infdtr:improved:root}
        Assume that either the conditions of \ref{cor:infdtr:improved:simple} together with $\sup_{n\in\N}\Eof{h(2 \hat \sigma_{\dtran})}<\infty$, or the conditions of \ref{cor:infdtr:improved:low}, are fulfilled.
        Then
        \begin{equation}
            \tran^{-1}\brOf{\EOf{\tran(\ol m{m_n})}} = \mo O(n^{-\frac12})
            \quad\text{and, in particular,}\quad
            \EOf{\ol m{m_n}} = \mo O(n^{-\frac12})
            \eqfs
        \end{equation}
    \end{enumerate}
\end{corollary}
\cref{thm:power:main} on the power Fr\'echet means is effectively an application of part \ref{cor:infdtr:improved:simple} of \cref{cor:infdtr:improved} to $\tran(x)=x^\alpha$, but with additional care taken to improve the constants. As another example, we take the transformation with $\dtran(x) = \log(x+1)$, where $\log$ denotes the natural logarithm.
\begin{example}\label{exa:logp1}
    We have
    \begin{align}
        \tran(x) &= (x + 1) \log(x + 1) - x
        \eqcm
        &\dtran(x) &= \log(x + 1)
        \eqcm\\
        \invdtran(x) &= \exp(x)-1
        \eqcm
        &\ddtran(x) &= \frac{1}{1+x}
        \eqcm\\
        g(x) &= 1 + 7 x
        \eqcm
        &h(x) &= 7\exp(12 x) - 6
        \eqfs
    \end{align}
    Assume $\Eof{\ol {Y}m^{25}} < \infty$.
    Then \cref{cor:infdtr:improved} \ref{cor:infdtr:improved:low} implies
    \begin{equation}
        \EOf{\min\brOf{\ol m{m_n}^2, \ol m{m_n}}}
        \leq
        \frac{C}{n}
        \eqfs
    \end{equation}
    for large enough $n$ with
    \begin{equation}
        C := c \EOf{\log(\ol Ym + 1)^2} \exp\brOf{18 \EOf{\log(\ol Ym + 1)}}
    \end{equation}
    and $c\in\Rpp$ is a universal constant.
\end{example}
\begin{remark}
    If $g$ and $\invdtran$ are subadditive up to a constant, in the sense that $f(x_1 + x_2) \leq c (f(x_1) + f(x_2))$ for $c>0$, then the constants in \cref{thm:infdtr:improved}, e.g., in the definition of $g$ and $h$, play only a minor role. This subadditivity condition is fulfilled for $\tran(x) = x^\alpha$. But, as seen in the example above, where $\invdtran(x) = \exp(x)-1$, it is not always true. In this case, these constants may lead to suboptimal requirements. In the example, the high moment requirement $\Eof{\ol {Y}m^{25}} < \infty$ comes from the condition $\Eof{ h(2\hat\sigma_{\dtran})^{1+\epsilon}} < \infty$ with $\epsilon = 1/24$ and Jensen's inequality for the convex function $x\mapsto \exp(cx)$. Intuitively, this requirement is suboptimal. For a specific transformation, following and adapting the proof of \cref{thm:power:main} and \cref{thm:infdtr:improved} suitably yields preferable results.
\end{remark}
\begin{remark}
    A lower bound analogous to \cref{thm:power:lower:lin} can also be established for general transformations $\tran$ with unbounded influence; see \cref{lmm:general:lower}. However, obtaining a linear lower bound as in \cref{thm:power:lower:lin} requires linearity of $\invdtran(\epsilon \dtran(x))$, specifically,
    \begin{equation}
      \forall \epsilon\in\Rpp \colon \exists \delta\in\Rpp \colon \forall x\geq x_0 \colon \quad \invdtran(\epsilon \dtran(x)) \geq \delta x
    \end{equation}
    for some $x_0\in\Rpp$. While this condition is satisfied for power transformations $\tran(x)=x^\alpha$, it fails for the logarithmic transformation considered in \cref{exa:logp1}. We therefore refrain from stating a general result here and instead refer to \cref{lmm:general:lower}. Improving the theory to obtain tight upper and lower bounds for logarithmic-type transformations such as in \cref{exa:logp1} remains an interesting direction for future research.
\end{remark}

%% file: sec_bounded.tex
\section{Bounded Influence}\label{sec:bounded}
We discuss convergence rates and exponential tail bounds for the $\tran$-Fr\'echet mean $m$ assuming $\lim_{x\to\infty}\dtran(x) < \infty$. Examples of such transformations are the identity (yielding the Fr\'echet median), the Huber loss, and the pseudo-Huber loss. In \cref{sec:largedevi}, we first consider deterministic bounds that quantify the maximum distance of $m$ from a set with mass $>\frac12$. As a corollary, we obtain that such transformed Fr\'echet means have a breakdown point of $\frac12$. Thereafter, we show that the estimator $m_n$ stays in a bounded region around $m$ with high probability. These results will be important for proving rates of convergence in expectation with minimal moment assumptions in \cref{sec:posddrtran}.
\subsection{Breakdown Point and Exponential Tail Bounds}\label{sec:largedevi}
First we find a deterministic bound on the distance between the $\tran$-Fr\'echet mean and a set with high mass.
\begin{notation}
    Denote the diameter of a set $\mc B \subset \mc Q$ as $\diam(\mc B) := \sup_{q,p\in\mc B} \ol qp$ and the distance from a point $p\in\mc Q$ to the set $\mc B$ as $d(p, \mc B) := \inf_{q\in\mc B} \ol qp$.
\end{notation}
Recall \cref{def:convex} for the definition of convex sets in Hadamard spaces. 
\begin{theorem}\label{thm:finiteD:convex}
    Assume $\lim_{x\to\infty}\dtran(x) =: D < \infty$.
    Let $\mc B \subset \mc Q$ be a convex and closed set with diameter $\diam(\mc B) \leq \delta \in \Rp$. Set $\rho := \Prof{Y \in \mc B}$.
    Let $R \in \Rpp$  and $\lambda \in (0,1]$ such that $\tran(R) \geq \lambda D R$.
    Assume $\rho > \frac{1}{1+\lambda}$.
    Set $a := \frac{1-\rho}{\rho}$ and
    \begin{equation}\label{eq:thm:finiteD:convex:x0}
        x_0
        :=
        \frac{\delta}{\lambda - a} \frac{a + \lambda \sqrt{1 -\lambda^2 + a^2}}{a + \lambda}
        \eqfs
    \end{equation}
    Then, 
    \begin{equation}
        d(m, \mc B)^2 \leq \max\brOf{x_0^2, R^2 - \delta^2}
        \eqfs
    \end{equation}
\end{theorem}
\begin{remark}\mbox{}
    \begin{enumerate}[label=(\roman*)]
        \item
        In \cref{thm:finiteD:convex}, we have
        \begin{equation}
            x_0
            \leq 
            \frac{\delta}{\lambda - a}
            \eqfs
        \end{equation}
        \item 
        As $\dtran$ is nondecreasing and $\lim_{x\to\infty}\dtran(x) = D$, the condition on $R$ and $\lambda$ can always be fulfilled: For all $\lambda \in [0,1)$ there is $R\in\Rpp$ such that $\tran(R) \geq \lambda D R$.
        In this case, we have $Dx \geq \tran(x) \geq \lambda D x$ for all $x \geq R$.
    \end{enumerate}
\end{remark}
\begin{corollary}\label{cor:finiteD:convex}
    Let $\tran(x)=x$ so that $m$ is a Fr\'echet median. 
    Let $\mc B \subset \mc Q$ be a convex and closed set with diameter $\diam(\mc B) \leq \delta \in \Rp$. Set $\rho := \Prof{Y \in \mc B}$.
    Assume $\rho > \frac{1}{2}$.
    Then
    \begin{equation}
        d(m, \mc B) \leq 2\rho \delta \frac{1 - \rho}{2\rho - 1}
        \eqfs
    \end{equation}
\end{corollary}
\begin{example}\label{exa:robustmedian}
    Let $\tran(x)=x$ so that $m$ is a Fr\'echet median. 
    Use $\mc Q = \R^2$ with the Euclidean norm $\normof{\cdot}$ and $\Prof{Y = (-1,0)} = \Prof{Y = (1,0)} = \rho/2$ where $\rho\in(\frac12, 1]$. Without knowing anything about the remaining $(1-\rho)$ mass of $Y$, \cref{cor:finiteD:convex} provides us with a bound on the location of $m$ using $\mc B = \setByEleInText{(x,0)}{x\in[-1,1]}$ and $\delta=2$:
    \begin{equation}
        \min_{x\in[-1,1]}\brOf{\normOf{m - \begin{pmatrix} x \\ 0 \end{pmatrix}}} \leq f(\rho)
        \qquad\text{with}\qquad f(\rho) = 4\rho \frac{1 - \rho}{2\rho - 1}
        \eqfs
    \end{equation}
    For example, $f(\frac23) = \frac83$ and $f(\frac34) = \frac32$. This is illustrated in \cref{fig:exa:robustmedian}.
\end{example}
\begin{figure}
    \begin{center}
        \begin{tikzpicture}
            \def\r{8/3}
            \def\s{3/2}
            
            \begin{scope}[opacity=.5, transparency group] 
                \fill[orange!40] (-1,0) circle (\r);
                \fill[orange!40] (1,0) circle (\r);
                \fill[orange!40] (-1,-\r) rectangle (1,\r);
            \end{scope}
            
            \begin{scope}[opacity=.5, transparency group] 
                \fill[orange!40] (-1,0) circle (\s);
                \fill[orange!40] (1,0) circle (\s);
                \fill[orange!40] (-1,-\s) rectangle (1,\s);
            \end{scope}
            
            \draw[->] (-4.5,0) -- (4.5,0) node[below right] {$x$};
            \draw[->] (0,-3.3) -- (0,3.3) node[above right] {$y$};
            
            \draw (-0.1,\r) -- (0.1,\r) node[right] {$\frac{8}{3}$};
            \draw (-0.1,-\r) -- (0.1,-\r) node[right] {$-\frac{8}{3}$};
            \draw (-0.1,\s) -- (0.1,\s) node[right] {$\frac{3}{2}$};
            \draw (-0.1,-\s) -- (0.1,-\s) node[right] {$-\frac{3}{2}$};
            \draw (1+\r,-0.1) -- (1+\r,0.1) node[below] {$\frac{11}{3}$};
            \draw (-1-\r,-0.1) -- (-1-\r,0.1) node[below] {$-\frac{11}{3}$};
            \draw (1+\s,-0.1) -- (1+\s,0.1) node[below] {$\frac{5}{2}$};
            \draw (-1-\s,-0.1) -- (-1-\s,0.1) node[below] {$-\frac{5}{2}$};
            
            \draw[very thick, orange] (-1,0) -- (1,0) node[pos=0.7, above] {$\mathcal{B}$};
            
            \filldraw[green!60!black] (-1,0) circle (2pt) node[below] {-1};
            \filldraw[green!60!black] (1,0) circle (2pt) node[below] {1};
            
            \def\x{6.2}
            \draw[->] (\x,2) -- (2.2,1.8);
            \node[draw, align=left, fill=white] at (\x,2) {
                $\rho = \frac23$
            };
            
            \draw[->] (\x,-2) -- (1.6,-0.8);
            \node[draw, align=left, fill=white] at (\x,-2) {
                $\rho = \frac34$
            };
        \end{tikzpicture}
    \end{center}
    \caption{Illustration of \cref{exa:robustmedian}.}\label{fig:exa:robustmedian}
\end{figure}
Using \cref{thm:finiteD:convex}, we can show that the breakdown point of $\tran$-Fr\'echet means with $\lim_{x\to\infty}\dtran(x) < \infty$ is $1/2$. The breakdown point of a statistic is the fraction of the mass of a probability distribution that an adversary needs to corrupt to let the statistic diverge.
\begin{definition}\label{def:breakdown}
    Let $\epsilon>0$. An \emph{$\epsilon$-contamination} of a probability distribution $P$ on $\mc Q$ is any probability distribution $\tilde P = \check P + \mu$, where $\check P$ is a measure with $\check P(\mc Q) = 1-\epsilon$ and $\check P(B) \leq P(B)$ for all measurable sets $B \subset \mc Q$ and $\mu$ is a measure with $\mu(\mc Q) = \epsilon$.
    
    Let $\mc P$ be a set of probability distributions. Let $T \colon \mc P \to \mc Z$ be a statistic with values in the measurable space $(\mc Z, \Sigma_{\mc Z})$. Let $\delta \colon \mc Z  \times \mc Z \to [0, \infty]$ be a function. The \emph{breakdown point} of $T$ at $P \in \mc P$ with respect to $\mc P$ and $\delta$ is
    \begin{equation}
        \varepsilon(P, \delta, \mc P, T) := \inf\setByEle{\epsilon>0}{\sup\setByEle{\delta(T(P), T(\tilde P))}{\tilde P \in \mc P \text{ is } \epsilon\text{-contamination of } P} = \infty}
        \eqfs
    \end{equation}
\end{definition}
Let $\mc P_0(\mc Q)$ be the set of all probability distributions on $\mc Q$ that are concentrated on a separable set.
For $\tran\in\setcciz$, let $\mc P_{\dtran}(\mc Q)$ be the set of all $P \in \mc P_0(\mc Q)$ such that $\Eof{\dtran(\ol Yq)} < \infty$ for one (and hence all) $q \in \mc Q$, where $Y \sim P$. 
\begin{theorem}\label{thm:breakdown}
    Assume $\diam(\mc Q) = \infty$.
    For $P \in \mc P_{\dtran}(\mc Q)$, let $M(P)$ be the set of $\tran$-Fr\'echet means of $Y \sim P$.
    For $A, B \subset \mc Q$, define
    \begin{equation}
        \delta(A, B) := \sup_{a\in A, b\in B} \ol ab
        \eqfs
    \end{equation}
    \begin{enumerate}[label=(\roman*)]
        \item
        Assume $\lim_{x\to\infty}\dtran(x) < \infty$. Then $ \mc P_{\dtran}(\mc Q)  =  \mc P_0(\mc Q) $ and
        \begin{equation}
            \forall P\in \mc P_0(\mc Q) \colon \varepsilon(P, \delta, \mc P_0(\mc Q), M) = \frac12
            \eqfs
        \end{equation}
        \item
        Assume $\lim_{x\to\infty}\dtran(x) = \infty$. Then
        \begin{equation}
            \forall P\in \mc P_{\dtran}(\mc Q) \colon\varepsilon(P, \delta, \mc P_{\dtran}(\mc Q), M) = 0
            \eqfs
        \end{equation}
    \end{enumerate}
\end{theorem}
We now turn the deterministic bounds on $m$ into exponential tail bounds on $\ol m{m_n}$ using the Chernoff bound. We denote
\begin{equation}
    \kullback(t, p)
    := t \log\brOf{\frac{t}{p}} + (1-t) \log\brOf{\frac{1-t}{1-p}}
    \eqcm
    \qquad t, p \in [0,1]
    \eqcm
\end{equation}
the binary relative entropy, with the conventions $0 \log 0 = 0$ and $\kullback(t,p) = +\infty$ if $t \notin \{0,1\}$ and $p \in \{0,1\}$.
\begin{theorem}\label{thm:tailbound}
    Assume $\lim_{x\to\infty}\dtran(x) =: D < \infty$.
    Let $R \in \Rpp$ and $\lambda \in (0,1]$ such that $\tran(R) \geq \lambda D R$.
    Let $r \in \Rp$ and set $\rho := \Prof{\ol Ym \leq r}$.
    Let $\eta \in (0,1]$.
    Assume $(\lambda + 1)\eta\rho > 1$ and $r \geq \frac12R$.
    Then
    \begin{equation}
        \PrOf{\ol m{m_n} > \br{\frac{\br{3 + \lambda}\eta\rho - 1}{\br{1 + \lambda}\eta\rho - 1}} r}
        \leq
        \exp\brOf{-n \kullback(\eta\rho, \rho)}
        \eqfs
    \end{equation}
\end{theorem}
\cref{thm:tailbound} applied with $\lambda = \frac9{10}$ and $\eta\rho = \frac23$ yields the
following corollary.
\begin{corollary}\label{cor:tailbound}
    Assume $\lim_{x\to\infty}\dtran(x) =: D < \infty$.
    Let $R \in \Rpp$ such that $\tran(R) \geq \frac9{10} D R$.
    Let $r \in \Rp$ such that $\rho := \Prof{\ol Ym \leq r} > \frac23$ and $r \geq \frac12 R$.
    Then
    \begin{equation}
        \PrOf{\ol m{m_n} > 6r}
        \leq
        \br{\frac{27}{4} (1-\rho) \rho^2}^{\frac n3}
        \leq
        \br{2 (1-\rho)^{\frac13}}^{n}
        \eqfs
    \end{equation}
\end{corollary}
\begin{remark}
     For transformations with $\lim_{x\to\infty}\dtran(x) < \infty$, \cref{thm:tailbound} and \cref{cor:tailbound} show that there is a finite radius $r$ depending on $\tran$ and the distribution of $Y$, such that we obtain exponential concentration of $m_n$ within a ball of radius $r$ centered at $m$. Unfortunately, these results do not allow us to reduce $r$ towards $0$ as $n\to\infty$. But, for any $\zeta\in\Rpp$, they ensure $\Eof{\ol{m}{m_n}^\zeta} < \infty$ for large enough $n$ under a minimal moment assumption, see \cref{lmm:posddrtran:expect}.
\end{remark}
Recall that the Fr\'echet median is the $\tran$-Fr\'echet mean with $\tran(x) = x$.
Let $q\in\mc Q$, $r\in \Rp$, and set $\rho := \Prof{\ol Yq \leq r}$. Assume  $\rho > \frac{1}{2}$.
By \cref{cor:finiteD:convex} applied to $\mc B = \ballclosed(q, r)$ with $\delta = 2r$, we have for all Fr\'echet medians $m$,
\begin{equation}
    \ol qm \leq \br{1 + 4 \rho\frac{1 - \rho}{2\rho - 1}} r
    \eqfs
\end{equation}
Combining this bound with the Chernoff bound as in the proof of \cref{thm:tailbound}, we obtain exponential tail bounds for the empirical Fr\'echet median. As $\tran(x) = x$ satisfies $\tran(R) \geq \lambda D R$ with $D = \lambda = 1$ for every $R\in\Rpp$, no condition of the form $r \geq \frac12 R$ is needed, and the resulting radius factor is smaller than the one of \cref{thm:tailbound} with $\lambda = 1$ by exactly $2\eta\rho$:
\begin{theorem}\label{thm:median:tailbound}
    Let $\tran(x)=x$ so that $m$ is a Fr\'echet median. 
    Let $r \in \Rp$ and set $\rho := \Prof{\ol Ym \leq r}$.
    Let $\eta \in (0,1]$.
    Assume $2\eta\rho > 1$.
    Then
    \begin{equation}
        \PrOf{\ol m{m_n} > \br{\frac{6\eta\rho - 1 - 4\eta^2\rho^2}{2\eta\rho - 1}} r}
        \leq
        \exp\brOf{-n \kullback(\eta\rho, \rho)}
        \eqfs
    \end{equation}
\end{theorem}
\cref{thm:median:tailbound} applied with $\eta\rho = \frac23$ yields the following corollary.
\begin{corollary}\label{cor:median:tailbound}
    Let $\tran(x)=x$ so that $m$ is a Fr\'echet median. 
    Let $r \in \Rp$ such that $\rho := \Prof{\ol Ym \leq r} > \frac23$.
    Then
    \begin{equation}
        \PrOf{\ol m{m_n} > \frac{11}3 r}
        \leq
        \br{\frac{27}{4} (1-\rho) \rho^2}^{\frac n3}
        \leq
        \br{2 (1-\rho)^{\frac13}}^{n}
        \eqfs
    \end{equation}
\end{corollary}
\subsection{Convergence Rates}\label{sec:posddrtran}
Here, we investigate the convergence rates for the $\tran$-Fr\'echet mean with $\lim_{x\to\infty}\dtran(x) < \infty$. First we consider $\ddrtran(x) > 0$ for all $x\in\Rpp$, which excludes the median and the standard Huber loss but includes the pseudo-Huber loss. Later we consider the Fr\'echet median separately.
\begin{theorem}\label{thm:posddrtran:main}
    Assume $\lim_{x\to\infty}\dtran(x) < \infty$.
    Assume $\ddrtran(x) > 0$ for all $x\in\Rpp$.
    Assume $\xi \in \Rpp$ exists with $\Eof{\ol Ym^\xi} < \infty$.
    Then
    \begin{equation}
        \EOf{\ol m{m_n}^2} = \mo O\brOf{\frac1n}
        \eqfs
    \end{equation}
\end{theorem}
\begin{remark}
    We need a minimal polynomial moment condition in the form of $\Eof{\ol Ym^\xi}$ for an arbitrarily small $\xi > 0$. This is a weak moment condition, but it excludes distributions with $\Eof{\log(\ol Ym + 1)} = \infty$.
\end{remark}
Next, we examine the rate of convergence for the Fr\'echet median, i.e., the $\tran$-Fr\'echet mean with $\tran(x) = x$.
Since $\ddrtran(x) = 0$, the standard variance inequality \cref{thm:infdtr:vi} is not useful and must be replaced by \cref{thm:median:vi}, whose lower bound involves an integral over $\mc Q$ excluding the bow tie region $\bowtieset(m, q, w)$ (see \cref{def:bowtie}). 

Some events appearing in \cref{thm:median:main} below and its proof, such as the one in \eqref{eq:thm:median:main:bowtie:sample}, involve an existential quantifier or an infimum over $q, p \in \ballclosed(m, R)$ and may fail to be measurable in general. To avoid distracting from the main arguments, we assume that all such sets and functions are measurable.
\begin{theorem}\label{thm:median:main}
    Let $\tran(x)=x$ so that $m$ is a Fr\'echet median. 
    Assume $\xi \in \Rpp$ exists with $\Eof{\ol Ym^\xi} < \infty$.
    Let $r\in\Rpp$ such that $\PrOf{\ol Ym > r} < \frac1{27}$. Set $R:=6r$.
    Assume there are $\ell \in \N$ and $w\in(0,1]$ such that
    \begin{align}
        \label{eq:thm:median:main:bowtie:pop}
        \sup_{p \in \ballclosed(m, R)} \PrOf{Y \in \bowtieset(m, p, w)} &< 1 \qquad \text{and}\\
        \label{eq:thm:median:main:bowtie:sample}
        \PrOf{\exists q,p \in \ballclosed(m, R) \colon Y_1, \dots, Y_\ell \in \bowtieset(q, p, w) \cup \ballclosed(m, R)\compl} &< 1
        \eqfs
    \end{align}
    Then
    \begin{equation}
        \EOf{\ol m{m_n}^2} = \mo O\brOf{\frac1n}
        \eqfs
    \end{equation}
\end{theorem}
A Hadamard manifold is a (finite dimensional) Riemannian manifold that is complete, simply connected, and everywhere of nonpositive sectional curvature.
\begin{corollary}\label{cor:median:notonageodesic}
    Let $\tran(x)=x$ so that $m$ is a Fr\'echet median.
    Let $\mc Q$ be a Hilbert space or a Hadamard manifold.
    Assume $\xi \in \Rpp$ exists with $\Eof{\ol Ym^\xi} < \infty$.
    Assume $Y$ is not concentrated on a geodesic, i.e.,
    \begin{equation}
        \PrOf{Y \in \gamma(\R)} < 1
    \end{equation}
    for every geodesic $\gamma\colon \R \to \mc Q$.
    Then
    \begin{equation}
        \EOf{\ol m{m_n}^2} = \mo O\brOf{\frac1n}
        \eqfs
    \end{equation}
\end{corollary}
\begin{remark}\mbox{}
    \begin{enumerate}[label=(\roman*)]
        \item 
        The bow tie set $\bowtieset(m, q, w)$ is the set of all points on geodesics that intersect $m$ or $q$ at an angle $\alpha_0$ or less, where $\alpha_0$ depends on $w$. See \cite[Remark 6.16]{varinequ} and \cref{fig:bowtie}. In Hilbert spaces, if we set the widening to zero ($w = 0$) then $\alpha_0 = 0$, and $\bowtieset(m, q, w) = \geodft{m}{q}(\R)$, i.e., the bow tie between $m$ and $q$ is the line through $m$ and $q$.
        \item 
        The condition \eqref{eq:thm:median:main:bowtie:pop} roughly translates to $Y$ not being concentrated on a bow tie. For the Fr\'echet median in Hilbert spaces (spatial/geometric median), a typical assumption for convergence results is that $Y$ is not concentrated on a line \cite[Theorem 3.3]{Chakraborty2014}. If this condition holds, we can find a widening $w>0$ small enough such that $Y$ is also not concentrated on any bow tie $\bowtieset(m, p, w)$ with $p \in \ballclosed(m, R)$ (see \cref{cor:median:notonageodesic}, where this argument is also extended to Hadamard manifolds). The restriction to a bounded set, $p \in \ballclosed(m, R)$, is needed as otherwise we could always find a $p$ far enough from $m$ and $Y$ such that the geodesic from $Y$ to $p$ intersects the geodesic between $p$ and $m$ at an arbitrarily small angle.
        \item 
        The condition \eqref{eq:thm:median:main:bowtie:sample} is a sample version of the requirement that $Y$ is not concentrated on a bow tie. In Hilbert spaces and Hadamard manifolds (\cref{cor:median:notonageodesic}), if $Y$ is not concentrated on a line/geodesic, then the probability that the iid sample $Y_1,Y_2,Y_3$ lies on a line is smaller than $1$. Furthermore, we can find $w>0$ small enough so that this statement can be extended from lines to bow ties $\bowtieset(q, p, w)$ with knots $q,p$ in a bounded region $\ballclosed(m, R)$. Thus, we can choose $\ell = 3$ in these spaces.
    \end{enumerate}    
\end{remark}

%% file: sec_discuss.tex
\section{Fast and Optimal Rates}\label{sec:optimal}

In this section, we first demonstrate that, for certain distributions, $\alpha$-Fr\'echet means with $\alpha<2$ may converge at rates faster than the parametric rate. We then prove that the parametric rates derived in the previous sections nevertheless remain optimal in any nondegenerate Hadamard space for a natural class of probability distributions.

\subsection{Fast Rates via Algorithm Stability}\label{ssec:fast}
So far, we have used variance inequalities (VIs) that are of order $\ol qm^2$ for points $q$ close to the $\tran$-Fr\'echet mean $m$. This yields the classical parametric rate of convergence. If $Y$ has large small ball probabilities near $m$, the VI can exhibit steeper growth. For example, in the extreme case of $\Prof{Y=m} = 1$, we have $\Eof{\ol Yq^\alpha - \ol Ym^\alpha} = \ol qm^\alpha$ for $\alpha \in \Rpp$. If a VI with steeper-than-squared growth holds for $q$ close to $m$, we obtain rates of convergence faster than parametric.

\begin{theorem}\label{thm:fast}
    Let $\alpha \in (1,2]$ and $\tran(x)=x^\alpha$.
    Assume $\Eof{\ol{Y}o^\alpha}<\infty$.
    Assume there are $\epsilon,b\in\Rpp$ and $\beta \in [\alpha, 2]$ such that 
    \begin{equation}\label{eq:fast:condition}
        \forall x\in (0,\epsilon]\colon \PrOf{\ol Ym \leq x} \geq b x^{\beta-\alpha}
        \eqfs
    \end{equation}
    Then, there is $C \in \Rpp$ such that 
    \begin{equation}\label{eq:fast:result}
        \EOf{\min\brOf{\ol m{m_n}^\beta, \ol m{m_n}^{\alpha}}}
        \leq
        C n^{-1}
    \end{equation}
    for all $n\in\N$.
\end{theorem}
For random variables $X_n$ with values in $\Rp$ and $a_n\in\Rpp$, we write
$X_n = \Op(a_n)$ if and only if
$\lim_{t\to\infty}\limsup_{n\to\infty}\Prof{X_n > t\, a_n} = 0$.
\begin{remark}
	The result \eqref{eq:fast:result} implies
	\begin{equation}
		\ol m{m_n} = \Op\brOf{n^{-\frac1\beta}}
	\end{equation}
	which is faster than the parametric rate $\Op(n^{-1/2})$ if $\beta < 2$. If $\ol Ym$ has a density, condition \eqref{eq:fast:condition} implies that the density goes to $\infty$ at $0$.
\end{remark}
Similar results can be obtained for other transformations $\tran\in\setcciz$. We do not extend this discussion further, as even these faster rates are sub-optimal for distributions with large small ball probabilities near $m$ \eqref{eq:fast:condition}: In \cref{thm:fast:entropy} we show rates of order $\Op(n^{-\frac1{2(\beta-1)}})$ for finite dimensional Hadamard spaces. 

Next, let us explain where the proof of \cref{thm:fast} loses tightness: We use the VI not only when relating the excess risk to the risk on $\ol m{m_n}$, but also when obtaining a bound on $\ol{m_n}{m_n^i}$, where $m_n^i$ is the $\alpha$-Fr\'echet mean of a sample where the $i$-th entry is replaced by an independent copy. In the second case, we still apply the (potentially suboptimal) quadratic VI; as here, we apply the VI to empirical distributions, we do not directly obtain steepness from condition \eqref{eq:fast:condition}. 
\subsection{Fast Rates via Chaining and Peeling}
So far, we have used the algorithm stability proof technique as in \cite{Escande2024,brunel2025}. Prior to these publications a proof technique based on \textit{chaining} \cite{vaart96, talagrand21} and \textit{peeling} (also called \textit{slicing}) \cite{geer00} was commonly used to obtain convergence rates for Fr\'echet means \cite{Petersen2019,schoetz19,Ahidar2020}. A downside of this approach is that it typically requires an \textit{entropy} upper bound such as a bound on the \textit{entropy integral}
\begin{equation}\label{eq:entropy:inte}
    J(B) = \int_0^{\infty} \!\!\sqrt{\log(N(B,r))} \,\,\dl r
    \eqcm
\end{equation}
where $N(B,r)$ is the covering number of the set $B \subset \mc Q$ with balls of radius $r$, i.e., 
\begin{equation}
    N(B,r) := \min\setByEle{k \in \N}{\exists q_1,\dots, q_k \in \mc Q\colon B \subset \bigcup_{\ell=1}^k \ballclosed(q_\ell, r)}
    \eqfs
\end{equation}
Requiring $J(B) < \infty$ for all balls $B$ of sufficiently small radius effectively requires the space $\mc Q$ to be finite dimensional (it is not fulfilled in infinite dimensional Hilbert spaces). 

We state here a result following the entropy approach in \cite{schoetz19} and using the assumption that $Y$ has large small ball probabilities near $m$. This allows us to obtain even faster rates than what we obtained using the stability approach. The rates shown here for the power Fr\'echet mean, \cref{cor:fast:entropy:power}, are optimal as can be seen from \cref{thm:distri:alpha} below. 

To ensure $\limsup_{\diam(B) \to 0}J(B)/\diam(B)<\infty$, we assume
\begin{equation}\label{eq:entropy}
    \exists s\in(0,2)\colon  
    \limsup_{R\to0} \ \sup_{q\in\mc Q} \ \sup_{r \in (0,R)} \  r^s R^{-s} \log\brOf{N(\ballclosed(q,R),r)} \ < \ \infty
    \eqfs
\end{equation}
This condition is fulfilled, e.g., in Euclidean spaces and the finite dimensional hyperbolic spaces.
\begin{theorem}\label{thm:fast:entropy}
    Let $\tran \in \setcciz$.
    Assume the entropy condition \eqref{eq:entropy}. 
    Assume $\sigma_{(\dtran)^2} := \Eof{\dtran(\ol Ym)^2} < \infty$.
    Use either of the following settings:
    \begin{enumerate}[label=(\roman*)]
        \item
        Set $\beta := 2$.
        Set $x_0 := \inf\setByEleInText{x\in\Rpp}{\ddrtran(x)=0}$.
        Assume $\Prof{\ol Ym < x_0} > 0$.
        \item 
        Let $\beta\in(1,2)$. Assume
        \begin{equation}\label{eq:smallball}
            \liminf_{t\to 0} \, t^{2-\beta} \ddrtran(2t) \PrOf{\ol Ym \leq t} > 0
            \eqfs
        \end{equation}
    \end{enumerate}
    Then
    \begin{equation}
        \ol{m}{m_n} = \Op\brOf{n^{-\frac{1}{2(\beta-1)}}}
        \eqfs
    \end{equation}
\end{theorem}
\begin{corollary}\label{cor:fast:entropy:power}
    Let $\alpha \in (1, 2]$ and set $\tran(x) = x^\alpha$.
    Assume the entropy condition \eqref{eq:entropy}. 
    Assume $\sigma_{2\alpha-2} := \Eof{\ol Ym^{2\alpha-2}} < \infty$.
    Set $\beta := 2$ or let $\beta\in[\alpha,2)$ and assume
    \begin{equation}\label{eq:smallball:alpha}
        \liminf_{t\to 0}  \, t^{\alpha-\beta} \PrOf{\ol Ym \leq t} > 0
        \eqfs
    \end{equation}
    Then
    \begin{equation}
        \ol{m}{m_n} = \Op\brOf{n^{-\frac{1}{2(\beta-1)}}}
        \eqfs
    \end{equation}
\end{corollary}
\begin{remark}
    \begin{enumerate}[label=(\roman*)]
        \item 
        The condition $\PrOf{\ol Ym < x_0} > 0$ (case $\beta=2$ of \cref{thm:fast:entropy}) ensures that the $\tran$-Fr\'echet mean is unique and the VI yields at least a quadratic growth. It is trivially fulfilled for all transformations with $\ddrtran(x) >0$ for all $x\in\Rpp$ such as $\tran(x)=x^\alpha$ with $\alpha\in (1,2]$. But it excludes the Fr\'echet median. For the case $\beta\in(1,2)$, \eqref{eq:smallball} implies $\PrOf{\ol Ym < x_0} > 0$.
        \item 
        The $\Op$-result gives us neither finite-sample bounds nor a bound in expectation. Following the proof as in \cite{schoetz19}, it is possible to obtain explicit tail bounds under stronger entropy conditions. Those do translate to finite sample bounds in expectation, but they tend to be less tight than our main results based on algorithm stability.
        \item 
        As the theorem does not show a result in expectation, the moment condition can be reduced to $\Eof{\dtran(\ol Yq)^2} < \infty$. The $\tran$-Fr\'echet mean is well defined even if $\Eof{\tran(\ol Yq)} = \infty$; we only require $\Eof{\dtran(\ol Yq)} < \infty$ for this, see \cref{ssec:tfm}.
        \item 
        Aside from not being fulfilled or not being efficient in infinite dimensional spaces, if the goal is to obtain rates in expectation or tail bounds, it is not enough to bound $J(B)$ in \eqref{eq:entropy:inte} for small sets $B$ as obtained by \eqref{eq:entropy}, but for all sets, see \cite{schoetz19}. This can weaken the results in negatively curved spaces: In the hyperbolic $k$-space $\mathbb H^k$ with $k\geq2$, we have
        $J(\ballclosed(q,R)) = \frac23 \sqrt{k-1}\, R^{\frac32} (1 + \mo o(1))$ as $R \to \infty$, which would lead to worse tail bounds following \cite{schoetz19}.
        \item 
        Under a suitable lower bound for the small ball probability $\Prof{\ol Ym \leq t}$ and assuming $\ddrtran(t) \xrightarrow{t\to0}\infty$ fast enough, we obtain a faster rate of convergence than the parametric rate and faster than we have proven in \cref{thm:fast} using the algorithm stability approach as can be seen in \cref{cor:fast:entropy:power}.
    \end{enumerate}
\end{remark}
\subsection{Asymptotic Distribution on the Real Line}
In this section, we study the asymptotic distribution of $\tran$-Fr\'echet means on the real line, which allows us to lower-bound the convergence rates in general Hadamard spaces in the subsequent sections.

For a sequence $(Z_n)_{n\in\N}$ of random variables and a distribution or another random variable $D$, we denote convergence in distribution using $Z_n \xrsquigarrow{n\to\infty} D$.
Let $\tran\in\setcciz$.
Let $X,X_1,X_2,\dots$ be independent and identically distributed $\R$-valued random variables.
Assume $\Eof{\dtran(|X|)}<\infty$. 
Let $m \in \argmin_{x\in\R} \Eof{\tran(|X-x|) - \tran(|X|)}$ and $m_n \in \argmin_{x\in\R}\sum_{i=1}^n \tran(|X_i-x|)$ a measurable selection.

\begin{theorem}\label{thm:distri:tran}
    Assume $\dtran(0) = 0$.
    Set $x_0 := \inf\setByEleInText{x\in\Rpp}{\ddrtran(x)=0}$.
    Assume $\Prof{|X-m| < x_0} > 0$.
    Assume $\sigma_{(\dtran)^2} := \Eof{\dtran(|X-m|)^{2}} < \infty$.
    Assume
    \begin{equation}\label{eq:thm:distri:tran:ddrtran}
        \sigma_{\ddrtran} := \Eof{\ddrtran(|X-m|)} < \infty
        \quad\text{and}\quad
        \frac1s \int_0^s \EOf{\ddrtran(|X-m-t|)} \dl t \xrightarrow{s\to0} \sigma_{\ddrtran}
        \eqcm
    \end{equation}
    where we use the convention $\int_0^s = -\int_s^0$ for $s < 0$.
    Then
    \begin{equation}
        \sqrt n\,(m_n-m)
        \xrsquigarrow{n\to\infty}
        \mathcal N\brOf{
            0,
            \frac{\sigma_{(\dtran)^2}}{\sigma_{\ddrtran}^2}
        }
        \eqfs
    \end{equation}
\end{theorem}
\begin{remark}
    \begin{enumerate}[label=(\roman*)]
        \item The condition $\dtran(0)=0$ excludes the median (and other $L^1$-type losses), but is fulfilled for the power Fr\'echet means with $\alpha>1$ and for the (pseudo-)Huber loss. The complementary case $\dtran(0)>0$ requires a somewhat different technical treatment, for which we refer to the classical asymptotic theory of sample medians \cite{Vaart1998, Knight1998}.
        \item The condition $\Prof{|X-m| < x_0} > 0$ ensures uniqueness of the $\tran$-Fr\'echet mean. If $\ddrtran(x)>0$ for all $x\in\Rpp$, we have $x_0=\infty$ and it is trivially fulfilled.
        \item The condition $\sigma_{(\dtran)^2} < \infty$ is a tail bound that ensures at least the parametric convergence rate. If $\lim_{x\to\infty}\dtran(x) < \infty$ such as for the Huber loss, it is trivially fulfilled.
        \item The condition $\sigma_{\ddrtran} < \infty$ ensures that the small ball probabilities near $m$ are not too large so that the convergence rate is not faster than parametric. We further require, via \eqref{eq:thm:distri:tran:ddrtran}, that this second derivative moment can be approached in Ces\`aro sense. A sufficient condition for \eqref{eq:thm:distri:tran:ddrtran} is $\EOf{\ddrtran(|X-m-s|)} \xrightarrow{s\to0} \sigma_{\ddrtran}$. If $X$ has a locally bounded density at $m$ and $\ddrtran$ is locally integrable at 0, both conditions are fulfilled.
    \end{enumerate}
\end{remark}
The next theorem discusses the case of power Fr\'echet means, i.e., the case $\tran(x)=x^\alpha$. The first part follows from \cref{thm:distri:tran}. The second part shows fast rates under large small ball probabilities near $m$.
\begin{theorem}\label{thm:distri:alpha}
    Let $\alpha\in(1,2]$ and set $\tran(x)=x^\alpha$.
    Assume
    \begin{equation}
        \sigma_{2\alpha-2}:=\Eof{|X-m|^{2\alpha-2}}<\infty
        \eqfs
    \end{equation}
    \begin{enumerate}[label=(\roman*)]
        \item
        Assume there is $\beta>2$ such that
        \begin{equation}\label{eq:thm:distri:alpha:slow}
            \lim_{t\searrow0} t^{\alpha-\beta}\,\Prof{|X-m|\le t}=0
            \eqfs
        \end{equation}
        Then $\sigma_{\alpha-2} := \Eof{|X-m|^{\alpha-2}}<\infty$ and
        \begin{equation}
            \sqrt n\,(m_n-m)
            \xrsquigarrow{n\to\infty}
            \mathcal N\brOf{
                0,
                (\alpha-1)^{-2}\,\sigma_{\alpha-2}^{-2}\,\sigma_{2\alpha-2}
            }
            \eqfs
        \end{equation}
        \item\label{thm:distri:alpha:fast}
        Assume there are $\beta\in(\alpha,2)$ and $b\in\Rpp$ such that
        \begin{equation}\label{eq:distri:alpha:tail}
            \frac{\PrOf{0\leq X-m\le t}}{b\,t^{\beta-\alpha}}
            \xrightarrow{t\searrow0}1
            \qquad\text{and}\qquad
            \frac{\PrOf{0\geq X-m\ge-t}}{b\,t^{\beta-\alpha}}
            \xrightarrow{t\searrow0}1
            \eqfs
        \end{equation}
        Then
        \begin{equation}
            n^{\frac1{2(\beta-1)}}\,(m_n-m)
            \xrsquigarrow{n\to\infty}
            C_{\alpha,\beta,b}\,\sign(W)\,|W|^{\frac1{\beta-1}}
            \eqcm
        \end{equation}
        where $W\sim\mathcal N(0,\sigma_{2\alpha-2})$ and
        $C_{\alpha,\beta,b}\in\Rpp$ depends only on $\alpha,\beta,b$.
    \end{enumerate}
\end{theorem}
If we assume $\tran(x) = cx^\alpha(1+\mo o(1))$ for $x\to0$, we could extend \cref{thm:distri:tran} to a part with fast rates as in \cref{thm:distri:alpha}. As this requires $\tran(x)$ to look like $x^\alpha$ close to zero, this does not seem to add much new insight and we keep the asymptotic distribution under fast rates for the canonical case only.

\subsection{Optimality}

The image of a geodesic in $\mc Q$ is convex and isometric to an interval. Furthermore, transformed Fr\'echet means of distributions that are supported within a closed and convex set also lie in that set \cite[Proposition 5.2]{varinequ}. Thus, lower bounds obtained in $\R$ carry over to general Hadamard spaces. In particular, the Portmanteau theorem together with \cref{thm:distri:tran} and \cref{thm:distri:alpha} yields the following corollaries.

\begin{corollary}\label{cor:lower:tran}
    Assume that $Y$ is concentrated on a unit-speed geodesic $\gamma$.
    Assume $\dtran(0) = 0$.
    Set $x_0 := \inf\setByEleInText{x\in\Rpp}{\ddrtran(x)=0}$.
    Assume $\Prof{\ol Ym < x_0} > 0$.
    Assume $\sigma_{(\dtran)^2} := \Eof{\dtran(\ol Ym)^{2}} < \infty$.
    Assume
    \begin{equation}
        \sigma_{\ddrtran} := \Eof{\ddrtran(\ol Ym)} < \infty
        \quad\text{and}\quad
        \frac1s \int_0^s \EOf{\ddrtran(|\gamma^{-1}(Y)-\gamma^{-1}(m)-t|)} \dl t \xrightarrow{s\to0} \sigma_{\ddrtran}
        \eqcm
    \end{equation}
    where we use the convention $\int_0^s = -\int_s^0$ for $s < 0$.
    Let $g \colon \Rp \to \Rp$ be lower semi-continuous.
    Then
    \begin{equation}
        \liminf_{n\to\infty} \EOf{g\brOf{\sqrt{n}\, \ol m{m_n}}} \geq \EOf{g(|Z|)}
        \eqcm
    \end{equation}
    where $Z \sim \mathcal N(0, \sigma_{(\dtran)^2}\sigma_{\ddrtran}^{-2})$.
\end{corollary}
\begin{corollary}\label{cor:lower:alpha}
    Assume that $Y$ is concentrated on a unit-speed geodesic $\gamma$.
    Let $\alpha\in(1,2]$ and set $\tran(x)=x^\alpha$.
    Assume
    \begin{equation}
        \sigma_{2\alpha-2}:=\Eof{\ol Ym^{2\alpha-2}}<\infty
        \eqfs
    \end{equation}
    Let $g \colon \Rp \to \Rp$ be lower semi-continuous.
    \begin{enumerate}[label=(\roman*)]
        \item
        Assume there is $\beta>2$ such that
        \begin{equation}
            \lim_{t\searrow0} t^{\alpha-\beta}\,\Prof{\ol Ym\le t}=0
            \eqfs
        \end{equation}
        Then $\sigma_{\alpha-2} := \Eof{\ol Ym^{\alpha-2}}<\infty$ and
        \begin{equation}
            \liminf_{n\to\infty} \EOf{g\brOf{\sqrt{n}\, \ol m{m_n}}} \geq \EOf{g(|Z|)}
            \eqcm
        \end{equation}
        where $Z \sim \mathcal N(0, (\alpha-1)^{-2} \sigma_{\alpha-2}^{-2} \sigma_{2\alpha-2})$.
        \item
        Assume there are $\beta\in(\alpha,2)$ and $b\in\Rpp$ such that
        \begin{equation}
            \frac{\PrOf{0\leq\gamma^{-1}(Y)-\gamma^{-1}(m)\leq t}}{b t^{\beta-\alpha}}
            \xrightarrow{t\searrow 0}
            1
            \ \text{and}\ 
            \frac{\PrOf{0 \geq \gamma^{-1}(Y)-\gamma^{-1}(m) \geq -t}}{b t^{\beta-\alpha}}
            \xrightarrow{t\searrow 0}
            1
            \eqfs
        \end{equation}
        Then
        \begin{equation}
            \liminf_{n\to\infty} \EOf{g\brOf{n^{\frac1{2(\beta-1)}}\, \ol m{m_n}}} \geq \EOf{g\brOf{C_{\alpha,\beta,b} |W|^{\frac1{\beta-1}}}}
            \eqcm
        \end{equation}
        where $W\sim\mathcal N(0,\sigma_{2\alpha-2})$ and $C_{\alpha,\beta,b} \in\Rpp$ is a constant depending only on $\alpha,\beta,b$.
    \end{enumerate} 
\end{corollary}
Next, we apply \cref{cor:lower:alpha} to a two-point distribution and do explicit calculations for a three point distribution to investigate the optimality of the upper bound in \cref{cor:power:stdloss:sharp}.
\begin{proposition}\label{prop:optimality}
    Let $\alpha \in (1,2)$ and $\tran(x)=x^\alpha$. Set $\phi := \frac{2-\alpha}{\alpha-1}$ and
    \begin{equation}
        \nu_\alpha := \frac1{\alpha(\alpha-1)}-\frac12 \in\Rpp
        \eqfs
    \end{equation}
    Let $h\in\Rpp$ and let $\gamma\colon[-h,h]\to\mc Q$ be a unit-speed geodesic; such $h$ and $\gamma$ exist whenever $\mc Q$ contains more than one element.
    For $\rho\in(0,1]$, let $Y$ be distributed according to
    \begin{equation}
        P^\rho := \frac\rho2\brOf{\delta_{\gamma(-h)}+\delta_{\gamma(h)}}+(1-\rho)\,\delta_{\gamma(0)}
        \eqfs
    \end{equation}
    Then $m = \gamma(0)$ and $\sigma_\varphi = \rho\,h^\varphi$ for all $\varphi\in\Rpp$, so that the first and the last summand of \eqref{eq:power:stdloss:sharp} are
    \begin{equation}
        L_1(n) = \sigma_{\alpha-1}^{\phi}\, \sigma_{2\alpha-2}^{\frac12}\, n^{-\frac12} = \rho^{\phi+\frac12}\, h\, n^{-\frac12}
        \eqcm\qquad
        L_2(n) = \sigma_\alpha^{\frac1\alpha}\, n^{-\frac{1}{\alpha(\alpha-1)}} =  \rho^{\frac1\alpha}\, h\, n^{-\frac1{\alpha(\alpha-1)}}
    \end{equation}
    with $\frac{L_1(n)}{L_2(n)} = \br{n\rho}^{\nu_\alpha}$.
    \begin{enumerate}[label=(\roman*)]
        \item \label{prop:optimality:first}
        \emph{(The first summand is attained.)}
        Let $\rho=1$. Then $\sigma_{\alpha-2}^{-1} = \sigma_{\alpha-1}^{\phi} = h^{2-\alpha}$ and
        \begin{equation}
            \liminf_{n\to\infty} \frac{\EOf{\ol m{m_n}^\alpha}^{\frac1\alpha}}{L_1(n)}
            \geq
            \frac34
            \eqcm\qquad\text{while}\qquad
            \frac{L_2(n)}{L_1(n)} = n^{-\nu_\alpha} \xrightarrow{n\to\infty} 0
            \eqfs
        \end{equation}
        \item \label{prop:optimality:last}
        \emph{(The last summand is attained.)}
        Let $n\geq2$ and $\rho = \rho_n = n^{-2}$. Then
        \begin{equation}
            \EOf{\ol m{m_n}^\alpha}^{\frac1\alpha}
            \geq
            \frac18\, L_2(n)
            \eqcm\qquad\text{while}\qquad
            \frac{L_1(n)}{L_2(n)} = n^{-\nu_\alpha} \xrightarrow{n\to\infty} 0
            \eqfs
        \end{equation}
    \end{enumerate}
\end{proposition}
\begin{remark}
    \begin{enumerate}[label=(\roman*)]
        \item
        \Cref{prop:optimality} shows that terms of the form $\sigma_{\alpha-1}^{\phi} \sigma_{2\alpha-2}^{\frac12} n^{-\frac12}$ (or at least $\sigma_{\alpha-2}^{-1} \sigma_{2\alpha-2}^{\frac12} n^{-\frac12}$) and $\sigma_\alpha^{\frac1\alpha} n^{-\frac{1}{\alpha(\alpha-1)}}$ are required in an upper bound of the $L^\alpha$-risk. This matches two out of three terms in \cref{cor:power:stdloss:sharp}. It is not clear whether the term, $\sigma_{\alpha-1}^{\frac\phi2} \sigma_{\alpha}^{\frac12} n^{-\frac{1}{2(\alpha-1)}}$, in \cref{cor:power:stdloss:sharp} is necessary.
        \item 
        While the rate and the dominating moment of order $(2\alpha-2)$ are optimal, the leading term of \cref{cor:power:stdloss:sharp} involves the moment $\sigma_{\alpha-1}$ of positive order $(\alpha-1)$, whereas the lower bound of \cref{cor:lower:alpha} involves the inverse moment $\sigma_{\alpha-2}^{-1}$ of negative order $(\alpha-2)$. This discrepancy is precisely the one introduced by Jensen's inequality in the bound
        \begin{equation}
            \sigma_{\alpha-2}^{-1}
            \leq
            \sigma_{\alpha-1}^{\frac{2-\alpha}{\alpha-1}}
            \eqfs
        \end{equation}
        For $\alpha<2$, this inequality is strict unless $\ol Ym$ is almost surely constant, as it is for the distribution used in \cref{prop:optimality}\,\ref{prop:optimality:first}. A bound in terms of $\sigma_{\alpha-2}^{-1}$ would thus be strictly stronger for every other distribution; whether such a bound holds, at least for distributions with $\sigma_{\alpha-2}<\infty$, is an open question.
    \end{enumerate}
\end{remark}

\begin{remark}
    The lower bounds in \cref{cor:lower:tran,cor:lower:alpha} above show that the parametric rate in general, and the
    fast rate $n^{-\frac{1}{2(\beta-1)}}$ under large small ball probabilities
    near $m$, are optimal in every Hadamard space with more than one element,
    provided the bound is required to hold uniformly over a suitably rich class of distributions.
    For an individual Hadamard space together with an individual distribution,
    faster rates are of course still possible. \cref{ex:tripod} below provides
    such a case: on the tripod, for distributions that are symmetric about the
    branch point and have vanishing small ball probabilities near $m$, the
    $\alpha$-Fr\'echet mean converges faster than at the parametric rate.
    
    This behavior should not be considered typical. It is caused by the
    stickiness phenomenon \cite{Hotz2013,Huckemann2015}: the branch point
    attracts the empirical $\alpha$-Fr\'echet means so strongly that
    $\Prof{m_n = m} \xrightarrow{n\to\infty} 1$. However, stickiness by itself does
    not always yield super-parametric rates, as \cite{Hotz2013} show the
    parametric rate for the $2$-Fr\'echet mean on open books in a sticky setting.
    Moreover, in Hadamard manifolds (including spaces with a negative curvature upper bound), central limit theorems at the parametric
    rate are available for the $2$-Fr\'echet mean for a general class of
    distributions with finite second moment
    \cite{Bhattacharya2005,Bhattacharya2017}. 
    Thus, in many settings, the parametric rate cannot be improved for
    $\alpha=2$. It remains an open question whether in these settings the parametric rate under small and the fast rate under large small
    ball probabilities near $m$ are similarly universally optimal for
    $\alpha$-Fr\'echet means with $\alpha\in(1,2)$.
\end{remark}

\begin{example}[Stickiness of transformed Fr\'echet means on a tripod]
    \label{ex:tripod}
    Let $\mc Q = (\{1,2,3\}\times[0,\infty))/\sim$ be the tripod obtained by gluing three copies of $\Rp$ at their origins, where $(j,0)\sim(k,0)$ for all $j,k$. We write $o$ for the common origin (the branch point), $(j,r)$ for the point at distance $r$ on leg $j$, and equip $\mc Q$ with the metric
    \begin{equation}
        d\brOf{(j,r),(k,s)}
        =
        \begin{cases}
            |r-s|, & j=k,\\
            r+s,   & j\neq k ,
        \end{cases}
    \end{equation}
    so that $(\mc Q,d)$ is a Hadamard space.
    
    Let $\alpha\in(1,2]$ and set $\tran(x)=x^\alpha$.
    Let $Y=(L,R)$, where $L$ is uniform on $\{1,2,3\}$, $R$ is an $\Rp$-valued random variable with $\Eof{R^{\alpha-1}}<\infty$, and $L,R$ are independent.
    Then $m = o$ because of symmetry.
    
    For positive sequences $(a_n)_{n\in\N}$ and $(b_n)_{n\in\N}$ we write
    $a_n=\Theta(b_n)$ if there are constants $0<c\leq C<\infty$, independent of $n$, such
    that $c b_n\leq a_n\leq C b_n$ for all sufficiently large $n$; the constants are
    allowed to depend on $\alpha$ and the distribution of $R$.
    
    \begin{enumerate}[label=(\roman*)]
        \item\emph{(Point masses.)} 
        Let $R = 1$ almost surely. Then, $\ol{m}{m_n} \leq 1$ almost surely, and, with $B_n\sim\mathrm{Bin}(n,1/3)$,
        \begin{equation}
            \PrOf{m_n\neq m} = 3 \PrOf{B_n > \frac n2}\eqcm
            \qquad\text{where}\qquad
            \PrOf{B_n > \frac n2} = \Theta\brOf{n^{-\frac12}\br{\frac89}^{\frac n2}}
            \eqfs
        \end{equation}
        Moreover,
        \begin{equation}
            \EOf{\ol{m}{m_n}^{\alpha}}
            =
            \Theta\brOf{n^{-\alpha-\frac12}\br{\frac89}^{\frac n2}}
            \eqfs
        \end{equation}        
        Thus, the convergence rate in $L^\alpha$-norm is exponential.
        \item\emph{(Pareto legs.)}
        Let $\lambda > \alpha$.
        Let $\Prof{R\geq r}=r^{-\lambda}$ for $r\geq1$ and $\Prof{R < 1} = 0$.
        Then $\Eof{\ol Ym^\alpha} < \infty$ and
        \begin{equation}
            \Prof{m_n\neq m} = \Theta\brOf{n^{1-\frac{\lambda}{\alpha-1}}}
            \eqfs
        \end{equation}
        Moreover,
        \begin{equation}
            \Eof{\ol{m}{m_n}^{\alpha}} = \Theta\brOf{n^{1-\frac{\lambda}{\alpha-1}}}
            \eqfs
        \end{equation}
        Thus, the convergence rate in $L^\alpha$-norm is $\mo o(n^{-\frac{1}{\alpha(\alpha-1)}})$, which is faster than parametric.
    \end{enumerate}
\end{example}

%% file: sec_powerfun.tex
\section{Elementary Properties of Power Functions}\label{sec:powfun}
\begin{lemma}\label{lmm:power:subadd}
    Let $x_1, x_2 \in\Rp$ and $\varphi\in\R$. 
    \begin{enumerate}[label=(\roman*)]
        \item Assume $\varphi \geq 0$. 
        Then
        \begin{equation}
            (x_1 + x_2)^\varphi \leq 2^{\max(0, \varphi-1)}(x_1^\varphi + x_2^\varphi)
        \end{equation}
        with the convention $0^0:=1$.
        \item Assume $\varphi \leq 0$ and $x_1, x_2 >0$.
        Then
        \begin{equation}
            (x_1 + x_2)^\varphi \leq 2^{\varphi-1}(x_1^\varphi + x_2^\varphi)
            \eqfs
        \end{equation}
    \end{enumerate}
\end{lemma}
\begin{proof}
    The assertions follow directly for $\varphi=0$. If
    $\varphi\in(0,1]$, the function $x\mapsto x^\varphi$ is
    subadditive on $\Rp$, and hence
    \begin{equation}
        (x_1+x_2)^\varphi
        \leq
        x_1^\varphi+x_2^\varphi
        \eqfs
    \end{equation}
    If $\varphi>1$, convexity of $x\mapsto x^\varphi$ and Jensen's
    inequality give
    \begin{equation}
        \br{\frac{x_1+x_2}{2}}^\varphi
        \leq
        \frac{x_1^\varphi+x_2^\varphi}{2},
    \end{equation}
    which proves \textup{(i)}. If $\varphi<0$, the same argument applies
    on $(0,\infty)$ and proves \textup{(ii)}.
\end{proof}
\begin{lemma}\label{lmm:power:ana:diff}
    Let $x_1, x_2 \in\Rp$ and $\varphi\in[1, \infty)$. Assume $x_1 \neq x_2$.
    Use the convention $0^0:=1$.
    Then
    \begin{equation}
        \min\brOf{1, \frac\varphi2}
        \leq \frac{\absof{x_1^\varphi - x_2^\varphi}}{\absof{x_1 - x_2}\br{x_1^{\varphi-1} + x_2^{\varphi-1}}} \leq 
        \max\brOf{1, \frac\varphi2}
    \end{equation}
    and
    \begin{equation}
        \min\brOf{1, \frac\varphi{2^{\varphi-1}}}
        \leq \frac{\absof{x_1^\varphi - x_2^\varphi}}{\absof{x_1 - x_2}\br{x_1 + x_2}^{\varphi-1}} \leq 
        \max\brOf{1, \frac\varphi{2^{\varphi-1}}}
        \eqfs
    \end{equation}
\end{lemma}
\begin{proof}
    By symmetry, we may assume $x_1>x_2$. Set
    $p:=\varphi-1$ and $t:=x_2/x_1\in[0,1)$.     
    For the first quotient, write
    \begin{equation}
        Q_1
        :=
        \frac{x_1^\varphi-x_2^\varphi}
        {(x_1-x_2)(x_1^p+x_2^p)}
        =
        \frac{1-t^{p+1}}{(1-t)(1+t^p)}
        =
        \frac{\varphi}{1+t^p}
        \frac{1}{1-t}\int_t^1 s^p\,\mathrm{d}s
        \eqfs
    \end{equation}
    By the Hermite--Hadamard inequality, applied to $s\mapsto s^p$,
    \begin{equation}
        Q_1
        \begin{cases}
            \geq \varphi/2, & 0\leq p\leq1,\\
            \leq \varphi/2, & p\geq1.
        \end{cases}
    \end{equation}
    Moreover,
    \begin{equation}
        Q_1-1
        =
        \frac{t-t^p}{(1-t)(1+t^p)}
        \begin{cases}
            \leq0, & 0\leq p\leq1,\\
            \geq0, & p\geq1.
        \end{cases}
    \end{equation}
    Since $p=\varphi-1$, these two estimates prove
    \begin{equation}
        \min\brOf{1,\frac{\varphi}{2}}
        \leq Q_1\leq
        \max\brOf{1,\frac{\varphi}{2}}
        \eqfs
    \end{equation}
    
    For the second quotient, put
    \begin{equation}
        u:=\frac{t}{1+t},
        \qquad
        v:=\frac{1}{1+t},
        \qquad
        h:=\frac{v-u}{2}.
    \end{equation}
    Then $u=1/2-h$, $v=1/2+h$, and
    \begin{equation}
        Q_2
        :=
        \frac{x_1^\varphi-x_2^\varphi}
        {(x_1-x_2)(x_1+x_2)^p}
        =
        \frac{\varphi}{v-u}\int_u^v s^p\,\mathrm{d}s
        =
        \varphi G_p(h),
    \end{equation}
    where
    \begin{equation}
        G_p(h)
        :=
        \frac{1}{2h}
        \int_{\frac12-h}^{\frac12+h}s^p\,\mathrm{d}s
        \eqfs
    \end{equation}
    The Hermite--Hadamard inequality shows that $G_p$ is nonincreasing
    for $0\leq p\leq1$ and nondecreasing for $p\geq1$. Since
    \begin{equation}
        \lim_{h\searrow0}G_p(h)=2^{-p},
        \qquad
        G_p\brOf{\frac12}
        =
        \int_0^1s^p\,\mathrm{d}s
        =
        \frac{1}{p+1}
        =
        \frac{1}{\varphi},
    \end{equation}
    it follows that
    \begin{equation}
        \min\brOf{\frac{1}{\varphi},\frac{1}{2^p}}
        \leq G_p(h)\leq
        \max\brOf{\frac{1}{\varphi},\frac{1}{2^p}}
        \eqfs
    \end{equation}
    Multiplication by $\varphi$ and the identity $p=\varphi-1$ give
    \begin{equation}
        \min\brOf{1,\frac{\varphi}{2^{\varphi-1}}}
        \leq Q_2\leq
        \max\brOf{1,\frac{\varphi}{2^{\varphi-1}}}
        \eqfs
    \end{equation}
    This proves both assertions.
\end{proof}
\begin{lemma}\label{lmm:power:ana:farvi}
    Let $x_1,x_2\in\Rp$ and $\alpha\in[1,2]$, and use the convention
    $0^0:=1$. Then:
    \begin{enumerate}[label=(\roman*)]
        \item
        \begin{equation}\label{eq:power:dsubadd}
            (x_1+x_2)^{\alpha-1}
            \leq
            x_1^{\alpha-1}+x_2^{\alpha-1}
            \leq
            2^{2-\alpha}(x_1+x_2)^{\alpha-1}
            \eqcm
        \end{equation}
        
        \item
        \begin{equation}\label{eq:power:farvi}
            x_1^\alpha+x_2^\alpha
            \leq
            \absof{x_1-x_2}^{\alpha}
            +
            2^{2-\alpha}\alpha
            \min\brOf{x_1,x_2}
            \max\brOf{x_1,x_2}^{\alpha-1}
            \eqcm
        \end{equation}
        
        \item
        \begin{equation}\label{eq:power:diff:alpha:sum}
            \frac{\alpha}{2}
            \absof{x_1-x_2}
            \br{x_1^{\alpha-1}+x_2^{\alpha-1}}
            \le
            \absof{x_1^\alpha-x_2^\alpha}
            \le
            \absof{x_1-x_2}
            \br{x_1^{\alpha-1}+x_2^{\alpha-1}}
            \eqcm
        \end{equation}
        
        \item
            \begin{equation}\label{eq:power:diff:alpha}
            \absof{x_1-x_2}
            (x_1+x_2)^{\alpha-1}
            \le
            \absof{x_1^\alpha-x_2^\alpha}
            \le
            2^{1-\alpha}\alpha
            \absof{x_1-x_2}
            (x_1+x_2)^{\alpha-1}
            \eqfs
        \end{equation}
    \end{enumerate}
\end{lemma}
\begin{proof}
    \begin{enumerate}[label=(\roman*)]
        \item
        Since $\alpha-1\in[0,1]$, the first inequality in
        \eqref{eq:power:dsubadd} follows from
        \cref{lmm:power:subadd}. The second is Jensen's inequality.
        \item 
        To prove \eqref{eq:power:farvi}, set
        $b:=\min\brOf{x_1,x_2}$ and $B:=\max\brOf{x_1,x_2}$. Then
        $\absof{x_1-x_2}=B-b$. By \eqref{eq:power:dsubadd}, for every
        $s\in[0,b]$,
        \begin{equation}
            s^{\alpha-1}+(B-s)^{\alpha-1}
            \leq
            2^{2-\alpha}B^{\alpha-1}.
        \end{equation}
        Consequently,
        \begin{align}
            x_1^\alpha+x_2^\alpha-\absof{x_1-x_2}^{\alpha}
            &=
            B^\alpha+b^\alpha-(B-b)^\alpha \\
            &=
            \alpha\int_0^b
            \br{s^{\alpha-1}+(B-s)^{\alpha-1}}
            \,\mathrm{d}s \\
            &\leq
            2^{2-\alpha}\alpha bB^{\alpha-1},
        \end{align}
        which proves \eqref{eq:power:farvi}.
        \item 
        Equation \eqref{eq:power:diff:alpha:sum} follows directly from
        \cref{lmm:power:ana:diff}, because
        \begin{equation}
            \min\brOf{1,\frac{\alpha}{2}}
            =
            \frac{\alpha}{2},
            \qquad
            \max\brOf{1,\frac{\alpha}{2}}
            =
            1
        \end{equation}
        for $\alpha\in[1,2]$.
        \item 
        Similarly, \eqref{eq:power:diff:alpha} follows directly from
        \cref{lmm:power:ana:diff}, because
        \begin{equation}
            \min\brOf{1,\frac{\alpha}{2^{\alpha-1}}}
            =1,
            \qquad
            \max\brOf{1,\frac{\alpha}{2^{\alpha-1}}}
            =
            \frac{\alpha}{2^{\alpha-1}}
        \end{equation}
        for $\alpha\in[1,2]$.
    \end{enumerate}
\end{proof}

%% file: sec_tran.tex
\section{Elementary Properties of Nondecreasing Convex Functions with Concave Derivative}\label{sec:tran}

Let $\tran\in\setcciz$ as defined in \cref{def:tran}.

\begin{lemma}[{\cite[Lemma 3]{quadruple}}]\label{lmm:dtran:subadd}
    For $x_1, x_2\in\Rp$, we have 
    \begin{equation}
        \dtran(x_1 + x_2) \leq \dtran(x_1) + \dtran(x_2) \leq 2\dtran\brOf{\frac{x_1 + x_2}2}
    \end{equation}
\end{lemma}
\begin{lemma}[{\cite[Lemma 3]{quadruple}}]\label{lmm:dtran:factor}
    For $x,a\in\Rp$, we have 
    \begin{align}
        \dtran(ax) \geq a\dtran(x) &\text{if } a\leq1 \eqcm\\
        \dtran(ax) \leq a\dtran(x) &\text{if } a\geq1 \eqfs
    \end{align}
\end{lemma}
\begin{lemma}[{\cite[Lemma 2]{quadruple}}]\label{lmm:tran:diff}
    For $x_1, x_2\in\Rp$, we have 
    \begin{equation}
        \absof{x_1 - x_2}\frac{\dtran(x_1)+\dtran(x_2)}2 \leq \absof{\tran(x_1) - \tran(x_2)} \leq \absof{x_1 - x_2}\dtran\brOf{\frac{x_1+x_2}2}
        \eqfs
    \end{equation}    
\end{lemma}
\begin{lemma}[{\cite[Lemma 2]{quadruple}}]\label{lmm:tran:sqr}
    The function $x\mapsto \tran(x)/x^2$ is nonincreasing. Furthermore, for $x_1, x_2\in\Rpp$, we have 
    \begin{equation}
        \absof{\tran(x_1) - \tran(x_2)} \leq \absOf{x_1 - x_2} \br{\frac{\tran(x_1)}{x_1} + \frac{\tran(x_2)}{x_2}}
        \eqfs
    \end{equation}    
\end{lemma}
\begin{proof}
    The derivative of $g(x) := \tran(x)/x^2$ is
    \begin{equation}
        g\pr(x) = \frac{x\dtran(x) - 2\tran(x)}{x^3}
        \eqfs
    \end{equation}
    As $\dtran$ is concave, we have
    \begin{equation}
        \frac{x}{2} \br{\dtran(x)+\dtran(0)} \leq \int_0^x \dtran(z) \dl z = \tran(x)
        \eqfs
    \end{equation}
    Hence, $g\pr(x) \leq 0$, and $g$ is nonincreasing.
    
    Now to the second part of the lemma: Without loss of generality, assume $x_1 > x_2 > 0$. Then, because of the first part of the lemma,
    \begin{equation}
        0
        \geq
        x_1 x_2 \br{\frac{\tran(x_1)}{x_1^2} - \frac{\tran(x_2)}{x_2^2}}
        = 
        \br{\tran(x_1) - \tran(x_2)} - (x_1-x_2) \br{\frac{\tran(x_1)}{x_1} + \frac{\tran(x_2)}{x_2}}
        \eqcm
    \end{equation}
    which shows the claim.    
\end{proof}

\begin{lemma}\label{lmm:tran:factor}
    Let $a,x\in\Rp$.
    \begin{enumerate}[label=(\roman*)]
        \item If $a\in[0,1]$, then
        \begin{equation}\label{eq:tran:factor:small}
            a^2\tran(x)
            \leq
            \tran(ax)
            \leq
            a\tran(x)
            \eqfs
        \end{equation}
        
        \item If $a\geq1$, then
        \begin{equation}\label{eq:tran:factor:large}
            a\tran(x)
            \leq
            \tran(ax)
            \leq
            a^2\tran(x)
            \eqfs
        \end{equation}
        
        \item For $x_1,x_2\in\Rp$, set
        \begin{equation}
            b:=\min\brOf{x_1,x_2},
            \qquad
            B:=\max\brOf{x_1,x_2}.
        \end{equation}
        If $B>0$, then
        \begin{equation}\label{eq:tran:diff:tran}
            \frac{B-b}{B}\tran(B)
            \leq
            \absof{\tran(x_1)-\tran(x_2)}
            \leq
            \frac{B^2-b^2}{B^2}\tran(B)
            \eqfs
        \end{equation}
    \end{enumerate}
\end{lemma}
\begin{proof}
    By convexity and $\tran(0)=0$, the function
    $x\mapsto\tran(x)/x$ is nondecreasing on $\Rpp$. By
    \cref{lmm:tran:sqr}, the function
    $x\mapsto\tran(x)/x^2$ is nonincreasing on $\Rpp$.
    
    Let first $a\in(0,1]$ and $x>0$. The two monotonicity properties
    yield
    \begin{equation}
        \frac{\tran(ax)}{ax}
        \leq
        \frac{\tran(x)}x
        \qquad\text{and}\qquad
        \frac{\tran(ax)}{a^2x^2}
        \geq
        \frac{\tran(x)}{x^2},
    \end{equation}
    and hence
    \begin{equation}
        a^2\tran(x)
        \leq
        \tran(ax)
        \leq
        a\tran(x)
        \eqfs
    \end{equation}
    This proves \eqref{eq:tran:factor:small}. The cases $a=0$ or
    $x=0$ follow from $\tran(0)=0$.
    
    Let now $a\geq1$ and $x>0$. The same monotonicity properties give
    \begin{equation}
        \frac{\tran(ax)}{ax}
        \geq
        \frac{\tran(x)}x
        \qquad\text{and}\qquad
        \frac{\tran(ax)}{a^2x^2}
        \leq
        \frac{\tran(x)}{x^2},
    \end{equation}
    which proves \eqref{eq:tran:factor:large}.
    
    Finally, assume $B>0$. Applying part~\textup{(i)} with $a=b/B$ and $x=B$ gives
    \begin{equation}
        \frac{b^2}{B^2}\tran(B)
        \leq
        \tran(b)
        \leq
        \frac{b}{B}\tran(B).
    \end{equation}
    Since $\tran$ is nondecreasing, we have
    \begin{equation}
        \absof{\tran(x_1)-\tran(x_2)}
        =
        \tran(B)-\tran(b).
    \end{equation}
    Subtracting the preceding bounds from $\tran(B)$ proves
    \eqref{eq:tran:diff:tran}.
\end{proof}
\begin{lemma}\label{lmm:tran:tran}
	Let $x \in \Rp$. Then
	\begin{equation}
		\frac x2 \dtran(x)
		\leq
		\tran(x)
		\leq 
		x\dtran\brOf{\frac{x}2}
		\leq 
		4\tran\brOf{\frac{x}2}
	\quad\text{and}\quad
		x \dtran(2x)
		\leq
		\tran(2x)
		\leq 
		2x\dtran\brOf{x}
		\leq 
		4\tran\brOf{x}
		\eqfs
	\end{equation}
\end{lemma}
\begin{proof}
	\cref{lmm:tran:diff} implies
	\begin{equation}
		\frac{x-y}{2}\br{\dtran(x)+\dtran(y)} \leq \tran(x)-\tran(y) \leq (x-y) \dtran\brOf{\frac{x+y}2}
	\end{equation}
	for $x \geq y \geq 0$.
	Setting $y=0$ and using $\tran(0) = 0$ as well as $\dtran(0) \geq 0$, we obtain
	\begin{equation}
		\frac{x}{2}\dtran(x) \leq \tran(x) \leq x \dtran\brOf{\frac{x}2}
	\end{equation}
	for all $x\geq 0$. Applying this inequality twice yields
	\begin{equation}
		\frac12x\dtran(x) \leq \tran(x) \leq x \dtran\brOf{\frac{x}2} = 4 \frac12\frac{x}{2} \dtran\brOf{\frac{x}2}  \leq 4 \tran\brOf{\frac{x}2}
		\eqfs
        \qedhere
	\end{equation}
\end{proof}
\begin{lemma}\label{lmm:tran:add}
    For $x,y\in\Rp$, we have
    \begin{equation}\label{eq:tran:add}
        \tran(x)+\tran(y)
        \leq
        \tran(x+y)
        \leq
        2\tran(x)+2\tran(y)
        \eqfs
    \end{equation}
\end{lemma}
\begin{proof}
    The assertion is immediate if $x+y=0$. Otherwise, part~\textup{(i)}
    of \cref{lmm:tran:factor} gives
    \begin{equation}
        \tran(x)
        =
        \tran\brOf{\frac{x}{x+y}(x+y)}
        \leq
        \frac{x}{x+y}\tran(x+y)
    \end{equation}
    and, analogously,
    \begin{equation}
        \tran(y)
        \leq
        \frac{y}{x+y}\tran(x+y).
    \end{equation}
    Adding proves the lower bound.
    
    By convexity and \eqref{eq:tran:factor:large},
    \begin{align}
        \tran(x+y)
        &=
        \tran\brOf{2\frac{x+y}{2}}\\
        &\leq
        4\tran\brOf{\frac{x+y}{2}}\\
        &\leq
        2\tran(x)+2\tran(y),
    \end{align}
    which proves the upper bound.
\end{proof}
\begin{lemma}\label{lmm:tran:sqrddtran}
	 Let $x, x_0 \in \Rpp$ with $x_0 \leq x$.
	 Then
	 \begin{equation}
	 	\frac12 x^2 \ddrtran(x) 
	 	\leq 
	 	\tran(x) 
	 	\leq 
	 	\tran(x_0)
	 	+
	 	\dtran(x_0) (x- x_0)
	 	+
	 	\frac12 \ddrtran(x_0)(x - x_0)^2
	 	\eqfs
	 \end{equation}
\end{lemma}
\begin{proof}
    The case $x_0=x$ in the upper bound is trivial. Let therefore
    $0<x_0<x$. Since $\dtran$ is concave, it is locally Lipschitz, and hence
    absolutely continuous, on compact subintervals of $\Rpp$. Moreover, its
    right derivative $\ddrtran$ is nonincreasing and agrees a.e. with the
    Lebesgue derivative of $\dtran$. Thus,
    \begin{equation}
        \tran(x)
        =
        \tran(x_0)
        +
        \dtran(x_0)(x-x_0)
        +
        \int_{x_0}^x \int_{x_0}^t \ddrtran(s)\,\dl s\,\dl t .
    \end{equation}
    
    Since $\tran(x_0)\geq0$, $\dtran(x_0)\geq0$, and $\ddrtran$ is
    nonincreasing, we have $\ddrtran(s)\geq\ddrtran(x)$ for
    $x_0\leq s\leq t\leq x$. Hence
    \begin{equation}
        \tran(x)
        \geq
        \ddrtran(x)
        \int_{x_0}^x\int_{x_0}^t 1\,\dl s\,\dl t
        =
        \frac12 (x-x_0)^2 \ddrtran(x).
    \end{equation}
    Letting $x_0\downarrow0$ gives
    \begin{equation}
        \tran(x)\geq \frac12 x^2\ddrtran(x).
    \end{equation}
    
    On the other hand, since $\ddrtran$ is nonincreasing,
    $\ddrtran(s)\leq\ddrtran(x_0)$ for $x_0\leq s\leq t\leq x$. Therefore
    \begin{align}
        \tran(x)
        &\leq
        \tran(x_0)
        +
        \dtran(x_0)(x-x_0)
        +
        \ddrtran(x_0)
        \int_{x_0}^x\int_{x_0}^t 1\,\dl s\,\dl t
        \\
        &=
        \tran(x_0)
        +
        \dtran(x_0)(x-x_0)
        +
        \frac12 \ddrtran(x_0)(x-x_0)^2 .
    \end{align}
\end{proof}

%% file: sec_loss.tex
\section{Auxiliary Results for Nonstandard Risks}
\begin{lemma}\label{lmm:loss:toalpha}
    Let $\alpha\in(1,2]$ and $b\in\Rpp$. Set $\phi:=\frac{2-\alpha}{\alpha-1}$ and
    \begin{equation}
        \psi\colon\Rp\to\Rp, x\mapsto x^2(b+x^{\alpha-1})^{-\phi}
        \eqfs
    \end{equation}
    \begin{enumerate}[label=(\roman*)]
        \item
        Let $X$ be a nonnegative random variable. 
        Then
        \begin{equation}
            \EOf{\min\brOf{X^\alpha, X^2}} \leq (1+b)^\phi \EOf{\psi(X)}
        \end{equation}
        and
        \begin{equation}
            \EOf{X^\alpha} \leq b^{\phi\alpha/2} \EOf{\psi(X)}^{\alpha/2} \br{1+b^{-\alpha/2}\EOf{\psi(X)}^{(\alpha-1)/2}}^{\phi}
            \eqfs
        \end{equation}
        \item
        Let $(X_n)_{n\in\N}$ be a sequence of nonnegative random variables.
        Assume $\Eof{\psi(X_n)} = \mo o(1)$.
        Then
        \begin{equation}
            \EOf{X_n^\alpha}^{\frac{2}{\alpha}}
            \leq 
            b^{\phi} \Eof{\psi(X_n)} \br{1+\mo o(1)}
            \eqfs
        \end{equation}
    \end{enumerate}
\end{lemma}
\begin{proof}
    For $\alpha = 2$ we have $\phi = 0$ and $\psi(x) = x^2$, so all three
    assertions hold with equality. Assume $\alpha\in(1,2)$.
    For the first inequality of (i), we use the pointwise inequality
    \begin{equation}
        \min\brOf{x^\alpha, x^2} \leq (1+b)^\phi \psi(x)
        \eqfs
    \end{equation}
    Now we show the second inequality of (i). For $x\in\Rpp$, a direct calculation shows
    \begin{equation}
        \psi\pr(x) = x(b+x^{\alpha-1})^{-\phi-1}\br{2b+\alpha x^{\alpha-1}}>0\eqfs
    \end{equation}
    As $\psi(x) \xrightarrow{x\to0} 0$ and $\psi(x) \xrightarrow{x\to\infty} \infty$, $\psi$ is a strictly increasing bijection with inverse function $\psi^{-1}$. As $b > 0$, $\psi \in \mc C^\infty(\Rpp)$. As $\psi\pr(x)>0$ for all $x\in\Rpp$, by the inverse function theorem $\psi^{-1} \in \mc C^\infty(\Rpp)$. Set $g:=(\psi^{-1})^\alpha$. Both $\psi^{-1},g$ extend continuously by $\psi^{-1}(0)=g(0)=0$.
    
    Set $r := \Eof{\psi(X)}$.
    Below, we show that $g$ is concave as $g\prr < 0$. Then we use Jensen's inequality to show
    \begin{equation}
        \EOf{X^\alpha} = \EOf{g(\psi(X))} \leq g(\EOf{\psi(X)}) = g(r)
        \eqfs 
    \end{equation}
    Finally, we show an upper bound $g(r) \leq b^{\phi\alpha/2} r^{\alpha/2} \br{1+b^{-\alpha/2}r^{(\alpha-1)/2}}^{\phi}$ to conclude the proof.
    
    \underline{\smash{Concavity of $g$:}} Differentiating $g(\psi(x))=x^\alpha$ yields
    \begin{equation}
        g\pr(\psi(x))
        =
        \alpha h(x)
        \eqcm\quad
        h(x)
        :=
        \frac{x^{\alpha-1}}{\psi\pr(x)}
        =
        \frac{x^{\alpha-2}(b+x^{\alpha-1})^{\phi+1}}{2b+\alpha x^{\alpha-1}}
        \eqfs
    \end{equation}
    Differentiating $g\pr(\psi(x)) = \alpha h(x)$ yields
    \begin{equation}
        g\prr(\psi(x))\psi\pr(x)
        =
        \alpha h\pr(x)
        \eqfs
    \end{equation}
    Since $\psi$ is an increasing bijection, $g''<0$ on $\Rpp$ iff $h'<0$ on $\Rpp$.
    Using $(\phi+1)(\alpha-1) = 1$, we obtain
    \begin{equation}
        \frac{\dl}{\dl x}\br{\log(h(x))}
        =
        \frac{\alpha-2}{x}+\frac{x^{\alpha-2}}{b+x^{\alpha-1}}-\frac{\alpha(\alpha-1)x^{\alpha-2}}{2b+\alpha x^{\alpha-1}}
        \eqfs
    \end{equation}
    Hence, writing $t:=x^{\alpha-1}$ a direct calculation shows
    \begin{equation}
        \frac{x h\pr(x)}{h(x)}
        =
        x \frac{\dl}{\dl x}\br{\log(h(x))}
        =
        \frac{(\alpha-2) b (2b+t)}{(b+t)(2b+\alpha t)}
        \eqfs
    \end{equation}
    As $\alpha<2$ the right-hand side is negative. Furthermore, for all $x\in\Rpp$, we have $x, h(x) > 0$. Hence $h\pr <0$ and therefore $g\prr < 0$ on $\Rpp$. Thus, $g$ is concave.
    
    \underline{\smash{Conclusion.}} 
    With $r = \EOf{\psi(X)}$, Jensen's inequality yields $\Eof{X^\alpha} = \Eof{g(\psi(X))} \leq g(r)$ (trivial for $r = \infty$, so assume $r < \infty$). 
    Set $z := z(r) := \psi^{-1}(r)$, so $g(r)=z^\alpha$. Using $(\alpha-1)\phi=2-\alpha$,
    we get
    \begin{equation}
        r
        =
        \psi(\psi^{-1}(r))
        =
        z^2(b + z^{\alpha-1})^{-\phi}
        =
        z^{\alpha}\br{b z^{-(\alpha-1)} + 1}^{-\phi}
        \eqfs
    \end{equation}
    Hence,
    \begin{equation}
        g(r) = 
        z^\alpha
        =
        r\br{1+b z^{-(\alpha-1)}}^{\phi}
        \eqfs
    \end{equation}
    We have $z^2 \geq r b^\phi$ as $r=z^2(b+z^{\alpha-1})^{-\phi}\leq z^2 b^{-\phi}$ and $\phi\geq0$.
    Hence, using $1-\tfrac{(\alpha-1)\phi}{2}=\tfrac\alpha2$, we obtain
    \begin{equation}
        g(r)
        \leq 
        r\br{1+b\br{b^{\phi}r}^{-\frac{\alpha-1}2}}^{\phi}
        =
        r\br{1+b^{\alpha/2}r^{-(\alpha-1)/2}}^{\phi}
        \eqfs
    \end{equation}
    Pulling $b^{\alpha/2}r^{-(\alpha-1)/2}$ out of the bracket, we arrive at
    \begin{equation}
        \EOf{X^\alpha}
        \leq 
        g(r)
        \leq 
        b^{\phi\alpha/2}\,r^{\alpha/2}\br{1+b^{-\alpha/2}r^{(\alpha-1)/2}}^{\phi}
        \eqfs
    \end{equation}
    Part (ii) follows directly from part (i).
\end{proof}

\begin{lemma}\label{lmm:loss:totran}
    Let $\tran\in\setcciz$.
    Let $a,b,\delta\in\Rpp$.
    Set
    \begin{equation}
        \psi(x) := b x^2\indOf{x\leq\delta}
        +
        a\tran(x)\indOf{x>\delta}
        \eqfs
    \end{equation}
    \begin{enumerate}[label=(\roman*)]
        \item
        Let $X$ be a nonnegative random variable.
        Then
        \begin{equation}
            \EOf{\min\brOf{X^2,\tran(X)}}
            \leq
            \br{\frac{1}{a} + \frac{1}{b}} \EOf{\psi(X)}
        \end{equation}
        and
        \begin{equation}
            \EOf{\tran(X)}
            \leq
            \tran\brOf{\sqrt{\frac{\EOf{\psi(X)}}{b}}}
            +
            \frac{\EOf{\psi(X)}}{a}
            \eqfs
        \end{equation}
        \item
        Let $(X_n)_{n\in\N}$ be a sequence of nonnegative random variables.
        Assume $\Eof{\psi(X_n)} = \mo o(1)$.
        Then, for $n\in\N$ large enough,
        \begin{equation}
            \tran^{-1}\brOf{\EOf{\tran(X_n)}}^2
            \leq
            \frac{\Eof{\psi(X_n)}}{b}\br{1 + \delta \sqrt{\frac{b}{a\tran(\delta)}}}^2
            \eqfs
        \end{equation}   
    \end{enumerate}
\end{lemma}

\begin{proof}
    As $\tran$ is convex, strictly increasing with $\tran(0)=0$ and $\tran(x)\to\infty$, $\tran^{-1}$ is well-defined on $[0,\infty)$.
    Set $r:=\Eof{\psi(X)}$. Then
    \begin{equation}
        \EOf{X^2\indOf{X\leq\delta}}\leq\frac{r}{b}
        \qquad\text{and}\qquad
        \EOf{\tran(X)\indOf{X>\delta}}\leq\frac{r}{a}
        \eqfs
    \end{equation}
    
    \emph{First display of (i).}
    On $\{X>\delta\}$, $\min\brOf{X^2,\tran(X)}\leq\tran(X)$; on
    $\{X\leq\delta\}$, $\min\brOf{X^2,\tran(X)}\leq X^2$. Therefore, 
    \begin{equation}
        \EOf{\min\brOf{X^2,\tran(X)}}
        \leq \EOf{\tran(X)\indOf{X>\delta}}
        + \EOf{X^2\indOf{X\leq\delta}}
        \leq \frac{r}{a}+\frac{r}{b}
        \eqfs
    \end{equation}
    
    \emph{Second display of (i).}
    Set $\trans(x):=\tran(\sqrt x)$. Then $\trans$ is concave:
    We have 
    \begin{equation}
        \dtrans(x)
        =
        \frac{\dtran(\sqrt x)}{2\sqrt x}
        \qquad\text{for }x>0\eqfs
    \end{equation}
    Since $\dtran$ is concave and nonnegative, the map
    $x\mapsto \dtran(x)/x$ is nonincreasing on $(0,\infty)$.
    As, additionally, $x\mapsto\sqrt x$ is increasing, $\dtrans$ is nonincreasing. Therefore, $\trans$ is
    concave on $\Rpp$, and hence on $\Rp$ by continuity.
    
    As $\trans$ is concave and $\trans(0) = 0$, $Z\geq0$ and any event $S$ satisfy the sub-probability Jensen inequality
    $\Eof{\trans(Z)\indOf{S}}\leq \trans\brOf{\EOf{Z\indOf{S}}}$: with $p:=\Prof{S}\in(0,1]$,
    conditional Jensen gives $\Eof{\trans(Z)\indOf S}\leq p\,\trans\brOf{\EOf{Z\indOf S}/p}$,
    and concavity with $\trans(0)=0$ yields $\trans(\lambda x)\leq\lambda \trans(x)$ for
    $\lambda\geq1$, applied with $\lambda=1/p$. Taking $Z=X^2$ and
    $S=\{X\leq\delta\}$, we obtain
    \begin{equation}
        \EOf{\tran(X)\indOf{X\leq\delta}}
        =\EOf{\trans(X^2)\indOf{X\leq\delta}}
        \leq \trans\brOf{\EOf{X^2\indOf{X\leq\delta}}}
        \leq \tran\brOf{\sqrt{r/b}}
        \eqfs
    \end{equation}
    Adding $\Eof{\tran(X)\indOf{X>\delta}}\leq\frac{r}{a}$ gives the second display.
    
    \emph{(ii).}
    Apply (i) with $r_n:=\Eof{\psi(X_n)}$. Since $\tran^{-1}$ is concave with
    $\tran^{-1}(0)=0$, it is subadditive, so
    \begin{equation}
        \tran^{-1}\brOf{\EOf{\tran(X_n)}}
        \leq \tran^{-1}\brOf{\tran\brOf{\sqrt{r_n/b}}+\tfrac{r_n}{a}}
        \leq \sqrt{r_n/b}+\tran^{-1}\brOf{\tfrac{r_n}{a}}
        \eqfs
    \end{equation}
    We have
    \begin{equation}
        \tran^{-1}\brOf{\frac{r_n}{a}}
        =\sqrt{\trans^{-1}\brOf{\frac{r_n}{a}}}
        =
        \sqrt{\frac{r_n}{b}} \sqrt{\frac{b}{a} \frac{\trans^{-1}\brOf{\frac{r_n}{a}} }{\frac{r_n}{a}}}
        \eqfs
    \end{equation}
    Hence,
    \begin{equation}
        \tran^{-1}\brOf{\EOf{\tran(X_n)}}
        \leq
        \sqrt{r_n/b} \br{1 + \sqrt{\frac{b}{a} \frac{\trans^{-1}\brOf{\frac{r_n}{a}} }{\frac{r_n}{a}}}}
        \eqfs
    \end{equation}    
    The map $z\mapsto \trans^{-1}(z)/z$ is nondecreasing as $\trans^{-1}$ is convex and $\trans^{-1}(0)=0$.
    As $r_n=\mo o(1)$ and $a \tran(\delta)>0$ is fixed, $r_n/a\leq\tran(\delta)$ for all large $n$. 
    Thus,
    \begin{equation}
        \tran^{-1}\brOf{\EOf{\tran(X_n)}}
        \leq
        \sqrt{r_n/b} \br{1 + \sqrt{\frac{b}{a} \frac{\trans^{-1}\brOf{\tran(\delta)} }{\tran(\delta)}}}
        \leq
        \sqrt{r_n/b} \br{1 + \delta \sqrt{\frac{b}{a\tran(\delta)}}}
        \eqfs
    \end{equation}    
\end{proof}

%% file: sec_algo.tex
\section{Algorithm Stability}\label{sec:algostabi}
The convergence rate proofs in this article rely on algorithmic stability. The initial steps are closely related to the arguments from \cite{Escande2024} for M-estimators and \cite{brunel2025} for the $2$-Fr\'echet mean. These steps extend to $\tran$-Fr\'echet means using the quadruple inequality \cref{thm:infdtr:qi} and the variance inequality \cref{thm:infdtr:vi}. Because the lower bound in this variance inequality depends on the underlying distribution, additional arguments are needed to establish convergence rates for $\tran$-Fr\'echet means. The arguments derived from the classical line of reasoning are presented here; additional techniques that lead to our main results appear in later sections.

Let $Y_1\pr, \dots, Y_n\pr$ be an additional independent set of iid copies of $Y$.
Denote the samples with $i$-th position replaced as $Y_j^i := Y_j$ if $i \neq j$ and $Y_i^i := Y_i\pr$.
Let $m_n^i$ be the sample Fr\'echet mean of the sample set $Y_1^i, \dots, Y_n^i$, 
\begin{equation}
    m_n^i \in \argmin_{q\in\mc Q} \sum_{j=1}^n \tran(\ol {Y_j^i}q)
    \eqfs
\end{equation}
That is, $m_n^i$ is the $\tran$-Fr\'echet mean of the sample where the $i$-th entry is replaced by an independent copy.
\begin{proposition}\label{prp:stabi}
	Use the convention $0^{-1} = \infty$.
	Then
	\begin{equation}
		\EOf{\tran(\ol Y{m_n}) - \tran(\ol Ym)}
        \leq
        \frac{16}{n^2} \sum_{i=1}^n \EOf{\dtran(\ol {Y_i}m)^2 H_i^{-1}}
		\eqcm
	\end{equation}
	where
	\begin{equation}
		H_i := \frac1n \sum_{j=1}^n \ddrtran\brOf{\ol {Y_j}{m_n} + \ol {m_n}{m_n^i}}
		\eqfs
	\end{equation}
\end{proposition}
Before proving \cref{prp:stabi}, observe that in the special case $\tran(x) = x^2$, we have $\dtran(x) = 2x$ and $\ddrtran(x) = 2$. Hence, \cref{prp:stabi} yields
\begin{equation}
	\EOf{\ol m{m_n}^2} \leq 32 n^{-1} \EOf{\ol {Y}m^2}
	\eqcm
\end{equation}
cf.\ \cite[Theorem 3]{brunel2025}. In the general case, however, the $\ddrtran$-terms make the results unsatisfactory at this stage. A large portion of the proofs in this paper are devoted to addressing these terms to obtain clean results from \cref{prp:stabi} for general $\tran \in \setcciz$.

For the proof of \cref{prp:stabi}, first define the \emph{double excess risk} as
\begin{equation}
	V_n := \EOf{\tran(\ol Y{m_n}) - \tran(\ol Ym)} + \EOf{\frac1n \sum_{i=1}^n \br{\tran(\ol {Y_i}{m}) - \tran(\ol {Y_i}{m_n})} }
	\eqfs
\end{equation}
Recall that $Y, Y_1, \dots, Y_n$ are iid, so that $Y$ is independent of $m_n$.
\begin{lemma}\label{lmm:variancebound}
	We have
	\begin{equation}
		V_n \leq \frac1n \sum_{i=1}^n \EOf{\ol {m_n}{m_n^i}\, \dtran(\ol {Y_i}{Y_i^i})}
		\eqfs
	\end{equation}
\end{lemma}
\begin{proof}
	As $Y$ has the same distribution as $Y_i$ and $(Y, m_n)$ has the same distribution as $(Y_i, m_n^i)$, we have
	\begin{equation}\label{eq:infdtr:var:simple1}
		V_n 
		=
		\EOf{\tran(\ol Y{m_n}) - \frac1n \sum_{i=1}^n \tran(\ol {Y_i}{m_n})} 
		=
		\frac1n \sum_{i=1}^n \EOf{\tran(\ol {Y_i}{m_n^i}) - \tran(\ol {Y_i}{m_n})} 
		\eqfs
	\end{equation}
	By the quadruple inequality, \cref{thm:infdtr:qi}, we have 
	\begin{equation}
		\br{\tran(\ol {Y_i}{m_n^i}) - \tran(\ol {Y_i}{m_n})}
		+ 
		\br{\tran(\ol {Y_i^i}{m_n}) - \tran(\ol {Y_i^i}{m_n^i})}
		\leq
		2 \, \ol {m_n}{m_n^i}\, \dtran(\ol {Y_i}{Y_i^i})
		\eqfs
	\end{equation}
	As $(Y_i, m_n, m_n^i)$ has the same distribution as $(Y_i^i, m_n^i, m_n)$, we obtain
	\begin{equation}\label{eq:infdtr:var:distri}
		2 \Eof{\tran(\ol {Y_i}{m_n^i}) - \tran(\ol {Y_i}{m_n})} \leq 2 \EOf{\ol {m_n}{m_n^i} \,\dtran(\ol {Y_i}{Y_i^i})}
		\eqfs
	\end{equation}
	Combining \eqref{eq:infdtr:var:simple1} and \eqref{eq:infdtr:var:distri} yields the desired result.
\end{proof}
\begin{lemma}\label{lmm:mnmnibound}\mbox{}
	We have
	\begin{equation}
		\ol {m_n}{m_n^i} \, \tilde H_i \leq \frac4n \dtran(\ol {Y_i}{Y_i^i})
		\eqfs
	\end{equation}
	where 
	\begin{equation}\label{eq:lmm:mnmnibound:hi:tran}
		\tilde H_i := \frac1n \sum_{j=1}^n \br{\ddrtran\brOf{\ol {Y_j}{m_n} + \ol {m_n}{m_n^i}} + \ddrtran\brOf{\ol {Y_j^i}{m_n^i} + \ol {m_n}{m_n^i}}}
		\eqfs
	\end{equation}
\end{lemma}
\begin{proof}
	The variance inequality \cref{thm:infdtr:vi} applied to the empirical distributions yields, for $q\in\mc Q$,
	\begin{align}
		\frac12 \ol q{m_n}^2 \frac1n \sum_{j=1}^n \ddrtran\brOf{\ol {Y_j}{m_n} + \ol q{m_n}}
		&\leq
		\frac1n \sum_{j=1}^n \br{\tran(\ol {Y_j}q) - \tran(\ol {Y_j}{m_n})} 
		\eqcm
		\\
		\frac12 \ol q{m_n^i}^2 \frac1n \sum_{j=1}^n \ddrtran\brOf{\ol {Y_j^i}{m_n^i} + \ol q{m_n^i}}
		&\leq
		\frac1n \sum_{j=1}^n \br{\tran(\ol {Y_j^i}q) - \tran(\ol {Y_j^i}{m_n^i})} 
		\eqfs
	\end{align}
	Thus, plugging in $q = m_n^i$ and $q = m_n$ respectively, adding the two inequalities, and noting $Y_j^i = Y_j$ for $i\neq j$, yields
	\begin{align}
		\frac{1}2 \ol {m_n}{m_n^i}^2 \tilde H_i 
		&\leq 
		\frac1n \sum_{j=1}^n \br{
			\tran(\ol {Y_j}{m_n^i}) -
			\tran(\ol {Y_j}{m_n}) +
			\tran(\ol {Y_j^i}{m_n}) -
			\tran(\ol {Y_j^i}{m_n^i})
		}  
		\\&=
		\frac1n \br{
			\tran(\ol {Y_i}{m_n^i}) -
			\tran(\ol {Y_i}{m_n}) +
			\tran(\ol {Y_i^i}{m_n}) -
			\tran(\ol {Y_i^i}{m_n^i})
		}  
		\eqfs
	\end{align}
	Hence, the quadruple inequality, \cref{thm:infdtr:qi}, implies
	\begin{equation}
		\frac{1}2 \ol {m_n}{m_n^i}^2 \tilde H_i 
		\leq
		2
		\frac1n \,
		\ol {m_n}{m_n^i}\,
		\dtran(\ol {Y_i}{Y_i^i})
		\eqfs
	\end{equation}
	Rearranging the terms yields the desired result.
\end{proof}
\begin{proof}[Proof of \cref{prp:stabi}]
	Combining \cref{lmm:variancebound} and \cref{lmm:mnmnibound}, we obtain
	\begin{equation}
		V_n \leq \frac{4}{n^2}\sum_{i=1}^n \EOf{\dtran(\ol {Y_i}{Y_i^i})^2 \tilde H_i^{-1}}
	\end{equation}
	with $\tilde H_i$ given in \eqref{eq:lmm:mnmnibound:hi:tran}.
	With \cref{lmm:dtran:subadd} and the triangle inequality we get
	\begin{equation}
		\dtran(\ol {Y_i}{Y_i^i})^2 \leq 2 \dtran(\ol {Y_i}{m})^2 + 2\dtran(\ol {Y_i^i}m)^2
		\eqfs
	\end{equation}
	As $(\tilde H_i, Y_i)$ has the same distribution as $(\tilde H_i, Y_i^i)$, this yields
	\begin{equation}
		\EOf{\dtran(\ol {Y_i}{Y_i^i})^2 \tilde H_i^{-1}} \leq 4 \EOf{\dtran(\ol {Y_i}{m})^2 \tilde H_i^{-1}}
		\eqfs
	\end{equation} 
	Furthermore, $\tilde H_i \geq H_i$. Thus,
	\begin{equation}\label{eq:prp:stabi:upper}
		V_n \leq \frac{16}{n^2}\sum_{i=1}^n \EOf{\dtran(\ol {Y_i}{m})^2 H_i^{-1}}
		\eqfs
	\end{equation}
	By the minimizing property of $m_n$,
	\begin{equation}\label{eq:prp:stabi:vi}
		V_n 
		\geq 
		\EOf{\tran(\ol Y{m_n}) - \tran(\ol Ym)}
		\eqfs
	\end{equation}
	Now \eqref{eq:prp:stabi:upper} and \eqref{eq:prp:stabi:vi} together show the desired result.
\end{proof}

%% file: sec_chernoff.tex
\section{Chernoff Bounds}
For reference, we state the versions of the well-known Chernoff bound used
throughout. Here and below, $\log$ denotes the natural logarithm and
\begin{equation}\label{eq:kl}
    \kullback(t, p)
    := t \log\brOf{\frac{t}{p}} + (1-t) \log\brOf{\frac{1-t}{1-p}}
    \eqcm
    \qquad t, p \in [0,1]
    \eqcm
\end{equation}
denotes the binary relative entropy, with the conventions $0 \log 0 = 0$ and
$\kullback(t,p) = +\infty$ if $t \notin \{0,1\}$ and $p \in \{0,1\}$. Note that
$\kullback(t,p) \geq 0$ with equality iff $t = p$, and that
$\kullback(t,p) = \kullback(1-t,1-p)$.
\begin{proposition}[Chernoff--Hoeffding Bound]\label{prp:chernoff:addi}
    Let $X_1, \dots, X_n$ be independent random variables with values in $\{0,1\}$
    and let $\rho := \frac1n \sum_{i=1}^n \EOf{X_i}$.
    Then, for all $t \in [0,1]$,
    \begin{align}
        \text{if $t \geq \rho$}\colon\quad
        \PrOf{\frac1n \sum_{i=1}^n X_i \geq t} &\leq \exp\brOf{-n \kullback(t, \rho)}
        \eqcm\\
        \text{if $t \leq \rho$}\colon\quad
        \PrOf{\frac1n \sum_{i=1}^n X_i \leq t} &\leq \exp\brOf{-n \kullback(t, \rho)}
        \eqfs
    \end{align}
\end{proposition}
\begin{proof}
    Write $S_n := \sum_{i=1}^n X_i$ and let $t \geq \rho$. For every $\lambda \geq 0$,
    Markov's inequality applied to $e^{\lambda S_n}$, independence, and
    $\EOf{e^{\lambda X_i}} = 1 - p_i + p_i e^{\lambda}$ with $p_i := \Eof{X_i}$ give
    \begin{equation}
        \PrOf{S_n \geq t n}
        \leq
        e^{-\lambda t n} \prod_{i=1}^n \br{1 - p_i + p_i e^{\lambda}}
        \leq
        \brOf{e^{-\lambda t} \br{1 - \rho + \rho e^{\lambda}}}^n
        \eqcm
    \end{equation}
    where the second step is the arithmetic--geometric mean inequality (the map
    $p \mapsto \log(1-p+pe^{\lambda})$ being concave, the product is maximal at
    $p_i \equiv \rho$). Choosing
    $\lambda := \log\brOf{\frac{t(1-\rho)}{\rho(1-t)}} \geq 0$ yields
    $1 - \rho + \rho e^{\lambda} = \frac{1-\rho}{1-t}$ and hence
    \begin{equation}
        e^{-\lambda t} \br{1 - \rho + \rho e^{\lambda}}
        =
        \frac{1-\rho}{1-t} \brOf{\frac{\rho(1-t)}{t(1-\rho)}}^{t}
        =
        \exp\brOf{-\kullback(t,\rho)}
        \eqfs
    \end{equation}
    The case $t \leq \rho$ follows by applying the above to $1 - X_1, \dots, 1-X_n$
    together with $\kullback(1-t,1-\rho) = \kullback(t,\rho)$.
\end{proof}
The relative entropy is the exact large-deviation rate, but it is often
convenient to relax it. The following ``Poisson form'' isolates the dependence
on the relative deviation $\eta = t/\rho$.
\begin{corollary}[Poisson Form]\label{cor:chernoff:poisson}
    Let $\psi(\eta) := \eta \log \eta - \eta + 1 \geq 0$ for $\eta \geq 0$
    (with $\psi(0) = 1$). Then, for all $\rho \in (0,1)$ and
    $\eta \in \br{0, \frac1\rho}$,
    \begin{equation}\label{eq:kl:poisson}
        \kullback(\eta\rho, \rho) \geq \rho\, \psi(\eta)
        \eqcm
    \end{equation}
    and consequently
    $\exp\brOf{-n \kullback(\eta\rho,\rho)} \leq \exp\brOf{-\rho n\, \psi(\eta)}$.
    Moreover,
    \begin{equation}\label{eq:psi:quadratic}
        \psi(\eta) \geq \frac{(1-\eta)^2}{2} \ \ \text{ for $\eta \in [0,1]$}
        \eqcm
        \qquad
        \psi(\eta) \geq \frac{3(\eta-1)^2}{2(\eta+2)} \ \ \text{ for $\eta \geq 1$}
        \eqfs
    \end{equation}
\end{corollary}
\begin{proof}
    By $u \log\brOf{\frac uv} \geq u - v$ for $u, v > 0$, applied with
    $u = 1 - \eta\rho$ and $v = 1 - \rho$,
    \begin{equation}
        \kullback(\eta\rho,\rho)
        =
        \eta\rho \log(\eta) + (1-\eta\rho) \log\brOf{\frac{1-\eta\rho}{1-\rho}}
        \geq
        \eta\rho \log(\eta) + \rho \br{1 - \eta}
        =
        \rho\, \psi(\eta)
        \eqfs
    \end{equation}
    The bounds in \eqref{eq:psi:quadratic} are elementary: both sides vanish at
    $\eta = 1$ together with their first derivatives, and the claim follows by
    comparing second derivatives.
\end{proof}
\begin{proposition}[Multiplicative Chernoff Bound]\label{prp:chernoff:multi}
    Let $X_1, \dots, X_n$ be independent random variables with values in
    $\{0, 1\}$ and $p := \frac1n \sum_{i=1}^n \EOf{X_i}$.
    Then, for any $\delta \in [0,1]$,
    \begin{equation}
        \PrOf{\frac1n \sum_{i=1}^n X_i \leq (1 - \delta)p}
        \leq
        \exp\brOf{-\frac{\delta^2 p n}{2}}
        \eqcm
    \end{equation}
    and, for any $\delta \geq 0$ with $(1+\delta)p \leq 1$,
    \begin{equation}
        \PrOf{\frac1n \sum_{i=1}^n X_i \geq (1 + \delta)p}
        \leq
        \exp\brOf{-\frac{3 \delta^2 p n}{2(\delta+3)}}
        \leq
        \exp\brOf{-\frac{\delta^2 p n}{2 + \delta}}
        \eqfs
    \end{equation}
\end{proposition}
\begin{proof}
    Combine \cref{prp:chernoff:addi} with \cref{cor:chernoff:poisson}, taking
    $\eta = 1 \mp \delta$ in \eqref{eq:psi:quadratic}.
\end{proof}
\begin{lemma}\label{lmm:tail:prob}
    Let $A_1, \dots, A_n$ be independent events with the same probability
    $\rho = \Prof{A_k} \in (0,1)$. Set the rate of occurrence of such events as
    \begin{equation}
        \rho_n := \frac1n \sum_{i=1}^n \ind_{A_i}
        \eqfs
    \end{equation}
    Let $\eta \in [0,1]$. Then
    \begin{equation}\label{eq:tail:lower}
        \PrOf{\rho_n \leq \eta\rho}
        \leq
        \exp\brOf{-n \kullback(\eta\rho, \rho)}
        \leq
        \exp\brOf{-\rho n \br{\eta\log\eta - \eta + 1}}
        \leq
        \exp\brOf{-\frac{\rho n (1-\eta)^2}{2}}
        \eqfs
    \end{equation}
    Let $\eta \in \br{1, \frac1\rho}$. Then
    \begin{equation}\label{eq:tail:upper}
        \PrOf{\rho_n \geq \eta \rho}
        \leq
        \exp\brOf{-n \kullback(\eta\rho, \rho)}
        \leq
        \exp\brOf{-\rho n \br{\eta\log\eta - \eta + 1}}
        \leq
        \exp\brOf{-\frac{3 \rho n (\eta-1)^2}{2(\eta+2)}}
        \eqfs
    \end{equation}
\end{lemma}
\begin{proof}
    Direct consequence of \cref{prp:chernoff:addi} and
    \cref{cor:chernoff:poisson}, applied with $t = \eta\rho$.
\end{proof}
\begin{remark}[Tightness]\label{rmk:tail:tight}
    The exponents in \cref{lmm:tail:prob} are optimal. At the extreme deviations
    they are attained exactly: $\kullback(0,\rho) = \log\frac{1}{1-\rho}$ and
    $\kullback(1,\rho) = \log\frac1\rho$ reproduce
    $\PrOf{\rho_n = 0} = (1-\rho)^n$ and $\PrOf{\rho_n = 1} = \rho^n$. In general,
    Stirling's formula gives the matching lower bound
    $\PrOf{\rho_n = t} \geq \frac{1}{n+1} \exp\brOf{-n \kullback(t,\rho)}$ for
    $t n \in \{0, \dots, n\}$, so only the polynomial prefactor can be improved.
\end{remark}

%% file: sec_proof_prelim.tex
\section{Proofs of Section: Preliminaries}\label{sec:proof:prelim}

\subsection{Basic Setup}

\begin{lemma}\label{lmm:convsepa}
    Let $(\mc Q, d)$ be a Hadamard space. Let $\mc Y \subset \mc Q$ be separable. Then the closed convex hull of $\mc Y$, $\overline{\ms{conv}}(\mc Y)$, is separable.
\end{lemma}

\begin{proof}
    Geodesics in $\mc Q$ are unique and depend continuously on their endpoints
    \cite[Proposition 2.3]{sturm03}; for $x,y\in\mc Q$ write
    $\gamma_{x,y}\colon[0,1]\to\mc Q$ for the constant-speed geodesic with
    $\gamma_{x,y}(0)=x$ and $\gamma_{x,y}(1)=y$, so that for each fixed
    $t\in[0,1]$ the map $(x,y)\mapsto\gamma_{x,y}(t)$ is continuous. The closed
    convex hull $\overline{\ms{conv}}(\mc Y)$ is the smallest closed convex subset
    of $\mc Q$ containing $\mc Y$ \cite[Remark 2.7]{sturm03}; recall that a set
    $C\subset\mc Q$ is convex if and only if $\gamma_{x,y}([0,1])\subset C$
    for all $x,y\in C$.
    
    Since $\mc Y$ is separable, fix a countable dense set $D\subset\mc Y$, and
    let $T:=\Q\cap[0,1]$, a countable dense subset of $[0,1]$. Define
    countable sets $C_n\subset\mc Q$ recursively by
    \[
        C_0:=D,\qquad
        C_{n+1}:=\bigl\{\,\gamma_{x,y}(t): x,y\in C_n,\ t\in T\,\bigr\},
    \]
    and put $C:=\bigcup_{n\ge0}C_n$. Each $C_n$ is countable, hence so is $C$.
    Taking $y=x$ gives $\gamma_{x,x}(t)=x$, so $C_n\subset C_{n+1}$ for every
    $n$; the family is therefore nondecreasing, and any finite subset of $C$
    lies in a single $C_n$. Set $K:=\overline{C}$. Then $K$ is closed and
    separable, and $\mc Y\subset\overline D\subset K$.
    
    Next we show that $K$ is convex. Let $u,v\in K$ and choose $u_k,v_k\in C$ with
    $u_k\xrightarrow{k\to\infty} u$ and $v_k\xrightarrow{k\to\infty} v$. 
    For each $k$, since $C=\bigcup_{n\in\Nn}C_n$ and the family $(C_n)_n$ is increasing, there exists $n_k$ such that $u_k,v_k\in C_{n_k}$.
    Hence $\gamma_{u_k,v_k}(t)\in C_{n_k+1}\subset C$ for every $t\in T$.
    Fix $t\in T$. By continuity of $(x,y)\mapsto\gamma_{x,y}(t)$ we have
    $\gamma_{u_k,v_k}(t)\xrightarrow{k\to\infty}\gamma_{u,v}(t)$, so $\gamma_{u,v}(t)\in\overline C=K$.
    Thus $\gamma_{u,v}(t)\in K$ for all $t\in T$; since $t\mapsto\gamma_{u,v}(t)$
    is continuous, $T$ is dense in $[0,1]$, and $K$ is closed, it follows that
    $\gamma_{u,v}([0,1])\subset K$. Hence $K$ is convex.
    
    Now we show $K=\overline{\ms{conv}}(\mc Y)$. As $K$ is closed, convex, and contains
    $\mc Y$, minimality gives $\overline{\ms{conv}}(\mc Y)\subset K$. Conversely,
    $C_0=D\subset\mc Y\subset\overline{\ms{conv}}(\mc Y)$, and if
    $C_n\subset\overline{\ms{conv}}(\mc Y)$ then, by convexity of
    $\overline{\ms{conv}}(\mc Y)$, every geodesic $\gamma_{x,y}$ with $x,y\in C_n$
    stays in $\overline{\ms{conv}}(\mc Y)$, whence
    $C_{n+1}\subset\overline{\ms{conv}}(\mc Y)$. By induction
    $C\subset\overline{\ms{conv}}(\mc Y)$, and taking closures gives
    $K\subset\overline{\ms{conv}}(\mc Y)$.
    
    Therefore $\overline{\ms{conv}}(\mc Y)=K=\overline C$ admits the countable
    dense subset $C$ and is separable.
\end{proof}

\subsection{Variance Inequality}\label{sec:viconv}

We use the setup described in \cref{ssec:prelim:setup}. We prove \cref{thm:infdtr:vi} based on the results in \cite{varinequ}.

\subsubsection{Preliminaries}

By \cite[Proposition 5.2]{varinequ} the variance functional 
\begin{equation}
    F\colon\mc Q \to \Rp, q\mapsto \Eof{\tran(\ol Yq)-\tran(\ol Ym)}
\end{equation}
is convex and we have 
\begin{equation}
    v(t)=\inf_{q\in \mc Q\setminus \ballopen(m,t)} F(q)
    \eqfs
\end{equation}
As $F(q) \geq 0$, we have $v(t)\geq 0$. As $F(m) = 0$, we have $v(0) = 0$. As $\mc Q\setminus \ballopen(m,t)$ decreases with increasing $t$, $v(t)$ is nondecreasing.

Next, we show that $t\mapsto v(t)/t$ is nondecreasing.
Let $0<t_1<t_2$. For $q\in\mc Q$ with $\ol qm\ge t_2$, let
$\gamma_{m\to q}:[0,\ol qm]\to\mc Q$ be the unit-speed geodesic from $m$ to $q$.
By convexity of $F$,
\begin{equation}
    F(\gamma_{m\to q}(t_1))
    \leq
    \frac{t_1}{t_2}F(\gamma_{m\to q}(t_2))
    \eqfs
\end{equation}
Taking infima over all such $q$ yields
\begin{equation}
    \frac{v(t_1)}{t_1} \leq \inf_{q\in \mc Q\setminus \ballopen(m, t_2)} \frac{F(\gamma_{m\to q}(t_1))}{t_1} \leq \frac{v(t_2)}{t_2}
    \eqfs
\end{equation}
Hence, $t\mapsto v(t)/t$ is nondecreasing.

For the lower bounds on $v(t)$, we first reduce the infimum over $\mc Q\setminus\ballopen(m,t)$ to an infimum over $q\in\mc Q$ with $\ol qm = t$:
Let $t\in\Rpp$ and let $q\in\mc Q\setminus\ballopen(m,t)$.
Set $s:=\ol qm\ge t$ and let
$\gamma:[0,s]\to\mc Q$
be the unit-speed geodesic from $m$ to $q$.
By convexity of $F$,
\begin{equation}
    F(\gamma(t))
    \le
    \frac{t}{s}F(q)
    \eqfs
\end{equation}
As $s\ge t$, this implies
\begin{equation}
    F(q)\ge F(\gamma(t))
\end{equation}
and therefore
\begin{equation}
    v(t)
    =
    \inf_{q\in\mc Q\setminus\ballopen(m,t)}F(q)
    =
    \inf_{q\in\mc Q:\,\ol qm=t}F(q)
    \eqfs
\end{equation}

\subsubsection{Part \ref{thm:infdtr:vi:default}}

Next we show the first variance inequality given in \cref{thm:infdtr:vi},
\begin{equation}\label{eq:vi:first}
    v(t) \geq \frac12 t^2 \EOf{\ddrtran\brOf{\ol Ym + t}}
    \eqfs
\end{equation}
By \cite[Theorem 5.4]{varinequ}, for every $q\in\mc Q$ with $\ol qm=t$,
\begin{equation}
    F(q)
    \ge
    \frac12 t^2
    \EOf{\ddrtran\brOf{\max(\ol Ym,t)}}
    \eqfs
\end{equation}
Hence
\begin{equation}
    v(t)
    \ge
    \frac12 t^2
    \EOf{\ddrtran\brOf{\max(\ol Ym,t)}}
    \eqfs
\end{equation}
As $\ddrtran$ is nonincreasing, we have $\ddrtran(\max(\ol Ym, t)) \geq \ddrtran(\ol Ym + t)$, which shows \eqref{eq:vi:first}.

\subsubsection{Part \ref{thm:infdtr:vi:convex}}

We now want to show that we can obtain a better bound under the assumption of $\ddrtran$ being convex:
\begin{equation}
    v(t) \geq \frac12 t^2 \EOf{\ddrtran\brOf{\ol Ym + \frac13 t}}
    \eqfs
\end{equation}
We adapt the proof of \cite[Theorem 5.4]{varinequ}, including prerequisite results shown in the same paper, and inject the convexity of $\ddrtran$ via Jensen's inequality at a suitable place.

The arguments in \cite{varinequ} are based on the fact that the distance functions to points $q\in\mc Q$, $t \mapsto \ol q\gamma(t)$, where $\gamma\colon\R\to\mc Q$ is a geodesic, are $\mc G$-convex, where 
\begin{equation}
    \mc G := \setByEle{t\mapsto \sqrt{(t-t_0)^2+h^2}}{t_0 \in \R, h\geq 0}
\end{equation}
and $\mc G$-convexity is defined as follows.
\begin{definition}\label{def:Fconvex}
    Let $I\subset\R$. Let $\mc G$ be a set of functions $I \to \R$.
    A function $f\colon I\to\R$ is called \emph{$\mc G$-convex} if and only if for every $t_0\in I$ there is a function $g\in\mc G$ such that $g(t_0) = f(t_0)$ and $g(t) \leq f(t)$ for all $t\in I$. In this case, the function $g$ is called \emph{$\mc G$-tangent} of $f$ at $t_0$.
\end{definition}
The technical basis for \cite[Theorem 5.4]{varinequ} is \cite[Theorem 4.17 (ii)]{varinequ}, which states the following:

Let $I\subset \R$ be convex with $0\in I$.
Let $f\colon I\to\R$ be $\mc G$-convex.
Let $\tran\in\setcc$.
Then, for all $t\in I$, $t\neq0$,
\begin{equation}
    (\tran \circ f)(t) - (\tran \circ f)(0) \geq t (\tran \circ f)\rd(0) + \frac12 t^2 \ddrtran\brOf{\max\brOf{f(0), f(t)}}
    \eqfs
\end{equation}
From the proof of \cite[Theorem 4.17 (ii)]{varinequ}, we extract an intermediate result as a Lemma:
\begin{lemma}\label{lmm:semitaylor}
    Let $I\subset \R$ be convex with $0\in I$.
    Let $f\colon I\to\R$ be $\mc G$-convex.
    Let $\tran\in\setcc$.
    Then, for all $t\in I$, $t\neq0$,
    \begin{equation}\label{eq:semitaylor:tran}
        (\tran \circ f)(t) - (\tran \circ f)(0) \geq t (\tran \circ f)\rd(0) + \int_{0}^t (t-r) \ddrtran(f(r)) \dl r
        \eqfs
    \end{equation}
\end{lemma}
\begin{proof}
    If $t<0$, we can replace $f$ by $t\mapsto f(-t)$. Thus, we can assume $t>0$ without loss of generality.
    As $\tran$ and $f$ are convex and $\tran$ is nondecreasing, $\tran \circ f$ is convex. Thus, $\tran \circ f$ is twice differentiable Lebesgue almost everywhere by Alexandrov's theorem (cf.\ \cite[Lemma A.2]{varinequ}). Let $B\subset I$ be the set of points where $\tran \circ f$ is not twice differentiable.
    As $\tran \circ f$ is convex on $I$, it is Lipschitz on $[0, t]$ and, hence, absolutely continuous on $[0, t]$. Thus, by the fundamental theorem of calculus for Lebesgue integrals \cite[Lemma A.1]{varinequ},
    \begin{equation}
        (\tran \circ f)(t) - (\tran \circ f)(0) = \int_0^t (\tran \circ f)\rd(s) \dl s
        \eqfs
    \end{equation}
    Furthermore, as $(\tran \circ f)\rd$ is nondecreasing, by fundamental theorem of calculus for Lebesgue integrals \cite[Lemma A.1]{varinequ},
    \begin{equation}
        (\tran \circ f)\rd(s) - (\tran \circ f)\rd(0) \geq \int_{[0,s]\setminus B} (\tran \circ f)\prr(r) \dl r
        \eqfs
    \end{equation}
    For $r \in I\setminus B \subset \mathring I$, by \cite[Lemma 4.16]{varinequ},
    \begin{equation}
        (\tran \circ f)\prr(r) \geq \ddrtran(f(r))
        \eqfs
    \end{equation}
    Hence,
    \begin{align}
        (\tran \circ f)(t) - (\tran \circ f)(0)
        &\geq
        \int_{[0,t]} \br{  (\tran \circ f)\rd(0) + \int_{[0,s]\setminus B} \ddrtran(f(r)) \dl r } \dl s
        \\&=
        t (\tran \circ f)\rd(0) + \int_{[0,t]} \int_{[0,s]} \ddrtran(f(r)) \dl r  \dl s
        \\&=
        t (\tran \circ f)\rd(0) + \int_{0}^t (t-r) \ddrtran(f(r)) \dl r
        \eqfs
    \end{align}
\end{proof}
Next we follow the proof of \cite[Theorem 5.4]{varinequ} but use this lemma instead of \cite[Theorem 4.17 (ii)]{varinequ}.
\begin{theorem}\label{thm:extended:varinequ}
    Let $\tran\in\setcciz$.
    Assume $(\mc Q, d)$ is Hadamard.
    Assume $\Eof{\dtran(\ol Yo)} < \infty$ for an arbitrary point $o \in\mc Q$.
    Let $m \in\argmin_{q\in\mc Q}\Eof{\tran(\ol Yq) - \tran(\ol Yo)}$. Let $q\in\mc Q\setminus\cb{m}$.
    Let $\gamma \colon [0, 1] \to \mc Q$ be the constant-speed geodesic from $\gamma(0) = m$ to $\gamma(1) =q$. 
    Then
    \begin{equation}
        \EOf{\tran(\ol Yq) - \tran(\ol Ym)} \geq \ol qm^2 \EOf{\int_{0}^{1} \ddrtran(\ol Y\gamma(s)) (1-s) \dl s}
        \eqfs
    \end{equation}
\end{theorem}
\begin{proof}
    Let $\tilde\gamma \colon [0, \ol qm] \to \mc Q$ be the unit-speed geodesic from $\tilde\gamma(0) = m$ to $\tilde\gamma(\ol qm) =q$, i.e., $\tilde\gamma(s\,\ol qm) = \gamma(s)$ for $s\in [0,1]$. 
    For $t\in[0, \ol qm]$ and $y\in \mc Q$, define $V_y(t) := \tran(\ol y{\tilde\gamma}(t)) - \tran(\ol y{\tilde\gamma}(0))$. Thus, $V_y(0) = 0$.
    By \cref{lmm:semitaylor},
    \begin{equation}\label{eq:taylorraw}
        V_y(t)
        \geq
        t  V_y\rd(0) + \int_{0}^t (t-r) \ddrtran(\ol y{\tilde\gamma}(r)) \dl r
        \eqfs
    \end{equation}
    Following the proof of \cite[Theorem 5.4]{varinequ} now identically, we obtain
    \begin{equation}
        \EOf{V_Y(t)} \geq \EOf{\int_{0}^t (t-r) \ddrtran(\ol Y{\tilde\gamma}(r)) \dl r}
        \eqfs
    \end{equation}
    Plugging in $t = \ol qm$ and substituting $r=s\,\ol qm$ yields
    \begin{align}
        \EOf{\tran(\ol Yq) - \tran(\ol Ym)} 
        &\geq 
        \EOf{\int_{0}^{\ol qm} (\ol qm-r) \ddrtran(\ol Y{\tilde\gamma}(r)) \dl r}
        \\&=
        \ol qm^2\EOf{\int_{0}^{1} (1-s) \ddrtran(\ol Y{\gamma}(s)) \dl s}
        \eqfs
    \end{align}
\end{proof}
\begin{corollary}
    Let $\tran\in\setcciz$.
    Assume $(\mc Q, d)$ is Hadamard.
    Assume $\Eof{\dtran(\ol Yo)} < \infty$ for an arbitrary point $o \in\mc Q$.
    Let $m \in\argmin_{q\in\mc Q}\Eof{\tran(\ol Yq) - \tran(\ol Yo)}$. Let $q\in\mc Q\setminus\cb{m}$.
    Assume that $\ddrtran$ is convex.
    Then
    \begin{equation}
        \EOf{\tran(\ol Yq) - \tran(\ol Ym)} \geq \frac12 \ol qm^2 \EOf{\ddrtran\brOf{\ol Ym + \frac13\, \ol qm}}
        \eqfs
    \end{equation}
\end{corollary}
\begin{proof}
    From \cref{thm:extended:varinequ}, we have
    \begin{equation}
        \EOf{\tran(\ol Yq) - \tran(\ol Ym)} \geq \ol qm^2 \EOf{\int_{0}^{1} \ddrtran(\ol Y\gamma(s)) (1-s) \dl s}
        \eqcm
    \end{equation}
    where $\gamma \colon [0, 1] \to \mc Q$ is the constant-speed geodesic from $\gamma(0) = m$ to $\gamma(1) =q$. 
    As $\ol Y{\gamma}(s)$ is $\ol qm$-Lipschitz and $\ddrtran$ is nonincreasing, we have $\ddrtran(\ol Y\gamma(s)) \geq \ddrtran(\ol Ym + s\,\ol qm)$. Furthermore, as $\int_0^1 2(1-s)\dl s = 1$, if $\ddrtran$ is convex, then
    \begin{align}
        \EOf{\int_{0}^{1} (1-s) \ddrtran(\ol Ym + s\,\ol qm) \dl s}
        &\geq
        \frac12\EOf{\ddrtran\brOf{\int_{0}^{1} 2 (1-s) \br{\ol Ym + s\,\ol qm} \dl s}}
        \\&=
        \frac12\EOf{\ddrtran\brOf{\ol Ym + \frac13 \ol qm}}
        \eqfs
    \end{align}
\end{proof}
Using $v(t) = \inf_{q\in\mc Q:\,\ol qm=t}F(q)$, this corollary shows part \ref{thm:infdtr:vi:convex} of \cref{thm:infdtr:vi}, namely
 \begin{equation}
    v(t) \geq \frac12 t^2 \EOf{\ddrtran\brOf{\ol Ym + \frac13 t}}
    \eqfs
\end{equation}
\subsubsection{Part \ref{thm:infdtr:vi:quantile}}
\begin{lemma}\label{lmm:infdtr:vi:near}
    Let $r\in\Rpp$ and set $\rho := \Prof{\ol Ym\leq r}$.
    For all $q\in\mc Q$, we have
    \begin{equation}
        \EOf{\tran(\ol Y{q}) - \tran(\ol Ym)}
        \geq
        \frac\rho2\, \ol qm^2 \ddrtran(r +  \ol qm)
        \eqfs
    \end{equation}
\end{lemma}
\begin{proof}
    The variance inequality \cref{thm:infdtr:vi} states
    \begin{equation}
        \EOf{\tran(\ol Y{q}) - \tran(\ol Ym)}
        \geq
        \frac12
        \ol qm^2
        \EOf{\ddrtran(\ol Ym +  \ol qm)}
        \eqfs
    \end{equation}
    We need to find a suitable lower bound on the expectation in the last term. As $\ddrtran$ is nonincreasing, we have
    \begin{align}
        \EOf{\ddrtran(\ol Ym +  \ol qm)}
        &\geq
        \EOf{\ddrtran(\ol Ym +  \ol qm)\ind_{[0,r]}(\ol Ym)}
        \\&\geq
        \rho \ddrtran(r +  \ol qm)
        \eqfs
    \end{align}
    Thus,
    \begin{equation}
        \EOf{\tran(\ol Y{q}) - \tran(\ol Ym)}
        \geq
        \frac\rho2  \ol qm^2
        \ddrtran(r +  \ol qm)
        \eqfs
        \qedhere
    \end{equation}
\end{proof}
Using $v(t) = \inf_{\ol qm = t} F(q)$, \cref{lmm:infdtr:vi:near} implies part \ref{thm:infdtr:vi:quantile} of \cref{thm:infdtr:vi}.
\subsubsection{Part \ref{thm:infdtr:vi:infty}}

\cite[Theorem 3.5]{varinequ} shows part \ref{thm:infdtr:vi:infty} of \cref{thm:infdtr:vi}, that is
\begin{equation}
    \lim_{t\to\infty} \frac{v(t)}{\tran(t)} = 1
    \eqfs
\end{equation}

\subsubsection{Part \ref{thm:infdtr:vi:combined}}

We first discuss the \emph{near} part, i.e., small $t$. Using $v(t) = \inf_{\ol qm = t} F(q)$, \cref{lmm:infdtr:vi:near} implies
\begin{equation}
    v(t) \geq \frac\rho2 t^2 \ddrtran\brOf{r + t}
\end{equation}
for $r\in\Rpp$ and $\rho =\Prof{\ol Ym\leq r}$.
Next we prove a result for the \emph{far} part, i.e., $t$ large.
\begin{lemma}\label{lmm:infdtr:vi:far}
    Let $r\in\Rpp$ and set $\rho := \Prof{\ol Ym\leq r}$. 
    Let $t_0\in\Rpp$.
    Assume $\rho > 0$ and $\ddrtran(r+t_0) > 0$. Set
    \begin{equation}
        a := \inf_{t\geq t_0}\frac{v(t)}{\tran(t)}
        \eqfs
    \end{equation}
    Then $a>0$. In particular $v(t) \geq a \tran(t) > 0$ for all $t\in [t_0, \infty)$.
\end{lemma}
\begin{proof}
    By \cref{lmm:infdtr:vi:near}, we have $v(t_0) \geq \frac{\rho}2 t_0^2 \ddrtran(r+t_0)$. Set $b_0 := v(t_0)/t_0$.
    Since, $v(t)/t$ is nondecreasing, we have $v(t) \geq b_0 t$ for all $t\geq t_0$.
    If $\diam(\mc Q) = \infty$, by \ref{thm:infdtr:vi:infty}, there is $T \geq t_0$ such that $v(t) \geq \tran(t)/2$ for all $t\geq T$. The existence of such a $T$ is trivial for $\diam(\mc Q) < \infty$.
    For $t\in [t_0, T]$, we have
    \begin{equation}
        \frac{v(t)}{\tran(t)} \geq \frac{b_0 t}{\tran(t)} \geq \frac{b_0 t}{\tran(T)}
        \eqfs
    \end{equation}
    Hence,
    \begin{equation}
        a \geq \min\brOf{\frac{b_0 t_0}{\tran(T)}, \frac12} > 0
    \end{equation}
    and $v(t) \geq a\tran(t)$ for all $t\geq t_0$.
\end{proof}

Together, \cref{lmm:infdtr:vi:near} and \cref{lmm:infdtr:vi:far} applied with $r = t_0$ and using the bound $\ddrtran(r+t) \geq \ddrtran(2t_0)$ for $t \leq t_0$ yield \ref{thm:infdtr:vi:combined} of \cref{thm:infdtr:vi}.

\subsubsection{Convexity Counter Example}

\begin{example}\label{exa:v:notconvex}
    Let $\mc Q=\R$ and let $\tran(r)=r^{3/2}$. Define
    \begin{equation}
        F(x)
        =
        \frac16 |x-1|^{3/2}
        +
        \frac13 |x-16|^{3/2}
        +
        \frac12 |x+9|^{3/2}
        \eqfs
    \end{equation}
    Since $x\mapsto |x-a|^{3/2}$ is strictly convex, $F$ is strictly
    convex. Moreover,
    \begin{equation}
        F'(0)
        =
        -\frac32
        \biggl(
        \frac16\sqrt1
        +
        \frac13\sqrt{16}
        -
        \frac12\sqrt9
        \biggr)
        =
        0
        \eqfs
    \end{equation}
    Therefore $m=0$ is the unique minimizer of $F$.
    Define
    \begin{equation}
        v(t)
        =
        \inf_{|x|\ge t}
        \bigl(F(x)-F(0)\bigr),
        \qquad t\ge0
        \eqfs
    \end{equation}
    Since $F$ is convex and minimized at $0$,
    \begin{equation}
        v(t)
        =
        \min(F(t),F(-t))-F(0)
        \eqfs
    \end{equation}
    A direct computation yields
    \begin{align}
        v(9)&=F(-9)-F(0)
        =
        \frac{20+5\sqrt{10}}{3}
        \eqcm
        \\
        v(11)&=F(11)-F(0)
        =
        \frac{5\sqrt{10}+65\sqrt5-105}{3}
        \eqcm
        \\
        v(13)&=F(13)-F(0)
        =
        5\sqrt3+11\sqrt{22}-35
        \eqfs
    \end{align}
    Thus,
    \begin{align}
        v(11)-\frac{v(9)+v(13)}2
        &=
        \frac{
            5\sqrt{10}
            +130\sqrt5
            -15\sqrt3
            -33\sqrt{22}
            -125
        }{6}
        \\
        &\approx 0.123
        >
        0
        \eqfs
    \end{align}
    Hence,
    \begin{equation}
        v(11)
        >
        \frac{v(9)+v(13)}2
        \eqcm
    \end{equation}
    so $v$ is not convex.
    
    The non-convexity arises because $v(t)=\min(F(t),F(-t))-F(0)$ is the minimum of two convex functions. The active branch switches between $t=9$ and $t=13$, creating a kink with decreasing slope and thus violating convexity.
\end{example}

%% file: sec_proof_unbounded.tex
\section{Proofs of Section: Unbounded Influence}

\subsection{Power Fr\'echet Means}\label{sec:powerproof}
In this section we prove \cref{thm:power:main}. Throughout this section, assume the setting and conditions of \cref{ssec:prelim:setup} and \cref{thm:power:main}. Furthermore, let $Y_1\pr, \dots, Y_n\pr$ be an additional independent set of iid copies of $Y$.
Denote the samples with $i$-th position replaced as $Y_j^i := Y_j$ if $i \neq j$ and $Y_i^i := Y_i\pr$.
Let $m_n^i$ be the sample $\alpha$-Fr\'echet mean of the sample set $Y_1^i, \dots, Y_n^i$, 
\begin{equation}
    m_n^i \in \argmin_{q\in\mc Q} \sum_{j=1}^n \ol {Y_j^i}q^\alpha
    \eqfs
\end{equation}

Define the \emph{double excess risk} as
\begin{equation}
    V_n := \EOf{\ol Y{m_n}^\alpha - \ol Ym^\alpha} + \EOf{\frac1n \sum_{i=1}^n \br{\ol {Y_i}{m}^\alpha - \ol {Y_i}{m_n}^\alpha} }
    \eqfs
\end{equation}
Recall that $Y, Y_1, \dots, Y_n$ are iid, so that $Y$ is independent of $m_n$.
\begin{lemma}\label{lmm:power:variancebound}
	We have
    \begin{equation}
        V_n \leq \frac{2^{1-\alpha} \alpha}n \sum_{i=1}^n \EOf{\ol {m_n}{m_n^i}\, \ol {Y_i}{Y_i^i}^{\alpha-1}}
        \eqfs
    \end{equation}
\end{lemma}
\begin{proof}
    As $Y$ has the same distribution as $Y_i$ and $(Y, m_n)$ has the same distribution as $(Y_i, m_n^i)$, we have
    \begin{equation}\label{eq:power:var:simple1}
        V_n 
        =
        \EOf{\ol Y{m_n}^\alpha - \frac1n \sum_{i=1}^n \ol {Y_i}{m_n}^\alpha} 
        =
        \frac1n \sum_{i=1}^n \EOf{\ol {Y_i}{m_n^i}^\alpha - \ol {Y_i}{m_n}^\alpha} 
        \eqfs
    \end{equation}
    By the quadruple inequality, \cref{thm:power:qi}, we have 
    \begin{equation}
        \br{\ol {Y_i}{m_n^i}^\alpha - \ol {Y_i}{m_n}^\alpha}
        + 
        \br{\ol {Y_i^i}{m_n}^\alpha - \ol {Y_i^i}{m_n^i}^\alpha}
        \leq
        2^{2-\alpha} \alpha \,\ol {m_n}{m_n^i} \,\ol {Y_i}{Y_i^i}^{\alpha-1}
        \eqfs
    \end{equation}
    As $(Y_i, m_n, m_n^i)$ has the same distribution as $(Y_i^i, m_n^i, m_n)$, we obtain
    \begin{equation}\label{eq:power:var:distri}
        2 \EOf{\ol {Y_i}{m_n^i}^\alpha - \ol {Y_i}{m_n}^\alpha} \leq 2^{2-\alpha} \alpha \EOf{\ol {m_n}{m_n^i} \,\ol {Y_i}{Y_i^i}^{\alpha-1}}
        \eqfs
    \end{equation}
    Taking \eqref{eq:power:var:simple1} and \eqref{eq:power:var:distri} together yields
    \begin{equation}
        V_n \leq 2^{1-\alpha} \alpha \frac1n \sum_{i=1}^n \EOf{\ol {m_n}{m_n^i} \,\ol {Y_i}{Y_i^i}^{\alpha-1}}
        \eqfs
		\qedhere
    \end{equation}
\end{proof}
Define
\begin{equation}
    \tilde H_i := \frac1n \sum_{j=1}^n \br{\br{\ol {Y_j}{m_n} + \frac13\ol {m_n}{m_n^i}}^{\alpha-2} + \br{\ol {Y_j^i}{m_n^i} + \frac13\ol{m_n}{m_n^i}}^{\alpha-2}}
    \eqfs
\end{equation}
\begin{lemma}\label{lmm:power:mnmnibound}
    We have
    \begin{equation}
        \ol {m_n}{m_n^i} \tilde H_i \leq \frac{2^{3-\alpha}}{\alpha-1}\frac1n \ol {Y_i}{Y_i^i}^{\alpha-1}
        \eqfs
    \end{equation}
\end{lemma}
\begin{proof}
    The variance inequality \cref{thm:infdtr:vi} applied to $\tran(x)=x^\alpha$ with convex second derivative $\ddtran(x) = \alpha(\alpha-1)x^{\alpha-2}$ on the empirical distributions yields, for $q\in\mc Q$,
    \begin{align}
        \frac{\alpha(\alpha-1)}2 \ol q{m_n}^2 \frac1n \sum_{j=1}^n \br{\ol {Y_j}{m_n} + \frac13\ol q{m_n}}^{\alpha-2}
        &\leq
        \frac1n \sum_{j=1}^n \br{\ol {Y_j}q^{\alpha}-\ol {Y_j}{m_n}^{\alpha}} 
        \eqcm
        \\
        \frac{\alpha(\alpha-1)}2 \ol q{m_n^i}^2 \frac1n \sum_{j=1}^n \br{\ol {Y_j^i}{m_n^i} + \frac13\ol q{m_n^i}}^{\alpha-2}
        &\leq
        \frac1n \sum_{j=1}^n \br{\ol {Y_j^i}q^{\alpha}-\ol {Y_j^i}{m_n^i}^{\alpha}} 
        \eqfs
    \end{align}
    Thus, plugging in $q = m_n^i$ and $q = m_n$ respectively, adding the two inequalities, and using the quadruple inequality, \cref{thm:power:qi}, we get
    \begin{align}
        \frac{\alpha(\alpha-1)}2 \ol {m_n}{m_n^i}^2 \tilde H_i 
        &\leq 
        \frac1n \sum_{j=1}^n \br{
            \ol {Y_j}{m_n^i}^{\alpha} -
            \ol {Y_j}{m_n}^{\alpha} +
            \ol {Y_j^i}{m_n}^{\alpha} -
            \ol {Y_j^i}{m_n^i}^{\alpha}
        }  
        \\&\leq
        2^{2-\alpha}\alpha
        \frac1n \sum_{j=1}^n
        \ol {m_n}{m_n^i}
        \,
        \ol {Y_j}{Y_j^i}^{\alpha-1}
        \eqfs
    \end{align}
    As $Y_j^i = Y_j$ for $i\neq j$, we obtain
    \begin{equation}
        \frac{\alpha(\alpha-1)}2 \ol {m_n}{m_n^i}^2 \tilde H_i 
        \leq
        2^{2-\alpha}\alpha
        \frac1n \,
        \ol {m_n}{m_n^i}\,
        \ol {Y_i}{Y_i^i}^{\alpha-1}
        \eqfs
    \end{equation}
    Rearranging the terms yields the desired result.
\end{proof}
\begin{notation}\label{def:power:moments:empirical}
	For $\varphi\in\Rp$, define
	\begin{align}
		\hat \sigma_{\varphi} &:= \frac1n\sum_{j=1}^n\ol {Y_j}m^\varphi\eqcm
		&
		\hat \sigma_{\varphi}^i &:= \frac1n\sum_{j=1}^n\ol {Y_j^i}m^\varphi
	\end{align}
	with the convention $0^0:= 1$.
\end{notation}
\begin{lemma}\label{lmm:power:rough}
    We have
    \begin{equation}
        \ol m{m_n}^{\alpha -1}
        \leq
        2^{2-\alpha} \alpha \br{
            2 \sigma_{\alpha-1} + \hat \sigma_{\alpha-1}
        }
        \eqfs
    \end{equation}
\end{lemma}
\begin{proof}
    Let $y,q,p\in\mc Q$. The quadruple inequality, \cref{thm:power:qi}, applied with $z:=q$ yields
    \begin{equation}
        \ol yp^\alpha - \ol yq^\alpha
        \geq 
        \ol qp^\alpha - 2^{2-\alpha} \alpha \,\ol qp\, \ol yq^{\alpha-1}
        \eqfs
    \end{equation}
    In particular, we have
    \begin{equation}
        \EOf{\ol Y{m_n}^\alpha - \ol Ym^\alpha | Y_1, \dots, Y_n} \geq 
        \ol m{m_n}^\alpha - 2^{2-\alpha} \alpha \,\ol m{m_n}\, \EOf{\ol Ym^{\alpha-1}}
        \eqcm
    \end{equation}
    where we recall that $Y, Y_1, \dots, Y_n$ are iid, $m_n$ is $\sigma(Y_1, \dots, Y_n)$-measurable and, hence, $m_n$ and $Y$ are independent.
    By the minimizing property of $m_n$ we also have
    \begin{equation}
        \frac1n\sum_{i=1}^n\br{\ol {Y_i}{m}^\alpha - \ol {Y_i}{m_n}^\alpha} \geq 0
        \eqfs
    \end{equation}
    Putting the last two inequalities together and using quadruple inequality, \cref{thm:power:qi}, we get
    \begin{align}
        \ol m{m_n}^\alpha - 2^{2-\alpha} \alpha \,\ol m{m_n}\, \EOf{\ol Ym^{\alpha-1}}
        &\leq 
        \EOf{\ol Y{m_n}^\alpha - \ol Ym^\alpha | Y_1, \dots, Y_n} + 
        \frac1n\sum_{i=1}^n\br{\ol {Y_i}{m}^\alpha - \ol {Y_i}{m_n}^\alpha}
        \\&\leq
        2^{2-\alpha} \alpha \, \ol m{m_n} \frac1n \sum_{i=1}^n \EOf{\ol Y{Y_i}^{\alpha-1} | Y_i}
        \eqfs
    \end{align}
    Rearranging the terms yields
    \begin{equation}
        \ol m{m_n}^{\alpha -1}
        \leq
        2^{2-\alpha} \alpha \br{
            \frac1n \sum_{i=1}^n \EOf{\ol Y{Y_i}^{\alpha-1} | Y_i}
            +
            \EOf{\ol Ym^{\alpha-1}}
        }
        \eqfs
    \end{equation}
    As $\alpha-1\in(0,1]$, we have $\ol Y{Y_i}^{\alpha-1} \leq \ol Y{m}^{\alpha-1} + \ol {Y_i}m^{\alpha-1}$, which concludes the proof.
\end{proof}
\begin{lemma}\label{lmm:power:iid:weights}
    Set $\phi := \frac{2-\alpha}{\alpha-1}$ and
    \begin{align}
        w_1 &:= \br{\br{\frac43}^{\alpha-1}  + \br{\frac13}^{\alpha-1}} 2^{3-\alpha} \alpha
        \eqcm\\ 
        w_2 &:= 1 + \br{\frac43}^{\alpha-1} 2^{2-\alpha} \alpha
        \eqcm\\
        w_3 &:= \br{\frac13}^{\alpha-1} 2^{2-\alpha} \alpha
        \eqcm
    \end{align}
    as well as $w := w_1+w_2+w_3$.
	\begin{enumerate}[label=(\roman*)]
		\item
		If $\phi \leq 1$, then
		\begin{align}
			&\EOf{\ol {Y_i}{Y_i^i}^{2\alpha-2} \tilde H_i^{-1}}
			\\&\leq
			2^{2\alpha-3} \br{w_1^{\phi}+w_2^{\phi}+w_3^{\phi}} \sigma_{2\alpha-2} (\sigma_{\alpha-1})^{\phi}
			+
			2^{2\alpha-4}  \br{w_2^{\phi}+w_3^{\phi}} n^{-\phi} \sigma_\alpha 
            \eqfs
		\end{align}
		\item 
		If $\phi\geq 1$, then
		\begin{align}
			&\EOf{
				\ol {Y_i}{Y_i^i}^{2\alpha-2}
				\tilde H_i^{-1}
			}
			\\&\leq
			w_1 w^{\phi-1} \sigma_{2\alpha-2} (\sigma_{\alpha-1})^\phi
			\\&\quad +
			\frac12 (w_2+w_3) w^{\phi-1} \br{
				2^{\phi-1} n^{-\phi} \sigma_{\alpha} + \br{1+2^{\phi-1}} \sigma_{2\alpha-2} \EOf{(\hat\sigma_{\alpha-1})^{\phi}}
			}
            \eqfs
		\end{align}
	\end{enumerate}
\end{lemma}
\begin{proof}
    We use the convexity of $t\mapsto t^{-\phi}$ with Jensen's inequality and $s^{-1} + t^{-1} \geq 4(s+t)^{-1}$ for $s,t\in\Rpp$ to obtain
	\begin{align}
		\tilde H_i^{-1}
		&=
		\br{\frac1n \sum_{j=1}^n \br{(\ol{Y_j}{m_n} + \frac13\ol{m_n}{m_n^i})^{-\phi(\alpha-1)} + (\ol{Y_j^i}{m_n^i} + \frac13\ol{m_n}{m_n^i})^{-\phi(\alpha-1)}}}^{-1}
		\\&\leq
		\br{\br{\frac1n \sum_{j=1}^n (\ol{Y_j}{m_n} + \frac13\ol{m_n}{m_n^i})^{\alpha-1}}^{-\phi} + \br{\frac1n \sum_{j=1}^n (\ol{Y_j^i}{m_n^i} + \frac13\ol{m_n}{m_n^i})^{\alpha-1}}^{-\phi}}^{-1}
		\\&\leq
		\frac14\br{\br{\frac1n \sum_{j=1}^n (\ol{Y_j}{m_n} + \frac13\ol{m_n}{m_n^i})^{\alpha-1}}^{\phi} + \br{\frac1n \sum_{j=1}^n (\ol{Y_j^i}{m_n^i} + \frac13\ol{m_n}{m_n^i})^{\alpha-1}}^{\phi}}
        \eqfs
	\end{align}
	Because of the symmetry $(\mo Y, m_n, \mo Y^i, m_n^i) \stackrel{d}{=} (\mo Y^i, m_n^i, \mo Y, m_n)$, we have
	\begin{align}
		\EOf{
			\ol {Y_i}{Y_i^i}^{2\alpha-2}
			\tilde H_i^{-1}
		}
		\leq 
		\frac12\EOf{
			\ol {Y_i}{Y_i^i}^{2\alpha-2}
			\br{\frac1n \sum_{j=1}^n (\ol{Y_j}{m_n} + \frac13\ol{m_n}{m_n^i})^{\alpha-1}}^{\phi}
		}
        \eqfs
	\end{align}
    We apply the triangle inequality together with the subadditivity of $t \mapsto t^{\alpha-1}$ as well as \cref{lmm:power:rough} for the bound 
	\begin{align}
		&\frac1n \sum_{j=1}^n (\ol{Y_j}{m_n} + \frac13\ol{m_n}{m_n^i})^{\alpha-1}
		\\&\leq
		\frac1n \sum_{j=1}^n \ol{Y_j}{m}^{\alpha-1} + \br{\frac43\, \ol{m}{m_n}}^{\alpha-1} + \br{\frac13\ol{m}{m_n^i}}^{\alpha-1}
		\\&=
		\hat\sigma_{\alpha-1} + \br{\frac43}^{\alpha-1} \ol{m}{m_n}^{\alpha-1} + \br{\frac13}^{\alpha-1}\ol{m}{m_n^i}^{\alpha-1}
		\\&\leq
		\hat\sigma_{\alpha-1} +\br{\frac43}^{\alpha-1} 2^{2-\alpha} \alpha \br{
				2 \sigma_{\alpha-1} + \hat \sigma_{\alpha-1}
			} + \br{\frac13}^{\alpha-1} 2^{2-\alpha} \alpha \br{
				2 \sigma_{\alpha-1} + \hat \sigma^i_{\alpha-1}
		}
		\\&\leq
		w_1 \sigma_{\alpha-1}
		+
		w_2 \hat\sigma_{\alpha-1}
		+
		w_3 \hat\sigma^i_{\alpha-1}
		\eqfs
	\end{align}
	\begin{enumerate}[label=(\roman*)]
		\item
	If $0 \leq \phi \leq 1$, then $t\mapsto t^\phi$ is subadditive and
	\begin{equation}
		\EOf{\ol {Y_i}{Y_i^i}^{2\alpha-2} \tilde H_i^{-1}}
		\leq
		\frac12\EOf{
			\ol {Y_i}{Y_i^i}^{2\alpha-2}
			\br{
				w_1^{\phi} (\sigma_{\alpha-1})^{\phi}
				+
				w_2^{\phi} (\hat\sigma_{\alpha-1})^{\phi}
				+
				w_3^{\phi} (\hat\sigma^i_{\alpha-1})^{\phi}
			}
		}
		\eqfs
	\end{equation}
	As $\phi\leq 1$, we have $\alpha \geq \frac32$. Hence, $2\alpha-2\geq 1$. Therefore, $(s+t)^{2\alpha-2} \leq 2^{2\alpha-3} (s^{2\alpha-2} + t^{2\alpha-2})$ for all $s,t\in\Rp$, see \cref{lmm:power:subadd}. Hence,
	\begin{align}
		&\EOf{\ol {Y_i}{Y_i^i}^{2\alpha-2} \tilde H_i^{-1}}
		\\&\leq
		2^{2\alpha-4} \EOf{
				\ol {Y_i}{m}^{2\alpha-2}
				\br{
					2 w_1^{\phi} (\sigma_{\alpha-1})^{\phi}
					+
					\br{w_2^{\phi}+w_3^{\phi}} (\hat\sigma_{\alpha-1})^{\phi}
					+
					\br{w_2^{\phi}+w_3^{\phi}} (\hat\sigma^i_{\alpha-1})^{\phi}
				}
			}
		\eqfs
	\end{align}
	Using $(\hat\sigma_{\alpha-1})^{\phi} \leq n^{-\phi} \ol{Y_i}m^{2-\alpha} + (\hat\sigma^i_{\alpha-1})^{\phi}$, we obtain
	\begin{align}
		\EOf{
			\ol {Y_i}{m}^{2\alpha-2}
			(\hat\sigma_{\alpha-1})^{\phi}
		}
		\leq
		n^{-\phi} \sigma_\alpha 
		+
		\sigma_{2\alpha-2}\EOf{(\hat\sigma_{\alpha-1})^{\phi}}
        \eqfs
	\end{align}
	Hence, using $\EOf{\br{\hat\sigma_{\alpha-1}}^{\phi}} \leq \EOf{\hat\sigma_{\alpha-1}}^{\phi} =  \sigma_{\alpha-1}^{\phi}$, we obtain
	\begin{align}
		&\EOf{\ol {Y_i}{Y_i^i}^{2\alpha-2} \tilde H_i^{-1}}
		\\&\leq
		2^{2\alpha-4} 
		\br{
			2 w_1^{\phi} \sigma_{2\alpha-2} (\sigma_{\alpha-1})^{\phi}
			+
			2 \br{w_2^{\phi}+w_3^{\phi}} \sigma_{2\alpha-2}\EOf{(\hat\sigma_{\alpha-1})^{\phi}}
			+
			\br{w_2^{\phi}+w_3^{\phi}} n^{-\phi} \sigma_\alpha 
		}
		\\&\leq
		2^{2\alpha-3} \br{w_1^{\phi}+w_2^{\phi}+w_3^{\phi}} \sigma_{2\alpha-2} (\sigma_{\alpha-1})^{\phi}
		+
		2^{2\alpha-4}  \br{w_2^{\phi}+w_3^{\phi}} n^{-\phi} \sigma_\alpha 
		\eqfs
	\end{align}
	\item 
	If $\phi\geq1$, set $w := w_1+w_2+w_3$. Then, using Jensen's inequality,
	\begin{align}
		&\EOf{
			\ol {Y_i}{Y_i^i}^{2\alpha-2}
			\tilde H_i^{-1}
		}
		\\&\leq
		\frac{w^\phi}2
		\EOf{
			\ol {Y_i}{Y_i^i}^{2\alpha-2}
			\br{
				\frac{w_1}{w} \sigma_{\alpha-1}
				+
				\frac{w_2}{w} \hat\sigma_{\alpha-1}
				+
				\frac{w_3}{w} \hat\sigma^i_{\alpha-1}
			}^{\phi}
		}
		\\&\leq
		\frac{1}2
		\EOf{
			\ol {Y_i}{Y_i^i}^{2\alpha-2}
			\br{
				w_1 w^{\phi-1} (\sigma_{\alpha-1})^{\phi}
				+
				w_2 w^{\phi-1} (\hat\sigma_{\alpha-1})^{\phi}
				+
				w_3 w^{\phi-1} (\hat\sigma^i_{\alpha-1})^{\phi}
			}
		}
        \eqfs
	\end{align}	
	As $\phi\geq 1$, we have $1 < \alpha \leq \frac32$. Hence, $0< 2\alpha-2\leq 1$. Therefore, $(s+t)^{2\alpha-2} \leq s^{2\alpha-2} + t^{2\alpha-2}$ for all $s,t\geq 0$.
	Hence,
		\begin{align}
		&\EOf{
			\ol {Y_i}{Y_i^i}^{2\alpha-2}
			\tilde H_i^{-1}
		}
		\\&\leq
		\EOf{
			\ol {Y_i}{m}^{2\alpha-2}
			\br{
				w_1 w^{\phi-1} (\sigma_{\alpha-1})^{\phi}
				+
				\frac12(w_2 + w_3) w^{\phi-1} (\hat\sigma_{\alpha-1})^{\phi}
				+
				\frac12(w_2 + w_3) w^{\phi-1} (\hat\sigma^i_{\alpha-1})^{\phi}
			}
		}
        \eqfs
	\end{align}	
	Using 
	\begin{align}
		(\hat\sigma_{\alpha-1})^{\phi} 
		&\leq
		\br{\frac1n \ol{Y_i}m^{\alpha-1}+\hat\sigma^i_{\alpha-1}}^\phi
		\\&\leq
		2^{\phi-1} n^{-\phi} \ol{Y_i}m^{2-\alpha} + 2^{\phi-1} (\hat\sigma^i_{\alpha-1})^{\phi}
        \eqcm
	\end{align}
	we obtain
	\begin{equation}
		\EOf{
			\ol {Y_i}{m}^{2\alpha-2}
			(\hat\sigma_{\alpha-1})^{\phi}
		}
		\leq
		2^{\phi-1} n^{-\phi} \sigma_{\alpha} + 2^{\phi-1} \sigma_{2\alpha-2} \EOf{(\hat\sigma_{\alpha-1})^{\phi}}
        \eqfs
	\end{equation}
	Hence,
    \begin{align}
        &\EOf{
            \ol {Y_i}{Y_i^i}^{2\alpha-2}
            \tilde H_i^{-1}
        }
        \\&\leq
        w_1 w^{\phi-1} \sigma_{2\alpha-2} (\sigma_{\alpha-1})^\phi
        \\&\quad +
        \frac12 (w_2+w_3) w^{\phi-1} \br{
            2^{\phi-1} n^{-\phi} \sigma_{\alpha} + \br{1+2^{\phi-1}} \sigma_{2\alpha-2} \EOf{(\hat\sigma_{\alpha-1})^{\phi}}
        }
        \eqfs
    \end{align}
	\end{enumerate}
\end{proof}
For later reference, we state the Rosenthal inequality.
\begin{lemma}\label{lmm:rosenthal}
    Let $X_1, \dots, X_n$ be real-valued, independent, centered random variables. 
    Let $\varphi \geq 2$.
    Then we have the Rosenthal inequality \cite{Rosenthal1970,Johnson1985}
    \begin{equation}
        \EOf{\absOf{\sum_{i=1}^n X_i}^\varphi} \leq R_\varphi \max\brOf{ \sum_{i=1}^n\EOf{|X_i|^\varphi},\ \br{\sum_{i=1}^n \EOf{X_i^2}}^{\frac \varphi2}}
        \eqcm
    \end{equation}
    where the constant $R_\varphi$ depends only on $\varphi$ and, by \cite[Theorem 4.1, Corollary 2.6 and the remark on p.~247]{Johnson1985}, may be taken equal to
    \begin{equation}\label{eq:rosenthal-const}
        R_\varphi := \min\brOf{\br{\frac{14.7 \varphi}{\mathrm{Log}(\varphi)}}^{\varphi},\, \br{\frac{2\varphi}{\sqrt{\mathrm{Log}(\varphi)}}}^{\varphi}}
    \end{equation}
    with $\mathrm{Log}(\varphi) := \max(1, \log(\varphi))$ and $\log$ is the natural logarithm.
\end{lemma}
\begin{lemma}\label{lmm:power:iid:empmom}
	\mbox{}
	\begin{enumerate}[label=(\roman*)]
		\item
		If $\varphi \in [1,2]$, then
		\begin{equation}
			\EOf{\br{\hat\sigma_{\alpha-1}}^{\varphi}} 
			\leq
			2^{\varphi-1} \sigma_{\alpha-1}^{\varphi} + 2^{\varphi-1} n^{-\frac{\varphi}{2}} \sigma_{2\alpha-2}^{\frac{\varphi}{2}} 
		\end{equation}
		\item
		If $\varphi \geq 2$, then
        \begin{equation}
            \EOf{\br{\hat\sigma_{\alpha-1}}^{\varphi}}
            \leq 
            2^{\varphi-1} \sigma_{\alpha-1}^{\varphi} + 2^{\varphi-1} R_\varphi \max\brOf{ n^{1-\varphi} \sigma_{\varphi(\alpha-1)},\ n^{-\frac{\varphi}{2}} \sigma_{2\alpha-2}^{\frac{\varphi}{2}}}
            \eqcm
        \end{equation}
        where $R_\varphi$ is the Rosenthal constant from \cref{lmm:rosenthal}.
	\end{enumerate}
\end{lemma}
\begin{proof}
	Define $X_j := \ol{Y_j}{m}^{\alpha-1} - \Eof{\ol{Y_j}{m}^{\alpha-1}}$. Then $X_1, \dots, X_n$ are iid centered real-valued random variables.
	Define $\overline{X}_n := \frac1n \sum_{j=1}^n X_j$.
	Then, using $\Eof{\ol{Y_j}{m}^{\alpha-1}} = \sigma_{\alpha-1}$, we have
	\begin{align}
		\EOf{\br{\hat\sigma_{\alpha-1}}^{\varphi}}
		&=
		\EOf{\br{\sigma_{\alpha-1} + \overline{X}_n}^{\varphi}}
		\\&\leq 
		2^{\varphi-1} \sigma_{\alpha-1}^{\varphi} + 2^{\varphi-1} \EOf{\absOf{\overline{X}_n}^{\varphi}}
	\end{align}
	\begin{enumerate}[label=(\roman*)]
			\item
			If $1 \leq \varphi \leq 2$, then
			\begin{equation}
				\EOf{\absOf{\sum_{j=1}^n X_j}^\varphi}
				\leq
				\EOf{{\absOf{\sum_{j=1}^nX_j}^2}}^{\frac{\varphi}{2}}
				\leq
				n^{\frac{\varphi}{2}} \EOf{X_1^2}^{\frac{\varphi}{2}}
			\end{equation}
			and
			\begin{equation}
				\EOf{X_1^2}
				= 
				\VOf{\ol{Y_1}{m}^{\alpha-1}}
				=
				\sigma_{2\alpha-2} - (\sigma_{\alpha-1})^2
                \leq
                \sigma_{2\alpha-2}
				\eqfs
			\end{equation}
			Hence,
			\begin{equation}
				\EOf{\absOf{\overline{X}_n}^{\varphi}}
				\leq
                n^{-\frac{\varphi}{2}} \sigma_{2\alpha-2}^{\frac{\varphi}{2}}
				\eqfs
			\end{equation}
			In total, we obtain
			\begin{equation}
				\EOf{\br{\hat\sigma_{\alpha-1}}^{\varphi}} 
				\leq 
				2^{\varphi-1} \sigma_{\alpha-1}^{\varphi} + 2^{\varphi-1} n^{-\frac{\varphi}{2}} \sigma_{2\alpha-2}^{\frac{\varphi}{2}}
				\eqfs
			\end{equation}
		\item
			As $X_j$ are iid and centered and $\varphi \geq 2$, we have by Rosenthal's inequality, \cref{lmm:rosenthal},
            \begin{equation}
                \EOf{\absOf{\overline{X}_n}^\varphi} \leq R_\varphi \max\brOf{ n^{1-\varphi} \EOf{|X_1|^\varphi},\ n^{-\frac{\varphi}{2}} \EOf{X_1^2}^{\frac{\varphi}{2}}}
            \end{equation}
			If $Z \geq 0$ and $p\geq 2$ and $\Eof{Z^p} <\infty$ and $\mu := \Eof{Z}$, then
			\begin{equation}
				\EOf{Z^p} \geq \EOf{|Z-\mu|^p} + \mu^p
				\eqfs
			\end{equation}
			Hence,
			\begin{equation}
				\EOf{|X_1|^\varphi}
				=
				\EOf{\absOf{\ol{Y_1}{m}^{\alpha-1} - \EOf{\ol{Y_1}{m}^{\alpha-1}}}^\varphi}
				\leq
				\sigma_{\varphi(\alpha-1)}
				\eqfs
			\end{equation}
			Thus
			\begin{equation}
				\EOf{\br{\hat\sigma_{\alpha-1}}^{\varphi}}
				\leq 
				2^{\varphi-1} \sigma_{\alpha-1}^{\varphi} + 2^{\varphi-1} R_\varphi \max\brOf{ n^{1-\varphi} \sigma_{\varphi(\alpha-1)},\ n^{-\frac{\varphi}{2}} \sigma_{2\alpha-2}^{\frac{\varphi}{2}}}
				\eqfs
			\end{equation}
	\end{enumerate}
\end{proof}
\begin{lemma}\label{lmm:power:vi}
    Set $\phi := \frac{2-\alpha}{\alpha-1}$.
	For all $q\in\mc Q$, we have 
	\begin{equation}
		\frac12 \alpha (\alpha-1) \ol{q}{m}^2 \br{\sigma_{\alpha-1} + 3^{1-\alpha}\ol{q}{m}^{\alpha-1}}^{-\phi} \leq \EOf{\ol Yq^\alpha - \ol Y{m}^\alpha}
	\end{equation}
    and, in particular,
    \begin{equation}
        \frac12 \alpha (\alpha-1) \ol{q}{m}^2 \br{\sigma_{\alpha-1} + \ol{q}{m}^{\alpha-1}}^{-\phi} \leq \EOf{\ol Yq^\alpha - \ol Y{m}^\alpha}
        \eqfs
    \end{equation}
\end{lemma}
\begin{proof}
	\Cref{thm:infdtr:vi} \ref{thm:infdtr:vi:convex} applied to $\tran(x) = x^\alpha$ shows
	\begin{equation}
		\frac12 \alpha (\alpha-1) \ol{q}{m}^2 \EOf{\br{\ol Y{m}+\frac13\ol q{m}}^{\alpha-2}}  \leq \EOf{\ol Yq^\alpha - \ol Y{m}^\alpha}
		\eqfs
	\end{equation}
	Set $\phi := \frac{2-\alpha}{\alpha-1}$. Then using Jensen's inequality and subadditivity for $x\mapsto x^{\alpha-1}$ yields
	\begin{align}
		\EOf{\br{\ol Y{m} + 3^{-1}\ol q{m}}^{\alpha-2}}
		&\geq 
		\EOf{\br{\ol Y{m} + 3^{-1}\ol{q}{m}}^{\alpha-1}}^{-\phi}
		\\&\geq
		\br{\sigma_{\alpha-1} + 3^{1-\alpha}\ol{q}{m}^{\alpha-1}}^{-\phi}
		\eqfs
	\end{align}
	Hence,
	\begin{equation}
		\frac12 \alpha (\alpha-1) \ol{q}{m}^2 \br{\sigma_{\alpha-1} + 3^{1-\alpha}\ol{q}{m}^{\alpha-1}}^{-\phi} \leq \EOf{\ol Yq^\alpha - \ol Y{m}^\alpha}
		\eqfs
	\end{equation}
\end{proof}

\begin{lemma}\label{lmm:power:momentcomp}
    Let $\alpha\in(1,2]$ and $\phi := \frac{2-\alpha}{\alpha-1}$. Assume $\sigma_\alpha<\infty$. Then
    \begin{enumerate}[label=(\roman*)]
        \item \label{lmm:power:momentcomp:interpol}
        $\sigma_{2\alpha-2} \leq \sigma_{\alpha-1}^{2-\alpha}\,\sigma_{\alpha}^{\alpha-1}$
        and, equivalently,
        $\sigma_{2\alpha-2}^{1+\phi} \leq \sigma_{\alpha-1}^{\phi}\,\sigma_{\alpha}$;
        \item \label{lmm:power:momentcomp:chord}
        if $\phi\geq2$, then
        $\sigma_{2-\alpha}^{\phi}\,\sigma_{2\alpha-2}^{\phi-1} \leq \sigma_{\alpha-1}^{\phi}\,\sigma_{\alpha}^{\phi-1}$.
    \end{enumerate}
\end{lemma}
\begin{proof}
    Set $X := \ol Ym$ and $L(\varphi) := \log \EOf{X^\varphi}$ for $\varphi\in[0,\alpha]$.
    All moments occurring below have order in $[0,\alpha]$ and are finite, as $X^\varphi \leq 1 + X^\alpha$ for $\varphi\in[0,\alpha]$.
    If $\sigma_\varphi = 0$ for some $\varphi\in(0,\alpha]$, then $X=0$ almost surely, all asserted inequalities read $0\leq0$, and there is nothing to show. Hence, assume $\sigma_\varphi\in\Rpp$ for all $\varphi \in (0,\alpha]$, so that $L$ is well-defined and, by H\"older's inequality, convex.
    
    \ref{lmm:power:momentcomp:interpol}:
    With the weights $\lambda := 2-\alpha \in[0,1)$ and $1-\lambda = \alpha-1$, we have
    $\lambda (\alpha-1) + (1-\lambda) \alpha = (\alpha-1)(2-\alpha) + (\alpha-1)\alpha = 2\alpha-2$.
    Thus, convexity of $L$ yields $\sigma_{2\alpha-2} \leq \sigma_{\alpha-1}^{2-\alpha}\sigma_{\alpha}^{\alpha-1}$.
    Raising this to the power $\frac1{\alpha-1} = 1+\phi$ gives the second formulation.
    
    \ref{lmm:power:momentcomp:chord}:
    Note that $\phi\geq2$ is equivalent to $\alpha\leq\frac43$. Hence,
    $\alpha - 1 \leq 2\alpha-2 \leq 2-\alpha \leq \alpha$ with $3-2\alpha>0$ and $2-\alpha>0$.
    As $L$ is convex, its difference quotients are nondecreasing in both endpoints, so that
    \begin{equation}
        \frac{L(2-\alpha) - L(\alpha-1)}{3-2\alpha}
        \leq
        \frac{L(\alpha) - L(2\alpha-2)}{2-\alpha}
        \eqfs
    \end{equation}
    Multiplying by $\frac{(2-\alpha)(3-2\alpha)}{\alpha-1} \in \Rpp$ and using
    $\phi = \frac{2-\alpha}{\alpha-1}$ and $\phi - 1 = \frac{3-2\alpha}{\alpha-1}$ gives
    \begin{equation}
        \phi\br{L(2-\alpha) - L(\alpha-1)} \leq (\phi-1)\br{L(\alpha) - L(2\alpha-2)}
        \eqcm
    \end{equation}
    which is the assertion after exponentiating.
\end{proof}
\begin{lemma}\label{lmm:power:absorbS}
    Let $\alpha\in(1,2)$ and $\phi := \frac{2-\alpha}{\alpha-1} \in \Rpp$. Assume $\sigma_\alpha<\infty$ and set
    \begin{equation}
        B_1 := \sigma_{\alpha-1}^{\phi}\, \sigma_{2\alpha-2}
        \eqcm\qquad
        B_2 := n^{-\phi}\, \sigma_\alpha
        \eqfs
    \end{equation}
    Then, for all $n\in\N$,
    \begin{enumerate}[label=(\roman*)]
        \item \label{lmm:power:absorbS:var}
        $n^{-\frac\phi2}\, \sigma_{2\alpha-2}^{1+\frac\phi2} \leq \max(B_1, B_2)$;
        \item \label{lmm:power:absorbS:ros}
        if $\phi\geq2$, then $n^{1-\phi}\, \sigma_{2\alpha-2}\, \sigma_{2-\alpha} \leq \max(B_1, B_2)$.
    \end{enumerate}
\end{lemma}
\begin{proof}
    All moments occurring below are of order in $[0,\alpha]$ and hence finite.
    
    \ref{lmm:power:absorbS:var}:
    Assume $n^{-\frac\phi2}\sigma_{2\alpha-2}^{1+\frac\phi2} > B_1$, as otherwise there is nothing to show.
    In particular $\sigma_{2\alpha-2}\in\Rpp$.
    Dividing by $\sigma_{2\alpha-2}$ and raising to the power $\frac2\phi$, the assumption is equivalent to $\sigma_{2\alpha-2}\,n^{-1} > \sigma_{\alpha-1}^{2}$, whence $\sigma_{\alpha-1}^{-\phi} \geq \br{\sigma_{2\alpha-2}n^{-1}}^{-\frac\phi2}$.
    Together with \cref{lmm:power:momentcomp}\,\ref{lmm:power:momentcomp:interpol} in the form $\sigma_\alpha \geq \sigma_{2\alpha-2}^{1+\phi}\sigma_{\alpha-1}^{-\phi}$, we obtain
    \begin{equation}
        B_2
        \geq
        n^{-\phi}\, \sigma_{2\alpha-2}^{1+\phi}\, \br{\sigma_{2\alpha-2}n^{-1}}^{-\frac\phi2}
        =
        n^{-\frac\phi2}\, \sigma_{2\alpha-2}^{1+\frac\phi2}
        \eqfs
    \end{equation}
    
    \ref{lmm:power:absorbS:ros}:
    Assume $n^{1-\phi}\sigma_{2\alpha-2}\sigma_{2-\alpha} > B_2$, as otherwise there is nothing to show.
    Multiplying by $n^{\phi}$, this is equivalent to $n^{-1} < \sigma_{2\alpha-2}\sigma_{2-\alpha}\sigma_\alpha^{-1}$; in particular $\sigma_\alpha\in\Rpp$.
    As $\phi-1>0$, this and \cref{lmm:power:momentcomp}\,\ref{lmm:power:momentcomp:chord} give
    \begin{equation}
        n^{1-\phi}\, \sigma_{2\alpha-2}\, \sigma_{2-\alpha}
        =
        \sigma_{2\alpha-2}\, \sigma_{2-\alpha} \br{n^{-1}}^{\phi-1}
        \leq
        \sigma_{2\alpha-2} \cdot \frac{\sigma_{2-\alpha}^{\phi}\, \sigma_{2\alpha-2}^{\phi-1}}{\sigma_\alpha^{\phi-1}}
        \leq
        \sigma_{\alpha-1}^{\phi}\, \sigma_{2\alpha-2}
        =
        B_1
        \eqfs
        \qedhere
    \end{equation}
\end{proof}

\begin{proof}[Proof of \cref{thm:power:main}]
    Set $\phi := \frac{2-\alpha}{\alpha-1}$ and
    \begin{equation}
        B_1 := \sigma_{\alpha-1}^{\phi}\, \sigma_{2\alpha-2}
        \eqcm\qquad
        B_2 := n^{-\phi}\, \sigma_\alpha
        \eqcm
    \end{equation}
    so that the assertion reads $\EOf{\ol{m}{m_n}^2 (\sigma_{\alpha-1} + \ol{m}{m_n}^{\alpha-1})^{-\phi}} \leq C_\alpha n^{-1}\br{B_1+B_2}$.
    Note that $\Eof{\ol Yo^\alpha}<\infty$ implies $\sigma_\alpha<\infty$.
    Using \cref{lmm:power:vi}, the minimizing property of $m_n$, \cref{lmm:power:variancebound}, and \cref{lmm:power:mnmnibound}, we have
    \begin{align}
        \frac12 \alpha (\alpha-1) \EOf{\ol{m}{m_n}^2 \br{\sigma_{\alpha-1} + \ol{m}{m_n}^{\alpha-1}}^{-\phi}}
        &\leq
        \EOf{\ol Y{m_n}^\alpha - \ol Y{m}^\alpha}
        \\&\leq 
        V_n
        \\&\leq 
        \frac{2^{1-\alpha} \alpha}n \sum_{i=1}^n \EOf{\ol {m_n}{m_n^i}\, \ol {Y_i}{Y_i^i}^{\alpha-1}}
        \\&\leq 
        \frac{2^{4-2\alpha} \alpha}{(\alpha-1)n^2} \sum_{i=1}^n \EOf{\ol {Y_i}{Y_i^i}^{2\alpha-2}  \tilde H_i^{-1}}
        \eqfs
    \end{align}
    Furthermore, by \cref{lmm:power:iid:weights},
    \begin{equation}
        \EOf{\ol {Y_i}{Y_i^i}^{2\alpha-2}  \tilde H_i^{-1}} \leq \begin{cases}
            C_1 B_1
            +
            C_2 B_2 
            & \text{for }
            \phi \leq 1,
            \\
            C_3 B_1
            +
            C_4 B_2
            +
            C_5 \sigma_{2\alpha-2} \EOf{(\hat\sigma_{\alpha-1})^{\phi}}
            & \text{for }
            \phi \geq 1,
        \end{cases}
    \end{equation}
    where we use the definitions of $w_1,w_2,w_3,w$ in \cref{lmm:power:iid:weights} and set
    \begin{align}
        C_1 &:= 2^{2\alpha-3} \br{w_1^{\phi}+w_2^{\phi}+w_3^{\phi}}\eqcm\\ 
        C_2 &:= 2^{2\alpha-4} \br{w_2^{\phi}+w_3^{\phi}}\eqcm\\
        C_3 &:= w_1 w^{\phi-1}\eqcm\\
        C_4 &:= \frac12 (w_2+w_3) w^{\phi-1} 2^{\phi-1}\eqcm\\
        C_5 &:= \frac12 (w_2+w_3) w^{\phi-1} \br{1+2^{\phi-1}}\eqfs
    \end{align}
    Now we apply \cref{lmm:power:iid:empmom} with $\varphi = \phi$, where we use $\phi(\alpha-1) = 2-\alpha$, and obtain
    \begin{align}
        &\EOf{\ol {Y_i}{Y_i^i}^{2\alpha-2}  \tilde H_i^{-1}} 
        \\&\leq 
        \begin{cases}
            C_1 B_1
            +
            C_2 B_2 
            & \text{for }
            \phi \leq 1\eqcm
            \\
            C_6 B_1
            +
            C_4 B_2
            +
            C_7 n^{-\frac\phi2} \sigma_{2\alpha-2}^{1 + \frac\phi2}
            & \text{for }
            \phi \in [1,2]\eqcm
            \\
            C_6 B_1
            +
            C_4 B_2
            +
            C_8 n^{-\frac\phi2} \sigma_{2\alpha-2}^{1 + \frac\phi2}
            +
            C_8 n^{1-\phi} \sigma_{2\alpha-2} \sigma_{2-\alpha}
            & \text{for }
            \phi \geq 2\eqcm
        \end{cases}
    \end{align}
    where
    \begin{align}
        C_6 &:= C_3 + 2^{\phi-1}C_5\eqcm\\ 
        C_7 &:= 2^{\phi-1}C_5\eqcm\\
        C_8 &:= 2^{\phi-1}C_5 R_\phi\eqcm
    \end{align}
    and $R_\phi$ is the Rosenthal constant from \cref{lmm:rosenthal}.
    
    It remains to remove the case distinction.
    For $\phi\geq1$ we have $\alpha<2$ and $\phi\in\Rpp$, so that \cref{lmm:power:absorbS} applies and bounds each of the terms carrying the constants $C_7$ and $C_8$ by $\max(B_1, B_2) \leq B_1 + B_2$.
    Hence, in all three cases,
    \begin{equation}
        \EOf{\ol {Y_i}{Y_i^i}^{2\alpha-2}  \tilde H_i^{-1}} 
        \leq
        D_1 B_1 + D_2 B_2
        \eqcm
    \end{equation}
    where
    \begin{align}\label{eq:powerproof:final:d}
        \br{D_1, D_2} := \begin{cases}
            \br{C_1,\ C_2}
            & \text{for }
            \phi \leq 1
            \eqcm
            \\
            \br{C_6+C_7,\ C_4+C_7}
            & \text{for }
            \phi \in [1,2]
            \eqcm
            \\
            \br{C_6+2C_8,\ C_4+2C_8}
            & \text{for }
            \phi \geq 2
            \eqfs
        \end{cases}
    \end{align}
    As the summands $\Eof{\ol {Y_i}{Y_i^i}^{2\alpha-2}  \tilde H_i^{-1}}$ do not depend on $i$, we conclude
    \begin{equation}\label{eq:powerproof:final}
        \EOf{\ol{m}{m_n}^2 \br{\sigma_{\alpha-1} + \ol{m}{m_n}^{\alpha-1}}^{-\phi}}
        \leq
        \frac{2^{5-2\alpha}}{(\alpha-1)^2\,n} \br{D_1 B_1 + D_2 B_2}
        \eqcm
    \end{equation}
    which is the assertion with $C_\alpha := \frac{2^{5-2\alpha}}{(\alpha-1)^2} \max(D_1, D_2)$.
\end{proof}
\begin{remark}\label{rem:power:consts}
    \Cref{eq:powerproof:final} is a slightly sharper, two-constant version of \cref{thm:power:main}: it weights the two summands separately. The constants therein can be specified as follows. With $\phi = \frac{2-\alpha}{\alpha-1}$,
    \begin{equation}
        w_1 = \br{\br{\frac43}^{\alpha-1}  + \br{\frac13}^{\alpha-1}} 2^{3-\alpha} \alpha
        \eqcm\quad
        w_2 = 1 + \br{\frac43}^{\alpha-1} 2^{2-\alpha} \alpha
        \eqcm\quad
        w_3 = \br{\frac13}^{\alpha-1} 2^{2-\alpha} \alpha
        \eqcm
    \end{equation}
    $w = w_1+w_2+w_3$, and
    \begin{equation}
        R_\phi = \min\brOf{\br{\frac{14.7 \phi}{\mathrm{Log}(\phi)}}^{\phi},\, \br{\frac{2\phi}{\sqrt{\mathrm{Log}(\phi)}}}^{\phi}}
        \eqcm
        \qquad 
        \mathrm{Log}(\phi) = \max(1, \log(\phi))
        \eqcm
    \end{equation}
    the constants of \eqref{eq:powerproof:final:d} are
    \begin{align}
        \br{D_1, D_2}
        &=
        \br{2^{2\alpha-3}\br{w_1^{\phi}+w_2^{\phi}+w_3^{\phi}},\ 
            2^{2\alpha-4}\br{w_2^{\phi}+w_3^{\phi}}}
        && \text{for } \alpha \in [\tfrac32, 2]
        \eqcm\\
        \br{D_1, D_2}
        &=
        \br{w_1 + 2 \kappa_\phi,\ 2^{\phi-2}\br{w_2+w_3} + \kappa_\phi} w^{\phi-1}
        && \text{for } \alpha \in [\tfrac43, \tfrac32)
        \eqcm\\
        \br{D_1, D_2}
        &=
        \br{w_1 + \kappa_\phi + 2R_\phi\kappa_\phi,\ 2^{\phi-2}\br{w_2+w_3} + 2R_\phi\kappa_\phi} w^{\phi-1}
        && \text{for } \alpha \in (1, \tfrac43)
        \eqcm
    \end{align}
    where $\kappa_\phi := \br{2^{\phi-2}+2^{2\phi-3}}\br{w_2+w_3}$.
    Numerically, for all $\alpha \in [\frac32, 2]$, we have $\phi \in [0, 1]$ and
    \begin{equation}
        \frac{2^{5-2\alpha}}{(\alpha-1)^2} \leq 16
        \eqcm\qquad
        D_1 \leq 13
        \eqcm\qquad
        D_2 \leq 3
        \eqcm
    \end{equation}
    so that \cref{thm:power:main} holds with $C_\alpha \leq 208$.
    For all $\alpha \in [\frac43, \frac32]$, we have $\phi \in [1, 2]$ and
    \begin{equation}
        \frac{2^{5-2\alpha}}{(\alpha-1)^2} \leq 46
        \eqcm\qquad
        D_1 \leq 452
        \eqcm\qquad
        D_2 \leq 239
        \eqcm
    \end{equation}
    so that \cref{thm:power:main} holds with $C_\alpha \leq 20792$.
    For $\alpha \in (1, \frac43]$, we have $\phi \geq 2$, and $D_1$, $D_2$, as well as the prefactor $(\alpha-1)^{-2}$, are unbounded as $\alpha \searrow 1$, owing to the factors $w^{\phi-1}$, $2^{2\phi-3}$, and $R_\phi$ that diverge as $\phi \to \infty$.
    
    Furthermore, note that for stating \cref{thm:power:main}, we chose to apply the cleaner but weaker version of \cref{lmm:power:vi}, i.e., we use a factor $1$ in the loss instead of $3^{1-\alpha}$. In particular, the upper bound given in \cref{thm:power:main} is also true for
    \begin{equation}
        \EOf{\ol{m}{m_n}^2 \br{\sigma_{\alpha-1} + 3^{1-\alpha}\ol{m}{m_n}^{\alpha-1}}^{-\phi}}
        \eqfs
    \end{equation}
\end{remark}
\begin{proof}[Proof of \cref{cor:power:stdloss:sharp}]
     Denote by $c_\alpha\in\Rpp$ a constant depending only on $\alpha$, which may change from line to line.
     Abbreviate the three terms on the right-hand side of \eqref{eq:power:stdloss:sharp} as $L_1$, $T$, and $L_2$, respectively.
     We may assume $\sigma_{\alpha-1}\in\Rpp$: if $\sigma_{\alpha-1} = 0$, then $Y = m$ almost surely, whence $m_n = m$ almost surely and the left-hand side of \eqref{eq:power:stdloss:sharp} vanishes.
     Furthermore, $\Eof{\ol Yo^\alpha}<\infty$ implies $\sigma_\varphi<\infty$ for all $\varphi\in[0,\alpha]$.
     Set
     \begin{equation}
         r := \EOf{\ol{m}{m_n}^2 \br{\sigma_{\alpha-1} + \ol{m}{m_n}^{\alpha-1}}^{-\phi}}
         \eqcm\quad
         A_1 := \sigma_{\alpha-1}^{\phi}\,\sigma_{2\alpha-2}\,n^{-1}
         \eqcm\quad
         A_2 := \sigma_\alpha\, n^{-1-\phi}
         \eqcm
     \end{equation}
     as well as $\theta := \sigma_{\alpha-1}^{2+\phi} = \sigma_{\alpha-1}^{\frac{\alpha}{\alpha-1}}$.
     With this notation, \cref{thm:power:main} states
     \begin{equation}\label{eq:cor:sharp:main}
         r \leq c_\alpha \br{A_1 + A_2}
         \eqfs
     \end{equation}
     
     \emph{Step 1: Conversion of the loss.}
     \cref{lmm:loss:toalpha}\,(i) with $b = \sigma_{\alpha-1}$, raised to the power $\frac1\alpha$, together with $(1+x)^{\frac\phi\alpha} \leq c_\alpha(1+x^{\frac\phi\alpha})$ for $x\in\Rp$ and the identities $\frac\phi2 - \frac\alpha2\cdot\frac\phi\alpha = 0$ and $\frac12 + \frac{\alpha-1}{2}\cdot\frac\phi\alpha = \frac1\alpha$, which hold by $(\alpha-1)\phi = 2-\alpha$, gives
     \begin{equation}\label{eq:cor:sharp:toalpha}
         \EOf{\ol{m}{m_n}^\alpha}^{\frac1\alpha}
         \leq
         c_\alpha\, \Phi(r)
         \eqcm\qquad\text{where}\qquad
         \Phi \colon \Rp\to\Rp,\ P \mapsto \sigma_{\alpha-1}^{\frac\phi2} P^{\frac12} + P^{\frac1\alpha}
         \eqfs
     \end{equation}
     The map $\Phi$ is nondecreasing and subadditive, as $t\mapsto t^{\frac12}$ and $t\mapsto t^{\frac1\alpha}$ are.
     Hence, \eqref{eq:cor:sharp:main} and \eqref{eq:cor:sharp:toalpha} yield
     \begin{equation}\label{eq:cor:sharp:split}
         \EOf{\ol{m}{m_n}^\alpha}^{\frac1\alpha}
         \leq
         c_\alpha \br{\Phi(A_1) + \Phi(A_2)}
         \eqcm
     \end{equation}
     and, using $1+\phi = \frac1{\alpha-1}$, we compute
     \begin{equation}\label{eq:cor:sharp:phiA}
         \Phi(A_1) = L_1 + A_1^{\frac1\alpha}
         \eqcm\qquad
         \Phi(A_2) = T + L_2
         \eqfs
     \end{equation}
     It remains to show that $A_1^{\frac1\alpha}$ is dominated by $L_1$ and $L_2$.
     
     \emph{Step 2: $A_1^{\frac1\alpha} \leq L_1 + L_2$.}
     The exponent identities $\phi + 2 - \alpha = \phi + (\alpha-1)\phi = \alpha\phi = (2+\phi)(2-\alpha)$ and $(1+\phi)(\alpha-1) = 1$ together with \cref{lmm:power:momentcomp}\,\ref{lmm:power:momentcomp:interpol} give
     \begin{equation}
         A_1
         \leq
         \sigma_{\alpha-1}^{\phi + 2-\alpha}\, \sigma_\alpha^{\alpha-1}\, n^{-1}
         =
         \sigma_{\alpha-1}^{(2+\phi)(2-\alpha)} \br{\sigma_\alpha n^{-1-\phi}}^{\alpha-1}
         =
         \theta^{2-\alpha}\, A_2^{\alpha-1}
         \leq
         \max(\theta, A_2)
         \eqcm
     \end{equation}
     where the last step is the weighted arithmetic-geometric mean inequality with the weights $2-\alpha$ and $\alpha-1$, which are nonnegative and sum to $1$.
     If $A_1 \leq A_2$, then $A_1^{\frac1\alpha} \leq A_2^{\frac1\alpha} = L_2$.
     Otherwise, $A_1 \leq \theta$, and, as $\theta^{\frac{2-\alpha}{2\alpha}} = \sigma_{\alpha-1}^{\frac{\alpha}{\alpha-1}\cdot\frac{2-\alpha}{2\alpha}} = \sigma_{\alpha-1}^{\frac\phi2}$ and $\frac1\alpha = \frac12 + \frac{2-\alpha}{2\alpha}$, we obtain
     \begin{equation}
         A_1^{\frac1\alpha}
         =
         A_1^{\frac12}\, A_1^{\frac{2-\alpha}{2\alpha}}
         \leq
         \sigma_{\alpha-1}^{\frac\phi2}\, A_1^{\frac12}
         =
         L_1
         \eqfs
     \end{equation}
     
     \emph{Step 3: Conclusion.}
     Combining \eqref{eq:cor:sharp:split}, \eqref{eq:cor:sharp:phiA}, and Step 2, we arrive at
     \begin{equation}
         \EOf{\ol{m}{m_n}^\alpha}^{\frac1\alpha}
         \leq
         c_\alpha \br{L_1 + \br{L_1 + L_2} + T + L_2}
         \eqcm
     \end{equation}
     which is the assertion.
 \end{proof}

\begin{proof}[Proof of \cref{cor:power:threehalfs}]
    For $\alpha=\frac32$ we have
    \begin{equation}
        \phi = \frac{2-\alpha}{\alpha-1} = 1
        \eqcm\qquad
        \sigma_{\alpha-1} = \sigma_{\frac12}
        \eqcm\qquad
        \sigma_{2\alpha-2} = \sigma_{1}
        \eqcm\qquad
        \sigma_{\alpha} = \sigma_{\frac32}
        \eqcm
    \end{equation}
    as well as $\frac1{2(\alpha-1)} = 1$ and $\frac1{\alpha(\alpha-1)} = \frac43$, so that the three terms of \eqref{eq:power:stdloss:sharp} read
    \begin{equation}
        L_1 := \sigma_{\frac12}\, \sigma_1^{\frac12}\, n^{-\frac12}
        \eqcm\qquad
        T := \sigma_{\frac12}^{\frac12}\, \sigma_{\frac32}^{\frac12}\, n^{-1}
        \eqcm\qquad
        L_2 := \sigma_{\frac32}^{\frac23}\, n^{-\frac43}
        \eqfs
    \end{equation}
    If $\sigma_{\frac12} = 0$, then $Y = m$ almost surely, hence $m_n = m$ almost surely
    and both claims are trivial. Thus, assume $\sigma_{\frac12}\in\Rpp$; moreover, $\sigma_\varphi<\infty$ for all $\varphi\in[0,\frac32]$.
    
    \emph{Step 1: The bound in the nonstandard loss.}
    As $\phi \leq 1$, \eqref{eq:powerproof:final} and \eqref{eq:powerproof:final:d} yield
    \begin{equation}
        r
        :=
        \EOf{\ol{m}{m_n}^2 \br{\sigma_{\frac12} + \ol{m}{m_n}^{\frac12}}^{-1}}
        \leq
        \frac{2^{5-2\alpha}}{(\alpha-1)^2 n}
        \br{C_1\, \sigma_{\frac12}\, \sigma_1 + C_2\, n^{-1} \sigma_{\frac32}}
        \eqfs
    \end{equation}
    By \cref{rem:power:consts} with $\phi = 1$ and $\alpha = \frac32$, we have
    \begin{equation}
        \frac{2^{5-2\alpha}}{(\alpha-1)^2} = \frac{4}{\br{\frac12}^2} = 16
        \eqcm\qquad
        C_1 = w_1+w_2+w_3
        \eqcm\qquad
        C_2 = \frac12\br{w_2+w_3}
        \eqcm
    \end{equation}
    and, using $\br{\frac43}^{\frac12} + \br{\frac13}^{\frac12} = \frac{3}{\sqrt3} = \sqrt3$,
    \begin{equation}
        w_1 = \sqrt3 \cdot 2^{\frac32} \cdot \frac32 = 3\sqrt6
        \eqcm\qquad
        w_2 = 1 + \br{\frac43}^{\frac12} 2^{\frac12} \frac32 = 1 + \sqrt6
        \eqcm\qquad
        w_3 = \br{\frac13}^{\frac12} 2^{\frac12} \frac32 = \frac{\sqrt6}{2}
        \eqfs
    \end{equation}
    Hence $16\,C_1 = 16 + 72\sqrt6 \leq 193$ and $16\,C_2 = 8 + 12\sqrt6 \leq 38$, which is the first claim,
    \begin{equation}\label{eq:cor:threehalfs:r}
        r
        \leq
        \hat A_1 + \hat A_2
        \eqcm\qquad\text{where}\qquad
        \hat A_1 := \frac{193\, \sigma_1\, \sigma_{\frac12}}{n}
        \eqcm\quad
        \hat A_2 := \frac{38\, \sigma_{\frac32}}{n^2}
        \eqfs
    \end{equation}
    
    \emph{Step 2: The bound in $L^\alpha$-norm.}
    We apply the second inequality of \cref{lmm:loss:toalpha}\,(i) with $b = \sigma_{\frac12}$,
    $X = \ol{m}{m_n}$, and $\phi = 1$, $\frac\alpha2 = \frac34$, $\frac{\alpha-1}2 = \frac14$:
    \begin{equation}
        \EOf{\ol{m}{m_n}^{\frac32}}
        \leq
        \sigma_{\frac12}^{\frac34} r^{\frac34}
        \br{1 + \sigma_{\frac12}^{-\frac34} r^{\frac14}}
        =
        \sigma_{\frac12}^{\frac34} r^{\frac34} + r
        \eqfs
    \end{equation}
    Taking the power $\frac23$ and using subadditivity of $t\mapsto t^{\frac23}$, we obtain
    $\EOf{\ol{m}{m_n}^{\frac32}}^{\frac23} \leq \sigma_{\frac12}^{\frac12} r^{\frac12} + r^{\frac23}$.
    Plugging in \eqref{eq:cor:threehalfs:r} and using subadditivity of $t\mapsto t^{\frac12}$
    and of $t\mapsto t^{\frac23}$ once more, we get
    \begin{equation}\label{eq:cor:threehalfs:four}
        \EOf{\ol{m}{m_n}^{\frac32}}^{\frac23}
        \leq
        \sqrt{193}\, L_1
        +
        \sqrt{38}\, T
        +
        193^{\frac23} A_1^{\frac23}
        +
        38^{\frac23}\, L_2
        \eqcm\qquad\text{where}\qquad
        A_1 := \frac{\sigma_{\frac12}\,\sigma_1}{n}
        \eqfs
    \end{equation}
    
    \emph{Step 3: Absorption of the cross term.}
    We claim
    \begin{equation}\label{eq:cor:threehalfs:gm}
        A_1^{\frac23} \leq L_1^{\frac45}\, L_2^{\frac15}
        \eqfs
    \end{equation}
    Indeed, the powers of $n$ on both sides agree, as $\frac45\cdot\frac12 + \frac15\cdot\frac43 = \frac25+\frac4{15} = \frac23$, so that \eqref{eq:cor:threehalfs:gm} is equivalent to
    $\sigma_{\frac12}^{\frac23} \sigma_1^{\frac23} \leq \sigma_{\frac12}^{\frac45}\sigma_1^{\frac25}\sigma_{\frac32}^{\frac2{15}}$
    and hence, after division, to
    $\sigma_1^{\frac4{15}} \leq \br{\sigma_{\frac12}\,\sigma_{\frac32}}^{\frac2{15}}$,
    i.e., to $\sigma_1^2 \leq \sigma_{\frac12}\,\sigma_{\frac32}$.
    The latter is \cref{lmm:power:momentcomp}\,\ref{lmm:power:momentcomp:interpol} for $\alpha=\frac32$, and it is here simply the Cauchy--Schwarz inequality,
    \begin{equation}
        \sigma_1
        =
        \EOf{\ol Ym^{\frac14}\, \ol Ym^{\frac34}}
        \leq
        \EOf{\ol Ym^{\frac12}}^{\frac12} \EOf{\ol Ym^{\frac32}}^{\frac12}
        \eqfs
    \end{equation}
    By \eqref{eq:cor:threehalfs:gm} and the weighted arithmetic-geometric mean inequality with the weights $\frac45$ and $\frac15$, we obtain
    $A_1^{\frac23} \leq \frac45 L_1 + \frac15 L_2$.
    
    \emph{Step 4: Conclusion.}
    Inserting Step 3 into \eqref{eq:cor:threehalfs:four} gives
    \begin{equation}
        \EOf{\ol{m}{m_n}^{\frac32}}^{\frac23}
        \leq
        \br{\sqrt{193} + \tfrac45\, 193^{\frac23}} L_1
        +
        \sqrt{38}\, T
        +
        \br{38^{\frac23} + \tfrac15\, 193^{\frac23}} L_2
        \eqcm
    \end{equation}
    and the second claim follows from
    $\sqrt{193} \leq 13.9$, $193^{\frac23}\leq 33.4$, $\sqrt{38}\leq 6.17$, $38^{\frac23}\leq 11.31$, which give
    $\sqrt{193} + \frac45 193^{\frac23} \leq 40.62$, and
    $38^{\frac23} + \frac15 193^{\frac23} \leq 17.99$.
\end{proof}
\subsection{General Transformations}
In this section we prove \cref{thm:infdtr:improved}. We use the algorithm stability approach. The general first steps in this approach are executed in \cref{sec:algostabi}, see \cref{prp:stabi}. Here we detail the additional arguments required to prove \cref{thm:infdtr:improved}.
Throughout this section, assume the setting and conditions of \cref{ssec:prelim:setup} and \cref{thm:infdtr:improved}. Furthermore, let $Y_1\pr, \dots, Y_n\pr$ be an additional independent set of iid copies of $Y$.
Denote the samples with $i$-th position replaced as $Y_j^i := Y_j$ if $i \neq j$ and $Y_i^i := Y_i\pr$.
Let $m_n^i$ be the sample $\tran$-Fr\'echet mean of the sample set $Y_1^i, \dots, Y_n^i$, 
\begin{equation}
    m_n^i \in \argmin_{q\in\mc Q} \sum_{j=1}^n \tran(\ol {Y_j^i}q)
    \eqfs
\end{equation}

\begin{notation}
    Let $f\colon \Rp \to \Rp$ be a measurable function. 
    Define
    \begin{equation}
        \hat\sigma_{f}^i := \frac1n \sum_{j=1}^n f(\ol {Y_j^i}m)
        \eqfs
    \end{equation}
\end{notation}
The first lemma shows that a moment bound of $(\dtran)^2 / \ddrtran$ is at least as strong as a $\tran$-moment bound.
\begin{lemma}\label{lmm:fractranmoment}
    Assume $\Eof{\dtran(\ol Yo)^2 / \ddrtran(\ol Yo)} < \infty$.
    Then $\Eof{\tran(\ol Yq)} < \infty$	for all $q\in\mc Q$.
\end{lemma}
\begin{proof}
    By \cref{lmm:tran:tran} and \cref{lmm:tran:sqrddtran}, we have for all $x \in \Rpp$
    \begin{equation}
        \tran(x) \leq x \dtran(x)
        \qquad\text{and}\qquad
        \frac12 x^2 \ddrtran(x) \leq \tran(x)
        \eqfs
    \end{equation}
    Thus,
    \begin{equation}\label{eq:tran:dtran2invddtran}
        \frac{\dtran(x)^2}{\ddrtran(x)} 
        \geq 
        \frac{\br{\frac{\tran(x)}x}^2}{\frac{2\tran(x)}{x^2}} = \frac12 \tran(x)
        \eqfs
    \end{equation}
    By \cref{lmm:tran:tran} and the triangle inequality, we have $\tran(\ol yq) \leq 2\tran(\ol yo) + 2\tran(\ol qo)$ for all $q,y\in\mc Q$.
    Thus, we obtain
    \begin{equation}
        \EOf{\tran(\ol Yq)} \leq 4\EOf{\frac{\dtran(\ol Yo)^2}{\ddrtran(\ol Yo)}} + 2\tran(\ol qo)
        \eqfs
        \qedhere
    \end{equation}
\end{proof}
The proof of \cref{thm:infdtr:improved} relies on \cref{lmm:infdtr:technical}, whose proof in turn uses the rough error bound for $\ol{m}{m_n}$ established in \cref{lmm:mmn:dtran:bound}. Combining the latter with \cref{prp:stabi} yields the proof of \cref{thm:infdtr:improved}.
\begin{lemma}\label{lmm:mmn:dtran:bound}
    We have
    \begin{equation}
        \dtran(\ol m{m_n})
        \leq
        8 \sigma_{\dtran} + 4 \hat\sigma_{\dtran}
        \eqfs
    \end{equation}
\end{lemma}
\begin{proof}
    Let $y,q,p\in\mc Q$. The quadruple inequality, \cref{thm:infdtr:qi}, applied with $z := q$ yields
    \begin{equation}
        \tran(\ol yp) - \tran(\ol yq)
        \geq
        \tran(\ol qp) - 2 \,\ol qp\, \dtran(\ol yq)
        \eqfs
    \end{equation}
    In particular, we have
    \begin{equation}
        \EOf{\tran(\ol Y{m_n}) - \tran(\ol Ym) | Y_1,\dots,Y_n} \geq 
        \tran(\ol m{m_n}) - 2 \,\ol m{m_n}\, \EOf{\dtran(\ol Ym)}
        \eqcm
    \end{equation}
    where we recall that $Y, Y_1, \dots, Y_n$ are iid, $m_n$ is $\sigma(Y_1, \dots, Y_n)$-measurable and, hence, $m_n$ and $Y$ are independent.
    By the minimizing property of $m_n$ we also have
    \begin{equation}
        \frac1n\sum_{i=1}^n\br{\tran(\ol {Y_i}{m}) - \tran(\ol {Y_i}{m_n})} \geq 0
        \eqfs
    \end{equation}
    Using the last two inequalities and the quadruple inequality, \cref{thm:infdtr:qi}, we get
    \begin{align}
        &\tran(\ol m{m_n}) - 2 \,\ol m{m_n}\, \EOf{\dtran(\ol Ym)}
        \\&\leq 
        \EOf{\tran(\ol Y{m_n}) - \tran(\ol Ym) | Y_1,\dots,Y_n} + 
        \frac1n\sum_{i=1}^n\br{\tran(\ol {Y_i}{m}) - \tran(\ol {Y_i}{m_n})}
        \\&\leq
        2 \, \ol m{m_n} \frac1n \sum_{i=1}^n \EOf{\dtran(\ol Y{Y_i}) | Y_i}
        \eqfs
    \end{align}
    Rearranging the terms yields
    \begin{equation}
        \frac{\tran(\ol m{m_n})}{\ol m{m_n}}
        \leq
        2\frac1n \sum_{i=1}^n \EOf{\dtran(\ol Y{Y_i}) | Y_i}
        +
        2\EOf{\dtran(\ol Ym)}
        \eqfs
    \end{equation}
    Using $x \dtran(x) \leq 2 \tran(x)$ (\cref{lmm:tran:tran}) and $\dtran(\ol Y{Y_i}) \leq \dtran(\ol Y{m}) + \dtran(\ol {Y_i}m)$ (\cref{lmm:dtran:subadd} and triangle inequality) concludes the proof.
\end{proof}
Recall
\begin{equation}
    H_i = \frac1n \sum_{j=1}^n \ddrtran(\ol{Y_j}{m_n}+\ol{m_n}{m_n^i})
    \eqfs
\end{equation}
\begin{lemma}\label{lmm:infdtr:technical}
    Use the setting of \cref{thm:infdtr:improved}.
    Then
    \begin{equation}
        \EOf{\dtran(\ol{Y_i}m)^2 H_i^{-1}}
        \leq
        \min\brOf{V_{n,1} + 4 \sigma_{(\dtran)^{2}} S_{n,1}\eqcm \ \frac{4 \sigma_{(\dtran)^2}}{\ddrtran(4r_0)}
            +
            b_n}
        \eqfs
    \end{equation}
\end{lemma}
\begin{proof}
    For $r,\eta\in\Rpp$, define the following events
    \begin{align}
        A = A_{r,\eta,n} &:= \cb{\frac1n \sum_{j=1}^n \ind_{[0, r]}(\ol {Y_j}{m}) \geq \eta}
        \eqcm
        &
        B = B_{r,n} &:= \cb{\ol m{m_n} \leq r}
        \eqcm
        \\
        B^i = B^i_{r,n} &:= \cb{\ol m{m_n^i} \leq r}
        \eqcm
        &
        \Omega_i &:= A \cap B \cap B^i
        \eqfs
    \end{align}
    We split the expectation of our target term on $\Omega_i$,
    \begin{equation}\label{eq:technical:omegasplit}
        \EOf{\dtran(\ol{Y_i}m)^2 H_i^{-1}}
        \leq
        \EOf{\dtran(\ol{Y_i}m)^2 H_i^{-1} \indOf{\Omega_i}}
        +
        \EOf{\dtran(\ol{Y_i}m)^2 H_i^{-1} \indOf{\Omega_i\compl}}
        \eqfs
    \end{equation}
    By the triangle inequality and $\ddrtran$ nonincreasing, we have
    \begin{equation}
        \ddrtran(\ol{Y_j}{m_n}+\ol{m_n}{m_n^i}) \geq \ddrtran(\ol{Y_j}{m}+2\ol{m}{m_n}+\ol{m}{m_n^i})
        \eqfs
    \end{equation}
    Thus, for the first term on the right-hand side of \eqref{eq:technical:omegasplit}, we obtain
    \begin{equation}
        \EOf{\dtran(\ol{Y_i}m)^2 H_i^{-1} \indOf{\Omega_i}}
        \leq
        \EOf{\dtran(\ol{Y_i}m)^2 \br{\eta \ddrtran(4r)}^{-1} \indOf{\Omega_i}}
        \leq
        \frac{\sigma_{(\dtran)^2}}{\eta \ddrtran(4r)}
        \eqfs
    \end{equation}
    For the second term on the right-hand side of \eqref{eq:technical:omegasplit}, we use H\"older's inequality with $\frac1p + \frac1q = 1$,
    \begin{equation}\label{eq:technical:hoelder}
        \EOf{\dtran(\ol{Y_i}m)^2 H_i^{-1} \indOf{\Omega_i\compl}}
        \leq
        \EOf{\dtran(\ol{Y_i}m)^{2p} H_i^{-p}}^{\frac1p} \PrOf{\Omega_i\compl}^{\frac1q}
        \eqfs
    \end{equation}
    First, we aim to find an upper bound on the expectation term on the right-hand side of \eqref{eq:technical:hoelder}.
    The functions, $\invdtran$, $x\mapsto 1/\ddrtran(x)$, and $x\mapsto x^p$ are nondecreasing on $\Rp$, where we set $1/\ddrtran(0) := 0$.
    We will make use of the following property of nondecreasing functions $f \colon \Rp\to\Rp$: Let $\ell\in\N, x_1, \dots, x_\ell, w_1, \dots, w_\ell \in \Rp$. Set $W := \sum_{k=1}^\ell w_k$. Then
    \begin{equation}
        f\brOf{\sum_{k=1}^\ell w_k x_k} 
        \leq 
        f\brOf{W \max_{k=1,\dots,\ell} x_k}
        = 
        \max_{k=1,\dots,\ell} f\brOf{W x_k}
        \leq 
        \sum_{k=1}^\ell	f\brOf{W x_k}
        \eqfs
    \end{equation}
    Recall $g(x) = \ddrtran(7 x)^{-1}$ and note that $x\mapsto \ddrtran(x)^{-p}$ is nondecreasing. We obtain, using Jensen's inequality for the convex function $x\mapsto x^{-p}$,
    \begin{align}
        H_i^{-p}
        &\leq
        \frac1n \sum_{j=1}^n \ddrtran(\ol{Y_j}{m_n}+\ol{m_n}{m_n^i})^{-p}
        \\&\leq
        \frac1n \sum_{j=1}^n \ddrtran\brOf{1\cdot\ol{Y_j}{m}+ 4\cdot \frac12 \ol m{m_n} + 2\cdot \frac12 \ol m{m_n^i}}^{-p}
        \\&\leq
        \frac1n \sum_{j=1}^n g(\ol{Y_j}{m})^p + g\brOf{\frac12\ol{m}{m_n}}^p + g\brOf{\frac12\ol{m}{m_n^i}}^p
        \eqfs
    \end{align}
    By \cref{lmm:mmn:dtran:bound} with $\invdtran$ nondecreasing, we have
    \begin{equation}
        \ol m{m_n}
        \leq
        \invdtran\brOf{8 \sigma_{\dtran} + 4 \hat \sigma_{\dtran}} 
        \leq
        \invdtran\brOf{12 \sigma_{\dtran}} + \invdtran\brOf{12 \hat \sigma_{\dtran}} 
        \eqfs
    \end{equation}
    Recall $h(x) = g(\invdtran(12 x))$ and note that $x\mapsto g(x)^p$ is nondecreasing. We obtain
    \begin{equation}\label{eq:technical:Hisplit}
        H_i^{-p}
        \leq
        \frac1n \sum_{j=1}^n g(\ol{Y_j}{m})^p +
        2 h(\sigma_{\dtran})^p +
        h(\hat\sigma_{\dtran})^p
        +
        h(\hat\sigma^i_{\dtran})^p
        \eqfs
    \end{equation}
    We split our target upper bound into four terms using \eqref{eq:technical:Hisplit},
    \begin{equation}\label{eq:technical:Hisplit:T}
        \EOf{\dtran(\ol{Y_i}m)^{2p} H_i^{-p}} \leq T_1 + T_2 + T_3 + T_4
    \end{equation}
    with $T_1, \dots, T_4$ defined and bounded below.
    Recall
    \begin{equation}
            S_{n,p} = \max\brOf{
                \sigma_{g^p}, 
                2 h(\sigma_{\dtran})^p,
                \EOf{ h(2\hat\sigma_{\dtran})^p}
            }
            \eqfs
    \end{equation}
    First, distinguishing between $j=i$ and $j\neq i$ yields
    \begin{align}
        T_1 
        := 
        \EOf{\dtran(\ol{Y_i}m)^{2p} \frac1n \sum_{j=1}^n g(\ol{Y_j}{m})^p} 
        \leq
        \frac1n  \EOf{\dtran(\ol{Y}m)^{2p} g(\ol{Y}{m})^p} + \sigma_{(\dtran)^{2p}} S_{n,p}
        \eqfs
    \end{align}
    Second, as $\sigma_{(\dtran)^{2p}}$ is a constant, we have
    \begin{equation}
        T_2 
        := 
        2\EOf{\dtran(\ol{Y_i}m)^{2p} h(\sigma_{\dtran})^p} 
        =
        2\EOf{\dtran(\ol{Y}m)^{2p}} h(\sigma_{\dtran})^p
        \leq
        \sigma_{(\dtran)^{2p}} S_{n,p}
        \eqfs
    \end{equation}
    Third, we again distinguish between $j=i$ and $j\neq i$ for $\hat\sigma_{\dtran} = \frac1n\sum_{j=1}^n \dtran(\ol {Y_j}m)$ and use that $h$ is nondecreasing to obtain
    \begin{align}
        T_3 
        &:= 
        \EOf{\dtran(\ol{Y_i}m)^{2p} h(\hat\sigma_{\dtran})^p} 
        \\& \leq
        \EOf{\dtran(\ol{Y_i}m)^{2p} h(n^{-1} \dtran(\ol{Y_i}m)+\hat\sigma^i_{\dtran})^p} 
        \\& \leq
        \EOf{\dtran(\ol{Y_i}m)^{2p} \br{h(2n^{-1} \dtran(\ol{Y_i}m))^p+h(2\hat\sigma^i_{\dtran})^p} }
        \\&\leq
        \EOf{\dtran(\ol{Y}m)^{2p} h(2n^{-1} \dtran(\ol{Y}m))^p} + \sigma_{(\dtran)^{2p}} S_{n,p}
        \eqfs
    \end{align}
    For the final term, we recognize that $Y_i$ is independent of $\hat\sigma^i_{\dtran}$ and obtain
    \begin{equation}
        T_4 
        := 
        \EOf{\dtran(\ol{Y_i}m)^{2p} h(\hat\sigma^i_{\dtran})^p} 
        \leq
        \sigma_{(\dtran)^{2p}} S_{n,p}
        \eqfs
    \end{equation}
    Plugging the upper bounds on these four terms into \eqref{eq:technical:Hisplit:T}, we obtain
    \begin{equation}\label{eq:technical:Hipbound}
        \EOf{\dtran(\ol{Y_i}m)^{2p} H_i^{-p}} \leq V_{n,p} + 4 \sigma_{(\dtran)^{2p}} S_{n,p}
    \end{equation}
    recalling
    \begin{equation}
        V_{n,p} = \frac1n  \EOf{\dtran(\ol{Y}m)^{2p} g(\ol{Y}{m})^p}  + 	\EOf{\dtran(\ol{Y}m)^{2p} h(2n^{-1} \dtran(\ol{Y}m))^p}\eqfs
    \end{equation}
    Next, we aim to find an upper bound on the probability term on the right-hand side of \eqref{eq:technical:hoelder}.
    Assume $r \geq \chi \in  \median(\ol Ym)$ and $\eta \leq \frac14$. Then the classical Chernoff bound (\cref{prp:chernoff:multi}) yields
    \begin{equation}
        \PrOf{A\compl} \leq \exp\brOf{-\frac{n}{16}}
        \eqfs
    \end{equation}
    For the bound on the probability of $B\compl$ and $(B^i)\compl$, we first use \cref{lmm:mmn:dtran:bound} with $\invdtran$ nondecreasing and obtain
    \begin{align}
        \ol m{m_n}
        &\leq
        \invdtran\brOf{8 \sigma_{\dtran} + 4 \hat \sigma_{\dtran}} 
        \\&\leq
        \invdtran\brOf{12 \sigma_{\dtran} + 4 \absOf{ \hat \sigma_{\dtran}- \sigma_{\dtran} }} 
        \\&\leq
        \invdtran\brOf{16 \sigma_{\dtran}} + \invdtran\brOf{16 \absOf{ \hat \sigma_{\dtran}- \sigma_{\dtran} }} 
        \eqfs
    \end{align}
    Next, by Chebyshev's inequality,
    \begin{equation}
        \PrOf{\absOf{\hat \sigma_{\dtran}- \sigma_{\dtran} } > \sigma_{\dtran}} \leq \frac{\sigma_{(\dtran)^2} - \sigma_{\dtran}^2}{n \sigma_{\dtran}^2}
        \eqfs
    \end{equation}
    If $r\geq 2 \invdtran\brOf{16 \sigma_{\dtran}}$, we now have
    \begin{align}
        \PrOf{(B^i)\compl} 
        =
        \PrOf{B\compl}
        &= 
        \PrOf{\ol m{m_n} > r} 
        \\&\leq 
        \PrOf{\invdtran\brOf{16 \absOf{ \hat \sigma_{\dtran}- \sigma_{\dtran}}} > \invdtran\brOf{16 \sigma_{\dtran}}}
        \\&\leq 
        \PrOf{\absOf{ \hat \sigma_{\dtran}- \sigma_{\dtran}} > \sigma_{\dtran}}
        \\&\leq
        \frac1n \br{\frac{\sigma_{(\dtran)^2}}{\sigma_{\dtran}^2} -1}
        \eqfs
    \end{align}
    With this, we obtain
    \begin{equation}
        \PrOf{\Omega_i\compl}
        \leq 
        \PrOf{A\compl} + 2\PrOf{B\compl}
        \leq 
        \exp\brOf{-\frac{n}{16}} +  \frac2n \br{\frac{\sigma_{(\dtran)^2}}{\sigma_{\dtran}^2}-1}
        =:
        u_n
        \eqfs
    \end{equation}
    
    Finally, putting everything together and recalling $r_0 = \max\brOf{\chi, 2 \invdtran\brOf{16 \sigma_{\dtran}}}$, we obtain
    \begin{align}
        \EOf{\dtran(\ol{Y_i}m)^2 H_i^{-1}}
        &\leq
        \frac{4 \sigma_{(\dtran)^2}}{\ddrtran(4r_0)}
        +
        \br{V_{n,p} + 4 \sigma_{(\dtran)^{2p}} S_{n,p}}^{\frac1p} u_n^{\frac1q}
        \eqfs
    \end{align}
    Furthermore, from \eqref{eq:technical:Hipbound} with $p = 1$, we obtain
    \begin{equation}
        \EOf{\dtran(\ol{Y_i}m)^2 H_i^{-1}}
        \leq 
        V_{n,1} + 4 \sigma_{(\dtran)^{2}} S_{n,1}
        \eqfs
        \qedhere
    \end{equation}
\end{proof}
\begin{proof}[Proof of \cref{thm:infdtr:improved}]
    By \cref{prp:stabi} and \cref{lmm:infdtr:technical}, we have
    \begin{align}
        \EOf{\tran(\ol Y{m_n}) - \tran(\ol Ym)}
        &\leq
        \frac{16}{n^2} \sum_{i=1}^n \EOf{\dtran(\ol {Y_i}m)^2 H_i^{-1}}
        \\&\leq
        \frac{16}{n} \min\brOf{V_{n,1} + 4 \sigma_{(\dtran)^{2}} S_{n,1}, \frac{4 \sigma_{(\dtran)^2}}{\ddrtran(4r_0)}
            +
            b_n}
        \eqfs
        &\qedhere
    \end{align}
\end{proof}
\begin{proof}[Proof of \cref{cor:infdtr:improved}]
    We first record two observations used repeatedly below. As $\invdtran(12\dtran(x)) \geq x$ and $g$ is nondecreasing, we have $g \leq h \circ \dtran$. Furthermore, $h$ is nondecreasing with $h(\delta)\in\Rpp$ for $\delta\in\Rpp$ large enough, so that, for $p\geq1$, distinguishing $\dtran(\ol Ym) \lessgtr \delta$ yields
    \begin{equation}\label{eq:cor:infdtr:unif}
        \max\brOf{\sigma_{(\dtran)^{2p}},\, \sigma_{g^p}}
        \leq
        \max\brOf{\delta^{2p}, h(\delta)^p} + \br{1 + h(\delta)^{-p}} \EOf{\dtran(\ol{Y}m)^{2p}\, h(\dtran(\ol{Y}m))^{p}}
    \end{equation}
    as well as, using $h(2n^{-1}\dtran(\ol Ym)) \leq h(\dtran(\ol Ym))$ for $n\geq2$,
    \begin{equation}\label{eq:cor:infdtr:Vn}
        V_{n,p} \leq 2\, \EOf{\dtran(\ol{Y}m)^{2p}\, h(\dtran(\ol{Y}m))^{p}}
        \qquad\text{for all } n \geq 2
        \eqfs
    \end{equation}
    \begin{enumerate}[label=(\roman*)]
        \item
        Using the extended definition of $\invdtran$, we have $\lim_{x\to0}\invdtran(x) = 0$.
        As we assume $\EOf{\dtran(\ol{Y}m)^{2} h(\dtran(\ol{Y}m))} < \infty$, dominated convergence yields
        \begin{align}
            \limsup_{n\to\infty}\EOf{\dtran(\ol{Y}m)^{2} h(2n^{-1} \dtran(\ol{Y}m))}
            &\leq
            \EOf{\dtran(\ol{Y}m)^{2} \limsup_{n\to\infty} h(2n^{-1} \dtran(\ol{Y}m))}
            \\&\leq
            \EOf{\dtran(\ol{Y}m)^{2} \ddrtran(0)^{-1}}
            \eqcm
        \end{align}
        where $\ddrtran(0)^{-1} = \lim_{x\searrow0}\ddrtran(x)^{-1}$.
        Thus, if $\lim_{x\searrow0}\ddrtran(x) = \infty$, we have $\lim_{n\to\infty} V_{n,1}= 0$ and \cref{thm:infdtr:improved} combined with \cref{thm:infdtr:vi} \ref{thm:infdtr:vi:quantile} yields the claim.
        \item
        Apply \cref{thm:infdtr:improved} with $p := 1+\epsilon$ and $q := \frac{1+\epsilon}\epsilon$. Note that
        \begin{equation}
            \lim_{n\to\infty}\br{\exp\brOf{-\frac{n}{16}} +  \frac2n \br{\frac{\sigma_{(\dtran)^2}}{\sigma_{\dtran}^2}-1}}^{\frac1q} = 0
            \eqfs
        \end{equation}
        Furthermore, by \eqref{eq:cor:infdtr:unif}, \eqref{eq:cor:infdtr:Vn}, and the assumptions, we have $\sup_{n\geq2}\br{V_{n,p} + 4 \sigma_{(\dtran)^{2p}} S_{n,p}} < \infty$. Thus, $\lim_{n\to\infty} b_n = 0$. Finally, \cref{thm:infdtr:vi} \ref{thm:infdtr:vi:quantile} transforms the excess risk bound of \cref{thm:infdtr:improved} to the bound on $\Eof{\ol m{m_n}^2  \ddrtran(\chi + \ol m{m_n})}$.
        \item
        First, we show $\EOf{\tran(\ol Y{m_n}) - \tran(\ol Ym)} = \mo O(n^{-1})$ using \cref{thm:infdtr:improved}.
        Under the conditions of \ref{cor:infdtr:improved:low}, this follows from $\lim_{n\to\infty}b_n = 0$ as shown in the proof of \ref{cor:infdtr:improved:low}.
        Under the conditions of \ref{cor:infdtr:improved:simple}, the bounds \eqref{eq:cor:infdtr:unif} and \eqref{eq:cor:infdtr:Vn} with $p=1$, together with $h(\sigma_{\dtran}) < \infty$ and $\sup_{n\in\N}\Eof{h(2\hat\sigma_{\dtran})} < \infty$, yield $\sup_{n\geq2}\br{V_{n,1} + 4 \sigma_{(\dtran)^{2}} S_{n,1}} < \infty$, so that \cref{thm:infdtr:improved} again gives the claim.
        
        As $\ddrtran > 0$ and $\Prof{\ol Ym \leq t_0} > 0$ for $t_0$ large enough, \cref{thm:infdtr:vi} \ref{thm:infdtr:vi:combined} bounds the excess risk from below by $\Eof{\psi(\ol m{m_n})}$ with $\psi$ as in \cref{lmm:loss:totran}. Thus, \cref{lmm:loss:totran} yields $\tran^{-1}\brOf{\Eof{\tran(\ol m{m_n})}} = \mo O(n^{-\frac12})$. As $\tran$ is convex, Jensen's inequality implies $\Eof{\ol m{m_n}} \leq \tran^{-1}\brOf{\Eof{\tran(\ol m{m_n})}}$, which concludes the proof.
    \end{enumerate}
\end{proof}
\begin{proof}[Proof of \cref{exa:logp1}]
    The formulas for $\dtran, \invdtran, \ddtran, g, h$ follow by elementary calculus; note $h(\dtran(x)) = 7(1+x)^{12}-6$.
    If $\ol Ym = 0$ almost surely, the claim is trivial. Otherwise $\sigma_{\dtran}, \sigma_{(\dtran)^2} \in \Rpp$.
    Set $M := \Eof{(1+\ol Ym)^{25}} < \infty$, $\epsilon := \frac1{24}$, and $\nu := 1+\epsilon = \frac{25}{24}$.
    
    We verify the moment conditions of \cref{cor:infdtr:improved} \ref{cor:infdtr:improved:low}:
    As $\log(1+x)^{2\nu} \leq c\br{1+x}^{\frac{25}2}$, we have
    $\dtran(x)^{2\nu} h(\dtran(x))^{\nu} \leq c\br{1+x}^{25}$ and $g(x)^\nu \leq c\br{1+x}^{25}$, so the corresponding moments are bounded by $cM$.
    Moreover, $h(2\hat\sigma_{\dtran})^\nu \leq 7^\nu \exp\brOf{25\hat\sigma_{\dtran}} = 7^\nu \prod_{j=1}^n (1+\ol{Y_j}m)^{\frac{25}n}$, so independence and Lyapunov's inequality give, uniformly in $n$,
    \begin{equation}
        \EOf{h(2\hat\sigma_{\dtran})^{\nu}}
        \leq
        7^{\nu}\, \EOf{(1+\ol Ym)^{\frac{25}{n}}}^{n}
        \leq
        7^{\nu} M
        \eqfs
    \end{equation}
    Hence, \cref{cor:infdtr:improved} \ref{cor:infdtr:improved:low} yields, for $n$ large enough (absorbing the $\mo o(1)$-term into the positive main term),
    \begin{equation}
        \EOf{\frac{\ol m{m_n}^2}{1+\chi+\ol m{m_n}}}
        \leq
        \frac{512}{n}\, \sigma_{(\dtran)^2} \br{1+4r_0}
        \eqcm
    \end{equation}
    where $\chi\in\median(\ol Ym)$ and $r_0 = \max\brOf{\chi,\, 2\br{\exp(16\sigma_{\dtran})-1}}$.
    As $\Prof{\ol Ym\geq\chi}\geq\frac12$ implies $1+\chi \leq \exp(2\sigma_{\dtran})$, we obtain $1+4r_0 \leq 12\exp(16\sigma_{\dtran})$.
    Finally, distinguishing $x\lessgtr1$ shows $\min(x^2,x) \leq (2+\chi)\frac{x^2}{1+\chi+x}$ for $x\in\Rp$, and $2+\chi \leq 2\exp(2\sigma_{\dtran})$ concludes the proof.
\end{proof}
\subsection{Influence Lower Bound}
In this section, we prove \cref{thm:power:lower:lin}. We start with a technical lemma \cref{lmm:lower:helper}, which helps us to prove our main work horse, \cref{lmm:general:lower}. The latter lemma allows us to prove the two theorems.
\begin{lemma}\label{lmm:lower:helper}
    Let $\epsilon,a\in (0,1)$, $s,r \in \Rp$.
    Assume $as \geq r$.
    Then
    \begin{align}
        &\epsilon \br{\tran(s-r) - \tran((1-a)s)}
        - (1-\epsilon) (as + r) \br{\dtran\brOf{\frac{as}2} + \dtran\brOf{\frac{r}2}}
        \\&\geq
        \frac{\epsilon}2 \br{as-r} \dtran(s) 
        -
        4 as \dtran\brOf{\frac{as}2}
        \eqfs
    \end{align}
\end{lemma}
\begin{proof}
    The condition $as \geq r$ implies $s-r \geq (1-a) s$. Hence, by \cref{lmm:tran:diff},
    \begin{align}
        \tran(s-r) - \tran((1-a)s) 
        &\geq
        ((s-r)-(1-a)s) \frac{\dtran(s-r)+\dtran((1-a)s)}{2}
        \\&\geq 
        \frac{as-r}2 \br{\dtran(s)-\dtran(r)}
    \end{align}
    using the subadditivity of $\dtran$ and $\dtran((1-a)s)\geq0$.
    As $as \geq r$, we have $as+r \leq 2as$ and $\dtran\brOf{\frac{as}2} + \dtran\brOf{\frac{r}2} \leq 2 \dtran\brOf{\frac{as}2}$. Thus, 
    \begin{align}
        &\epsilon \br{\tran(s-r) - \tran((1-a)s)}
        - (1-\epsilon) (as + r) \br{\dtran\brOf{\frac{as}2} + \dtran\brOf{\frac{r}2}}
        \\&\geq
        \epsilon \frac{as-r}2 \br{\dtran(s)-\dtran(r)}
        - 4 (1-\epsilon) as \dtran\brOf{\frac{as}2}
        \\&\geq
        \epsilon \frac{as}2 \dtran(s)
        -
        \epsilon \frac{r}2 \dtran(s)
        -
        4 as \dtran\brOf{\frac{as}2}
        \eqfs
    \end{align}
    using $\frac{r}2 \dtran(r) \geq 0$, $\dtran(r) \leq  2\dtran\brOf{r/2} \leq 2\dtran\brOf{as/2}$, and $\epsilon + 4(1-\epsilon) \leq 4$.
\end{proof}
\begin{lemma}\label{lmm:general:lower}
    Let $P$ be the distribution of $Y$.
    Let $\epsilon\in(0,1)$ and $q\in\mc Q$.
    Set $\tilde P := (1-\epsilon)P + \epsilon \delta_q$. 
    Let $\tilde Y \sim \tilde P$. 
    Let $\tilde m$ be the $\tran$-Fr\'echet mean of $\tilde Y$.
    Assume $\lim_{x\to\infty} \dtran(x) = \infty$.
    \begin{enumerate}[label=(\roman*)]
        \item
        Assume 
        \begin{equation}
            \dtran(\ol qm) \geq \frac{32}{\epsilon} \EOf{\dtran(\ol {Y}m)}
            \eqfs
        \end{equation}
        Then,
        \begin{equation}
            \ol m{\tilde m} 
            \geq
            \frac12 \invdtran\brOf{\frac{\epsilon \dtran(\ol qm)}{32}}
            \eqfs
        \end{equation}
        \item 
        Let $\alpha \in (1, 2]$ and set $\tran(x) = x^\alpha$.
        Assume 
        \begin{equation}
            \ol qm\geq \br{\frac{16 \sigma_{\alpha-1}}{\epsilon}} ^{\frac{1}{\alpha-1}}
            \eqfs
        \end{equation}
        Then
        \begin{equation}
            \ol m{\tilde m}
            \geq 2^{-\frac{\alpha+4}{\alpha-1}}
            \epsilon^{\frac{1}{\alpha-1}} \ol qm
            \eqfs
        \end{equation}
    \end{enumerate} 
\end{lemma}
\begin{proof}
    \begin{enumerate}[label=(\roman*)]
        \item
        Let $p$ be a point on the geodesic between $m$ and $q$.
        Let $r\in\Rpp$.
        Let $z\in \ballclosed(m,r)$. Then we have
        \begin{equation}\label{eq:breakdown:first}
            \EOf{\tran(\ol {\tilde Y}z) - \tran(\ol {\tilde Y}p)}
            =
            (1-\epsilon)\EOf{\tran(\ol {Y}z) - \tran(\ol {Y}p)}
            +
            \epsilon\br{\tran(\ol qz) - \tran(\ol qp)}
            \eqfs
        \end{equation}
        For the first term on the right-hand side of \eqref{eq:breakdown:first}, we can use \cref{lmm:tran:diff}, \cref{lmm:dtran:subadd}, and the triangle inequality to obtain
        \begin{align}
            \EOf{\tran(\ol {Y}z) - \tran(\ol {Y}p)}
            &\geq
            - \ol pz\, \EOf{\dtran\brOf{\frac{\ol {Y}z + \ol {Y}p}2}}
            \\&\geq
            - \br{\ol pm + \ol mz}\br{\EOf{\dtran\brOf{\ol {Y}m}} + \dtran\brOf{\frac{\ol pm}2}+ \dtran\brOf{\frac{\ol mz}2}}
            \\&\geq
            - \br{as + r}\br{\EOf{\dtran\brOf{\ol {Y}m}} + \dtran\brOf{\frac{as}2}+ \dtran\brOf{\frac{r}2}}
        \end{align}
        with $s := \ol qm$ and $a := \ol pm/s$. 
        Now consider the second term on the right-hand side of \eqref{eq:breakdown:first}:
        As $p$ is on the geodesic between $m$ and $q$ we have $\ol qp = \ol qm - \ol pm$. Hence,
        \begin{equation}
            \tran(\ol qz) - \tran(\ol qp) \geq\tran(s-r)  -\tran((1-a)s) 
            \eqfs
        \end{equation}
        Thus, \eqref{eq:breakdown:first} becomes
        \begin{align}
            &\EOf{\tran(\ol {\tilde Y}z) - \tran(\ol {\tilde Y}p)}
            \\&\geq
            \epsilon \br{\tran(s-r) - \tran((1-a)s)}
            - (1-\epsilon) \br{as + r}\br{\EOf{\dtran\brOf{\ol {Y}m}} + \dtran\brOf{\frac{as}2}+ \dtran\brOf{\frac{r}2}}
            \eqfs
        \end{align}
        By \cref{lmm:lower:helper}, if $as \geq r$, we obtain 
        \begin{equation}\label{eq:lmm:general:lower:pre3bounds}
            \EOf{\tran(\ol {\tilde Y}z) - \tran(\ol {\tilde Y}p)}
            \geq
            \frac{\epsilon}2 \br{as-r} \dtran(s) 
            -
            4 as \dtran\brOf{\frac{as}2}
            - 2as \EOf{\dtran\brOf{\ol {Y}m}}
            \eqfs
        \end{equation}
        Assume the following three bounds
        \begin{enumerate}[label=(\alph*)]
            \item $\frac14as \geq r$,
            \item $\frac\epsilon8 \dtran(s) \geq 2 \Eof{\dtran(\ol {Y}m)}$,
            \item $ \frac{\epsilon}8 \dtran(s) \geq 4 \dtran\brOf{\frac{as}{2}}$.
        \end{enumerate}
        Then \eqref{eq:lmm:general:lower:pre3bounds} yields
        \begin{equation}
            \EOf{\tran(\ol {\tilde Y}z) - \tran(\ol {\tilde Y}p)}
            \geq
            \frac{\epsilon}8 as \dtran(s) > 0
            \eqfs
        \end{equation}
        Set 
        \begin{equation}
            s_0 := \invdtran\brOf{\frac{32}{\epsilon} \EOf{\dtran(\ol {Y}m)}}
            \eqcm\qquad
            a_0(s) := \frac2s \invdtran\brOf{\frac{\epsilon \dtran(s)}{32}}
        \end{equation}
        for $s \geq s_0$. Note that $\invdtran$ is only the true inverse of $\dtran$ on $[\dtran(0), \infty)$. This is accounted for, as the arguments of $\invdtran$ in the definition of $s_0$ and $a_0$ are at least $\dtran(0)$. 
        Furthermore, by \cref{lmm:dtran:factor},
        \begin{equation}
            a_0(s) \leq \frac2s \invdtran\brOf{\frac{\epsilon \dtran(s/2)}{16}} \leq \frac2s\invdtran\brOf{\dtran(s/2)} \leq 1
            \eqfs
        \end{equation}
        Thus, we can choose $p$, so that $a = a_0(s)$. Then, if $s \geq s_0$, (b) and (c) are fulfilled.
        Hence, we can set
        \begin{equation}
            r = r(s) = \frac14 a_0(s) s = \frac12 \invdtran\brOf{\frac{\epsilon \dtran(s)}{32}}
        \end{equation}
        to obtain 
        \begin{equation}
            \EOf{\tran(\ol {\tilde Y}z) - \tran(\ol {\tilde Y}p)} > 0
        \end{equation}
        for all $z \in \ballclosed(m, r)$. 
        Hence, for the $\tran$-Fr\'echet mean $\tilde m$ of $\tilde Y$, we have $\ol m{\tilde m} \geq r$.
        \item
        In the case $\tran(x)=x^\alpha$ with $\alpha\in(1,2]$, we have
        \begin{equation}
            \dtran(x) = \alpha x^{\alpha-1}
            \qquad\text{and}\qquad
            \invdtran(x) = \br{\frac{x}\alpha}^{\frac1{\alpha-1}}
            \eqfs
        \end{equation}
        As $\dtran(0)=0$, we can choose 
        \begin{equation}
            s_0 := \invdtran\brOf{\frac{16}{\epsilon} \EOf{\dtran(\ol {Y}m)}}
        \end{equation}
        in the proof of the first part.
        Hence,
        \begin{equation}
            r(s) 
            =
            \frac12 \invdtran\brOf{\frac{\epsilon \dtran(s)}{32}}
            =
            \frac12 \br{\frac{\epsilon s^{\alpha-1}}{32}}^{\frac{1}{\alpha-1}}
            =
            2^{-\br{1+\frac{5}{\alpha-1}}} \epsilon^{\frac{1}{\alpha-1}} s
            \eqfs
        \end{equation}
    \end{enumerate}
\end{proof}
\begin{proof}[Proof of \cref{thm:power:lower:lin}]
    \begin{enumerate}[label=(\roman*)]
        \item 
        We apply \cref{lmm:general:lower} to $P = \frac{1}{n-1}\sum_{i=2}^n \delta_{y_i}$, $\epsilon = \frac1n$, $q=y_1$.
        \item 
        Denote the sample $\alpha$-Fr\'echet mean of $Y_2, \dots, Y_n$ as $m_{2:n}$.
        By \cref{lmm:general:lower} there is a threshold $S(Y_2,\dots,Y_n)$ such that
        \begin{equation}
            \EOf{\ol {m_{n}}{m_{2:n}}^\xi | Y_2,\dots,Y_n} 
            \geq 
            \EOf{\br{c_\alpha n^{-\frac{1}{\alpha-1}} \ol {Y_1}{m_{2:n}}}^\xi \indOf{\ol {Y_1}{m_{2:n}} \geq S(Y_2,\dots,Y_n)}| Y_2,\dots,Y_n} 
            =
            \infty
        \end{equation}
        as $\Eof{\ol {Y}o^\xi} = \infty$ implies $\Eof{\ol {Y_1}q^\xi \indOf{\ol {Y_1}q \geq s}} = \infty$ for every $q\in\mc Q, s\in \R$, and $Y_1$ is independent of $Y_2, \dots, Y_n$.
        Hence, 
        \begin{align}
            \EOf{\ol {m_n}q^\xi} 
            &\geq 
            \EOf{\max\brOf{0, c_\xi \ol {m_{n}}{m_{2:n}}^\xi - \ol {m_{2:n}}{q}^\xi}} 
            \\ &\geq 
            \EOf{\max\brOf{0, c_\xi \EOf{\ol {m_{n}}{m_{2:n}}^\xi | Y_2,\dots,Y_n} - \ol {m_{2:n}}{q}^\xi}}
            \\ &= 
            \infty
            \eqfs
        \end{align}
        Thus, $\EOf{\ol {m_n}q^\xi}  = \infty$.
    \end{enumerate}
\end{proof}

%% file: sec_proof_bounded.tex
\section{Proofs of Section: Bounded Influence}

\subsection{Breakdown Point and Exponential Tail Bounds}
Hadamard spaces have unique projections to closed and convex sets.
\begin{proposition}[{\cite[Proposition 2.6]{sturm03}}]\label{prp:projection}
    Let $\mc B \subset \mc Q$ be a convex and closed set. For every $q \in \mc Q$, there is a $p \in \mc B$ such that
    \begin{equation}
        \forall y\in\mc B\colon\ \ol yq^2 \geq \ol yp^2 + \ol qp^2
        \eqfs 
    \end{equation}
    Furthermore,
    \begin{equation}
        \ol qp = \inf_{y \in \mc B} \,\ol qy =: d(q, \mc B)
        \eqfs
    \end{equation}
    We call $p$ the \emph{projection} of $q$ onto $\mc B$.
\end{proposition}
\begin{lemma}\label{lmm:finiteD:convex:lower}
    Assume $\lim_{x\to\infty}\dtran(x) =: D < \infty$.
    Let $\mc B \subset \mc Q$ be a convex and closed set with diameter $\diam(\mc B) \leq \delta \in\Rp$. Set $\rho := \Prof{Y \in \mc B}$.
    Then, for all $q \in \mc Q$, its projection $p$ onto $\mc B$ fulfills
    \begin{equation}
        \EOf{\tran(\ol Yq) - \tran(\ol Yp)} \geq \rho \br{\tran\brOf{\sqrt{\ol qp^2 + \delta^2}} + D (\ol qp - \delta)} - D\, \ol qp
        \eqfs
    \end{equation}
\end{lemma}
\begin{proof}[Proof of \cref{lmm:finiteD:convex:lower}]
    Let  $q \in \mc Q$. By \cref{prp:projection}, the \textit{projection} $p$ of $q$ onto $\mc B$ fulfills $p \in \mc B$ and $\ol yq^2 \geq \ol yp^2 + \ol qp^2$ for all $y\in \mc B$.
    Let $f\colon\Rp \to \R, x \mapsto \tran(\sqrt{x^2 + a}) - \tran(x)$ with $a \geq 0$.
    Then 
    \begin{align}
        f\pr(x) 
        &= 
        \frac{x}{\sqrt{x^2+a}} \dtran\brOf{\sqrt{x^2 + a}} - \dtran(x)
        \\&\leq
        \dtran\brOf{ \frac{x}{\sqrt{x^2+a}} \sqrt{x^2 + a}} - \dtran(x)
        \\&=
        0
        \eqcm
    \end{align}
    where we used $\frac{x}{\sqrt{x^2+a}} \in [0,1]$ and \cref{lmm:dtran:factor}.
    Thus, $f$ is nonincreasing.
    Together with $\tran(x_1) - \tran(x_2) \leq D \absof{x_1 - x_2}$ and $\tran(x) \leq D x$ for $x, x_1, x_2 \in\Rp$ (\cref{lmm:tran:diff} and \cref{lmm:tran:tran}), we obtain
    \begin{align}
        &\EOf{\tran(\ol Yq) - \tran(\ol Yp)} 
        \\&\geq 
        \EOf{\br{\tran\brOf{\sqrt{\ol Yp^2 + \ol qp^2}} - \tran(\ol Yp)} \ind_{\mc B}(Y)} 
        +
        \EOf{\br{\tran(\ol Yq) - \tran(\ol Yp)} \ind_{\mc Q \setminus \mc B}(Y)} 
        \\&\geq 
        \br{\tran\brOf{\sqrt{\delta^2 + \ol qp^2}} - \tran(\delta)} \rho
        -
        D \,\ol qp\, (1-\rho)
        \\&\geq 
        \rho \br{\tran\brOf{\sqrt{\delta^2 + \ol qp^2}} + D (\ol qp - \delta)} - D\, \ol qp
        \eqfs
    \end{align}
\end{proof}
\begin{proof}[Proof of \cref{thm:finiteD:convex}]
    Let $q\in\mc Q$ and $p \in \mc B$ its projection onto $\mc B$. Note that $\tran(R) \geq \lambda D R$ implies $\tran(x) \geq \lambda D x$ for all $x\geq R$. Assume that $\ol qp^2 + \delta^2 \geq R^2$. Then, by \cref{lmm:finiteD:convex:lower}, 
    \begin{align}
        \EOf{\tran(\ol Yq) - \tran(\ol Yp)} 
        &\geq 
        \rho \br{\tran\brOf{\sqrt{\ol qp^2 + \delta^2}} + D (\ol qp - \delta)} - D\, \ol qp
        \\&\geq 
        D \br{\rho \br{\lambda\sqrt{\ol qp^2 + \delta^2} + (\ol qp - \delta)} - \ol qp}
        \\&= 
        D \rho \br{\lambda \sqrt{\ol qp^2 + \delta^2} - \frac{1-\rho}{\rho}  \ol qp - \delta}
        \eqfs
    \end{align}
    Let $a := \frac{1-\rho}{\rho}$. Define
    \begin{equation}
        f(x) = \lambda \sqrt{x^2 + \delta^2} - a x - \delta
        \eqfs
    \end{equation}
    Thus,
    \begin{equation}\label{eq:finiteD:convex:basicvi}
        \EOf{\tran(\ol Yq) - \tran(\ol Yp)}  \geq D \rho f(\ol qp)
        \eqfs
    \end{equation}
    We have
    \begin{align}
        f\pr(x) = \lambda \frac{x}{\sqrt{x^2 + \delta^2}} - a
        \eqcm
        \qquad\text{and}\qquad
        f\prr(x) = \lambda \frac{\delta^2}{(x^2+\delta^2)^{\frac32}}
        \eqfs
    \end{align}
    We assume $\rho > \frac1{1+\lambda}$ and thus $\lambda > a$.
    If $\delta = 0$, $f$ is linear and $f(x)>0$ for all $x > 0 = x_0$. If $\delta > 0$, then $f$ is strictly convex. If $x_0$ is the largest $x\geq0$ with $f(x) = 0$, then $f$ is positive for all values larger than $x_0$. So, we calculate
    \begin{align}
        &\lambda \sqrt{x^2 + \delta^2} - a x - \delta = 0
        \\\Leftrightarrow\quad &
        \lambda^2 x^2 + \lambda^2 \delta^2 = a^2 x^2 + \delta^2 + 2 a \delta x
        \\\Leftrightarrow\quad &
        (\lambda^2 - a^2) x^2 - 2 a \delta x - (1-\lambda^2) \delta^2= 0
        \\\Leftrightarrow\quad &
        x = \frac{2 a \delta \pm \sqrt{(2 a \delta)^2 + 4 (\lambda^2 - a^2) (1-\lambda^2) \delta^2}}{2 (\lambda^2 - a^2)}
        \eqfs
    \end{align}
    The larger root is
    \begin{align}
        x_0 = \frac{2 a \delta + \sqrt{(2 a \delta)^2 + 4 (\lambda^2 - a^2) (1-\lambda^2) \delta^2}}{2 (\lambda^2 - a^2)}
        =
        \delta \frac{a + \lambda \sqrt{1 -\lambda^2 + a^2}}{\lambda^2 - a^2}
        \eqfs
    \end{align}
    Thus, for all $x  > x_0$, we have $f(x) > 0$.
    Applying this to \eqref{eq:finiteD:convex:basicvi} yields
    \begin{equation}
        \EOf{\tran(\ol Yq) - \tran(\ol Yp)} > 0 
    \end{equation}
    for all $q\in\mc Q$ that fulfill $\ol qp > x_0$ and $\ol qp^2 + \delta^2 \geq R^2$.
    Hence, $q$ is not a $\tran$-Fr\'echet mean of $Y$. In other words, for the projection $p_m$ of $m$ onto $\mc B$, we must have
    \begin{equation}
        \ol m{p_m} \leq x_0 \qquad\text{or}\qquad \ol m{p_m}^2 + \delta^2 < R^2
        \eqfs
    \end{equation}
    Hence,
    \begin{equation}
        \ol m{p_m}^2 \leq \max\brOf{x_0^2, R^2 - \delta^2}
        \eqfs
    \end{equation}
    We finish the proof by noting that the projection fulfills
    \begin{equation}
        \ol m{p_m} = \inf_{y \in \mc B} \ol ym = d(m, \mc B)
        \eqcm
    \end{equation}
    see \cref{prp:projection}.
\end{proof}
\begin{proof}[Proof of \cref{thm:breakdown}]
    \begin{enumerate}[label=(\roman*)]
        \item
        Set $D := \lim_{x\to\infty}\dtran(x) < \infty$.
        Let $\epsilon\in(0,\frac12)$.
        Let $P \in \mc P_0(\mc Q)$.
        Let $\zeta \in (0,1)$.
        Let $o\in\mc Q$ be an arbitrary reference point. We can choose $r$ large enough so that 
        \begin{equation}
            P(\ballclosed(o, r)) \geq  1 - \zeta =: \rho
            \eqfs
        \end{equation}
        Let $\tilde P \in \mc P_0(\mc Q)$ be an $\epsilon$-contamination of $P$ with $\tilde P = \check P + \mu$ as in \cref{def:breakdown}.
        We have
        \begin{equation}
            \tilde P(\ballclosed(o, r)) \geq P(\ballclosed(o, r)) - \mu(\mc Q) \geq 1 - \zeta - \epsilon =: \tilde\rho
            \eqfs
        \end{equation}
        As $\epsilon < \frac12$, we can make $\zeta$ small enough so that $\tilde\rho > \frac12$.
        Let $m \in M(P)$ be a $\tran$-Fr\'echet mean of $P$.
        Let $\tilde m \in M(\tilde P)$ be a $\tran$-Fr\'echet mean of $\tilde P$.
        We apply \cref{thm:finiteD:convex} by choosing $\lambda\in(0,1)$ large enough so that $\tilde\rho > \frac{1}{1+\lambda}$ and obtain
        \begin{equation}
            d(m, \ballclosed(o, r)) \leq K \qquad\text{and}\qquad d(\tilde m, \ballclosed(o, r)) \leq \tilde K
        \end{equation}
        for finite radii $K, \tilde K \in \Rpp$ that depend only on $\rho,\tilde\rho,\lambda,\tran$, and $r$. In particular, they do not depend on the specific contamination.
        Thus, 
        \begin{equation}
            \sup\setByEle{\delta(M(P), M(\tilde P))}{\tilde P \in \mc P_0(\mc Q) \text{ is } \epsilon\text{-contamination of } P} \leq 2r + K + \tilde K
            \eqfs
        \end{equation}
        Thus, $\varepsilon(P, \delta, \mc P_0(\mc Q), M) \geq \epsilon$ for all $\epsilon < \frac12$. Hence,
        \begin{equation}
            \varepsilon(P, \delta, \mc P_0(\mc Q), M) \geq \frac12
            \eqfs
        \end{equation}
        Equality follows easily by considering $\epsilon > \frac12$ and a sequence of $\epsilon$-contaminations $\tilde P_k$ where the contamination part $\mu_k$ is a point mass at $q_k$ and $(q_k)\subset \mc Q$ is a sequence with $\lim_{k\to\infty} \ol o{q_k} = \infty$.
        \item
        Let $P \in \mc P_{\dtran}(\mc Q)$.
        Fix $\epsilon\in(0,1)$.
        Let $q\in\mc Q$ and let $Q$ be the measure with $Q(\{q\})=1$. Let $\tilde P = (1-\epsilon)P + \epsilon Q$. Then $\tilde P$ is an $\epsilon$-contamination of $P$. Let $Y \sim P$ and $\tilde Y \sim \tilde P$. Let $m$ be the $\tran$-Fr\'echet mean of $Y$ and $\tilde m$ of $\tilde Y$. 
        Set
        \begin{equation}
           r(s) := \frac12 \invdtran\brOf{\frac{\epsilon \dtran(s)}{32}}
        \end{equation}
        where $\invdtran$ is the generalized inverse function of $\dtran$.
        Then \cref{lmm:general:lower} shows $\ol m{\tilde m} \geq r(s)$.
        For $\dtran(x) \xrightarrow{x\to\infty} \infty$, we have $r(s)\xrightarrow{s\to\infty} \infty$.
        As $\diam(\mc Q) = \infty$, we can realize $s = \ol qm \to \infty$. Hence, we can find a sequence of $\epsilon$-contaminations $\tilde P_k$ so that $\lim_{k\to\infty}\delta(M(P), M(\tilde P_k)) = \infty$. As the choice of $\epsilon\in(0,1)$ was arbitrary, we obtain 
        \begin{equation}
            \varepsilon(P, \delta, \mc P_{\dtran}(\mc Q), M) = 0
            \eqfs
            \qedhere
        \end{equation}
    \end{enumerate}
\end{proof}
\begin{proof}[Proof of \cref{thm:tailbound}]
    Set
    \begin{equation}
        \rho_n := \frac1n \sum_{i=1}^n \ind_{[0,r]}(\ol {Y_i}{m})
        \eqfs
    \end{equation}
    Then, by the Chernoff bound in form of \cref{lmm:tail:prob}, for $\eta\in(0,1]$,
    \begin{equation}
        \PrOf{\rho_n \leq \eta \rho}
        \leq
        \exp\brOf{-n \kullback(\eta\rho, \rho)}
        \eqfs
    \end{equation}
    Set $a := \frac{1-\eta\rho}{\eta\rho}$ and $x_0 := \frac{2r}{\lambda - a}$, which is
    positive since $\lambda - a = \frac{(\lambda+1)\eta\rho - 1}{\eta\rho} > 0$ by assumption.
    By applying \cref{thm:finiteD:convex} to the empirical distribution, on the event that
    $\rho_n \geq \eta \rho$, we obtain
    \begin{equation}
        d\brOf{m_n, \ballclosed(m, r)}^2 \leq \max\brOf{x_0^2, R^2 - 4r^2}
        \eqfs
    \end{equation}
    Thus, if $r\geq \frac12 R$, we have $\ol m{m_n} - r \leq x_0$. One can easily calculate
    \begin{align}
        x_0 + r = \br{\frac{2 + \lambda - a}{\lambda - a}} r
        = \br{\frac{\br{3 + \lambda}\eta\rho - 1}{\br{1 + \lambda}\eta\rho - 1}} r
    \end{align}
    to finish the proof.
\end{proof}
\begin{proof}[Proof of \cref{cor:tailbound}]
    Use \cref{thm:tailbound} with $\lambda = \frac9{10}$ and $\eta := \frac{2}{3\rho}$.
    This is admissible: $\rho > \frac23$ gives $\eta \in [\frac23, 1)$, and
    $\eta\rho = \frac23$ yields $(\lambda+1)\eta\rho = \frac{19}{15} > 1$. The radius factor is
    \begin{equation}
        \frac{(\lambda + 3)\eta\rho - 1}{(\lambda + 1)\eta\rho - 1}
        =
        \frac{\frac{39}{10}\cdot\frac23 - 1}{\frac{19}{10}\cdot\frac23 - 1}
        =
        \frac{8/5}{4/15}
        =
        6
        \eqfs
    \end{equation}
    It remains to evaluate the exponent. By the definition of $\kullback$,
    \begin{equation}
        \exp\brOf{-\kullback\brOf{\tfrac23, \rho}}
        =
        \br{\frac{\rho}{2/3}}^{\frac23} \br{\frac{1-\rho}{1/3}}^{\frac13}
        =
        \br{\frac{27}{4} (1-\rho) \rho^2}^{\frac13}
        \eqfs
    \end{equation}
    The last inequality in the statement of the corollary follows from $\rho^2 \leq 1$.
\end{proof}
\begin{proof}[Proof of \cref{thm:median:tailbound}]
    Set
    \begin{equation}
        \rho_n := \frac1n \sum_{i=1}^n \ind_{[0,r]}(\ol {Y_i}{m})
        \eqfs
    \end{equation}
    By the Chernoff bound in form of \cref{lmm:tail:prob}, we have, for $\eta\in(0,1]$,
    \begin{equation}\label{eq:median:tailbound:chernoff}
        \PrOf{\rho_n \leq \eta \rho} \leq \exp\brOf{-n \kullback(\eta\rho, \rho)}
        \eqfs
    \end{equation}
    The map
    \begin{equation}
        g \colon \br{\tfrac12, 1} \to \Rp
        \eqcm\qquad
        g(t) := \frac{2t(1-t)}{2t-1}
        \eqcm
    \end{equation}
    is decreasing, since $g\pr(t) = -\br{(2t-1)^2+1}(2t-1)^{-2} < 0$.
    Now consider the event $\cb{\rho_n \geq \eta\rho}$. The set $\mc B := \ballclosed(m, r)$ is closed and convex with $\diam(\mc B) \leq 2r$, and it carries empirical mass $\rho_n \geq \eta\rho > \frac12$. Hence, \cref{cor:finiteD:convex} applied to the empirical distribution $\frac1n\sum_{i=1}^n \delta_{Y_i}$, together with the monotonicity of $g$, yields
    \begin{equation}
        d\brOf{m_n, \mc B} \leq 2r g(\rho_n) \leq 2r g(\eta\rho) = 4 r \eta\rho \frac{1 - \eta\rho}{2\eta\rho - 1}
        \eqfs
    \end{equation}
    As $\ol m{m_n} \leq r + d(m_n, \mc B)$, we obtain
    \begin{equation}
        \ol m{m_n}
        \leq
        \br{1 + \frac{4\eta\rho(1-\eta\rho)}{2\eta\rho-1}} r
        =
        \br{\frac{6\eta\rho - 1 - 4\eta^2\rho^2}{2\eta\rho - 1}} r
    \end{equation}
    on the event $\cb{\rho_n \geq \eta\rho}$. Thus, the event in the claim is contained in $\cb{\rho_n < \eta\rho}$ and \eqref{eq:median:tailbound:chernoff} finishes the proof.
\end{proof}
\begin{proof}[Proof of \cref{cor:median:tailbound}]
    Use \cref{thm:median:tailbound} with $\eta := \frac{2}{3\rho}$.
    This is admissible: $\rho > \frac23$ gives $\eta \in [\frac23, 1)$, and $\eta\rho = \frac23$ yields $2\eta\rho = \frac43 > 1$. The radius factor is
    \begin{equation}
        \frac{6\eta\rho - 1 - 4\eta^2\rho^2}{2\eta\rho-1}
        =
        \frac{4 - 1 - \frac{16}9}{\frac13}
        =
        \frac{11}{3}
        \eqfs
    \end{equation}
    For the exponent, the definition of $\kullback$ gives
    \begin{equation}
        \exp\brOf{-\kullback\brOf{\tfrac23, \rho}}
        =
        \br{\frac{\rho}{2/3}}^{\frac23} \br{\frac{1-\rho}{1/3}}^{\frac13}
        =
        \br{\frac{27}{4} (1-\rho) \rho^2}^{\frac13}
        \eqfs
    \end{equation}
    The last inequality in the statement of the corollary follows from $\rho^2 \leq 1$.
\end{proof}
\subsection{Convergence Rates}
Set the double excess risk as
\begin{equation}
    V_n := \EOf{\tran(\ol Y{m_n}) - \tran(\ol Ym)} + \EOf{\frac1n \sum_{i=1}^n \br{\tran(\ol {Y_i}{m}) - \tran(\ol {Y_i}{m_n})} }
    \eqfs
\end{equation}
Recall that $Y, Y_1, \dots, Y_n$ are iid, so that $Y$ is independent of $m_n$.
\begin{proposition}\label{prp:posddrtran:Vn}
    Assume $\lim_{x\to\infty}\dtran(x) < \infty$.
    Assume $\ddrtran(x) > 0$ for all $x\in\Rpp$.
    Assume $\xi \in \Rpp$ exists with $\Eof{\ol Ym^\xi} < \infty$.
    Then there are $n_0 \in \N$ and $C \in \Rpp$ depending only on $\tran$, $\xi$, and the distribution of $Y$ such that
    \begin{equation}
        V_n \leq Cn^{-1}
    \end{equation}
    for all $n \geq n_0$.
\end{proposition}
\begin{lemma}\label{lmm:posddrtran:expect}
    Let $\zeta \in \Rpp$.
    Assume $\lim_{x\to\infty}\dtran(x) < \infty$.
    Assume $\xi \in \Rpp$ exists with $\Eof{\ol Ym^\xi} < \infty$.
    Assume $n_0 \geq 3\zeta\xi^{-1} + \max(1,\xi^{-1})$.
    Then there are constants $r_0, C_1, C_2\in\Rpp$ depending only on $\tran$, $\xi$, $\zeta$, and the distribution of $Y$ such that
    \begin{equation}
        \EOf{\ol m{m_n}^\zeta \indOfOf{[r_0, \infty)}{\ol m{m_n}}} \leq C_1 \exp\brOf{- C_2 n}
    \end{equation}
    for all $n \geq n_0$.
    
    In particular, there is a constant $C_3\in\Rpp$ depending only on $\tran$, $\xi$, $\zeta$, and the distribution of $Y$ such that
    \begin{equation}
        \EOf{\ol m{m_n}^\zeta} \leq C_3
    \end{equation}
    for all $n \geq n_0$.
\end{lemma}
\begin{proof}
    Let $R \in \Rpp$ as in \cref{cor:tailbound}.
    By Markov's inequality,
    \begin{equation}
        \PrOf{\ol Ym > \frac t6}
        =
        \PrOf{\ol Ym^\xi > \br{\frac t6}^\xi}
        \leq
        \frac{6^\xi\Eof{\ol Ym^\xi}}{t^\xi}
        \eqcm
    \end{equation}
    so that, with
    \begin{equation}
        \tilde C := 2 \cdot 6^{\frac{\xi}3} \Eof{\ol Ym^\xi}^{\frac13}
        \eqcm
        \qquad
        r_1 := \max\br{3R,\ \br{2\tilde C}^{\frac3\xi}}
        \eqcm
        \qquad
        r_0 := 2 r_1
        \eqcm
    \end{equation}
    we have $2\PrOf{\ol Ym > \frac t6}^{\frac13} \leq \tilde C t^{-\frac\xi3} \leq \frac12$ for all
    $t \geq r_1$. In particular $\PrOf{\ol Ym > \frac t6} \leq \frac1{64}$, so
    \cref{cor:tailbound} applies with $r = \frac t6 \geq \frac R2$ and yields
    \begin{equation}\label{eq:tail:t}
        \PrOf{\ol m{m_n} > t}
        \leq
        \br{2\PrOf{\ol Ym > \frac t6}^{\frac13}}^n
        \leq
        \br{\tilde C t^{-\frac{\xi}3}}^n
    \end{equation}
    for all $t \geq r_1$. Set $C' := \tilde C r_1^{-\frac{\xi}3} \leq \frac12$.
    Note that $\tilde C$, and hence $r_1$, $r_0$, and $C'$, depend only on $R$, $\xi$, and the
    distribution of $Y$; in particular they do not depend on $\zeta$.
    
    Since $\ol m{m_n} \geq 0$, the layer-cake formula gives, for any $\zeta \in \Rpp$,
    \begin{equation}
        \EOf{\ol m{m_n}^\zeta \indOfOf{[r_0,\infty)}{\ol m{m_n}}}
        =
        r_0^\zeta \PrOf{\ol m{m_n} \geq r_0}
        +
        \int_{r_0}^\infty \zeta t^{\zeta-1} \PrOf{\ol m{m_n} > t}\,\dl t
        \eqfs
    \end{equation}
    Since $r_0 > r_1$, the first term is controlled by \eqref{eq:tail:t} at $t = r_1$,
    \begin{equation}
        \PrOf{\ol m{m_n} \geq r_0}
        \leq
        \PrOf{\ol m{m_n} > r_1}
        \leq
        \br{\tilde C r_1^{-\frac\xi3}}^n
        =
        {C'}^n
        \eqcm
    \end{equation}
    and the second by \eqref{eq:tail:t} on $[r_0,\infty)$. As
    $n_0 \geq 3\zeta\xi^{-1} + \max(1,\xi^{-1})$, we have $\zeta - \frac{\xi n}3 < 0$
    for $n \geq n_0$, so the integral converges and, using
    $\tilde C r_0^{-\frac\xi3} \leq \tilde C r_1^{-\frac\xi3} = C'$,
    \begin{align}
        \int_{r_0}^\infty \zeta t^{\zeta-1} \PrOf{\ol m{m_n} > t}\,\dl t
        &\leq
        \zeta \tilde C^n \int_{r_0}^\infty t^{\zeta-1-\frac{\xi n}3}\,\dl t
        =
        \frac{3\zeta}{\xi n - 3\zeta} r_0^{\zeta} \br{\tilde C r_0^{-\frac\xi3}}^n
        \leq
        \frac{3\zeta}{\xi n - 3\zeta} r_0^\zeta {C'}^n
        \eqfs
    \end{align}
    Combining the two terms, for all $n \geq n_0$,
    \begin{equation}
        \EOf{\ol m{m_n}^\zeta \indOfOf{[r_0,\infty)}{\ol m{m_n}}}
        \leq
        \br{1 + \frac{3\zeta}{\xi n - 3\zeta}} r_0^\zeta {C'}^n
        \leq
        \br{1 + 3\zeta} r_0^\zeta {C'}^n
        \leq
        \br{1 + 3\zeta} r_0^\zeta 2^{-n}
        \eqfs
    \end{equation}
    Here we used that $n_0 \geq 3\zeta\xi^{-1} + \max(1,\xi^{-1})$ yields
    $\xi n - 3\zeta \geq \xi n_0 - 3\zeta \geq \max(\xi,1) \geq 1$, hence
    $\frac{3\zeta}{\xi n - 3\zeta} \leq 3\zeta$, and that $C' \leq \frac12$.
    This is the claim with $C_1 := (1+3\zeta) r_0^\zeta$ and $C_2 := \ln 2$.
\end{proof}
\begin{lemma}\label{lmm:posddrtran:event}
    Assume $\lim_{x\to\infty}\dtran(x) < \infty$.
    Assume $\xi \in \Rpp$ exists with $\Eof{\ol Ym^\xi} < \infty$. 
    For $r\in \Rpp$, $\eta \in [0,1]$, and $n \in \N$, define the events
    \begin{align}
        A = A_{r,\eta,n} &:= \cb{\frac1n \sum_{j=1}^n \ind_{[0, r]}(\ol {Y_j}{m}) \geq \eta}\eqcm
        &
        A^i = A^i_{r,\eta,n} &:= \cb{\frac1n \sum_{j=1}^n \ind_{[0, r]}(\ol {Y_j^i}{m}) \geq \eta}\eqcm
        \\
        B = B_{r,n} &:= \cb{\ol m{m_n} \leq r}\eqcm
        &
        B^i = B^i_{r,n} &:= \cb{\ol m{m_n^i} \leq r}\eqfs
    \end{align}
    Let $\eta_0\in[0,1)$.
    Then there are $r_0\in\Rpp$ and $n_0\in\N$ large enough with the following property:
    For all $r \geq r_0, n\geq n_0$, we have
    \begin{equation}
        \PrOf{(A_{r,\eta_0,n} \cap B_{r,n} \cap A^i_{r,\eta_0,n} \cap B^i_{r,n})\compl} \leq \exp(-cn)
        \eqcm
    \end{equation}
    where $c\in\Rpp$ does not depend on $n$.
\end{lemma}
\begin{proof}
    Set
    \begin{equation}
        \Omega^i := A \cap B \cap A^i \cap B^i\eqfs
    \end{equation}
    For the probability of the complement of $\Omega^i$, we use
    \begin{align}
        \PrOf{(\Omega^i)\compl}
        &\leq
        \PrOf{A\compl} +  \PrOf{B\compl} + \PrOf{(A^i)\compl} + \PrOf{(B^i)\compl}
        \\&=
        2\PrOf{A\compl} + 2\PrOf{B\compl}
        \eqfs
    \end{align}
    Set $M := \Eof{\ol Ym^\xi} < \infty$ and $q_r := \PrOf{\ol Ym > r}$, so that
    $q_r \leq M r^{-\xi}$ by Markov's inequality. In particular $q_r < \frac13$ for all
    $r$ large enough, so \cref{cor:tailbound} applies and yields $r_1 \in \Rpp$ with
    \begin{equation}
        \PrOf{B\compl} = \PrOf{\ol m{m_n} > r}
        \leq
        \br{2 q_r^{\frac13}}^n
        \leq
        \br{C_1 r^{-\frac{\xi}3}}^n
        \eqcm
        \qquad
        C_1 := 2 M^{\frac13}
        \eqcm
    \end{equation}
    for all $r \geq r_1$.
    For event $A$, set $\rho_r := \PrOf{\ol {Y}{m} \leq r}$ and $\rho_0 := \frac{1+\eta_0}2$.
    As $\lim_{r\to\infty}\rho_r = 1$ and $\eta_0 < 1$, there is $r_2 \in \Rpp$ with
    $\rho_r \geq \rho_0 > \eta_0$ for all $r \geq r_2$.
    Since $\rho \mapsto \kullback(\eta_0, \rho)$ is increasing on $[\eta_0, 1]$,
    \cref{prp:chernoff:addi} applied with $t = \eta_0$ yields
    \begin{equation}
        \PrOf{A\compl}
        =
        \PrOf{\frac1n \sum_{j=1}^n \ind_{[0, r]}(\ol {Y_j}{m}) < \eta_0}
        \leq
        \exp\brOf{-n \kullback(\eta_0, \rho_r)}
        \leq
        \exp(-c_2 n)
        \eqcm
    \end{equation}
    where $c_2 := \kullback\brOf{\eta_0, \frac{1+\eta_0}2} > 0$ depends only on $\eta_0$.
    Finally, choose $r_0 \geq \max\brOf{r_1, r_2, (2C_1)^{3/\xi}}$, so that
    $C_1 r^{-\xi/3} \leq \frac12$ for all $r \geq r_0$. Then
    \begin{equation}
        \PrOf{(\Omega^i)\compl}
        \leq
        4 \exp(-c' n)
        \eqcm
        \qquad
        c' := \min\brOf{\log 2, c_2}
        \eqcm
    \end{equation}
    and hence $\PrOf{(\Omega^i)\compl} \leq \exp(-cn)$ with $c := \frac12 c'$ for all
    $n \geq n_0 := \lceil 2\log(4)/c' \rceil$ and $r \geq r_0$.
\end{proof}
\begin{proof}[Proof of \cref{prp:posddrtran:Vn}]
    \cref{lmm:variancebound} together with $\dtran(x) \leq D$ shows
    \begin{equation}
        V_n \leq \frac{D}n \sum_{i=1}^n \EOf{\ol {m_n}{m_n^i}}
        \eqfs
    \end{equation}
    Set
    \begin{equation}
        \tilde H_i := \frac1n \sum_{j=1}^n \br{\ddrtran(\ol {Y_j}{m_n} + \ol {m_n}{m_n^{i}}) + \ddrtran(\ol {Y_j^{i}}{m_n^{i}} + \ol {m_n}{m_n^{i}})}
        \eqfs
    \end{equation}
    Then \cref{lmm:mnmnibound} implies
    \begin{equation}
        \ol {m_n}{m_n^{i}} \leq \frac {4 D}n \tilde H_i^{-1}
        \eqfs
    \end{equation}
    We now need to find a suitable bound on
    \begin{align}
        \tilde H_i^{-1} 
        &= 
        \br{\frac1n \sum_{j=1}^n \br{\ddrtran(\ol {Y_j}{m_n} + \ol {m_n}{m_n^{i}}) + \ddrtran(\ol {Y_j^{i}}{m_n^{i}} + \ol {m_n^{i}}{m_n})}}^{-1}
        \\&\leq
        \br{\frac1n \sum_{j=1}^n \br{\ddrtran(\ol {Y_j}{m} + 2\,\ol m{m_n} + \ol {m}{m_n^{i}}) + \ddrtran(\ol {Y_j^{i}}{m} +  \ol m{m_n} + 2\,\ol {m}{m_n^{i}})}}^{-1}
        \eqfs
    \end{align}
    Let $r\in \Rpp$ and $\eta \in (0,1)$. Define the events
    \begin{align}
        A &:= \cb{\frac1n \sum_{j=1}^n \ind_{[0, r]}(\ol {Y_j}{m}) \geq \eta}\eqcm
        &
        A^i &:= \cb{\frac1n \sum_{j=1}^n \ind_{[0, r]}(\ol {Y_j^i}{m}) \geq \eta}\eqcm
        \\
        B &:= \cb{\ol m{m_n} \leq r}\eqcm
        &
        B^i &:= \cb{\ol m{m_n^i} \leq r}
    \end{align}
    and
    \begin{equation}
        \Omega^i := A \cap B \cap A^i \cap B^i\eqfs
    \end{equation}
    On $\Omega^i$, we have
    \begin{equation}
        \tilde H_i^{-1} \leq \br{2\eta\ddrtran(4r)}^{-1}
        \eqfs
    \end{equation}
    We split $V_n$ on $\Omega^i$ as follows
    \begin{equation}
        V_n
        \leq 
        \frac D{n}\sum_{i=1}^n \EOf{\ol{m_n}{m_n^{i}}}
        =
        \frac D{n}\sum_{i=1}^n \EOf{\ol{m_n}{m_n^{i}} \ind_{\Omega^i}}
        +
        \frac D{n}\sum_{i=1}^n \EOf{\ol{m_n}{m_n^{i}} \ind_{(\Omega^i)\compl}}
        \eqfs
    \end{equation}
    For the first term, we have already shown
    \begin{equation}
        \EOf{\ol{m_n}{m_n^{i}} \ind_{\Omega^i}}
        \leq 
        \frac{4D}{n}\EOf{\tilde H_i^{-1} \ind_{\Omega^i}}
        \leq 
        \frac{2D}{\eta\ddrtran(4r)n}
        \eqfs
    \end{equation}
    Hence, 
    \begin{equation}
        \frac D{n}\sum_{i=1}^n \EOf{\ol{m_n}{m_n^{i}} \ind_{\Omega^i}} \leq \frac{2D^2}{\eta\ddrtran(4r)n}
        \eqfs
    \end{equation}
    For the second term, we use the triangle inequality and Cauchy--Schwarz and obtain
    \begin{align}
        \EOf{\ol{m_n}{m_n^{i}} \ind_{(\Omega^i)\compl}}
        &\leq 
        \EOf{\ol{m}{m_n} \ind_{(\Omega^i)\compl}}
        +
        \EOf{\ol{m}{m_n^{i}} \ind_{(\Omega^i)\compl}}
        \\&=
        2 \EOf{\ol{m}{m_n} \ind_{(\Omega^i)\compl}}
        \\&\leq 
        2 \br{\EOf{\ol{m}{m_n}^2} \PrOf{(\Omega^i)\compl}}^{\frac12}
        \eqfs
    \end{align}
    To finish the proof, we need to show that $\EOf{\ol{m}{m_n}^2}$ can be bounded by a constant $\tilde C\in\Rpp$ and the probability decreases exponentially in $n$, i.e., $\PrOf{(\Omega^i)\compl} \leq \exp(-cn)$ with $c\in\Rpp$.
    This is proven in \cref{lmm:posddrtran:expect,lmm:posddrtran:event}.
    
    Putting everything together, we get, for $r\in\Rpp$ and $n\in\N$ large enough and a fixed $\eta\in(0,1)$ (chosen to satisfy the conditions of \cref{lmm:posddrtran:expect,lmm:posddrtran:event}),
    \begin{equation}
        V_n
        \leq 
        \frac{2D^2}{\eta\ddrtran(4r)n}
        +
        2 D \br{\tilde C \exp(-cn)}^{\frac12}
        \eqfs
    \end{equation}
    Thus, there is $C \in \Rpp$ such that
    \begin{equation}
        V_n \leq Cn^{-1}
    \end{equation}
    for all $n \geq n_0$.
\end{proof}
\begin{proof}[Proof of \cref{thm:posddrtran:main}]
    By the minimizing property of $m_n$, we have
    \begin{equation}
        V_n
        \geq 
        \EOf{\tran(\ol Y{m_n})-\tran(\ol Ym)}
        \eqfs
    \end{equation}
    Let $r \in \Rpp$ with $\Prof{\ol Ym \leq r}>0$. We combine \cref{thm:infdtr:vi} \ref{thm:infdtr:vi:quantile} using $\ddrtran(2r) >0$ with \cref{prp:posddrtran:Vn} to obtain the following: There are $n_0\in\N$, $C_1 \in\Rpp$ depending only on $\tran$, $r$, $\xi$, and the distribution of $Y$, such that
    \begin{equation}
        \EOf{\ol m{m_n}^2 \indOfOf{[0, r]}{\ol m{m_n}}} \leq C_1 n^{-1}
    \end{equation}
    for all $n \geq n_0$. Moreover, \cref{lmm:posddrtran:expect} shows that there are $r_0, C_2\in\Rpp$ depending only on $\tran$, $\xi$, and the distribution of $Y$, such that
    \begin{equation}
        \EOf{\ol m{m_n}^2 \indOfOf{(r_0, \infty)}{\ol m{m_n}}} = \mo O(\exp(-C_2 n))
        \eqfs
    \end{equation}
    Thus, if $r$ is chosen so that $r\geq r_0$, we obtain 
    \begin{equation}
        \EOf{\ol m{m_n}^2} = \mo O\brOf{\frac1n}
        \eqfs
    \end{equation}
\end{proof}
To prove the \cref{thm:median:main}, we first need some additional lower bounds on $\EOf{\ol Yq - \ol Ym}$, i.e., variance inequalities. 
\begin{lemma}\label{lmm:median:vi:specific}
    Let $\tran(x)=x$ so that $m$ is a Fr\'echet median. 
    \begin{enumerate}[label=(\roman*)]
        \item
        Let $r \in \Rpp$ such that $\PrOf{\ol Ym > r} < \frac1{27}$. 
        Set $R := 6 r$.
        Then, for all $q \in \ballclosed(m, R)\compl$, we have
        \begin{equation}
            \EOf{\ol Yq - \ol Ym}  \geq \frac35 \ol qm
            \eqfs
        \end{equation}
        \item
        Let $R \in \Rpp$ and $w \in [0,1]$. Set
        \begin{equation}
            \rho := \inf_{p \in \ballclosed(m, R)} \PrOf{Y \in \bowtiesetcompl(m, p, w)}
            \eqfs
        \end{equation}
        Let $\tilde\chi  \in \Rpp$ such that $\PrOf{\ol Ym \leq \tilde\chi} \geq 1 - \frac12\rho$. Then, for all  $q \in \ballclosed(m, R)$,
        \begin{equation}
            \EOf{\ol Yq - \ol Ym} 
            \geq 
            \frac{\rho w^2}{4(\tilde\chi + R)}  \,\ol qm^2
            \eqfs
        \end{equation}
    \end{enumerate}
\end{lemma}
\begin{proof}
    \begin{enumerate}[label=(\roman*)]
        \item
        Let $r \in \Rpp$ and $\rho := \PrOf{\ol Ym \leq r}$. Then
        \begin{align}
            \EOf{\ol Yq - \ol Ym} 
            &= 
            \EOf{\br{\ol Yq - \ol Ym}\ind_{[0, r]}(\ol Ym)} 
            +
            \EOf{\br{\ol Yq - \ol Ym}\ind_{(r, \infty)}(\ol Ym)} 
            \\&\geq
            \EOf{\br{\ol qm - 2\ol Ym}\ind_{[0, r]}(\ol Ym)} 
            -
            \EOf{\ol qm\ind_{(r, \infty)}(\ol Ym)} 
            \\&\geq
            \br{\ol qm - 2r}\rho 
            -
            \ol qm\, (1-\rho)
            \\&=
            2\rho(\ol qm - r) - \ol qm
            \eqfs
        \end{align}
        As $\rho \geq \frac{26}{27}$, if $\ol qm \geq r$, we have
        \begin{equation}
            \EOf{\ol Yq - \ol Ym} \geq \frac{25}{27} \ol qm - \frac{52}{27} r
            \eqfs
        \end{equation}
        Thus, if $q \in \ballclosed(m, 6 r)\compl$, then
        \begin{equation}
            \EOf{\ol Yq - \ol Ym}  \geq \frac{6\cdot 25-52}{6 \cdot 27} \ol qm \geq \frac35 \ol qm
            \eqfs
        \end{equation}
        \item
        If $q \in \ballclosed(m, R)$ and $\delta \in \Rpp$, then \cref{thm:median:vi} implies
        \begin{align}
            &\EOf{\ol Yq - \ol Ym} 
            \\&\geq 
            \frac12 w^2  \,\ol qm^2\, \EOf{(\ol Ym + \ol qm)^{-1} \indOf{\bowtiesetcompl(m, q, w)}(Y) \ind_{[0,\delta]}(\ol Ym)}
            \\&\geq 
            \frac12 w^2 \,\ol qm^2\, \br{\delta+ \ol qm}^{-1} \PrOf{\cb{Y \in \bowtiesetcompl(m, q, w)}\cap \cb{\ol Ym \leq \delta}}
            \\&\geq 
            \inf_{p \in \ballclosed(m, R)}\frac12 w^2 \,\ol qm^2\, \br{\delta+ \ol qm}^{-1} \PrOf{\cb{Y \in \bowtiesetcompl(m, p, w)}\cap \cb{\ol Ym \leq \delta}}
            \\&\geq 
            \frac12 w^2 \,\ol qm^2\, \br{\delta+ \ol qm}^{-1} \br{\inf_{p \in \ballclosed(m, R)} \PrOf{Y \in \bowtiesetcompl(m, p, w)} + \PrOf{\ol Ym \leq \delta} -1}
            \eqfs
        \end{align}
        Choosing $\delta = \tilde\chi$ and using the definition of $\rho$ and $\tilde\chi$, we obtain, for all  $q \in \ballclosed(m, R)$,
        \begin{equation}
            \EOf{\ol Yq - \ol Ym} 
            \geq 
            \frac14\rho w^2 \,\ol qm^2\, \br{\tilde\chi+ \ol qm}^{-1}
            \geq 
            \frac{\rho w^2}{4(\tilde\chi+ R)}  \,\ol qm^2
            \eqfs
            \qedhere
        \end{equation}
    \end{enumerate}
\end{proof}
For the proof of \cref{thm:median:main}, we show an upper bound on the double excess risk,
\begin{equation}
    V_n := \EOf{\ol Y{m_n} - \ol Ym} + \EOf{\frac1n \sum_{i=1}^n \br{\ol {Y_i}{m} - \ol {Y_i}{m_n}} }
    \eqfs
\end{equation}
Recall that $Y, Y_1, \dots, Y_n$ are iid, so that $Y$ is independent of $m_n$.
\begin{proposition}\label{prp:median:Vn}
    Let $\tran(x)=x$ so that $m$ is a Fr\'echet median. 
    Assume $\xi \in \Rpp$ exists with $\Eof{\ol Ym^\xi} < \infty$.
    Let $r\in\Rpp$ such that $\PrOf{\ol Ym > r} < \frac1{27}$.
    Assume there are $\ell \in \N$ and $w\in(0,1]$ such that
    \begin{equation}
        \PrOf{\exists q,p \in \ballclosed(m, 6r) \colon Y_1, \dots, Y_\ell \in \bowtieset(q, p, w) \cup \ballclosed(m, 6r)\compl} < 1
        \eqfs
    \end{equation}
    Then there are $n_0 \in \N$ and $C \in \Rpp$, such that $V_n \leq C n^{-1}$ for all $n \geq n_0$.
\end{proposition}
\begin{proof}
    \cref{lmm:variancebound} shows
    \begin{equation}
        V_n \leq \frac1n \sum_{i=1}^n \EOf{\ol {m_n}{m_n^i}}
        \eqfs
    \end{equation}
    Set
    \begin{equation}
        H_i := \frac1n \sum_{j=1}^n (\ol {Y_j}{m_n} + \ol {m_n}{m_n^{i}})^{-1} \indOf{\bowtiesetcompl(m_n, m_n^i, w)}(Y_j)
        \eqfs
    \end{equation}
    Following the proof of \cref{lmm:mnmnibound} while using the variance inequality \cref{thm:median:vi} instead of \cref{thm:infdtr:vi} and omitting the positive term $\ol {Y_j^i}{m_n^i} + \ol {m_n}{m_n^{i}}$ implies
    \begin{equation}\label{eq:medi:mnmni:bound}
        \ol {m_n}{m_n^{i}} \leq \frac {4}{w^2 n} H_i^{-1}
        \eqfs
    \end{equation}
    Set $R := 6r$.
    Conditional on $m_n, m_n^i \in \ballclosed(m, R)$, we have
    \begin{align}
        H_i^{-1} 
        &\leq
        \sup_{q,p\in \ballclosed(m, R)}\br{\frac1n \sum_{j=1}^n (\ol {Y_j}{m_n} + \ol {m_n}{m_n^{i}})^{-1}\indOf{\bowtiesetcompl(q, p, w)}(Y_j)}^{-1}
        \\&\leq
        \sup_{q,p\in \ballclosed(m, R)}\br{\frac1n \sum_{j=1}^n (4R)^{-1}\indOf{\bowtiesetcompl(q, p, w)}(Y_j)\indOf{[0,R]}(\ol {Y_j}m)}^{-1}
        \\&\leq
        4 R \br{\inf_{q,p\in \ballclosed(m, R)} \frac1n \sum_{j=1}^n \indOf{\bowtiesetcompl(q, p, w)}(Y_j)\indOf{[0,R]}(\ol {Y_j}m)}^{-1}
        \eqfs
    \end{align}
    Let $\eta \in (0,1]$ to be specified later. Define the events
    \begin{align}
        B &:= \cb{\ol m{m_n} \leq R}\eqcm
        &
        \Gamma &:= \cb{\inf_{q,p\in \ballclosed(m, R)} \frac1n\sum_{j=1}^n \indOf{\bowtiesetcompl(q, p, w)}(Y_j)\indOf{[0,R]}(\ol {Y_j}m) \geq \eta}\eqcm
        \\
        B^i &:= \cb{\ol m{m_n^i} \leq R}\eqcm
        & 
        \Gamma^i &:= \cb{\inf_{q,p\in \ballclosed(m, R)} \frac1n\sum_{j=1}^n \indOf{\bowtiesetcompl(q, p, w)}(Y_j^i)\indOf{[0,R]}(\ol {Y_j^i}m) \geq \eta}
        \eqfs
    \end{align}
    Finally, denote the intersection of these events as
    \begin{equation}
        \Omega^i := B \cap B^i \cap \Gamma \cap \Gamma^i \eqfs
    \end{equation}
    On $\Omega^i$, we have
    \begin{equation}\label{eq:medi:Hiinv:bound}
        H_i^{-1} \leq
        \frac {4R}{\eta}
        \eqfs
    \end{equation}
    We split $V_n$ on $\Omega^i$ as follows
    \begin{equation}
        V_n
        \leq 
        \frac 1{n}\sum_{i=1}^n \EOf{\ol{m_n}{m_n^{i}}}
        =
        \frac 1{n}\sum_{i=1}^n \EOf{\ol{m_n}{m_n^{i}} \ind_{\Omega^i}}
        +
        \frac 1{n}\sum_{i=1}^n \EOf{\ol{m_n}{m_n^{i}} \ind_{(\Omega^i)\compl}}
        \eqfs
    \end{equation}
    For the first term, \eqref{eq:medi:mnmni:bound} and \eqref{eq:medi:Hiinv:bound} imply
    \begin{equation}
        \EOf{\ol{m_n}{m_n^{i}} \ind_{\Omega^i}}
        \leq 
        \frac{4}{w^2 n}\EOf{H_i^{-1} \ind_{\Omega^i}}
        \leq 
        \frac{16R}{w^2 \eta n}
        \eqfs
    \end{equation}
    Hence, 
    \begin{equation}
        \frac 1{n}\sum_{i=1}^n \EOf{\ol{m_n}{m_n^{i}} \ind_{\Omega^i}}
        \leq
        \frac{16R}{w^2 \eta n}
        \eqfs
    \end{equation}
    For the second term, we use the triangle inequality and Cauchy--Schwarz to obtain
    \begin{align}
        \EOf{\ol{m_n}{m_n^{i}} \ind_{(\Omega^i)\compl}}
        &\leq 
        \EOf{\ol{m}{m_n} \ind_{(\Omega^i)\compl}}
        +
        \EOf{\ol{m}{m_n^{i}} \ind_{(\Omega^i)\compl}}
        \\&=
        2 \EOf{\ol{m}{m_n} \ind_{(\Omega^i)\compl}}
        \\&\leq 
        2 \br{\EOf{\ol{m}{m_n}^2} \PrOf{(\Omega^i)\compl}}^{\frac12}
        \eqfs
    \end{align}
    To finish the proof, we show that $\Eof{\ol{m}{m_n}^2}$ can be bounded by a constant $\tilde C\in\Rpp$ (\cref{lmm:posddrtran:expect}) and the probability decreases exponentially in $n$, i.e., $\Prof{(\Omega^i)\compl} \leq \exp(-cn)$ with $c\in\Rpp$.
    
    We say that an event $\mc E_n$ depending on $n\in\N$ happens \emph{with high probability} or for short \emph{whp}, if there are $n_0 \in\N$ and $c \in \Rpp$ not depending on $n$ such that $\Prof{\mc E_n\compl} \leq \exp(-c n)$ for all $n \geq n_0$. Note that if $\mc E_n, \tilde{\mc E}_n$ happen whp, then $\mc E_n \cap\tilde{\mc E}_n$ whp.
    
    \cref{cor:median:tailbound} implies
    \begin{equation}
        \PrOf{\ol m{m_n} > R} \leq \br{2\Prof{\ol Ym > r}^{\frac13}}^n \leq \br{\frac23}^n
        \eqfs
    \end{equation}
    Thus, $B$ and $B^i$ whp.
    \cref{lmm:medi:Gamma} below shows that we can choose $\eta\in\Rpp$ so that $\Gamma$, $\Gamma^i$ whp. Thus, $\Omega^i$ whp.
    Hence, there are $n_0, C, c \in \Rpp$ such that for all $n \geq n_0$, we have
    \begin{equation}
        V_n
        \leq 
        \frac{16R}{w^2 \eta n}
        +
        2 \br{\tilde C \exp(-c n)}^{\frac12}
        \leq 
        C n^{-1}
        \eqfs
        \qedhere
    \end{equation}
\end{proof}
\begin{lemma}\label{lmm:medi:Gamma}
    Let $R\in\Rpp$.
    Assume there are $\ell \in \N$ and $w\in[0,1]$ such that
    \begin{equation}\label{eq:lmm:medi:Gamma:assu}
        \PrOf{\exists q,p \in \ballclosed(m, R) \colon Y_1, \dots, Y_\ell \in \bowtieset(q, p, w) \cup \ballclosed(m, R)\compl} < 1
        \eqfs
    \end{equation}
    For $\eta\in[0,1]$ and $n\in\N$, define
    \begin{equation}
        \Gamma_{n,\eta} := \cb{\inf_{q,p\in\ballclosed(m, R)} \frac1n\sum_{j=1}^n \indOf{\bowtiesetcompl(q, p, w)}(Y_j)\indOf{[0, R]}(\ol {Y_j}m) \geq \eta}
        \eqfs
    \end{equation}
    Then, there are $n_0 \in \N$, $\eta\in(0,1]$, and $c \in \Rpp$ such that
    \begin{equation}
        \PrOf{\Gamma_{n,\eta}\compl} 
        \leq 
        \exp(-c n)
    \end{equation}
    for all $n \geq n_0$.
\end{lemma}
\begin{proof}
    Because of \eqref{eq:lmm:medi:Gamma:assu}, we can choose $\tilde \ell \in\N$ large enough so that
    \begin{align}
        &\PrOf{\exists q,p \in \ballclosed(m, R) \colon Y_1, \dots, Y_{\tilde\ell} \in \bowtieset(q, p, w) \cup \ballclosed(m, R)\compl}
        \\&\leq
        \PrOf{\exists q,p \in \ballclosed(m, R) \colon Y_1, \dots, Y_{\ell} \in \bowtieset(q, p, w) \cup \ballclosed(m, R)\compl}^{\tilde\ell/\ell-1}
        \\&\leq
        \frac1{16}
        \eqfs
    \end{align}
    Set $n_0 := 6 \tilde\ell^2$. Let $n \geq n_0$.
    Let $K\in\N$ be the largest integer so that $n \geq K\tilde\ell$.
    For $k\in\{1,\dots,K\}$, define the events
    \begin{equation}
        G_k := \cb{\exists  q,p\in\ballclosed(m, R) \colon Y_{\tilde\ell(k-1)+1}, \dots,  Y_{\tilde\ell k} \in \bowtieset(q, p, w)  \cup \ballclosed(m, R)\compl}
        \eqfs
    \end{equation}
    Set
    \begin{equation}
        N := \left\lceil\frac{1}{2} K + (\tilde \ell-1)(K+1) \right\rceil
        \eqfs
    \end{equation}
    For $q,p\in\mc Q$, the event
    \begin{equation}
        \sum_{j=1}^n \indOf{\bowtieset(q, p, w)\cup \ballclosed(m, R)\compl}(Y_j) \geq N 
    \end{equation}
    implies that at least $N - (\tilde\ell-1) K - (n - K\tilde\ell)$ of the events $G_k$, $k = 1,\dots, K$ must occur.
    Set $\eta := \frac1{3 \tilde\ell}$. Then $1-\eta \geq N/n$ as
    \begin{equation}
        \frac{N}{n}
        \leq
        \frac{(\tilde \ell-\frac12) K + \tilde\ell}{K \tilde\ell}
        = 
        \frac{\tilde\ell-\frac12}{\tilde\ell} + \frac1K
        \leq
        1-\eta
        \eqcm
    \end{equation}
    since $n\geq n_0$ implies $K \geq 6\tilde\ell$.
    Hence,
    \begin{align}
        \PrOf{\Gamma_{n,\eta}\compl}
        &=
        \PrOf{\cb{\forall q,p \in\ballclosed(m, R) \colon \frac1n\sum_{j=1}^n \indOf{\bowtiesetcompl(q, p, w)}(Y_j)\indOf{[0, R]}(\ol {Y_j}m) \geq \eta}\compl}
        \\&=
        \PrOf{\exists q,p \in\ballclosed(m, R) \colon \frac1n\sum_{j=1}^n \indOf{\bowtiesetcompl(q, p, w)}(Y_j)\indOf{[0, R]}(\ol {Y_j}m) < \eta}
        \\&=
        \PrOf{\exists q,p \in\ballclosed(m, R) \colon \frac1n\sum_{j=1}^n \indOf{\bowtieset(q, p, w)\cup\ballclosed(m, R)\compl}(Y_j) > 1-\eta}
        \\&\leq  
        \PrOf{\exists  q,p\in\ballclosed(m, R) \colon\sum_{j=1}^n \indOf{\bowtieset(q, p, w)\cup \ballclosed(m, R)\compl}(Y_j) \geq N}
        \\&\leq 
        \PrOf{\sum_{k=1}^{K} \ind_{G_k} \geq N - (\tilde\ell-1)(K+1)}
        \\&\leq
        \PrOf{\frac{1}{K}\sum_{k=1}^{K} \ind_{G_k} \geq \frac12}
        \eqcm
    \end{align}
    as
    \begin{equation}
        N - (\tilde\ell-1) K - (n - K\tilde\ell) \geq N - (\tilde\ell-1)(K+1) 
        \geq
        \frac{1}{2} K
        \eqfs
    \end{equation}
    As the events $G_k$ are iid with $\rho_G := \PrOf{G_1} \leq \frac1{16}$, 
    \cref{prp:chernoff:addi} together with the identity 
    $\kullback\brOf{\frac12, \rho} = \frac12 \log\brOf{\frac{1}{4\rho(1-\rho)}}$
    and the monotonicity of $\rho \mapsto \rho(1-\rho)$ on $\br{0,\frac12}$ yield
    \begin{equation}
        \PrOf{\frac{1}{K}\sum_{k=1}^{K} \ind_{G_k} \geq \frac12}
        \leq
        \exp\brOf{-K \kullback\brOf{\tfrac12, \rho_G}}
        =
        \br{2\sqrt{\rho_G(1-\rho_G)}}^{K}
        \leq
        \br{\frac{\sqrt{15}}{8}}^{K}
        \leq
        \br{\frac12}^{K}
        \eqfs
    \end{equation}
    We arrive at
    \begin{equation}
        \PrOf{\Gamma_{n,\eta}\compl} \leq 2^{-K} \leq \exp\brOf{-\log(2) \br{\frac{n}{\tilde \ell} - 1}} \leq \exp(-cn)
    \end{equation}
    with $c = \frac12\log(2)/\tilde\ell$ as $n \geq 2 \tilde \ell$.
\end{proof}
\begin{proof}[Proof of \cref{thm:median:main}]
    By the minimizing property of $m_n$ and \cref{prp:median:Vn}, there are $n_0 \in \N$ and $C \in \Rpp$, such that
    \begin{equation}
        \EOf{\ol Y{m_n}-\ol Ym}
        \leq 
        V_n 
        \leq
        Cn^{-1}
    \end{equation}
    for all $n \geq n_0$.
    With $R, w$ as given in the theorem, set
    \begin{equation}
        \rho := \inf_{p \in \ballclosed(m, R)} \PrOf{Y \in \bowtiesetcompl(m, p, w)} > 0
        \eqfs
    \end{equation}
    Let $\tilde\chi \in\Rpp$ such that $\PrOf{\ol Ym \leq \tilde\chi} \geq 1 - \frac12\rho$. Then, \cref{lmm:median:vi:specific} yields, for all $q\in\mc Q$,
    \begin{align}
        \EOf{\ol Y{q}- \ol Ym}
        &= 
        \EOf{\ol Y{q}- \ol Ym}\ind_{[0, R]}(\ol qm)
        +
        \EOf{\ol Y{q}- \ol Ym}\ind_{(R, \infty)}(\ol qm)
        \\&\geq 
        \frac{\rho w^2}{4(\tilde\chi + R)}  \ol qm^2\ind_{[0, R]}(\ol qm)
        +
        \frac35 \ol qm\ind_{(R, \infty)}(\ol qm)
        \eqfs
    \end{align}
    Thus,
    \begin{align}
        \EOf{\ol Y{m_n}- \ol Ym} 
        &\geq
        \frac{\rho w^2}{4(\tilde\chi + R)}  \EOf{\ol m{m_n}^2\ind_{[0, R]}(\ol m{m_n})}
        +
        \frac35 \EOf{\ol m{m_n}\ind_{(R, \infty)}(\ol m{m_n})}
        \\&\geq
        \min\brOf{\frac35, \frac{\rho w^2}{4(\tilde\chi + R)}} \EOf{\min\brOf{\ol m{m_n}^2, \ol m{m_n}}}
        \eqfs
    \end{align}
    We obtain
    \begin{equation}
        \EOf{\min\brOf{\ol m{m_n}^2, \ol m{m_n}}}
        \leq 
        C \max\brOf{\frac53, \frac{4(\tilde\chi + R)}{\rho w^2}} \frac{1}{n}
        \eqfs
    \end{equation}
    Moreover, \cref{lmm:posddrtran:expect} shows that there are $r_0, C_2\in\Rpp$ depending only on $\xi$, and the distribution of $Y$, such that
    \begin{equation}
        \EOf{\ol m{m_n}^2 \indOfOf{(r_0, \infty)}{\ol m{m_n}}} = \mo O(\exp(-C_2 n))
        \eqfs
    \end{equation}
    Thus, as
    \begin{equation}
        \EOf{\ol m{m_n}^2 \indOfOf{[0, r_0]}{\ol m{m_n}}}
        \leq 
        \max(1, r_0) \EOf{\min\brOf{\ol m{m_n}^2, \ol m{m_n}}}
    \end{equation}
    we obtain
    \begin{equation}
        \EOf{\ol m{m_n}^2} = \mo O\brOf{\frac1n}
        \eqfs
    \end{equation}
\end{proof}
\subsection{Distributions That Are Not Concentrated on a Geodesic}\label{ssec:proof:median:geodesic}
The two bow tie conditions \eqref{eq:thm:median:main:bowtie:pop} and \eqref{eq:thm:median:main:bowtie:sample} of \cref{thm:median:main} couple the distribution of $Y$ to the geometry of $\mc Q$. Here we prove \cref{cor:median:notonageodesic}, which decouples the two: the distribution only has to fulfill the simple requirement of not being concentrated on a geodesic, whereas the geometric part is taken care of by the assumption that $\mc Q$ is a Hilbert space or a finite dimensional Hadamard manifold. The geometric property we actually use is that bow ties with knots in a bounded set are \emph{thin}, i.e., they contain no three points that are far from being collinear; see \cref{def:thinbowtie} and \cref{lmm:thinbowtie}.

Throughout, we identify a geodesic with its image whenever convenient. Recall the standing assumption that $\Prof{Y\in\mc Y} = 1$ for a separable set $\mc Y\subset\mc Q$. Hence, the \emph{support}
\begin{equation}
    \suppof Y := \setByEle{q\in\mc Q}{\forall t \in \Rpp\colon \PrOf{Y \in \ballopen(q, t)} > 0}
\end{equation}
is a closed set with $\Prof{Y\in\suppof Y} = 1$: its complement is a union of open balls of probability zero, and, as $\overline{\mc Y}$ is separable and hence Lindel\"of, countably many of them already cover $\overline{\mc Y}\setminus\suppof Y$.
\begin{notation}\label{nota:coll}
    For $y_1,y_2,y_3\in\mc Q$, set
    \begin{equation}
        \collOf{y_1,y_2,y_3}
        :=
        \min\brOf{d\brOf{y_1, \geodft{y_2}{y_3}},\ d\brOf{y_2, \geodft{y_1}{y_3}},\ d\brOf{y_3, \geodft{y_1}{y_2}}}
        \eqfs
    \end{equation}
    We call $y_1,y_2,y_3$ \emph{collinear} if and only if $\collOf{y_1,y_2,y_3} = 0$, i.e., if one of the three points lies on the geodesic connecting the other two. A set $A\subset\mc Q$ is called \emph{collinear} if and only if every triple of points in $A$ is collinear.
\end{notation}
\begin{lemma}\label{lmm:coll:lipschitz}
    Let $y_1,y_2,y_3,z_1,z_2,z_3\in\mc Q$ and $s\in\Rp$ with $\ol{y_i}{z_i}\leq s$ for $i\in\cb{1,2,3}$. Then
    \begin{equation}
        \absOf{\collOf{z_1,z_2,z_3} - \collOf{y_1,y_2,y_3}} \leq 2s
        \eqfs
    \end{equation}
    In particular, $\ms{coll}\colon\mc Q^3\to\Rp$ is continuous.
\end{lemma}
\begin{proof}
    As $\mc Q$ is Hadamard, the metric is convex. Thus, for the geodesics $\geodft{y_2}{y_3}$ and $\geodft{z_2}{z_3}$, parametrized affinely on $[0,1]$, we obtain
    \begin{equation}
        \ol{\geodft{y_2}{y_3}(t)}{\geodft{z_2}{z_3}(t)}
        \leq
        (1-t)\,\ol{y_2}{z_2} + t\,\ol{y_3}{z_3}
        \leq
        s
    \end{equation}
    for all $t\in[0,1]$, so the Hausdorff distance between these two geodesics is at most $s$. Hence,
    \begin{equation}
        \absOf{d\brOf{z_1,\geodft{z_2}{z_3}} - d\brOf{y_1,\geodft{y_2}{y_3}}}
        \leq
        \ol{y_1}{z_1} + s
        \leq
        2s
        \eqcm
    \end{equation}
    and the same bound holds for the two remaining terms in \cref{nota:coll}. As the minimum of finitely many functions inherits this bound, the claim follows.
\end{proof}
\begin{definition}[thin bow ties]\label{def:thinbowtie}
    Let $\mc B\subset\mc Q$. We say that \emph{bow ties are thin on $\mc B$} if and only if, for every $a\in\Rpp$, there is $w\in(0,1)$ such that
    \begin{equation}
        \collOf{z_1,z_2,z_3} \leq a
    \end{equation}
    for all $q,p\in\mc B$ and all $z_1,z_2,z_3\in\bowtieset(q,p,w)\cap\mc B$.
\end{definition}
\begin{lemma}\label{lmm:thinbowtie}
    Let $\mc B\subset\mc Q$ be bounded.
    \begin{enumerate}[label=(\roman*)]
        \item\label{lmm:thinbowtie:hilbert}
        If $\mc Q$ is a Hilbert space, then bow ties are thin on $\mc B$.
        \item\label{lmm:thinbowtie:manifold}
        If $\mc Q$ is a Hadamard manifold, then bow ties are thin on $\mc B$.
    \end{enumerate}
\end{lemma}
\begin{proof}
    Let $a\in\Rpp$ and $w\in(0,1)$. As $\bowtieset(q,q,w) = \cb q$, all triples of points in $\bowtieset(q,q,w)$ are collinear. Hence, we may restrict to knots $q,p\in\mc B$ with $q\neq p$ in both parts.
    \begin{enumerate}[label=(\roman*)]
        \item
        Let $q,p\in\mc B$ with $q\neq p$. Let $\gamma\colon\R\to\mc Q$ be the unit-speed geodesic extension of $\geodft qp$ to $\R$. Set $u := \frac{p-q}{\normof{p-q}}$. Then the geodesic image is $\gamma(\R) = \setByEleInText{q+tu}{t\in\R}$. For $y \neq q$, we have $\ol y{\gamma}\rd(0) = -\ipof{u}{\frac{y-q}{\normof{y-q}}}$ and thus
        \begin{equation}
            \ol y{\gamma}\rd(0)^2 \geq 1-w^2
            \ \ \Rightarrow\ \ 
            d(y,\gamma(\R))^2 = \normOf{y-q}^2 - \ipOf{u}{y-q}^2 \leq w^2 \normOf{y-q}^2
            \eqfs
        \end{equation}
        Analogously, $\ol y{\gamma}\ld(\ol qp) = -\ipof{u}{\frac{y-p}{\normof{y-p}}}$ for $y \neq p$, which yields $d(y,\gamma(\R))\leq w\normof{y-p}$, with the same line $\gamma(\R)$, as $p\in \gamma(\R)$. Furthermore, $d(y,\gamma(\R)) = 0$ for $y\in\cb{q,p}$. Hence,
        \begin{equation}\label{eq:thinbowtie:tube}
            d(y,\gamma(\R)) \leq w \,\diam(\mc B)
            \qquad\text{for all } y\in\bowtieset(q,p,w)\cap\mc B
            \eqfs
        \end{equation}
        Let $z_1,z_2,z_3\in\bowtieset(q,p,w)\cap\mc B$ and let $\hat z_i$ denote the orthogonal projection of $z_i$ onto $\gamma(\R)$, so that $\normof{z_i - \hat z_i}\leq w \,\diam(\mc B)$ by \eqref{eq:thinbowtie:tube}. As $\hat z_1,\hat z_2,\hat z_3$ lie on a line, one of them, say $\hat z_j$, lies between the other two, $\hat z_j\in\geodft{\hat z_i}{\hat z_k}$. By convexity of the metric, the Hausdorff distance between $\geodft{\hat z_i}{\hat z_k}$ and $\geodft{z_i}{z_k}$ is at most $w \,\diam(\mc B)$. Thus,
        \begin{equation}
            \collOf{z_1,z_2,z_3}
            \leq
            d\brOf{z_j,\geodft{z_i}{z_k}}
            \leq
            \normOf{z_j - \hat z_j} + w \,\diam(\mc B)
            \leq
            2w \,\diam(\mc B)
            \eqfs
        \end{equation}
        Hence, for $a\in\Rpp$, we can set $w = \min\brOf{\frac a{2 \,\diam(\mc B)},\frac12}$ to show thinness of bow ties according to \cref{def:thinbowtie}.
        \item
        We may assume that $\mc B$ is closed and hence, by the theorem of Hopf and Rinow, compact. For $q\in\mc Q$, denote the logarithmic map as $\exp_q^{-1}\colon \mc Q\to T_q\mc Q$. As $\exp_q$ is a diffeomorphism for every $q$ and $\mc Q$ has no cut locus, the map $(q,y)\mapsto \exp_q^{-1}(y)$ is continuous, and the first variation formula gives, for every unit-speed geodesic $\gamma$ and every $t_0$ in its domain,
        \begin{equation}\label{eq:thinbowtie:firstvar}
            \brOf{\frac{\dl}{\dl t}\, \ol{y}{\gamma(t)}\bigg|_{t = t_0}}^2 \geq 1 - w^2
            \ \ \Leftrightarrow\ \ 
            \ipOf{\gamma\pr(t_0)}{\exp^{-1}_{\gamma(t_0)}(y)}^2 \geq (1-w^2) \normOf{\exp^{-1}_{\gamma(t_0)}(y)}^2
            \eqcm
        \end{equation}
        where the right-hand side is also the correct reading for $y = \gamma(t_0)$, in which case both sides vanish and the one-sided derivatives are $\pm1$.
        
        Assume bow ties are not thin on $\mc B$. Then there are $a\in\Rpp$, a sequence $(w_k)\subset(0,1)$ with $w_k\searrow0$, knots $q_k,p_k\in\mc B$, and points $z_{k,1},z_{k,2},z_{k,3}\in\bowtieset(q_k,p_k,w_k)\cap\mc B$ with
        \begin{equation}\label{eq:thinbowtie:contra}
            \collOf{z_{k,1},z_{k,2},z_{k,3}} > a
            \eqfs
        \end{equation}
        In particular, $q_k \neq p_k$. Set $\gamma_k := \geodft{q_k}{p_k}$. As $\mc B$ is compact, so is the unit sphere bundle over $\mc B$. Hence, after passing to a subsequence, we may assume that $q_k \to q$, $p_k\to p$, $z_{k,i}\to z_i$, that
        \begin{equation}
            \gamma_k\pr(0) \to u \in T_q\mc Q
            \eqcm\qquad
            \gamma_k\pr(\ol{q_k}{p_k}) \to \tilde u \in T_p\mc Q
            \eqcm
        \end{equation}
        and that, for each $i\in\cb{1,2,3}$, the maximum in \cref{def:bowtie} is attained at the same knot for all $k$. Set $\gamma(t) := \exp_q(tu)$ and $\tilde\gamma(t) := \exp_p(t\tilde u)$.
        
        We claim $\gamma(\R) = \tilde\gamma(\R)$. If $q\neq p$, then $u$ and $\tilde u$ are the velocities of the geodesic $\gamma$ at its endpoints, so $\gamma$ and $\tilde\gamma$ are the unique complete geodesic through $q$ and $p$, up to reparametrization. If $q = p$, then $\ol{q_k}{p_k}\to0$ and, by continuity of the geodesic flow on $T\mc Q$, we obtain $\tilde u = \lim_k \gamma_k\pr(\ol{q_k}{p_k}) = \lim_k\gamma_k\pr(0) = u$, so $\gamma = \tilde\gamma$.
        
        Fix $i\in\cb{1,2,3}$ and assume that the maximum in \cref{def:bowtie} is attained at the knot $q_k$; the case of the knot $p_k$ is analogous with $\tilde\gamma$ in place of $\gamma$. By \eqref{eq:thinbowtie:firstvar},
        \begin{equation}
            \ipOf{\gamma_k\pr(0)}{\exp^{-1}_{q_k}(z_{k,i})}^2 \geq (1-w_k^2) \normOf{\exp^{-1}_{q_k}(z_{k,i})}^2
            \eqfs
        \end{equation}
        Letting $k\to\infty$ and using continuity of $(q,y)\mapsto \exp^{-1}_q(y)$ and of the inner product on $T\mc Q$, we obtain $\ipof{u}{\exp^{-1}_q(z_i)}^2 \geq \normof{\exp^{-1}_q(z_i)}^2$. As $\normof u = 1$, this is the equality case of the Cauchy--Schwarz inequality, so $\exp^{-1}_q(z_i) = \pm\normof{\exp^{-1}_q(z_i)} u$ and therefore $z_i \in\gamma(\R)$.
        
        Hence $z_1,z_2,z_3\in\gamma(\R)$, which implies $\collof{z_1,z_2,z_3} = 0$. On the other hand, \eqref{eq:thinbowtie:contra} and \cref{lmm:coll:lipschitz} yield $\collof{z_1,z_2,z_3}\geq a$, a contradiction.
        \qedhere
    \end{enumerate}
\end{proof}
\begin{lemma}\label{lmm:collinearset}
    Assume that geodesics in $\mc Q$ extend uniquely, i.e., every geodesic $\gamma\colon[a,b]\to\mc Q$ with $a<b$ is the restriction of exactly one geodesic $\R\to\mc Q$, up to reparametrization. This is the case if $\mc Q$ is a Hilbert space or a Hadamard manifold. Let $A\subset\mc Q$ be collinear. Then there is a geodesic $\gamma\colon\R\to\mc Q$ with $A\subset\gamma(\R)$.
\end{lemma}
\begin{proof}
    If $A$ contains at most one point, a constant map $\gamma\colon\R\to\mc Q$ with $A\subset\gamma(\R)$ is a geodesic in the sense of \cref{def:geodesic} and we are done. Otherwise, choose $y_1,y_2\in A$ with $y_1\neq y_2$ and let $\gamma\colon\R\to\mc Q$ be the unique geodesic with $\geodft{y_1}{y_2}\subset\gamma(\R)$. Let $y_3\in A$. As $A$ is collinear, one of the three points $y_1,y_2,y_3$ lies on the geodesic connecting the other two.
    
    If $y_3\in\geodft{y_1}{y_2}$, then $y_3\in\gamma(\R)$. If $y_1\in\geodft{y_3}{y_2}$, then $y_1$ and $y_2$ both lie on the geodesic $\geodft{y_3}{y_2}$, so, by uniqueness of geodesics in Hadamard spaces, $\geodft{y_1}{y_2}\subset\geodft{y_3}{y_2}$. Thus, the geodesic $\R\to\mc Q$ extending $\geodft{y_3}{y_2}$ also extends $\geodft{y_1}{y_2}$ and, by uniqueness of the extension, coincides with $\gamma$; in particular $y_3\in\gamma(\R)$. The case $y_2\in\geodft{y_1}{y_3}$ is analogous. Hence $A\subset\gamma(\R)$.
\end{proof}
\begin{proof}[Proof of \cref{cor:median:notonageodesic}]
    As $\PrOf{Y\in\suppof Y} = 1$, the assumption implies that $\suppof Y$ is not contained in the image of a geodesic $\R\to\mc Q$. Thus, \cref{lmm:collinearset} shows that $\suppof Y$ is not collinear, so there are $y_1,y_2,y_3\in\suppof Y$ with
    \begin{equation}
        a := \collOf{y_1,y_2,y_3} > 0
        \eqfs
    \end{equation}
    As $\lim_{r\to\infty}\PrOf{\ol Ym > r} = 0$, we may choose $r\in\Rpp$ with $\PrOf{\ol Ym>r}<\frac1{27}$ and $y_1,y_2,y_3\in\ballopen(m,6r)$. Set $R := 6r$ and $\mc B:=\ballclosed(m,R)$. Choose $s\in(0,\frac a4]$ small enough that $S_i := \ballclosed(y_i,s)\subset\ballopen(m,R)$ for $i\in\cb{1,2,3}$ and set
    \begin{equation}
        \pi_0 := \min_{i\in\cb{1,2,3}}\PrOf{Y\in S_i}
        \eqcm
    \end{equation}
    which is positive as $y_i\in\suppof Y$. By \cref{lmm:coll:lipschitz},
    \begin{equation}\label{eq:cor:median:noncoll}
        \collOf{z_1,z_2,z_3} \geq a - 2s \geq \frac a2
        \qquad\text{for all } z_i\in S_i,\ i\in\cb{1,2,3}
        \eqfs
    \end{equation}
    By \cref{lmm:thinbowtie}, bow ties are thin on the bounded set $\mc B$, so, applying \cref{def:thinbowtie} with $\frac a4$ in place of $a$, there is $w\in(0,1)$ such that $\collof{z_1,z_2,z_3} \leq \frac a4 < \frac a2$ for all $q,p\in\mc B$ and all $z_1,z_2,z_3\in\bowtieset(q,p,w)\cap\mc B$. As $S_i\subset\mc B$, this and \eqref{eq:cor:median:noncoll} imply
    \begin{equation}\label{eq:cor:median:onesetfree}
        \forall q,p\in\mc B\ \exists i\in\cb{1,2,3}\colon\quad S_i\cap\bowtieset(q,p,w) = \emptyset
        \eqcm
    \end{equation}
    as we could otherwise pick $z_i\in S_i\cap\bowtieset(q,p,w)$ for $i\in\cb{1,2,3}$ and obtain a contradiction.
    
    It remains to verify the conditions of \cref{thm:median:main} for this choice of $r$, $w$, and $\ell := 3$. Let $p\in\mc B$. Applying \eqref{eq:cor:median:onesetfree} to the knots $m,p\in\mc B$ yields an $i\in\cb{1,2,3}$ with
    \begin{equation}
        \PrOf{Y\in\bowtieset(m,p,w)} \leq 1 - \PrOf{Y\in S_i} \leq 1-\pi_0
        \eqcm
    \end{equation}
    so that \eqref{eq:thm:median:main:bowtie:pop} holds. Furthermore, set $E := \bigcap_{i=1}^3\cb{Y_i\in S_i}$, so that $\PrOf E \geq \pi_0^3 > 0$ by independence. On $E$, we obtain from \eqref{eq:cor:median:onesetfree}, for all $q,p\in\mc B$, an $i\in\cb{1,2,3}$ with $Y_i\in S_i\subset\ballopen(m,R)$ and $Y_i\notin\bowtieset(q,p,w)$. Hence, on $E$, the event in \eqref{eq:thm:median:main:bowtie:sample} does not occur, which yields
    \begin{equation}
        \PrOf{\exists q,p\in\ballclosed(m,R)\colon Y_1,Y_2,Y_3\in\bowtieset(q,p,w)\cup\ballclosed(m,R)\compl}
        \leq
        1 - \pi_0^3
        <
        1
        \eqfs
    \end{equation}
    Now \cref{thm:median:main} implies the claim.
\end{proof}

%% file: sec_proof_opti.tex
\section{Proofs of Section: Fast and Optimal Rates}\label{sec:proof:optimal}

\subsection{Fast Rates via Algorithm Stability}
\begin{proof}[Proof of \cref{thm:fast}]
    The proof of \cref{thm:power:main} shows that there is $\tilde C\in\Rpp$ such that
    \begin{equation}\label{eq:fast:excess}
        \EOf{\ol Y{m_n}^{\alpha}-\ol Ym^{\alpha}}
        \leq
        \tilde C n^{-1}
    \end{equation}
    for all $n\in \N$.
    By \cref{thm:infdtr:vi}, we have, for all $q\in\mc Q$,
    \begin{equation}
        \EOf{\ol Yq^\alpha-\ol Ym^\alpha} \geq \frac{\alpha(\alpha-1)}{2} \ol qm^2 \EOf{(\ol Ym + \ol qm)^{\alpha-2}}
        \eqfs
    \end{equation}
    Using \eqref{eq:fast:condition}, for $q\in \mc Q$ with $\ol qm \leq \epsilon$, we have
    \begin{align}
        \EOf{(\ol Ym + \ol qm)^{\alpha-2}} 
        &\geq 
        \EOf{(\ol Ym + \ol qm)^{\alpha-2} \ind_{[0, \ol qm]}(\ol Ym)} 
        \\&\geq 
        (2\, \ol qm)^{\alpha-2} \PrOf{\ol Ym \leq \ol qm}
        \\&\geq 
        2^{\alpha-2} b \, \ol qm^{\alpha-2} \ol qm^{\beta-\alpha}
        \eqfs
    \end{align}
    Hence,
    \begin{equation}\label{eq:fast:vi:fast}
        \EOf{\ol Yq^\alpha-\ol Ym^\alpha}
        \geq
        \alpha(\alpha-1) 2^{\alpha-3} b\, \ol qm^{\beta}
        \eqfs
    \end{equation}
    Let $\chi \in \median(\ol Ym)$.
    Then, for all $q\in \mc Q$, we have 
    \begin{align}
    	\EOf{(\ol Ym + \ol qm)^{\alpha-2}} 
    	\geq 
    	\frac12 (\chi + \ol qm)^{\alpha-2}
    	\geq 
    	2^{\alpha-3} \max(\chi, \ol qm)^{\alpha-2}
    	\eqfs
    \end{align}
    Hence,
    \begin{equation}\label{eq:fast:vi:std}
    	\EOf{\ol Yq^\alpha-\ol Ym^\alpha}
    	\geq
    	\alpha(\alpha-1) 2^{\alpha-4} \min(\chi^{\alpha-2} \ol qm^2, \ol qm^{\alpha})
    	\eqfs
    \end{equation}
    Combining \eqref{eq:fast:vi:fast} and \eqref{eq:fast:vi:std}, we can choose $c_0 \in \Rpp$ small enough so that
    \begin{equation} 
    	\EOf{\ol Yq^\alpha-\ol Ym^\alpha} \geq c_0 \min(\ol qm^{\beta}, \ol qm^{\alpha})
    \end{equation}
    for all $q\in\mc Q$.
    Applying this bound to \eqref{eq:fast:excess} yields
    \begin{equation}
    	\EOf{\min(\ol m{m_n}^{\beta}, \ol m{m_n}^{\alpha})} \leq \tilde C c_0^{-1} n^{-1}
    	\eqfs
		\qedhere
    \end{equation}
\end{proof}
\subsection{Fast Rates via Chaining and Peeling}
\begin{proof}[Proof of \cref{thm:fast:entropy}]
    As $\mc Q$ is Hadamard, \eqref{eq:entropy} together with the Hopf--Rinow theorem guarantees that $\mc Q$ has the Heine--Borel property (closed and bounded sets are compact). This allows us to use \cite[Corollary 2]{Park2026}, a strong law of large numbers for Fr\'echet means. In particular, for any $\delta>0$, $\PrOf{\ol m{m_n} > \delta} \xrightarrow{n\to\infty} 0$. Thus, we can focus on the event $\cb{\ol m{m_n} \leq \delta}$.
    
    \begin{table}
    \begin{tabular}{c|c}
        \cite{schoetz19} & here \\
        \hline
        $\alpha$ & $s/2$\\
        $\beta$ & $s/2$\\
        $\gamma$ & $\beta$\\
        $\zeta$ & $2$\\
        $\mc Q$ & $\ballclosed(m, \delta)$\\
        $\mc Y$ & $\mc Q$\\
        $\mf a(y,z)$ & $2 \dtran(\ol yz)$\\
        $\mf b(q,p)$ & $\ol qp$\\
        $\mf c(y,q)$ & $\tran(\ol yq) - \tran(\ol yo)$\\
        $\mf l(q,p)$ & $\ol qp$
    \end{tabular}
    \caption{Identification of symbols between \cite[Theorem 1]{schoetz19} and this section.}\label{tbl:ident}
    \end{table}
    
    We apply \cite[Theorem 1]{schoetz19} with the identification of variables as in  \cref{tbl:ident}. We now check the conditions of \cite[Theorem 1]{schoetz19} (in typewriter font): The moment condition implies \texttt{Existence}. \texttt{Weak Quadruple} is \cref{thm:infdtr:qi}. \texttt{Moment} is fulfilled due to our assumption $\Eof{\dtran(\ol Yq)^2} < \infty$. \texttt{Entropy} is assumed via \eqref{eq:entropy}. For \texttt{Growth}, we use the VI \cref{thm:infdtr:vi} and distinguish two cases as in the statement of the theorem: 
    \begin{enumerate}[label=(\roman*)]
        \item
        Case $\beta=2$.
        As $\Prof{\ol Ym < x_0} > 0$, there is $x_1<x_0$ such that $\PrOf{\ol Ym \leq x_1} > 0$. \cref{thm:infdtr:vi} implies
        \begin{equation}
            \EOf{\tran(\ol Yq) - \tran(\ol Ym)}
            \geq
            \frac12 \ol qm^2 \ddrtran(x_1 + \ol qm) \PrOf{\ol Ym \leq x_1}
            \eqfs
        \end{equation}
        We have $\Prof{\ol Ym \leq x_1}>0$ and $\ddrtran(x_1 + \ol qm)$ is bounded away from 0 for all $q$ with $\ol qm \leq (x_0-x_1)/2$.
        Furthermore, by \cref{thm:infdtr:vi}
        \begin{equation}
            r\mapsto \inf_{q\in\ballclosed(m,\delta)\setminus \ballclosed(m,r)} \frac{\EOf{\tran(\ol Yq)-\tran(\ol Ym)}}{r}
        \end{equation}
        is nondecreasing. Hence, we have
        \begin{equation}
            \liminf_{r\to0} \inf_{q\in\ballclosed(m,\delta)\setminus \ballclosed(m,r)} \frac{\EOf{\tran(\ol Yq)-\tran(\ol Ym)}}{r^2} > 0
            \eqfs
        \end{equation}
        \item 
        Case $\beta\in(1,2)$.
        \cref{thm:infdtr:vi} implies
        \begin{equation}
            \EOf{\tran(\ol Yq) - \tran(\ol Ym)}
            \geq
            \frac12 \ol qm^2 \ddrtran(2\ol qm) \PrOf{\ol Ym \leq \ol qm}
            \eqfs
        \end{equation}
        Hence, \eqref{eq:smallball} implies
        \begin{equation}
            \liminf_{r\to0} \inf_{q\in\ballclosed(m,\delta)\setminus \ballclosed(m,r)} \frac{\EOf{\tran(\ol Yq)-\tran(\ol Ym)}}{r^\beta} > 0
            \eqfs
        \end{equation}
    \end{enumerate}    
    Now all conditions of \cite[Theorem 1]{schoetz19} are validated and we obtain
    \begin{equation}
        \PrOf{\cb{n^{\frac{1}{2(\beta-1)}}\ol {m}{m_n} \geq t} \cap \cb{\ol m{m_n} \leq \delta}} \leq c \Eof{\dtran(\ol Ym)^2} t^{-2(\beta-1)}
        \eqfs
    \end{equation}
    Together with the result from the law of large numbers as mentioned at the beginning of the proof, we have
    \begin{equation}
        \ol {m}{m_n} = \Op\brOf{n^{-\frac{1}{2(\beta-1)}}}
        \eqfs
    \end{equation}
\end{proof}
\subsection{Asymptotic Distribution on the Real Line}
\begin{proof}[Proof of \cref{thm:distri:tran}]
    For $t\in\R$, write
    \begin{equation}
        \rho(t)
        :=
        \sign(t)\dtran(|t|)
        \eqfs
    \end{equation}
    Since $\dtran(0)=0$, the map
    $q\mapsto \tran(|x-q|)$ is continuously differentiable for every fixed
    $x\in\R$, with derivative $\rho(q-x)$.
    Moreover, by the subadditivity of $\dtran$ \cref{lmm:dtran:subadd}
    and \cref{lmm:tran:diff},
    \begin{equation}
        \absOf{\tran(|X-q|)-\tran(|X-m|)}
        \leq
        |q-m| \br{\dtran(|X-m|)+\dtran(|q-m|)}
        \eqcm
    \end{equation}
    which is uniformly integrable for $q$ in bounded sets, since
    $\Eof{\dtran(|X-m|)} < \infty$. Hence,
    \begin{equation}
        F(q)
        :=
        \Eof{\tran(|X-q|)-\tran(|X-m|)}
    \end{equation}
    is finite and differentiable, with
    \begin{equation}
        F\pr(q)
        =
        \Eof{\rho(q-X)}
        =
        \Eof{\sign(q-X)\dtran(|X-q|)}
        \eqfs
    \end{equation}
    Since $m$ minimizes $F$, we have
    \begin{equation}\label{eq:Fprime:m:zero}
        F\pr(m)=0.
    \end{equation}
    We first establish the local second order expansion of $F$. 
    Since $\dtran$ is concave, it is locally absolutely continuous on $\Rpp$.
    Consequently, $\rho$ is locally absolutely continuous on $\R$, and
    \begin{equation}
        \rho\pr(t)
        =
        \ddrtran(|t|)
    \end{equation}
    for Lebesgue-a.e. $t\in\R$. Therefore, for every $s\in\R$,
    \begin{align}
        F\pr(m+s)-F\pr(m)
        &=
        \EOf{\rho(m+s-X)-\rho(m-X)}
        \\
        &=
        \EOf{
            \int_0^s
            \ddrtran(|X-m-t|)
            \,\dl t
        }
        \eqfs
    \end{align}
    The last expression is understood as an oriented integral if $s<0$.
    Since the integrand is nonnegative, Tonelli's theorem gives
    \begin{equation}\label{eq:Fprime}
        F\pr(m+s)-F\pr(m)
        =
        \int_0^s
        \Eof{\ddrtran(|X-m-t|)}
        \,\dl t
        \eqfs
    \end{equation}
    By assumption \eqref{eq:thm:distri:tran:ddrtran}, and using \eqref{eq:Fprime:m:zero} as well as \eqref{eq:Fprime}, we obtain 
    \begin{equation}
        \frac{F'(m+s)}{s}
        =
        \frac1s\int_0^s \Eof{\ddrtran(|X-m-t|)} \,\dl t
        \xrightarrow{s \to 0}
        \sigma_{\ddrtran}
        \eqfs
    \end{equation}
    Hence,
    \begin{align}
        F(q)
        &=
        F(q)-F(m)
        \\&=
        \int_m^q F'(r)\,\dl r
        \\&=
        \int_0^{q-m} F'(m+s) \,\dl s
        \\&=
        \frac{\sigma_{\ddrtran}}2(q-m)^2
        +
        \mo o\brOf{(q-m)^2}
        \qquad
        \text{as }q\to m.
        \label{eq:F:expansion}
    \end{align}    
    By the definition of $x_0$ and the assumption $\Prof{|X-m|<x_0}>0$, we obtain below that $\sigma_{\ddrtran}>0$. Hence, the quadratic term in the expansion of $F(q)$ is dominating. 
    Indeed, choose $x_1<x_0$ such that $\Prof{|X-m|\le x_1}>0$. Since $\ddrtran$ is nonincreasing and $\ddrtran(x)>0$ for $x<x_0$, we have
    \begin{equation}
        \sigma_{\ddrtran}
        =
        \Eof{\ddrtran(|X-m|)}
        \geq
        \ddrtran(x_1)\Prof{|X-m|\le x_1}
        >
        0.
    \end{equation}
    It remains to transfer the expansion to the empirical minimizer. Define
    \begin{equation}
        \mathbb G_n [f]
        :=
        \frac1{\sqrt n}\sum_{i=1}^n
        \br{f(X_i)-\Eof{f(X)}}
        \eqfs
    \end{equation}
    For $q\in\R$, set
    \begin{equation}
        \psi_q(x) := \rho(q-x) = \sign(q-x)\dtran(|x-q|)
        \eqfs
    \end{equation}
    Since $F\pr(m)=0$, we have $\Eof{\psi_m(X)}=0$. Moreover,
    \begin{equation}
        \EOf{\psi_m(X)^2}
        =
        \EOf{\dtran(|X-m|)^2}
        =
        \sigma_{(\dtran)^2}
        <
        \infty
        \eqfs
    \end{equation}
    Hence, by the classical central limit theorem,
    \begin{equation}
        \mathbb G_n[\psi_m]
        \xrsquigarrow{n\to\infty}
        Z,
        \qquad
        Z\sim \mathcal N\brOf{0,\sigma_{(\dtran)^2}}.
    \end{equation}
    For $t\in\R$, define the localized empirical criterion
    \begin{equation}
        V_n(t)
        :=
        \sum_{i=1}^n
        \br{
            \tran(|X_i - m - t/\sqrt n|)
            -
            \tran(|X_i-m|)
        }
        \eqfs
    \end{equation}
    Then
    \begin{equation}
        \sqrt n(m_n-m)
        \in 
        \argmin_{t\in\R} V_n(t)
        \eqfs
    \end{equation}
    We claim that, for every fixed $t$,
    \begin{equation}\label{eq:V:expansion}
        V_n(t)
        =
        t \mathbb G_n[\psi_m]
        +
        \frac{\sigma_{\ddrtran}}2 t^2 
        +
        \mo o_{\Pr}(1)
        \eqfs
    \end{equation}
    Put $h=t/\sqrt n$ and
    \begin{equation}
        r_h(x)
        :=
        \tran(|x-m-h|)
        -
        \tran(|x-m|)
        -
        h\psi_m(x)
        \eqfs
    \end{equation}
    Then
    \begin{equation}
        V_n(t)
        =
        t\mathbb G_n[\psi_m]
        +
        n\Eof{r_h(X)}
        +
        \sum_{i=1}^n
        \br{r_h(X_i)-\Eof{r_h(X)}}
        \eqfs
    \end{equation}
    Using \eqref{eq:F:expansion} and $F(m) = F\pr(m)=0$, we obtain
    \begin{align}
        n\Eof{r_h(X)}
        &=
        n\br{F(m+h)-F(m)-hF'(m)}
        \\
        &=
        n\br{
            \frac{\sigma_{\ddrtran}}2 h^2+\mo o(h^2)
        }
        \\
        &=
        \frac{\sigma_{\ddrtran}}2 t^2
        +
        \mo o(1)
        \label{eq:mean:rh}
        \eqfs
    \end{align}
    It remains to show that the centered remainder is negligible. By the fundamental theorem of calculus,
    \begin{equation}
        \tran(|x-m-h|)
        -
        \tran(|x-m|)
        =
        \int_m^{m+h}\rho(q-x)\,\dl q
        \eqfs
    \end{equation}
    Hence,
    \begin{equation}\label{eq:rh}
        r_h(x)
        =
        h\int_0^1
            \rho(m+uh - x)-\rho(m-x)
        \,\dl u.
    \end{equation}
    Furthermore,
    \begin{equation}
        |\rho(q-X)-\rho(m-X)|
        \leq
        2\dtran(|X-m|)
        +
        \dtran(|q-m|)
        \eqfs
    \end{equation}
    For $q$ in a bounded neighborhood of $m$, the square of the right-hand side is bounded by
    \begin{equation}
        c\br{1+\dtran(|X-m|)^2}
    \end{equation}
    for some constant $c\in\Rpp$, which is integrable by the assumption $\sigma_{(\dtran)^2}<\infty$. Since, $\rho$ is continuous,
    \begin{equation}
        \rho(q-X)
        \xrightarrow{q\to m}
        \rho(m-X)
        \qquad\text{almost surely}\eqfs
    \end{equation}
    Thus, dominated convergence yields
    \begin{equation}
        \EOf{|\rho(q-X)-\rho(m-X)|^2}
        \xrightarrow{q\to m}
        0
        \eqfs
    \end{equation}
    Therefore, using \eqref{eq:rh}, by Jensen's inequality,
    \begin{align}
        \EOf{r_h(X)^2}
        &\leq
        h^2
        \int_0^1
        \EOf{|\rho(m+uh - X)-\rho(m-X)|^2}
        \,\dl u
        \\
        &=
        \mo o(h^2) \quad \text{as $h\to0$}
        \eqfs
    \end{align}
    Since $h=t/\sqrt n$, this yields
    \begin{align}
        \VOf{
            \sum_{i=1}^n
            \br{r_h(X_i)-\Eof{r_h(X)}}
        }
        &=
        n\VOf{r_h(X)}
        \\
        &\leq
        n\EOf{r_h(X)^2}
        \\
        &=
        \mo o(1) \quad \text{as $n\to\infty$}
        \eqfs
    \end{align}
    Hence the centered remainder is $\mo o_{\Pr}(1)$ by Chebyshev's inequality. 
    Together with \eqref{eq:mean:rh}, this proves \eqref{eq:V:expansion}.    
    Consequently, for every fixed $t$,
    \begin{equation}
        V_n(t)
        \xrsquigarrow{n\to\infty}
        V(t)
        :=
        tZ + \frac{\sigma_{\ddrtran}}2 t^2
        \eqfs
    \end{equation}
    The preceding expansion holds jointly for any finite collection $t_1,\dots,t_k$ with fixed $k$, because the remainders are coordinatewise $\mo o_{\Pr}(1)$ and $Z$ does not depend on $t$. 
    Hence, $V_n$ converges to $V$ in finite dimensional distributions.
    Each $V_n$ is convex in $t$, and $V$ is strictly convex because $\sigma_{\ddrtran}>0$. 
    Therefore, the argmin continuous mapping theorem for convex processes \cite[Theorem 1; with single-point index set $T=\{\star\}$]{Kato2009} yields
    \begin{equation}
        \sqrt n(m_n-m)
        =
        \argmin_{t\in\R} V_n(t)
        \xrsquigarrow{n\to\infty}
        \argmin_{t\in\R}
        \br{
            tZ+\frac12\sigma_{\ddrtran}t^2
        }
        \eqfs
    \end{equation}
    The unique minimizer of $t\mapsto tZ+\frac12t^2\sigma_{\ddrtran}$ is $-Z/\sigma_{\ddrtran}$.
    Hence, 
    \begin{equation}
        \sqrt n(m_n-m)
        \xrsquigarrow{n\to\infty}
        -\frac{Z}{\sigma_{\ddrtran}}
        \eqfs
    \end{equation}
    Since $Z\sim\mathcal N(0,\sigma_{(\dtran)^2})$, we obtain
    \begin{equation}
        -\frac{Z}{\sigma_{\ddrtran}}
        \sim
        \mathcal N\brOf{
            0,
            \frac{\sigma_{(\dtran)^2}}{\sigma_{\ddrtran}^2}
        }
        \eqfs
    \end{equation}
    This proves the claim.
\end{proof}

\begin{lemma}\label{lmm:alpha:slow:verify}
    Let $\alpha\in(1,2]$ and set $\tran(x)=x^\alpha$. 
    Assume there is $\beta>2$ such that
    \begin{equation}\label{eq:lmm:slow}
        \lim_{t\searrow0} t^{\alpha-\beta}\Prof{|X-m|\le t}=0
        \eqfs
    \end{equation}
    Then $\sigma_{\alpha-2} := \Eof{|X-m|^{\alpha-2}} < \infty$ and
    \begin{equation}\label{eq:lmm:avg}
        \frac1s\int_0^s\EOf{|X-m-t|^{\alpha-2}}\,\dl t
        \xrightarrow{s\to0}
        \sigma_{\alpha-2}
        \eqfs
    \end{equation}
\end{lemma}

\begin{proof}
    The statement is trivial for $\alpha=2$ as we use the convention $0^0 = 1$. Assume $\alpha\in(1,2)$.
    Without loss of generality assume $m=0$. Since $\Prof{|X|\leq t} \xrightarrow{t\to0}0$ by \eqref{eq:lmm:slow}, $\Prof{X=0}=0$, hence $|X|^{\alpha-2}$ is finite almost surely. 
    
    \emph{Step 1 ($\sigma_{\alpha-2}<\infty$).}
    For $x>0$, $x^{\alpha-2}=(2-\alpha)\int_x^\infty t^{\alpha-3}\,\dl t$
    (using $\alpha-2<0$). By Tonelli,
    \begin{equation}
        \Eof{|X|^{\alpha-2}}
        =(2-\alpha)\int_0^\infty t^{\alpha-3}\Prof{|X|\leq t} \,\dl t
        \eqfs
    \end{equation}
    By \eqref{eq:lmm:slow} there is $\delta\in(0,1]$ with $\Prof{|X|\leq t}\leq t^{\beta-\alpha}$
    for $t\le\delta$, so
    \begin{equation}
        \int_0^{\delta}t^{\alpha-3}\Prof{|X|\leq t}\,\dl t
        \leq
        \int_0^{\delta}t^{\beta-3}\,\dl t
        =
        \frac{\delta^{\beta-2}}{\beta-2}
        <
        \infty
        \eqcm
    \end{equation}
    finite precisely because $\beta-3>-1$, i.e.\ $\beta>2$. On $[\delta,\infty)$,
    $\Prof{|X|\leq t}\le1$ and $\int_\delta^\infty t^{\alpha-3}\,\dl t<\infty$ since
    $\alpha-3<-1$. Hence $\sigma_{\alpha-2}=\Eof{|X|^{\alpha-2}}<\infty$.
    
    \emph{Step 2 (bounds on the inner average).}
    Assume that $s>0$; the case for negative $s$ is symmetric.
    For $s>0$ and $x\in\R$ set
    \begin{equation}
        g_s(x)
        :=
        \frac1{s}\int_{0}^s |x-t|^{\alpha-2}\,\dl t
        =
        \frac1{s}\int_{x-s}^{x} |w|^{\alpha-2}\,\dl w
        \eqfs
    \end{equation}
    Since $|\cdot|^{\alpha-2}$ is even and decreasing in $|w|$, the average over
    a length-$s$ interval is largest when the interval is centered at $0$;
    therefore, uniformly in $x$,
    \begin{equation}\label{eq:gs:unif}
        g_s(x)
        \le
        \frac1{s}\int_{-s/2}^{s/2} |w|^{\alpha-2}\,\dl w
        =
        \frac{2^{2-\alpha}}{\alpha-1} s^{\alpha-2}
        \eqfs
    \end{equation}
    Moreover, if $|x|>2s$ then $|x-t|\ge|x|-s\ge|x|/2$ for $t\in[0,s]$, so
    \begin{equation}\label{eq:gs:far}
        g_s(x)\le\brOf{|x|/2}^{\alpha-2}=2^{2-\alpha}|x|^{\alpha-2}
        \eqfs
    \end{equation}
    Finally, for fixed $x\neq0$, $g_s(x)\to|x|^{\alpha-2}$ as $s\to0$ by
    continuity of $t\mapsto|x-t|^{\alpha-2}$ at $t = 0$ for $x\neq0$.
    
    \emph{Step 3.}
    By Tonelli,  for $\epsilon\in\Rpp$,
    \begin{equation}
        \frac1{s}\int_{0}^s \EOf{|X-t|^{\alpha-2}}\,\dl t
        =
        \Eof{g_s(X)}
        =
        \Eof{g_s(X)\mathbf 1\cb{|X|>\epsilon}}+\Eof{g_s(X)\mathbf 1\cb{|X|\le\epsilon}}
        \eqfs
    \end{equation}
   
    \emph{Far part.} On $\cb{|X|>\epsilon}$ and $s<\epsilon/2$ we have
    $|X-t|\ge\epsilon/2$ for $t\in[0,s]$, so
    $g_s(X)\le(\epsilon/2)^{\alpha-2}$. With the pointwise limit
    from Step 2, dominated convergence yields
    \begin{equation}
        \Eof{g_s(X)\mathbf 1\cb{|X|>\epsilon}}
        \xrightarrow{s\to0}
        \Eof{|X|^{\alpha-2}\mathbf 1\cb{|X|>\epsilon}}
        \eqfs
    \end{equation}
    
    \emph{Near part.} Split according to $|X|\lessgtr2s$ and use
    \eqref{eq:gs:far} on the first piece, \eqref{eq:gs:unif} on the second:
    \begin{equation}
        \Eof{g_s(X)\mathbf 1\cb{|X|\le\epsilon}}
        \le
        2^{2-\alpha} \Eof{|X|^{\alpha-2}\mathbf 1\cb{|X|\le\epsilon}}
        +
        c_\alpha s^{\alpha-2} \Prof{|X|\leq 2s}
        \eqfs
    \end{equation}
    By \eqref{eq:lmm:slow}, $s^{\alpha-2}\Prof{|X|\leq 2s}\xrightarrow{s\to0}0$ because $\beta>2$. Hence,
    \begin{equation}
        \limsup_{s\to0}\Eof{g_s(X)\mathbf 1\cb{|X|\le\epsilon}}
        \le
        2^{2-\alpha}\Eof{|X|^{\alpha-2}\mathbf 1\cb{|X|\le\epsilon}}
        \eqfs
    \end{equation}
    
    \emph{Combined.}
    Combining the two parts yields
    \begin{equation}
        \limsup_{s\to0}\absOf{\Eof{g_s(X)} - \sigma_{\alpha-2}}
        \le
        \Eof{|X|^{\alpha-2}\mathbf 1\cb{|X|\le\epsilon}}
        \eqfs
    \end{equation}
    The left-hand side is independent of $\epsilon$, and the right-hand side
    tends to $0$ as $\epsilon\to0$ because $\Eof{|X|^{\alpha-2}}<\infty$ (Step 1)
    and $\mathbf 1\cb{|X|\le\epsilon}\searrow0$ a.s. Therefore
    $\Eof{g_s(X)}\to\sigma_{\alpha-2}$, which implies \eqref{eq:lmm:avg}.
\end{proof}

\begin{lemma}\label{lmm:alpha:fast:secondderiv}
    Let $\alpha\in(1,2)$. Let $Z$ be a real-valued random variable.
    Assume there are $\beta\in(\alpha,2)$ and $b\in\Rpp$ such that
    \begin{equation}\label{eq:alpha:fast:largesmallball}
        \lim_{r\searrow 0}\frac{\mathbb{P}(0\le Z\le r)}{b r^{\beta-\alpha}} = 1
        \qquad\text{and}\qquad
        \lim_{r\searrow 0}\frac{\mathbb{P}(0\ge Z\ge -r)}{b r^{\beta-\alpha}} = 1
        \eqfs
    \end{equation}
    Then
    \begin{equation}
        s^{1-\beta} b^{-1}
        \EOf{\int_0^s |Z-t|^{\alpha-2} \dl t}
        \xrightarrow{s\searrow 0}
        C_{\alpha,\beta}
    \end{equation}
    where $C_{\alpha,\beta} \in\Rpp$ depends only on $\alpha$ and $\beta$.
\end{lemma}

\begin{proof}
    \emph{Step 1 (scaling).}
    For $x\in\R$ and $s>0$, the substitution $t=su$ gives
    \begin{equation}\label{eq:lem:homog}
        \int_0^s|x-t|^{\alpha-2}\,\dl t
        =
        s^{\alpha-1}G\brOf{\tfrac xs}
        \eqcm\qquad
        G(y):=\int_0^1|y-u|^{\alpha-2}\,\dl u
        \eqfs
    \end{equation}
    Since $w\mapsto\sign(w)|w|^{\alpha-1}/(\alpha-1)$ is an antiderivative of $|w|^{\alpha-2}$ on $\R$,
    \begin{equation}
        G(y)=\frac{\sign(y)|y|^{\alpha-1}-\sign(y-1)|y-1|^{\alpha-1}}{\alpha-1}
        \eqcm\qquad
        G'(y)=|y|^{\alpha-2}-|y-1|^{\alpha-2}
    \end{equation}
    for $y\notin\cb{0,1}$. Thus $G$ is continuous, positive and bounded, with $G(0)=1/(\alpha-1)$, and as $|y|\to\infty$
    \begin{equation}\label{eq:lem:Gtail}
        G(y)=|y|^{\alpha-2}\brOf{1+\mo o(1)}
        \eqcm\qquad
        G'(y)=\mo O\brOf{|y|^{\alpha-3}}
        \eqfs
    \end{equation}
    In particular $G'\in L^1(\R)$: the singularities at $0$ and $1$ have exponent $\alpha-2>-1$, and the tail \eqref{eq:lem:Gtail} is integrable since $\alpha-3<-1$. As
    $s^{1-\beta}b^{-1}\EOf{\int_0^s|Z-t|^{\alpha-2}\,\dl t}=b^{-1}s^{\alpha-\beta}\EOf{G(Z/s)}$, it suffices to prove
    \begin{equation}\label{eq:lem:goal}
        s^{-(\beta-\alpha)}\EOf{G(Z/s)}
        \xrightarrow{s\to0}
        bC_{\alpha,\beta}
        \eqcm\qquad
        C_{\alpha,\beta}:=(\beta-\alpha)\int_\R|y|^{\beta-\alpha-1}G(y)\,\dl y
        \eqfs
    \end{equation}
    
    \emph{Step 2 (positive part).}
    Let $H(r):=\Prof{0<Z\le r}$. By \eqref{eq:alpha:fast:largesmallball}, $\Prof{Z=0}=0$, so
    $\EOf{G(Z/s)}=I_+(s)+I_-(s)$ with $I_\pm(s):=\EOf{G(Z/s)\mathbf 1\cb{\pm Z>0}}$.
    For $z>0$ we have $G(z/s)=\int_{z/s}^\infty(-G'(y))\,\dl y$ since $G(\infty)=0$; inserting this and applying Fubini ($\int_\R|G'|<\infty$),
    \begin{equation}
        I_+(s)=\int_0^\infty(-G'(y))H(sy)\,\dl y
        \eqfs
    \end{equation}
    By \eqref{eq:alpha:fast:largesmallball} there is $\delta\in(0,1]$ with $H(r)\le 2br^{\beta-\alpha}$ for $r\le\delta$; combined with $H\le1$ this gives the global bound
    \begin{equation}\label{eq:lem:global}
        H(r)\le Kr^{\beta-\alpha}
        \quad\text{for all }r>0,
        \qquad
        K:=\max\cb{2b,\ \delta^{-(\beta-\alpha)}}
        \eqcm
    \end{equation}
    because $r^{\beta-\alpha}\ge\delta^{\beta-\alpha}$ when $r>\delta$. Hence
    \begin{equation}
        \absOf{(-G'(y))\frac{H(sy)}{s^{\beta-\alpha}}}
        \le
        K|G'(y)|y^{\beta-\alpha}=:D(y)
        \eqcm
    \end{equation}
    and $D\in L^1(0,\infty)$: near $0$, $D\asymp y^{\beta-2}$ with $\beta>1$; near $1$ the exponent is $\alpha-2>-1$; near $\infty$, $D\asymp y^{\beta-3}$ with $\beta<2$. Since
    $H(sy)/s^{\beta-\alpha}=\brOf{H(sy)/(sy)^{\beta-\alpha}}y^{\beta-\alpha}\to by^{\beta-\alpha}$ pointwise, dominated convergence and then integration by parts (boundary terms vanish as $y^{\beta-\alpha}G(y)\to0$ at $0$ and $\infty$, using $\alpha<\beta<2$) give
    \begin{equation}\label{eq:lem:plus}
        s^{-(\beta-\alpha)} I_+(s)
        \xrightarrow{s\to0}
        b\int_0^\infty(-G'(y)) y^{\beta-\alpha}\,\dl y
        =
        b (\beta-\alpha)\int_0^\infty y^{\beta-\alpha-1} G(y)\,\dl y
        \eqfs
    \end{equation}
    
    \emph{Step 3 (negative part and conclusion).}
    Applying Step~2 to $-Z$, which replaces $G$ by $y\mapsto G(-y)$ and, by the second part of \eqref{eq:alpha:fast:largesmallball}, satisfies \eqref{eq:lem:global} with the same $b$, yields
    \begin{equation}\label{eq:lem:minus}
        s^{-(\beta-\alpha)} I_-(s)
        \xrightarrow{s\to0}
        b (\beta-\alpha)\int_{-\infty}^0 |y|^{\beta-\alpha-1} G(y)\,\dl y
        \eqfs
    \end{equation}
    Adding \eqref{eq:lem:plus} and \eqref{eq:lem:minus} gives \eqref{eq:lem:goal}. Finally $C_{\alpha,\beta}\in\Rpp$: the integrand $|y|^{\beta-\alpha-1}G(y)$ is positive, behaves like $G(0) |y|^{\beta-\alpha-1}$ near $0$ (integrable as $\beta>\alpha$) and like $|y|^{\beta-3}$ at infinity (integrable as $\beta<2$), and depends only on $\alpha,\beta$.
\end{proof}
\begin{proof}[Proof of \cref{thm:distri:alpha}]
    The first part of the theorem follows from \cref{thm:distri:tran} together with \cref{lmm:alpha:slow:verify}. Now consider the second part.
    
    \emph{Step 0 (inheritance from \cref{thm:distri:tran}).}
    Set 
    \begin{equation}
        \rho(t) := \alpha\,\sign(t)|t|^{\alpha-1}
        \eqcm\quad
        \psi_q(x) := \rho(q-x)
        \quad\text{and}\quad
        F(q) := \Eof{|X-q|^\alpha-|X-m|^\alpha}
        \eqfs
    \end{equation}
    Following the proof of \cref{thm:distri:tran}, we obtain 
    \begin{equation}
        F\pr(q) = \EOf{\rho(q - X)} \quad\text{with}\quad F\pr(m) = 0
    \end{equation}
    as well as
    \begin{equation}\label{eq:distri:alpha:Fpr}
        F\pr(m+s) = \alpha (\alpha-1)  \EOf{\int_0^s|X-m-t|^{\alpha-2} \dl t}
    \end{equation}
    and
    \begin{equation}
        \Eof{\psi_m(X)^2} = \alpha^2 \Eof{|X-m|^{2\alpha-2}} = \alpha^2 \sigma_{2\alpha-2} < \infty
        \eqfs
    \end{equation}
    With 
    \begin{equation}
        \mathbb G_n [f]
        :=
        \frac1{\sqrt n}\sum_{i=1}^n
        \br{f(X_i)-\Eof{f(X)}}
    \end{equation}
    the central limit theorem implies
    \begin{equation}\label{eq:alpha:clt}
        \mathbb G_n[\psi_m]
        \xrsquigarrow{n\to\infty}
        Z,
        \qquad
        Z\sim\mathcal N\brOf{0,\alpha^2\sigma_{2\alpha-2}}
        \eqfs
    \end{equation}
    Only the local behavior of $F$ at $m$ differs from \cref{thm:distri:tran}; here the second derivative $x\mapsto\alpha(\alpha-1)x^{\alpha-2}$ is not integrable near $m$, so the quadratic expansion is replaced by one of order $\beta$.
    
    \emph{Step 1 (local expansion of $F$).}
    The tail hypothesis \eqref{eq:distri:alpha:tail} is precisely the assumption of \cref{lmm:alpha:fast:secondderiv}. The lemma gives
    \begin{equation}
        \EOf{\int_0^s |X-m-t|^{\alpha-2}\,\dl t}
        =
        \sign(s) b C_{\alpha,\beta} |s|^{\beta-1}+\mo o\brOf{|s|^{\beta-1}}
        \quad\text{as }s\to0
        \eqfs
    \end{equation}
    Hence, with
    \begin{equation}\label{eq:alpha:cstar}
        c_\star:=\alpha(\alpha-1) b C_{\alpha,\beta}\in(0,\infty)
    \end{equation}
    and \eqref{eq:distri:alpha:Fpr}, we obtain 
    \begin{equation}\label{eq:alpha:Phi}
        F\pr(m+s)=c_\star\,\sign(s) |s|^{\beta-1}+\mo o\brOf{|s|^{\beta-1}}
        \quad\text{as }s\to0
        \eqfs
    \end{equation}
    Integrating and using $F\pr(m)=0$,
    \begin{equation}\label{eq:alpha:F:expansion}
        F(m+s)-F(m)
        =\int_0^s F\pr(m+\sigma)\,\dl\sigma
        =\frac{c_\star}{\beta} |s|^{\beta}+\mo o\brOf{|s|^{\beta}}
        \quad\text{as }s\to0
        \eqfs
    \end{equation}
    
    \emph{Step 2 (localized criterion).}
    Set $a_n:=n^{\frac1{2(\beta-1)}}$ and, for $t\in\R$,
    \begin{equation}
        V_n(t):=\sum_{i=1}^n
        \brOf{|X_i-m-t/a_n|^\alpha-|X_i-m|^\alpha}
        \eqcm
    \end{equation}
    so that $a_n(m_n-m)\in\argmin_{t}V_n(t)$. With $h:=t/a_n$ and
    $r_h(x):=|x-m-h|^\alpha-|x-m|^\alpha-h\psi_m(x)$ as in
    \cref{thm:distri:tran},
    \begin{equation}
        V_n(t)
        =h\sum_{i=1}^n\psi_m(X_i)
        +n\Eof{r_h(X)}
        +\sum_{i=1}^n\brOf{r_h(X_i)-\Eof{r_h(X)}}
        \eqfs
    \end{equation}
    The two terms balance at the scale
    $\kappa_n:=\sqrt n/a_n=n/a_n^{\beta}$. We analyze $\tilde V_n:=V_n/\kappa_n$, whose argmin
    equals that of $V_n$.
    
    For the linear term, since $h/\kappa_n=(t/a_n)(a_n/\sqrt n)=t/\sqrt n$,
    \begin{equation}
        \frac1{\kappa_n} h\sum_{i=1}^n\psi_m(X_i)
        =t \mathbb G_n[\psi_m]
        \xrsquigarrow{n\to\infty} tZ
        \eqcm
    \end{equation}
    using $\Eof{\psi_m(X)}=0$ and \eqref{eq:alpha:clt}.
    For the drift term, since $n/\kappa_n=a_n^{\beta}$ and $F'(m)=0$,
    \begin{equation}
        \frac1{\kappa_n} n\Eof{r_h(X)}
        =a_n^{\beta}\brOf{F(m+h)-F(m)}
        \xrightarrow{n\to\infty}
        \frac{c_\star}{\beta} |t|^{\beta}
        \eqcm
    \end{equation}
    by \eqref{eq:alpha:F:expansion} with $h=t/a_n\to0$ and
    $a_n^\beta\cdot\mo o(a_n^{-\beta})=\mo o(1)$.
    The centered remainder is $\mo o_{\Pr}(1)$ by the same arguments as in the proof of \cref{thm:distri:tran} using  $\sigma_{2\alpha-2}<\infty$. Combining the three pieces,
    for every fixed $t$,
    \begin{equation}
        \tilde V_n(t)\xrsquigarrow{n\to\infty}
        V(t):=tZ+\frac{c_\star}{\beta} |t|^{\beta}
        \eqfs
    \end{equation}
    
    \emph{Step 3 (argmin).}
    Again with the same arguments as in the proof of \cref{thm:distri:tran} we can apply the argmin continuous mapping theorem for convex processes \cite{Kato2009} (using that $V$ is strictly convex and coercive because $\beta>1$ and $c_\star>0$). We obtain
    \begin{equation}
        a_n(m_n-m)=\argmin_t\tilde V_n(t)
        \rightsquigarrow
        \argmin_t V(t)
        \eqfs
    \end{equation}
    Solving $V\pr(t)=Z+c_\star\,\sign(t)|t|^{\beta-1}=0$ gives the minimizer
    \begin{equation}
        t^\star
        =-\sign(Z)\brOf{\frac{|Z|}{c_\star}}^{\frac1{\beta-1}}
        \eqfs
    \end{equation}
    
    \emph{Step 4 (identification of the limit).}
    Write $Z\overset d=\alpha W$ with $W\sim\mathcal N(0,\sigma_{2\alpha-2})$.
    Then $\sign(Z)=\sign(W)$ and
    $|Z|^{1/(\beta-1)}=\alpha^{1/(\beta-1)}|W|^{1/(\beta-1)}$, so
    \begin{equation}
        t^\star
        \overset d=
        -\brOf{\frac{\alpha}{c_\star}}^{\frac1{\beta-1}}
        \sign(W) |W|^{\frac1{\beta-1}}
        \overset d=
        C_{\alpha,\beta,b}\,\sign(W) |W|^{\frac1{\beta-1}}
        \eqcm
    \end{equation}
    the last step by symmetry of $W$, where
    $C_{\alpha,\beta,b}:=\brOf{\alpha/c_\star}^{1/(\beta-1)}\in\Rpp$ depends
    only on $\alpha,\beta,b$ through \eqref{eq:alpha:cstar}. Therefore
    \begin{equation}
        n^{\frac1{2(\beta-1)}}(m_n-m)
        \rightsquigarrow
        C_{\alpha,\beta,b}\,\sign(W) |W|^{\frac1{\beta-1}}
        \eqcm
    \end{equation}
    which is the claim.
\end{proof}

\subsection{Optimality}
\begin{proof}[Proof of \cref{prop:optimality}]
    As $\mc Q$ contains two distinct points $q,p$, the image of $\geodft qp$ is isometric to $[0,\ol qp]$, and restricting it to a subinterval centered at its midpoint yields a unit-speed geodesic $\gamma\colon[-h,h]\to\mc Q$ for every $h \in (0, \frac12 \ol qp]$.
    The image $\gamma([-h,h])$ is closed and convex, so the $\alpha$-Fr\'echet means of $P^\rho$ and of its empirical versions lie on $\gamma$ and may be computed in $\R$ \cite[Proposition 5.2]{varinequ}. We identify $\gamma(t)$ with $t\in[-h,h]$ throughout.
    The objective function $x\mapsto\frac\rho2\absof{x+h}^\alpha + \frac\rho2\absof{x-h}^\alpha + (1-\rho)\absof{x}^\alpha$ is strictly convex and symmetric about $0$, so $m=0$. Hence $\ol Ym = h$ with probability $\rho$ and $\ol Ym = 0$ with probability $1-\rho$, which gives $\sigma_\varphi = \rho h^\varphi$ for $\varphi\in\Rpp$ and, using $(\alpha-1)\phi = 2-\alpha$,
    \begin{equation}
        \sigma_{\alpha-1}^{\phi} \sigma_{2\alpha-2}^{\frac12}
        =
        \rho^{\phi+\frac12}  h^{(\alpha-1)\phi + (\alpha-1)}
        =
        \rho^{\phi+\frac12}  h
        \eqcm\qquad
        \sigma_\alpha^{\frac1\alpha} = \rho^{\frac1\alpha}  h
        \eqfs
    \end{equation}
    The identity for $L_1/L_2$ follows from $\phi + 1 - \frac1\alpha = \frac{\alpha(2-\alpha)+(\alpha-1)^2}{\alpha(\alpha-1)} = \frac1{\alpha(\alpha-1)}$, that is, $\phi + \frac12 - \frac1\alpha = \nu_\alpha$, together with $\frac1{\alpha(\alpha-1)}-\frac12 = \nu_\alpha$.
    
    \ref{prop:optimality:first}:
    Let $\rho=1$. Then $\sigma_{\alpha-2} = h^{\alpha-2}$ and $\sigma_{\alpha-1}^{\phi} = h^{(\alpha-1)\phi} = h^{2-\alpha}$, which are reciprocal.
    The hypothesis of \cref{cor:lower:alpha}\,(i) holds for every $\beta>2$, as $\PrOf{\ol Ym\leq t} = 0$ for $t<h$, and $\sigma_{2\alpha-2} = h^{2\alpha-2}<\infty$.
    The variance of the limit is $(\alpha-1)^{-2}\sigma_{\alpha-2}^{-2}\sigma_{2\alpha-2} = (\alpha-1)^{-2} h^{2(2-\alpha)} h^{2\alpha-2} = (\alpha-1)^{-2} h^2$.
    Hence, applying \cref{cor:lower:alpha}\,(i) with $g(x) = x^\alpha$ and taking the power $\frac1\alpha$,
    \begin{equation}
        \liminf_{n\to\infty} \sqrt n  \EOf{\ol m{m_n}^\alpha}^{\frac1\alpha}
        \geq
        \frac{h}{\alpha-1} \EOf{\absof{Z}^\alpha}^{\frac1\alpha}
        \eqcm
    \end{equation}
    where $Z\sim\mc N(0,1)$.
    As $\sqrt n  L_1(n) = h$ and $\alpha - 1 < 1$, we obtain
    \begin{equation}
        \liminf_{n\to\infty} \frac{\EOf{\ol m{m_n}^\alpha}^{\frac1\alpha}}{L_1(n)}
        \geq
        \frac{\EOf{\absof{Z}^\alpha}^{\frac1\alpha}}{\alpha-1}
        \geq
        \EOf{\absof{Z}}
        =
        \sqrt{\frac2\pi}
        >
        \frac34
        \eqfs
    \end{equation}
    
    \ref{prop:optimality:last}:
    Let $n\geq2$, $\rho=n^{-2}$, and let $E$ be the event that exactly one of the samples equals $\gamma(-h)$ and all others equal $\gamma(0)$. As $n\rho = \frac1n \leq \frac12$, Bernoulli's inequality gives $(1-\rho)^{n-1} \geq 1-n\rho \geq \frac12$ and hence
    \begin{equation}
        \PrOf{E} = \frac{n\rho}{2}\br{1-\rho}^{n-1} \geq \frac{n\rho}{4}
        \eqfs
    \end{equation}
    On $E$, the empirical objective function is $x \mapsto \absof{x+h}^\alpha + (n-1)\absof{x}^\alpha$, which is strictly convex with minimizer $-t_n$ for some $t_n \in (0,h)$ characterized by $(n-1) t_n^{\alpha-1} = (h-t_n)^{\alpha-1}$. Thus,
    \begin{equation}
        \ol m{m_n}
        =
        t_n
        =
        \frac{h}{(n-1)^{\frac1{\alpha-1}}+1}
        \geq
        \frac{h}{2}  n^{-\frac1{\alpha-1}}
        \eqfs
    \end{equation}
    Hence, using $\frac1\alpha - \frac1{\alpha-1} = -\frac1{\alpha(\alpha-1)}$ and $4^{\frac1\alpha}\leq4$,
    \begin{equation}
        \EOf{\ol m{m_n}^\alpha}^{\frac1\alpha}
        \geq
        \PrOf{E}^{\frac1\alpha}  \frac{h}{2}  n^{-\frac1{\alpha-1}}
        \geq
        \frac{\rho^{\frac1\alpha}  h}{2 \cdot 4^{\frac1\alpha}}  n^{-\frac1{\alpha(\alpha-1)}}
        \geq
        \frac18 L_2(n)
        \eqfs
        \qedhere
    \end{equation}
\end{proof}
\begin{proof}[Proof of \cref{ex:tripod}]
    Throughout, set $\nu := \alpha-1 \in (0,1]$ and write $Y_i = (L_i, R_i)$ for the sample.
    For $j\in\cb{1,2,3}$ let
    \begin{equation}
        I_j := \setByEle{i\in\cb{1,\dots,n}}{L_i = j}
        \eqcm\qquad
        N_j := \absof{I_j}
        \eqcm\qquad
        S_j := \sum_{i\in I_j} R_i^\nu
        \eqcm\qquad
        T := \sum_{i=1}^n R_i^\nu
        \eqfs
    \end{equation}
    In both parts we have $R\geq1$ almost surely, so $T\geq n>0$.
    
    \emph{Step 1: A stickiness criterion.}
    For $j\in\cb{1,2,3}$ and $s\in\Rp$ set
    \begin{equation}
        f_{n,j}(s)
        :=
        \sum_{i=1}^n d(Y_i, (j,s))^\alpha
        =
        \sum_{i\in I_j} \absof{R_i-s}^\alpha + \sum_{i\notin I_j} \br{R_i+s}^\alpha
        \eqcm
    \end{equation}
    \begin{equation}\label{eq:tripod:G}
        G_{n,j}(s)
        :=
        \sum_{i\in I_j} \sign(R_i-s)\absof{R_i-s}^{\nu} - \sum_{i\notin I_j} \br{R_i+s}^{\nu}
        \eqfs
    \end{equation}
    As $\alpha>1$, the function $x\mapsto\absof x^\alpha$ is strictly convex and continuously differentiable on $\R$, so $f_{n,j}$ is strictly convex and continuously differentiable on $\Rp$ with
    \begin{equation}\label{eq:tripod:deriv}
        f_{n,j}'(s) = -\alpha  G_{n,j}(s)
        \eqcm
    \end{equation}
    and $G_{n,j}$ is decreasing. We claim that, for every $s\in\Rp$,
    \begin{equation}\label{eq:tripod:crit}
        \cb{\ol o{m_n} > s}
        =
        \biguplus_{j=1}^3 \cb{G_{n,j}(s) > 0}
        \eqcm
    \end{equation}
    where $\uplus$ denotes a disjoint union.
    
    First, the union is disjoint: since $G_{n,j}$ is decreasing, $G_{n,j}(s)>0$ implies $G_{n,j}(0) = 2S_j - T > 0$; if this held for two distinct $j,k$, then $2T \geq 2(S_j+S_k) > 2T$, a contradiction.
    
    Next, let $G_{n,j}(s)>0$. By \eqref{eq:tripod:deriv} and monotonicity of $G_{n,j}$ we have $f_{n,j}' < 0$ on $[0,s]$, so the minimizer of the strictly convex function $f_{n,j}$ lies in $(s,\infty)$, say at $t>s$, and $f_{n,j}(t) < f_{n,j}(0)$. Since $f_{n,j}(0) = \sum_i \ol{Y_i}o^\alpha$ for every $j$, this yields $m_n \neq o$, hence $m_n = (k,u)$ with $u>0$ for some leg $k$. If we had $k\neq j$, then $u$ minimizes $f_{n,k}$, so $f_{n,k}'(0)<0$ by strict convexity, i.e., $G_{n,k}(0)>0$, while also $G_{n,j}(0)\geq G_{n,j}(s)>0$, contradicting the disjointness shown above. Hence $k=j$, $u=t>s$, and $\ol o{m_n}>s$.
    
    Conversely, let $\ol o{m_n}>s$, say $m_n = (j,t)$ with $t>s$. Then $t$ minimizes $f_{n,j}$, and strict convexity gives $f_{n,j}'(s)<0$, i.e., $G_{n,j}(s)>0$. This proves \eqref{eq:tripod:crit}. Since $m=o$, we have $\ol m{m_n} = \ol o{m_n}$ throughout.
    
    \emph{Step 2: Proof of (i).}
    Let $R=1$ almost surely. Then $G_{n,j}(0) = N_j - (n-N_j) = 2N_j-n$, so \eqref{eq:tripod:crit} with $s=0$ and $N_j\sim\ms{Bin}(n,\frac13)$ give
    \begin{equation}
        \PrOf{m_n\neq m}
        =
        \sum_{j=1}^3 \PrOf{2N_j > n}
        =
        3 \PrOf{B_n>\frac n2}
        \eqfs
    \end{equation}
    Moreover, as $\alpha$-Fr\'echet means stay within the closed convex hull of their distribution \cite[Proposition 5.2]{varinequ}, we have $\ol m{m_n}\leq1$ almost surely.
    
    Set $k_n := \lfloor n/2\rfloor +1$, so that $\cb{B_n>n/2} = \cb{B_n\geq k_n}$. For $k\geq k_n$ we have $n-k < n/2 < k+1$ and hence
    \begin{equation}\label{eq:tripod:ratio}
        \frac{\PrOf{B_n=k+1}}{\PrOf{B_n=k}}
        =
        \frac{n-k}{k+1}\cdot\frac{1/3}{2/3}
        \leq
        \frac12
        \eqfs
    \end{equation}
    Consequently, $\PrOf{B_n = k_n+l} \leq 2^{-l}\PrOf{B_n=k_n}$ for all $l\in\N_0$ and
    \begin{equation}\label{eq:tripod:tailloc}
        \PrOf{B_n=k_n}
        \leq
        \PrOf{B_n>\frac n2}
        \leq
        2 \PrOf{B_n=k_n}
        \eqfs
    \end{equation}
    By Stirling's formula there are constants $0<c\leq C<\infty$ with
    $c n^{-\frac12}e^{nH(k/n)} \leq \binom nk \leq C n^{-\frac12}e^{nH(k/n)}$
    for all $n\in\N$ and all $k$ with $k/n\in[\frac14,\frac34]$, where $H(u) := -u\log u - (1-u)\log(1-u)$. With $D(u) := u\log(3u) + (1-u)\log\br{\frac32(1-u)}$ this gives
    \begin{equation}
        \PrOf{B_n = k} = \Theta\brOf{n^{-\frac12}e^{-nD(k/n)}}
        \eqcm\qquad k/n\in\left[\tfrac14,\tfrac34\right]
        \eqfs
    \end{equation}
    As $D$ is Lipschitz on $[\frac14,\frac34]$, say with constant $L$, and $\absof{k_n/n - \frac12}\leq \frac1n$, we obtain $\absof{nD(k_n/n) - nD(\frac12)} \leq L$, and $D(\frac12) = \frac12\log\frac98$ yields
    \begin{equation}\label{eq:tripod:loc}
        \PrOf{B_n=k_n}
        =
        \Theta\brOf{n^{-\frac12}\br{\tfrac89}^{\frac n2}}
        \eqfs
    \end{equation}
    Together with \eqref{eq:tripod:tailloc} this proves the first display of (i).
    
    Turning to the moment bound, fix $j$ and $k>n/2$ and work on $\cb{N_j = k}$. By Step 1, $m_n = (j,t)$ with $t\in(0,1]$ the unique zero of $G_{n,j}$, i.e.,
    $k(1-t)^\nu = (n-k)(1+t)^\nu$. Setting $x := \frac{2k-n}{n}\in(0,1]$, so that $\frac{k}{n-k} = \frac{1+x}{1-x}$, this is equivalent to
    \begin{equation}
        \artanh(t) = \frac1\nu \artanh(x)
        \eqcm\qquad\text{i.e.}\qquad
        t = \tanh\brOf{\tfrac1\nu\artanh(x)}
        \eqcm
    \end{equation}
    with the convention $\tanh(\infty)=1$ for $x=1$. Since $\frac1\nu\geq1$ and $\tanh$ is concave and increasing on $\Rp$ with $\tanh(0)=0$, we get $\tanh(u) \leq \tanh(\frac1\nu u)\leq\frac1\nu\tanh(u)$ for $u\in\Rp$ and therefore
    \begin{equation}\label{eq:tripod:disp}
        x \leq \ol m{m_n} \leq \frac{x}{\nu}
        \eqcm\qquad
        x = \frac{2N_j-n}{n}
        \eqfs
    \end{equation}
    By \eqref{eq:tripod:crit}, \eqref{eq:tripod:disp}, and disjointness,
    \begin{align}
        \EOf{\ol m{m_n}^\alpha}
        &=
        \Theta\brOf{n^{-\alpha}\sum_{j=1}^3\ \sum_{k>n/2} (2k-n)^\alpha\PrOf{N_j = k}}
        \\&=
        \Theta\brOf{n^{-\alpha}\sum_{l=0}^\infty\br{2k_n-n+2l}^\alpha\PrOf{B_n = k_n+l}}
        \eqfs
    \end{align}
    The summand with $l=0$ is at least $\PrOf{B_n=k_n}$, and by \eqref{eq:tripod:ratio} and $2k_n-n\leq2$ the whole sum is at most $\PrOf{B_n=k_n}\sum_{l=0}^\infty(2l+2)^\alpha2^{-l}<\infty$. Hence, with \eqref{eq:tripod:loc},
    \begin{equation}
        \EOf{\ol m{m_n}^\alpha}
        =
        \Theta\brOf{n^{-\alpha}\PrOf{B_n=k_n}}
        =
        \Theta\brOf{n^{-\alpha-\frac12}\br{\tfrac89}^{\frac n2}}
        \eqfs
    \end{equation}
    
    \emph{Step 3: Setup for (ii).}
    Let $\lambda>\alpha$ and $\Prof{R\geq r} = r^{-\lambda}$ for $r\geq1$. Set $\beta := \frac\lambda\nu$. As $\alpha\leq2$, we have $\frac{\alpha}{\alpha-1}\geq2$ and hence
    \begin{equation}\label{eq:tripod:beta}
        \beta = \frac{\lambda}{\alpha-1} > \frac{\alpha}{\alpha-1}\geq 2
        \eqfs
    \end{equation}
    In particular $\Eof{R^\alpha} = \frac{\lambda}{\lambda-\alpha}<\infty$, i.e., $\Eof{\ol Ym^\alpha}<\infty$, and $\Eof{R^{2\nu}} = \frac{\lambda}{\lambda-2\nu}<\infty$ because $2\nu\leq\alpha<\lambda$.
    
    Fix $s\in\Rp$. By \eqref{eq:tripod:G} we may write $G_{n,1}(s) = \sum_{i=1}^n V_i(s)$ with $V_1(s),\dots,V_n(s)$ independent and identically distributed copies of
    \begin{equation}
        V(s) :=
        \begin{cases}
            \sign(R-s)\absof{R-s}^\nu, & L = 1,\\
            -\br{R+s}^\nu, & L\neq1,
        \end{cases}
    \end{equation}
    and we set $c(s) := -\Eof{V(s)} = \frac23\Eof{\br{R+s}^\nu} - \frac13\EOf{\sign(R-s)\absof{R-s}^\nu}$. Using $\absof{\sign(R-s)\absof{R-s}^\nu}\leq (R+s)^\nu$ and $R\leq R+s\leq R(1+s)$, we obtain
    \begin{equation}\label{eq:tripod:c}
        \frac13\br{1+s}^\nu
        \leq
        \frac13\EOf{\br{R+s}^\nu}
        \leq
        c(s)
        \leq
        \Eof{R^\nu}\br{1+s}^\nu
        \eqcm
    \end{equation}
    and, with $A := 9\Eof{R^{2\nu}}<\infty$,
    \begin{equation}\label{eq:tripod:var}
        \EOf{V(s)^2}
        \leq
        \EOf{\br{R+s}^{2\nu}}
        \leq
        \Eof{R^{2\nu}}\br{1+s}^{2\nu}
        \leq
        A c(s)^2
        \eqfs
    \end{equation}
    Finally, $V(s)\leq R^\nu$ almost surely with $\PrOf{R^\nu\geq v} = v^{-\beta}$ for $v\geq1$, and $v^{-\beta}>1$ for $v\in(0,1)$, so
    \begin{equation}\label{eq:tripod:Vtail}
        \PrOf{V(s)\geq v} \leq v^{-\beta}
        \qquad\text{for all } v\in\Rpp
        \eqfs
    \end{equation}
    
    \emph{Step 4: A uniform upper bound.}
    We claim that for every $K\geq1$ there is $C<\infty$, depending only on $\alpha$, $\lambda$, and $K$, such that
    \begin{equation}\label{eq:tripod:ub}
        \PrOf{\ol m{m_n}>s}
        \leq
        C n^{1-\beta}\br{1+s}^{-\lambda} + C n^{-\frac K2}
        \qquad\text{for all } s\in\Rp,\ n\geq 2K
        \eqfs
    \end{equation}
    Fix $s$, put $W_i := V_i(s)+c(s)$, $\eta := n c(s)$ and $b := \eta/K$, and truncate, $\bar W_i := \min(W_i, b)$. Then
    \begin{equation}
        \PrOf{G_{n,1}(s)>0}
        =
        \PrOf{\sum_{i=1}^n W_i > \eta}
        \leq
        n\PrOf{W_1 > b} + \PrOf{\sum_{i=1}^n \bar W_i > \eta}
        \eqfs
    \end{equation}
    For the first term, $n\geq2K$ implies $c(s)\leq\frac b2$, so that $b - c(s)\geq \frac b2$, and \eqref{eq:tripod:Vtail} gives
    \begin{align}
        n\PrOf{W_1>b}
        &=
        n\PrOf{V(s) > b - c(s)}
        \\&\leq
        n\PrOf{V(s) \geq \frac b2}
        \\&\leq
        n\brOf{\frac{n c(s)}{2K}}^{-\beta}
        \\&=
        (2K)^\beta  n^{1-\beta} c(s)^{-\beta}
        \eqfs
    \end{align}
    For the second term we use the one-sided Bennett inequality \cite[Theorem 2.9]{Boucheron2013}, which holds for independent summands with nonpositive mean bounded above by $b$: as $\Eof{\bar W_i}\leq\Eof{W_i} = 0$, $\bar W_i\leq b$, and $\Sigma := \sum_i\Eof{\bar W_i^2}\leq n\Eof{V(s)^2}\leq nAc(s)^2$ by \eqref{eq:tripod:var},
    \begin{equation}
        \PrOf{\sum_{i=1}^n \bar W_i > \eta}
        \leq
        \exp\brOf{-\frac{\Sigma}{b^2}h\brOf{\frac{b\eta}{\Sigma}}}
        \leq
        \brOf{1+\frac{b\eta}{\Sigma}}^{-\frac{\eta}{2b}}
        \leq
        \brOf{\frac{n}{KA}}^{-\frac K2}
        \eqcm
    \end{equation}
    where $h(u) = (1+u)\log(1+u)-u$ and we used $h(u)\geq\frac u2\log(1+u)$, which follows from $\log(1+u)\geq\frac{2u}{2+u}$, together with $\frac{b\eta}{\Sigma}\geq\frac{n}{KA}$ and $\frac{\eta}{2b} = \frac K2$. Combining both bounds, multiplying by $3$ according to \eqref{eq:tripod:crit}, and inserting $c(s)^{-\beta}\leq 3^\beta(1+s)^{-\nu\beta} = 3^\beta(1+s)^{-\lambda}$ from \eqref{eq:tripod:c} proves \eqref{eq:tripod:ub}.
    
    \emph{Step 5: A lower bound.}
    We claim that for every fixed $s\in\Rp$ there are $c>0$ and $n_0\in\N$ with
    \begin{equation}\label{eq:tripod:lb}
        \PrOf{\ol m{m_n}>s} \geq c n^{1-\beta}
        \qquad\text{for all } n\geq n_0
        \eqfs
    \end{equation}
    Let $E_i := \cb{V_i(s) > 2nc(s)}$ and $F_i := \cb{\sum_{k\neq i}V_k(s) > -2nc(s)}$; note that $E_i$ and $F_i$ are independent and that $E_i\cap F_i\subset\cb{G_{n,1}(s)>0}$. Since $R\geq1$, for $n$ large enough
    \begin{equation}
        p_n
        :=
        \PrOf{E_1}
        =
        \frac13\PrOf{R > s+\br{2nc(s)}^{\frac1\nu}}
        \geq
        \frac13\brOf{2\br{2nc(s)}^{\frac1\nu}}^{-\lambda}
        =
        \frac{2^{-\lambda}}{3}\br{2nc(s)}^{-\beta}
        \eqfs
    \end{equation}
    Conversely, as $2nc(s)\geq\frac{2n}3\geq1$ by \eqref{eq:tripod:c}, we also have $p_n\leq\frac13\PrOf{R^\nu > 2nc(s)}\leq\frac13\br{2nc(s)}^{-\beta}$, so that $np_n = \mo O\br{n^{1-\beta}}\to0$ by \eqref{eq:tripod:beta}.
    By \eqref{eq:tripod:var} and Chebyshev's inequality,
    $\PrOf{F_i\compl} \leq \frac{(n-1)\Eof{V(s)^2}}{\br{(n+1)c(s)}^2}\leq \frac An$,
    so $\PrOf{F_i}\geq\frac12$ for $n$ large. Hence, for $n$ large enough that additionally $np_n\leq\frac12$, Bonferroni's inequality gives
    \begin{align}
        \PrOf{G_{n,1}(s)>0}
        &\geq
        \PrOf{\bigcup_{i=1}^n \br{E_i\cap F_i}}
        \\&\geq
        \sum_{i=1}^n\PrOf{E_i}\PrOf{F_i} - \sum_{i<k}\PrOf{E_i}\PrOf{E_k}
        \\&\geq
        \frac n2 p_n - \frac{\br{np_n}^2}{2}
        \\&\geq
        \frac n4 p_n
        \eqcm
    \end{align}
    which is of order $n^{1-\beta}$. With \eqref{eq:tripod:crit} this proves \eqref{eq:tripod:lb}.
    
    \emph{Step 6: Proof of (ii).}
    Applying \eqref{eq:tripod:ub} with $s=0$ and any $K\geq2(\beta-1)$, and \eqref{eq:tripod:lb} with $s=0$, we obtain
    $\PrOf{m_n\neq m} = \Theta\br{n^{1-\beta}}$, which is the first assertion.
    
    Turning to the moment bound, note first that on the event $\cb{\max_{i\leq n}R_i < s}$ every summand in \eqref{eq:tripod:G} is negative, so $G_{n,j}(s)<0$ for all $j$ and \eqref{eq:tripod:crit} yields $\ol m{m_n}\leq s$. Hence, as $\PrOf{R\geq s} = s^{-\lambda}$ for $s\geq1$,
    \begin{equation}\label{eq:tripod:crude}
        \PrOf{\ol m{m_n}>s}
        \leq
        \PrOf{\max_{i\leq n}R_i\geq s}
        \leq
        n s^{-\lambda}
        \eqcm\qquad
        s\geq1
        \eqfs
    \end{equation}
    Choose $K \geq \frac{2(\beta-1+\alpha/\lambda)}{1-\alpha/\lambda}$, let $C$ be as in \eqref{eq:tripod:ub}, and set $s_n := n^{\frac{K/2+1}{\lambda}}$. Splitting the integral at $s_n$ and using \eqref{eq:tripod:ub} below and \eqref{eq:tripod:crude} above $s_n$,
    \begin{align}
        \EOf{\ol m{m_n}^\alpha}
        &=
        \int_0^\infty \alpha s^{\alpha-1}\PrOf{\ol m{m_n}>s}\dl s
        \\&\leq
        C n^{1-\beta}\int_0^\infty \alpha s^{\alpha-1}\br{1+s}^{-\lambda}\dl s
        + C n^{-\frac K2}s_n^\alpha
        + \frac{\alpha n}{\lambda-\alpha}s_n^{\alpha-\lambda}
        \eqfs
    \end{align}
    The first integral is finite because $\lambda>\alpha$. The last two terms are both of order $n^{-\frac K2\br{1-\frac\alpha\lambda}+\frac\alpha\lambda}$, which is at most $n^{1-\beta}$ by the choice of $K$. Hence $\EOf{\ol m{m_n}^\alpha} = \mo O\br{n^{1-\beta}}$. Conversely, \eqref{eq:tripod:lb} with $s=1$ gives
    $\EOf{\ol m{m_n}^\alpha} \geq \PrOf{\ol m{m_n}>1} \geq c  n^{1-\beta}$.
    Thus $\EOf{\ol m{m_n}^\alpha} = \Theta\br{n^{1-\beta}}$, and $1-\beta<1-\frac{\alpha}{\alpha-1} = -\frac1{\alpha-1}$ by \eqref{eq:tripod:beta}.
\end{proof}